\documentclass[12pt,PhD]{muthesis}
\pdfoutput=1

\usepackage{amssymb}
\usepackage{amsmath}
\usepackage{amsthm}
\usepackage[pdftex]{graphicx} 
\usepackage{tikz}
\usetikzlibrary{matrix,calc}
\usepackage{array}
\usepackage{multirow}
\usepackage{enumerate}
\usepackage{arydshln}
\usepackage{url}
\usepackage{mathrsfs}
\usepackage{nicefrac}

\usepackage{makeidx}

\makeindex

\newtheorem{theorem}{Theorem}[section]
\newtheorem{lemma}[theorem]{Lemma}
\newtheorem{corollary}[theorem]{Corollary}
\newtheorem{proposition}[theorem]{Proposition}
\newtheorem{conjecture}[theorem]{Conjecture}
\theoremstyle{definition}
\newtheorem{example}[theorem]{Example}
\newtheorem{remark}[theorem]{Remark}
\newtheorem{definition}[theorem]{Definition}
\newtheorem{question}[theorem]{Question}
\newtheorem{questions}[theorem]{Questions}
\newtheorem{conditions}[theorem]{Conditions}
\newtheorem{notation}[theorem]{Notation}
\newtheorem{ImpNot}{Important Global Convention}

\numberwithin{equation}{section}

\newcommand{\BB}{\mathfrak{B}}
\newcommand{\CC}{\mathcal{C}}
\newcommand{\DD}{\mathfrak{D}}
\newcommand{\RR}{\mathfrak{R}}

\newcommand{\SLZ}{SL_2(\mathbb{Z})}

\newcommand{\pa}{\theta_1}
\newcommand{\pb}{\theta_2}
\newcommand{\pc}{\theta_3}
\newcommand{\expq}[1]{exp_{\hat{q}}\left[#1\right]}

\newcommand{\bigslant}[2]{{\raisebox{.2em}{$#1$}\left/\raisebox{-.2em}{$#2$}\right.}}

\newcommand{\qr}[2]{\nicefrac{\mbox{\normalsize$#1$}}{\mbox{\normalsize$#2$}}}

\newcommand{\nom}[1]{\index{#1}}

\newcommand{\blurb}[2]{\begin{quote}#1 \begin{flushright}\textit{#2}\end{flushright}\end{quote}}

\begin{document}

\title{The $\lowercase{q}$-Division Ring, Quantum Matrices and Semi-classical Limits}
\author{Si\^an Fryer}
\school{Mathematics}
\faculty{Engineering and Physical Sciences}
\def\wordcount{33619}




\beforeabstract
Let $k$ be a field of characteristic zero and $q \in k^{\times}$ not a root of unity.  We may obtain non-commutative counterparts of various commutative algebras by twisting the multiplication using the scalar $q$: one example of this is the \textit{quantum plane} $k_q[x,y]$, which can be viewed informally as the set of polynomials in two variables subject to the relation $xy=qyx$.  We may also consider the full localization of $k_q[x,y]$, which we denote by $k_q(x,y)$ or $D$ and view as the non-commutative analogue of $k(x,y)$, and also the quantization $\mathcal{O}_q(M_n)$ of the coordinate ring of $n\times n$ matrices over $k$.

Our aim in this thesis will be to use the language of deformation-quantization to understand the quantized algebras by looking at certain properties of the commutative ones, and conversely to obtain results about the commutative algebras (upon which a Poisson structure is induced) using existing results for the non-commutative ones.

The $q$-division ring $k_q(x,y)$ is of particular interest to us, being one of the easiest infinite-dimensional division rings to define over $k$.  Very little is known about such rings: in particular, it is not known whether its fixed ring under a finite group of automorphisms should always be isomorphic to another $q$-division ring (possibly for a different value of $q$) nor whether the left and right indexes of a subring $E \subset D$ should always coincide.  

We define an action of $SL_2(\mathbb{Z})$ by $k$-algebra automorphisms on $D$ and show that the fixed ring of $D$ under any finite group of such automorphisms is isomorphic to $D$.  We also show that $D$ is a deformation of the commutative field $k(x,y)$ with respect to the Poisson bracket $\{y,x\} = yx$ and that for any finite subgroup $G$ of $SL_2(\mathbb{Z})$ the fixed ring $D^G$ is in turn a deformation of $k(x,y)^G$.  Finally, we describe the Poisson structure of the fixed rings $k(x,y)^G$, thus answering the Poisson-Noether question in this case.

A number of interesting results can be obtained as a consequence of this: in particular, we are able to answer several open questions posed by Artamonov and Cohn concerning the structure of the automorphism group $Aut(D)$.  They ask whether it is possible to define a conjugation automorphism by an element $z \in L \backslash D$, where $L$ is a certain overring of $D$, and whether $D$ admits any endomorphisms which are not bijective.  We answer both questions in the affirmative, and show that up to a change of variables these endomorphisms can be represented as non-bijective conjugation maps.

We also consider Poisson-prime and Poisson-primitive ideals of the coordinate rings $\mathcal{O}(GL_3)$ and $\mathcal{O}(SL_3)$, where the Poisson bracket is induced from the non-commutative multiplication on $\mathcal{O}_q(GL_3)$ and $\mathcal{O}_q(SL_3)$ via deformation theory.  This relates to one case of a conjecture made by Goodearl, who predicted that there should be a homeomorphism between the primitive (resp. prime) ideals of certain quantum algebras and the Poisson-primitive (resp. Poisson-prime) ideals of their semi-classical limits.  We prove that there is a natural bijection from the Poisson-primitive ideals of these rings to the primitive ideals of $\mathcal{O}_q(GL_3)$ and $\mathcal{O}_q(SL_3)$, thus laying the groundwork for verifying this conjecture in these cases.

\afterabstract

\prefacesection{Acknowledgements}
\setlength{\parindent}{0pt}
\setlength{\parskip}{1.5ex plus 0.5ex minus 0.2ex}
Thanks are naturally due to far more people than I could list or remember -- the postgrads, academic staff and support staff in the department have all contributed in various ways and your support and help has been appreciated.  Thanks are due in particular to EPSRC for the financial support.

This thesis would have been far more boring, confusing and generally incorrect without the advice and input of my supervisor, Toby Stafford, who has been unfailingly generous with his time and ideas, and patiently read and corrected early versions of my work until I learnt to write at least somewhat clearly.  Thank you!

To my parents, who always knew I'd end up here and provided unconditional support instead of trying to encourage me towards a ``real'' job -- this is for you.

Credit is due to Tom Withers and Lyndsey Clark for suggesting the alternative thesis titles ``Does Not Commute: The $q$-Division Ring'' and ``Commute Like You Give a Damn'' respectively.

For my lovely partner Andrew, who patiently listened to endless descriptions of which part of maths wasn't working right that day and was always there with sympathy and hugs -- you make everything better.

And of course, none of this would have been possible without the support and cuddles of the wonderful Stumpy, mightiest of stegosaurs.

\afterpreface

\newcommand{\HH}{\mathcal{H}}
\newcommand{\HP}{\mathcal{H}'}
\newcommand{\GL}[1]{\mathcal{O}(GL_{#1})}
\newcommand{\SL}[1]{\mathcal{O}(SL_{#1})}
\newcommand{\ML}[1]{\mathcal{O}(M_{#1})}
\newcommand{\QGL}[1]{\mathcal{O}_q(GL_{#1})}
\newcommand{\QSL}[1]{\mathcal{O}_q(SL_{#1})}
\newcommand{\QML}[1]{\mathcal{O}_q(M_{#1})}
\newcommand{\wt}[1]{\widetilde{#1}}
\newcommand{\ol}[1]{\overline{#1}}

\newcommand{\Hspec}[1]{\HH$-$spec(#1)}
\newcommand{\HPspec}[1]{\HP$-$spec(#1)}
\newcommand{\PHspec}[1]{\HH$-$Pspec(#1)}
\newcommand{\PHPspec}[1]{\HP$-$Pspec(#1)}

\newenvironment{smallarray}[1]
  {\gdef\MatName{#1}\begin{tikzpicture}
    \matrix[
    matrix of math nodes,
    row sep=-\pgflinewidth,
    column sep=-\pgflinewidth,
    nodes={inner sep=1pt,rectangle,text width=2mm,align=center},
    text depth=0mm,
    text height=2mm,
    nodes in empty cells
    ]  (#1)}
  {\end{tikzpicture}}
\def\MyZ(#1,#2){%
  \draw ([xshift=-3.6pt,yshift=3.6pt] $ (\MatName-#1-#2)!0.5!(\MatName-\the\numexpr#1+1\relax-\the\numexpr#2+1\relax) $ ) rectangle ([xshift=3.6pt,yshift=-3.6pt] $ (\MatName-#1-#2)!0.5!(\MatName-\the\numexpr#1+1\relax-\the\numexpr#2+1\relax) $ );
}

\chapter{Introduction}\label{c:introduction}
\section{Overview}\label{s:introduction overview}

In this thesis we examine the similarities between certain commutative and non-commutative algebras, with a focus on using the properties of one algebra to understand the structure of the other.  We also focus in detail on one specific non-commutative division ring, describing in detail some of its subrings and proving some striking results concerning its automorphism and endomorphism groups.

Throughout, we will assume that $k$ is a field of characteristic zero and $q \in k^{\times}$ is not a root of unity.  This thesis divides broadly into two parts, one considering the $q$-division ring $k_q(x,y)$ and the other concerning coordinate rings of matrices and their quantum analoges.  We will now describe each in turn.

We define the \textit{quantum plane} $k_q[x,y]$\nom{K@$k_q[x,y]$} as a quotient of the free algebra in two variables, namely $k_q[x,y] = k\langle x,y \rangle /(xy-qyx)$.  This is a Noetherian domain for all non-zero $q$, and hence by \cite[Chapter 6]{GW1} it has a division ring of fractions.  We denote this ring by $k_q(x,y)$\nom{K@$k_q(x,y)$} or $D$\nom{D@$D$}, and call it the \textit{$q$-division ring}\nom{Q@$q$-division ring}.  

When $q$ is not a root of unity, the centre of $D$ is trivial (see, for example, \cite[Exercise 6J]{GW1}) and hence $D$ is a division ring which is infinite-dimensional over its centre.  Very little is known about division rings of this type: for example, if $E$ is a non-commutative sub-division ring of $D$, it is known that $D$ must have finite index over $E$ on both the left and the right \cite[Theorem~34]{S1}, but it is not known if the two indexes must always be equal.  Similarly, if $G$ is a finite group of automorphisms of $D$, must its fixed ring $D^G = \{r \in D : g(r) = r \ \forall g \in G\}$ always be another $q$-division ring (possibly for a different value of $q$)?

One of the key motivations in studying division rings such as $D$ is a conjecture made by Artin concerning the classification of surfaces in non-commutative algebraic geometry.  In \cite{Artin}, Artin conjectured that all the non-commutative surfaces had already been described (up to birational equivalence, i.e. up to isomorphism of their function fields); nearly twenty years later, this conjecture still remains open.  Restated in terms of division rings, this says (informally) that the only division rings appearing as function fields of non-commutative surfaces must have one of the following forms:
\begin{itemize}
\item division rings of algebras finite-dimensional over function fields of transcendence degree 2; 
\item division rings of Ore extensions of function fields of curves; 
\item the degree 0 part of the graded division ring of the 3-dimensional Sklyanin algebra (defined in \cite[Example~8.3]{SV1}).
\end{itemize}
For a more precise statement of Artin's conjecture, including the definition of a non-commutative surface and its function field (which we do not define here as it will not be used in this thesis) see \cite{Artin,SV1}.

From a purely ring-theoretic point of view, one way of approaching Artin's conjecture is to examine the subrings of finite index within the division rings appearing on this list, as such rings must also fit into the framework of the conjecture.

For an arbitrary division ring $L$ and a finite subgroup $G$ of $Aut(L)$, a non-commutative version of Artin's lemma \cite[\S5.2.1]{dumas_invariants} states that the index of the fixed ring $L^G$ inside $L$ must satisfy the inequality $[L:L^G] \leq |G|$.  In particular, since the $q$-division ring $D$ is a division ring of an Ore extension of the function field $k(y)$ and hence one of the rings appearing in Artin's conjectured list above, its fixed rings under finite groups of automorphisms are of interest to us.

Chapter~\ref{c:fixed_rings_chapter} proves a number of results concerning the structure of various fixed rings of the $q$-division ring, and throws doubt on the idea that the automorphism and endomorphism groups of $D$ might be well-behaved by constructing several examples of counter-intuitive conjugation maps. (Apart from minor modifications, this chapter has also appeared in the Journal of Algebra as \textit{The $q$-Division Ring and its Fixed Rings} \cite{ME1}.)  In Chapter~\ref{c:deformation_chapter} we describe progress towards an alternative method for understanding fixed rings of $D$, via Poisson deformation of the function field $k(x,y)$.  
 
We will also consider prime and primitive ideals in quantum matrices and their commutative semi-classical limits, which at first glance seems completely unrelated to questions concerning the fixed rings of infinite-dimensional division algebras.  However, we will see that both topics can be studied by viewing the non-commutative algebras as deformations of certain commutative Poisson algebras, and in both cases we will be interested in moving from the non-commutative structure to the commutative one and back in order to better understand the properties of both.  

In particular, one of the main tools in understanding prime and primitive ideals in quantum algebras is the \textit{stratification theory} due to Goodearl and Letzter, which is described in detail in \cite{GBbook} and provides tools for describing the prime and primitive ideals of our algebra in terms of certain localizations.  In this thesis our aim will be to develop a commutative Poisson version of the results in \cite{GL1}, which explicitly describes the primitive ideals of quantum $GL_3$ and $SL_3$, and hence prove that there is a natural bijection between the two sets.

This is one small part of a larger conjecture, which in the case of quantum algebras and their semi-classical limits was stated by Goodearl in \cite{GoodearlSummary}.  We describe this in more detail in \S\ref{ss:H primes examples}, but informally stated the conjecture predicts the existence of a homeomorphism between the prime ideals of a quantum algebra and the Poisson-prime ideals of its semi-classical limit.  By \cite[Lemma~9.4]{GoodearlSummary}, this is equivalent to the existence of a bijection $\Phi$ between the two sets such that both $\Phi$ and $\Phi^{-1}$ preserve inclusions.  Approaching the question by direct computation of small-dimensional examples has been successful for e.g. $SL_2$ and $GL_2$, and so extending this analysis to $SL_3$ and $GL_3$ is a natural next step.

In Chapter~\ref{c:H-primes} we consider the relationship between the ideal structure of the quantum algebras $\QGL{3}$ and $\QSL{3}$ and Poisson ideal structure of their commutative counterparts $\GL{3}$ and $\SL{3}$, which we view as Poisson algebras for an appropriate choice of Poisson bracket.   By explicitly describing generators for the Poisson-primitive ideals of $\GL{3}$ and $\SL{3}$ and combining this with results of \cite{GL1}, we prove that there is a natural bijection these two sets.  We hope that in future work this can be extended to verify Goodearl's conjecture in these cases.

Finally, in Appendix~\ref{c:appendix} we provide the code we have used for computations in the computer algebra system Magma, which allows us to perform computations in the $q$-division ring by embedding it into a larger ring of non-commutative power series and to verify that certain ideals are prime in the commutative algebra $\ML{3}$.  In Appendix~\ref{c:H-prime figures} we collect together several figures relating to the $\HH$-prime computations in Chapter~\ref{c:H-primes}.

\section{Notation}\label{s:notation}
In this section we outline the notation and definitions we will need.  

\begin{ImpNot}\label{q is not a root of unity for now and forever}
Throughout, fix $k$\nom{K@$k$} to be a field of characteristic zero and $q \in k^{\times}$\nom{Q@$q$} not a root of unity, that is $q^n \neq 1$ for all $n\geq 1$.
\end{ImpNot}
In \S\ref{c:H-primes} we will further restrict our attention to the case where $k$ is algebraically closed.

Let $R$ be any ring, $\alpha$ an endomorphism of $R$ and $\delta$ a left $\alpha$-derivation.  The (left) \textit{Ore extension} $R[x;\alpha, \delta]$\nom{R@$R[x;\alpha,\delta]$} is an overring of $R$, which is free as a left $R$-module with basis $\{1,x,x^2, \dots\}$ and commutation relation
\[xr = \alpha(r)x + \delta(r).\]
We write $R[x; \alpha]$ or $R[x;\delta]$ when $\delta = 0$ or $\alpha = 1$ respectively.

The ring $k_q[x,y]$ can be viewed as the Ore extension $k[y][x;\alpha]$, where $\alpha$ is the automorphism defined on $k[y]$ by $\alpha(y) = qy$.  For $r \in k_q[x,y]$, let $deg_x(r)$\nom{D@$deg_x(r)$} be the degree of $r$ as a polynomial in $x$.

We say that a multiplicative subset $S$ of a ring $R$ satisfies the \textit{right Ore condition}\nom{Ore set} if 
\begin{equation}\label{eq:right Ore condition}
\forall r \in R\textrm{ and } x \in S,\ \  \exists s \in R \textrm{ and } y \in S\  \textrm{ such that } ry = xs.
\end{equation}
If $S$ consists only of regular elements then satisfying \eqref{eq:right Ore condition} is a sufficient (and indeed necessary) condition for the existence of the localization $RS^{-1}$ \cite[Theorem~6.2]{GW1}.  A \textit{left Ore set} is defined symmetrically, and $S$ is simply called an \textit{Ore set} if it satisfies both the left and the right Ore condition.

More generally, we call a multiplicative set $S$ in $R$ a \textit{right denominator set}\nom{denominator set} if it satisfies the right Ore condition and also the \textit{right reversibility condition}:
\[\textrm{If }r \in R \textrm{ and }x \in S \textrm{ such that }xr = 0, \textrm{ then there exists }x' \in S \textrm{ such that }rx'=0.\]
This allows us to form a right ring of fractions for $R$ with respect to $S$ even if $S$ contains zero-divisors \cite[Theorem~10.3]{GW1}.  The \textit{left denominator set} is again defined symmetrically, and $S$ is a denominator set if it satisfies the denominator set conditions on both sides.  We note that in a left/right Noetherian ring, the left/right Ore condition implies the left/right reversibility condition, and hence all Ore sets are denominator sets in this case \cite[Proposition~10.7]{GW1}.

By localizing $k_q[x,y]$ at the set of all its monomials, which is clearly both left and right Ore since monomials are normal in $k_q[x,y]$, we obtain the ring of \textit{quantum Laurent polynomials} $k_q[x^{\pm1},y^{\pm1}]$\nom{K@$k_q[x^{\pm1},y^{\pm1}]$}.  This ring sits strictly between $k_q[x,y]$ and the division ring $k_q(x,y)$, and the properties of it and its fixed rings are studied in \cite{Baudry}.  

The \textit{$q$-division ring} $D = k_q(x,y)$ embeds naturally into a larger division ring, namely the ring of Laurent power series\nom{K@$k_q(y)(\!(x)\!)$}
\begin{equation}\label{eq:def of Laurent power series ring,notation section}k_q(y)(\!(x)\!) = \left\{ \sum_{i \geq n} a_i x^i : n \in \mathbb{Z}, a_i \in k(y)\right\}\end{equation}
subject to the same relation $xy = qyx$.  It is often easier to do computations in $k_q(y)(\!(x)\!)$ than in $D$, and we will identify elements of $D$ with their image in $k_q(y)(\!(x)\!)$ without comment.

We will also need a generalization of the quantum plane, namely the \textit{uniparameter quantum affine space} $k_{\mathbf{q}}[x_1, \dots, x_n]$\nom{K@$k_{\mathbf{q}}[x_1, \dots, x_n]$}.  Here $\mathbf{q}$ is an additively anti-symmetric $n \times n$ matrix, and the relations are given by
\[x_ix_j = q^{a_{ij}}x_jx_i,\]
where $a_{ij}$ denotes the $(i,j)$th entry of $\mathbf{q}$.

The \textit{coordinate ring of the $2 \times 2$ matrices} over a field $k$ is simply the polynomial ring in four variables, that is $\ML{2} = k[x_{11},x_{12},x_{21},x_{22}]$.  The quantized version of this algebra is defined in \cite[Example~I.1.6]{GBbook} to be the quotient of the free algebra $k\langle X_{11},X_{12},X_{21},X_{22} \rangle$ by the six relations
\begin{equation}\label{eq:relns for 2x2 matrices notation section}\begin{aligned}
X_{11}X_{12} &- qX_{12}X_{11},\quad & X_{12}X_{22} &- qX_{22}X_{12}, \\
X_{11}X_{21} &- qX_{21}X_{11},\quad & X_{21}X_{22} &- qX_{22}X_{21}, \\
X_{12}X_{21}&-X_{21}X_{12},\quad & X_{11}X_{22}-X_{22}X_{11}&-(q-q^{-1})X_{12}X_{21};
\end{aligned}\
\end{equation}
for some $q \in k^{\times}$.  This algebra is denoted by $\QML{2}$\nom{O@$\mathcal{O}_q(M_2)$}.  From this construction we may obtain the \textit{quantum $m \times n$ matrices} $\mathcal{O}_q(M_{m \times n})$\nom{O@$\mathcal{O}_q(M_{m\times n})$} as the algebra in $mn$ variables $\{X_{ij}: 1 \leq i \leq m, 1 \leq j \leq n\}$, subject to the condition that any set of four variables $\{X_{ij}, X_{im}, X_{lj}, X_{lm}\}$ with $i<l$ and $j<m$ should satisfy the relations \eqref{eq:relns for 2x2 matrices notation section}.  

When $m=n$, we write $\QML{n}$ for $\mathcal{O}_q(M_{n \times n})$ and define the \textit{quantum determinant}\nom{D@$Det_q$}\nom{Quantum determinant} to be
\begin{equation}\label{eq:quantum det def}
Det_q = \sum_{\pi \in S_n}(-q)^{l(\pi)}X_{1,\pi(1)}X_{2,\pi(2)}\dots X_{n,\pi(n)},
\end{equation}
where $S_n$ is the symmetric group on $n$ elements and $l(\pi)$ denotes the length of the permutation $\pi \in S_n$.  The quantum determinant is central in $\QML{n}$ (see, for example, \cite[I.2.4]{GBbook}), and hence the set $\{1, Det_q, Det_q^2, \dots\}$ is an Ore set in $\QML{n}$ and $\langle Det_q -1\rangle$ defines an ideal in $\QML{n}$.  We therefore define \textit{quantum $GL_n$} and \textit{quantum $SL_n$} as follows\nom{O@$\mathcal{O}_q(GL_n)$}\nom{O@$\mathcal{O}_q(SL_n)$}:
\[\QGL{n} := \QML{n}[Det_q^{-1}], \qquad \QSL{n} := \QML{n}/\langle Det_q - 1 \rangle.\]
We may also generalise the definition of quantum determinant to obtain a notion of minors in $\QML{n}$. Using the notation of \cite{GL1}, if $I$ and $J$ are subsets of $\{1, \dots, n\}$ of equal cardinality then we define the \textit{quantum minor} $[I|J]_q$\nom{$[I"|J]_q$} to be the quantum determinant in the subalgebra of $\QML{n}$ generated by $\{X_{ij} : i \in I, j \in J\}$.  We will use the same notation for quantum minors of $\QGL{n}$ and $\QSL{n}$, where we simply mean the image of $[I|J]_q$ in the appropriate algebra.

Since we will work primarily with $3\times 3$ quantum matrices, we will often drop the set brackets in our notation and, for example, write $[12|13]_q$ for the minor $[\{1,2\}|\{1,3\}]_q$.  Similarly, let $\wt{I}$ denote the set $\{1, \dots, n\}\backslash I$, and so that $[12|13]_q$ may also be denoted by $[\wt{3}|\wt{2}]_q$.  Since $1 \times 1$ minors are simply the generators $X_{ij}$, we will use the notation $[i|j]_q$ and $X_{ij}$ interchangeably.

If $R$ is any ring and $z$ an invertible element, we denote the resulting conjugation map on $R$ by\nom{C@$c_z$ (conjugation)}
\[c_z: r \mapsto zrz^{-1} \quad \forall r \in R.\]
Since we will define conjugation maps on $D$ with $z \in k_q(y)(\!(x)\!) \backslash D$, the following distinction will be important: we call a conjugation map $c_z$ an \textit{inner automorphism}\nom{inner automorphism} of $R$ if $z, z^{-1} \in R$.

Meanwhile, if $G$ is a subgroup of $Aut(R)$ we define the fixed ring to be\nom{R@$R^G$}
\[R^G = \{r \in R: g(r) = r, \ \forall g \in G\}.\]
If $G = \langle \varphi \rangle$ is cyclic, we will also denote the fixed ring by $R^{\varphi}$.

Let $spec(R)$\nom{S@$spec(R)$} be the set of prime ideals in a ring $R$, and $prim(R)$\nom{P@$prim(R)$} the set of primitive ideals.  The \textit{Zariski topology} is defined on $spec(R)$ by defining the closed sets to be those of the form
\[V(I) = \{P \in spec(R) : P \supseteq I\}\]
for some ideal $I$ of $R$.  This induces a topology on $prim(R)$, where the closed sets are simply those of the form $V(I) \cap prim(R)$ for some closed set $V(I)$ in $spec(R)$.  

\section{Results on the structure of the $q$-division ring}

As described in \S\ref{s:introduction overview}, we are interested in understanding the structure of fixed rings of division rings.  In particular, we will focus on the $q$-division ring $D = k_q(x,y)$ and its fixed rings under finite groups of automorphisms, with a view to establishing whether and how these fixed rings fit into the list predicted by Artin's conjecture.

We first describe the existing results along these lines. In the simplest case, where the group of automorphisms restricts to automorphisms of the quantum plane, we have a full description of the fixed rings $D^G$ given by the following theorem.
\begin{theorem}\label{res:AD_prop_intro}\cite[Proposition~3.4]{AD1} Let $k$ be a field of characteristic zero, and $q \in k^{\times}$ not a root of unity.  Denote by $R_q$ the quantum plane $k_q[x,y]$ and by $D_q$ its full ring of fractions.  Then:
\begin{enumerate}[(i)]
\item For $q' \in k$, we have $D_q \cong D_{q'}$ if and only if $q' = q^{\pm1}$, if and only if $R_q \cong R_{q'}$.
\item For all finite subgroups $G$ of $Aut(R_q)$, $D_q^G \cong D_{q'}$ for $q' = q^{|G|}$.
\end{enumerate}
\end{theorem}
However, $D$ admits many other automorphisms of finite order which are not covered by this theorem.  The following theorem by Stafford and Van den Bergh considers one such example:
\begin{theorem}\cite[\S13.6]{SV1}\label{res:order_2_monomial_result}
Let $\tau$ be the automorphism defined on $D$ by
\[\tau: x \mapsto x^{-1}, \ y \mapsto y^{-1}\]
Then the fixed ring $D^{\tau}$ is isomorphic to $D$ as $k$-algebras.
\end{theorem}
The map $\tau$ is an example of a monomial automorphism, i.e. one where the images of $x$ and $y$ are both monomials (up to scalars).  In contrast to Theorem~\ref{res:AD_prop_intro}, the value of $q$ in $D^{\tau}$ does not depend on the order of $\tau$.  In a private communication to Stafford, Van den Bergh posed the question of whether the same result holds for the order 3 automorphism $\sigma: x \mapsto y, \ y \mapsto (xy)^{-1}$; we answer this in \S\ref{s:more fixed rings} as part of the following general theorem:
\begin{theorem}[Theorem~\ref{res:epic_theorem}, Theorem~\ref{res:thm_monomial_results_2}]\label{res:thm_monomial_results} Let $k$ be a field of characteristic zero and $q \in k^{\times}$.
\begin{enumerate}[(i)]
\item Define an automorphism of $D$ by
\[\varphi: x \mapsto (y^{-1}-q^{-1}y)x^{-1}\quad y \mapsto -y^{-1}\]
and let $G$ be the group generated by $\varphi$.  Then $D^G \cong D$ as $k$-algebras.
\item Suppose $k$ contains a third root of unity $\omega$, and both a second and third root of $q$.  If $G$ is a finite group of monomial automorphisms of $D$ then $D^G \cong D$ as $k$-algebras.
\end{enumerate}
\end{theorem}
This suggests that for finite groups of automorphisms which do not restrict to $k_q[x,y]$, we should expect the fixed ring to again be $q$-division for the same value of $q$.  There are several difficulties standing in the way of proving a general theorem of this form, however: in particular, the full automorphism group of $D$ is not yet fully understood, and the methods used in the proof of Theorem~\ref{res:thm_monomial_results} involve direct computation with elements of $D$ and do not easily generalise to automorphisms of large order.

In \S\ref{ss:autos of division ring background section} we describe what is currently known about $Aut(D)$, based on work by Alev and Dumas in \cite{AD3} and Artamonov and Cohn in \cite{AC1}.  In \S\ref{s:automorphism_consequences} we demonstrate why the structure of this group remains mysterious, by proving the following counter-intuitive result:
\begin{theorem}[Theorem~\ref{res:thm_AC_results_2}]\label{res:thm_AC_results} Let $k$ be a field of characteristic zero and $q \in k^{\times}$ not a root of unity.  Then:
\begin{enumerate}[(i)]
\item The $q$-division ring $D$ admits examples of bijective conjugation maps by elements $z \in k_q(y)(\!(x)\!)\backslash D$; these include examples satisfying $z^n \in D$ for some positive $n$, and also those such that $z^n \not\in D$ for all $n \geq 1$.
\item $D$ also admits an endomorphism which is not an automorphism, which can be represented in the form of a conjugation map.  
\end{enumerate}
\end{theorem}
Both parts of Theorem~\ref{res:thm_AC_results} illustrate different problems with understanding the automorphism group $Aut(D)$.  Part $(i)$ means that we must distinguish between the concepts of ``bijective conjugation map'' and ``inner automorphism'' when considering automorphisms of $D$, and raises the possibility that these non-inner conjugation automorphisms may be examples of \textit{wild automorphisms} (see \S\ref{s:background artamonov and cohn}).  Meanwhile, part $(ii)$ of the theorem flies in the face of our most basic intuitions concerning conjugation maps, and also allows us to construct interesting new division rings such as the following: if $c_z$ is a conjugation map predicted by Theorem~\ref{res:thm_AC_results} $(ii)$, then we can consider the limits
\[\bigcap_{i \geq 0}z^iDz^{-i} \textrm{\quad and \quad} \bigcup_{i \geq 0}z^{-i}Dz^i,\]
about which very little is currently known.

\subsection{Methods for computation in the $q$-division ring}
As noted above, one of the reasons that so many apparently-simple questions concerning $D$ remain open is that direct computation in non-commutative division rings is extremely difficult.  By \cite[Corollary~6.7]{GW1}, if $R$ is a right Noetherian domain then the set $S:=R\backslash\{0\}$ forms a right Ore set, i.e. satisfies the \textit{right Ore condition} defined in \eqref{eq:right Ore condition}.

This condition is what makes the addition and multiplication well-defined in the localization $RS^{-1}$: for example, when computing the product of two fractions $ab^{-1}cd^{-1}$, the Ore condition \eqref{eq:right Ore condition} guarantees the existence of $u \in R$, $v \in S$ such that $b^{-1}c = uv^{-1}$ and hence
\[ab^{-1}cd^{-1} = au(dv)^{-1} \ \in RS^{-1}.\]
The problem is that this is not a constructive result, and in practice finding the values of $u$ and $v$ is often all but impossible.  In order to get around this problem, we embed $D$ into a larger division ring, namely the ring of Laurent power series
\begin{equation}\label{eq:def of Laurent power series ring, intro}
k_q(y)(\!(x)\!) = \left\{\sum_{i \geq n} a_i x^i : a_i \in k(y), \ n \in \mathbb{Z}\right\}
\end{equation}
where $x$ and $y$ are subject to the same relation $xy = qyx$.  Addition and multiplication in this ring can be computed term-by-term, where each step involves only monomials in $x$ (see Appendix~\ref{c:appendix} for further details on this).

Computing in $k_q(y)(\!(x)\!)$ can therefore be reduced to computation of the coefficients for each power of $x$, and in Appendix~\ref{c:appendix} we provide the code used to implement this approach in the computer algebra system Magma.  We also prove several results which allow us to pull the answers of our computations back to elements in $D$, which we state next.
\begin{theorem}\label{res:rec reln thm, intro version}(Theorem~\ref{res:rec reln thm}) Let $K$ be a field, $\alpha$ an automorphism on $K$ and $K[x;\alpha]$ the Ore extension of $K$ by $\alpha$.  Denote by $K(x;\alpha)$ the division ring of $K[x;\alpha]$ and $K[\![x;\alpha]\!]$ the power series ring into which $K[x;\alpha]$ embeds.

The power series $\sum_{i \geq 0} a_ix^i \in K[\![x; \alpha]\!]$ represents a rational function $Q^{-1}P$ in $K(x;\alpha)$ if and only if there exists some integer $n$, and some constants $c_1, \dots, c_n \in K$ (of which some could be zero) such that for all $i \geq 0$ the coefficients of the power series satisfy the linear recurrence relation
\begin{equation*}a_{i+n} = c_1\alpha(a_{i+(n-1)}) + c_2\alpha^2(a_{i+(n-2)}) + \dots + c_n\alpha^n(a_i).\end{equation*}
If this is the case, then $P$ is a polynomial of degree $\leq n-1$ and $Q = 1 - \sum_{i=1}^n c_ix^i$.\end{theorem}
\begin{theorem}\label{res:rec reln det thm, intro version}(Theorem~\ref{res:rec reln det thm}) Keep the same notation as Theorem~\ref{res:rec reln thm, intro version}. A power series $\sum_{i \geq0}a_ix^i$ satisfies a linear recurrence relation
\[a_{i+n} = c_1\alpha(a_{i+(n-1)}) + c_2\alpha^2(a_{i+(n-2)}) + \dots + c_k\alpha^k(a_i)\]
if and only if there exists some $m \geq 1$ such that the determinants of the matrices
\begin{equation*}
 \Delta_k = \left[ \begin{array}{ccccc}
         \alpha^k(a_0) & \alpha^{k-1}(a_{1}) & \dots & \alpha(a_{k-1}) & a_{k} \\
	 \alpha^k(a_{1}) & \alpha^{k-1}(a_{2}) & \dots &\alpha(a_{k}) & a_{k+1} \\
	 \vdots & &\ddots & &\vdots\\
	 \alpha^k(a_{k}) & \alpha^{k+1}(a_{k+1}) & \dots& \alpha(a_{2k-1}) & a_{2k}
        \end{array} \right]
\end{equation*}
are zero for all $k \geq m$.\end{theorem}
Since we cannot in practice compute infinitely many terms of a series or infinitely many determinants of matrices, these results only provide the tools which allow us to approximate computation in $D$ and $k_q(y)(\!(x)\!)$.  However, we may then use the intuition gained from these computations to prove results by more standard methods.

\subsection{Approaching $D$ via Poisson deformation}\label{ss:intro, talking about D as a deformation}
An alternative method of understanding $D$ while avoiding the difficulties imposed by the non-commutativity is to translate the problem to a related commutative ring where localization is better behaved, and then use deformation theory to pull the results back to $D$.  In \cite{BaudryThesis}, Baudry constructed the algebra of $q$-commuting Laurent polynomials $k_q[x^{\pm1},y^{\pm1}]$ as a deformation of the commutative algebra $k[x^{\pm1},y^{\pm1}]$, and proved that for certain finite groups of automorphisms $G$ the fixed ring $k_q[x^{\pm1},y^{\pm1}]^G$ is in turn a deformation of $k[x^{\pm1},y^{\pm1}]^G$.

In Chapter~\ref{c:deformation_chapter} we build on Baudry's result to prove the corresponding result for $D$, and describe partial results towards understanding fixed rings $D^G$ as deformations of commutative rings.  In particular, we prove the following result:
\begin{theorem}\label{res:D is def of F intro statement}(Proposition~\ref{res:F is a deformation of D}, Theorem~\ref{res:fixed ring is deformation theorem})
Let $k(x,y)$ be the field of rational functions in two commuting variables with Poisson bracket defined by $\{y,x\} = yx$, and $G$ a finite subgroup of $SL_2(\mathbb{Z})$ acting on $k(x,y)$ by Poisson monomial automorphisms and on $D$ by monomial automorphisms.  Then $D$ is a deformation of $k(x,y)$, and the fixed ring $D^G$ is a deformation of $k(x,y)^G$.
\end{theorem}
The Poisson bracket on $k(x,y)$ captures some of the non-commutative behaviour of $D$, while on the other hand its commutative multiplication makes it a far easier ring to work with.  Theorem~\ref{res:D is def of F intro statement} tells us that if we can describe the Poisson structure of $k(x,y)^G$ and the possible Poisson deformations of this structure, this will allow us to also understand the fixed ring $D^G$.

In \S\ref{s:fixed Poisson rings} we achieve the first of these for the case of finite groups of monomial automorphisms on $k(x,y)$ with respect to the Poisson bracket $\{y,x\} = yx$, by proving the following result.
\begin{theorem}\label{res:Poisson fixed rings intro statment}(Theorem~\ref{res:Poisson fixed rings big theorem})
Let $k$ be a field of characteristic zero which contains a primitive third root of unity $\omega$, and let $G$ be a finite subgroup of $SL_2(\mathbb{Z})$ which acts on $k(x,y)$ by Poisson monomial automorphisms as defined in Definition~\ref{def:action of SL2, Poisson def}.  Then there exists an isomorphism of Poisson algebras $k(x,y)^G \cong k(x,y)$.
\end{theorem}
Unfortunately we have not yet managed to describe the possible deformations of $k(x,y)^G$, which means that we cannot yet replace the results of Chapter~\ref{c:fixed_rings_chapter} with this alternative Poisson approach.  However, the proof of Theorem~\ref{res:Poisson fixed rings intro statment} is far simpler and more intuitive than the proof of Theorem~\ref{res:thm_monomial_results}, which suggests that this may be a better way to approach a general theorem concerning the structure of fixed rings $D^G$ for arbitrary finite $G$.

\section{Primitive ideals in $\GL{3}$ and $\SL{3}$}\label{s:intro H primes}

In Chapter~\ref{c:H-primes} we apply the deformation theory techniques explored in the previous chapter to a completely different setting: the quantum algebras $\QML{n}$, $\QGL{n}$ and $\QSL{n}$ defined in \S\ref{s:notation} and their commutative counterparts $\ML{n}$, $\GL{n}$ and $\SL{n}$.  Informally, by letting $q=1$ we obtain the standard coordinate rings of $M_n$, $GL_n$ and $SL_n$, but as in \S\ref{ss:intro, talking about D as a deformation} this process induces a Poisson bracket on the commutative algebra which retains a ``first order impression'' of the non-commutative multiplication.

This relationship between the non-commutative and Poisson structures seems to force the ideal structures of the two algebras to match up quite closely: in the case of $\QSL{2}$ and $\SL{2}$, for example, it is fairly easy to show directly that there is a homeomorphism from the prime ideals of $\QSL{2}$ to the Poisson-prime ideals of $\SL{2}$, and further that this restricts to a homeomorphism from primitive ideals to Poisson-primitive ideals \cite[Example~9.7]{GoodearlSummary}.  Goodearl has conjectured in \cite[Conjecture~9.1]{GoodearlSummary} that the existence of this homeomorphism should be a general phenomenon, extending not just to all algebras of quantum matrices but other types of quantum algebra as well (the precise definition of ``quantum algebra'' remains an open question; some examples and common properties of these algebras are discussed in \S\ref{ss:H action quantum version}).

Let $A$ denote a quantum algebra and $B$ its semi-classical limit, and denote the set of Poisson-primes in $B$ by $Pspec(B)$; note that for all of the algebras we are interested in,  a Poisson-prime ideal is simply a prime ideal in the usual commutative sense which is closed under the Poisson bracket.  By \cite[Lemma~9.4]{GoodearlSummary}, a bijection $\Phi: spec(A) \rightarrow Pspec(B)$ is a homeomorphism if any only if  $\Phi$ and $\Phi^{-1}$ both preserve inclusions, hence for low dimensional examples of algebras $A$ and $B$ it is a valid tactic to try to obtain generating sets for all of the (Poisson-)primes and check the inclusions directly.  The aim of computing these examples explicitly is to provide evidence in favour of the conjecture, and also to provide intuition for a more general proof (or disproof) in the future.

In \cite{GL1}, Goodearl and Lenagan give explicit generating sets for the primitive ideals of $\QGL{3}$ and $\QSL{3}$ and lay the foundations for a full description of the prime ideals.  In Chapter~\ref{c:H-primes} we make use of their results and also techniques from deformation theory to obtain the corresponding description of Poisson-primitive ideals in $\GL{3}$ and $\SL{3}$.  We obtain the following theorem:
\begin{theorem}\label{res:bijection between primitives,intro} [Corollary~\ref{res:bijection between primitives,H-prime section}]
Let $k$ be algebraically closed of characteristic 0 and $q \in k^{\times}$ not a root of unity.  Let $A$ denote $\QGL{3}$ or $\QSL{3}$, and let $B$ denote the semi-classical limit of $A$.  Then there is a bijection of sets between $prim(A)$ and $Pprim(B)$, which is induced by the ``preservation of notation map'' 
\[A \rightarrow B: \ X_{ij} \mapsto x_{ij}, \ [\wt{i}|\wt{j}]_q \mapsto [\wt{i}|\wt{j}].\]
\end{theorem}
Here $[\wt{i}|\wt{j}]_q$ denotes a quantum minor in $A$ as defined in \S\ref{s:notation}, and $[\wt{i}|\wt{j}]$ is the corresponding minor in $B$ with $q=1$.

Although we are not able to verify that this bijection is actually a homeomorphism, Theorem~\ref{res:bijection between primitives,intro} does make it extremely likely that Goodearl's conjecture is true in these cases.

Note that the statement of \cite[Lemma~9.4]{GoodearlSummary} relating homeomorphisms to bijections preserving inclusions of primes does \textit{not} restrict to the corresponding statement for primitives; in order to verify the conjecture we would therefore need to prove that the bijection in Theorem~\ref{res:bijection between primitives,intro} was a homeomorphism using other techniques, or first extend it to a bijection on \textit{prime} ideals.  With this in mind, we also prove the following result for $\SL{3}$:
\begin{proposition}\label{res:UFD Poisson SL3 quotients,intro} [Proposition~\ref{res:UFD Poisson SL3 quotients}]
For any Poisson $\HH$-prime $I_{\omega}$ in $\SL{3}$, the quotient $\SL{3}/I_{\omega}$ is a commutative UFD.
\end{proposition}
The corresponding quantum version is proved in \cite[Theorem~5.2]{GBrown}.  We hope to use these results in future work to pull back generating sets for prime ideals to $\QSL{3}$ (resp. generating sets for Poisson-primes in $\SL{3}$); currently these are only known up to certain localizations.  This would allow us to extend the bijection in Theorem~\ref{res:bijection between primitives,intro} to a bijection $spec(\QSL{3}) \rightarrow Pspec(\SL{3})$.
\chapter{Background Material}\label{c:background}

The aim of this chapter is to provide the background material upon which the following chapters are built.  We begin in \S\ref{s:background artamonov and cohn} with an introduction to the concept of tame and wild automorphism groups, and focus in particular on what is known about the automorphism groups of $q$-commuting structures related to the $q$-division ring.  In \S\ref{ss:autos of division ring background section} we outline work done by Artamonov and Cohn in \cite{AC1}, upon which our results in \S\ref{s:automorphism_consequences} concerning strange conjugation maps of $k_q(x,y)$ are based.

In \S\ref{s:background deformation} we introduce Poisson algebras and Poisson deformation, which is the tool that will allow us to move between commutative and non-commutative algebras and compare the properties of the two.  Finally, in \S\ref{s:background H primes} we introduce stratification theory and $\HH$-primes for both quantum and Poisson algebras; this is an extremely powerful theory which allows us to partition the spectrum (respectively, Poisson spectrum) of certain types of algebra into smaller, more manageable pieces and hence describe the prime and primitive (resp. Poisson-prime and Poisson-primitive) ideals of the algebra.

Recall that as per Important Global Notation~\ref{q is not a root of unity for now and forever} we assume throughout that $k$ is a field of characteristic zero and $q \in k^{\times}$ is not a root of unity.

\section{On automorphism groups: tame, wild and the $q$-division ring}\label{s:background artamonov and cohn}

One way to get a feel for the structure of an algebra is to describe its automorphism group: the set of all possible $k$-linear automorphisms that can be defined on it, which is a group under composition of maps.  One way to waste a lot of time, on the other hand, is to try to describe an automorphism group that can't be described: informally, a \textit{wild} automorphism group.

The definition of tame and wild automorphisms varies from algebra to algebra, but the common theme is as follows: the tame automorphisms should be those in the group generated by some ``natural'' or ``elementary'' set of generators, while any remaining automorphisms not covered by this description are called ``wild''.  This concept is best illustrated by examples.

\begin{notation}
If $R$ is a $k$-algebra, the notation $Aut(R)$ will always mean the group of $k$-linear automorphisms of $R$.
\end{notation}

\begin{example}\label{ex:tame autos poly 2 vars}  Let $k[x,y]$ be the commutative polynomial ring in two variables.  Define two subgroups of $Aut(k[x,y])$ as follows:
\begin{gather*}A = \big\{(x,y) \mapsto (\lambda_1 x + \lambda_2 y + \lambda_3,\ \mu_1 x + \mu_2 y  + \mu_3) : \lambda_1\mu_2 \neq \lambda_2\mu_1, \lambda_i,\mu_i \in k\big\},\\
B = \big\{(x,y) \mapsto (\lambda x + \mu,\ \eta y + f(x)) : \lambda, \eta \in k^{\times}, \mu \in k, f(x) \in k[x]\big\};
\end{gather*}
the \textit{affine} and \textit{triangular} automorphism respectively.  The tame automorphisms of $k[x,y]$ are defined to be those in the group generated by $A \cup B$; it is a well-known result (due to Jung \cite{jung} in characteristic 0 and van der Kulk \cite{vanderkulk} in arbitrary characteristic) that the group of tame automorphisms equals the whole group $Aut(k[x,y])$.
\end{example}
\begin{example}\label{ex: tame autos poly n vars} More generally, let $k[x_1, \dots, x_n]$ be the polynomial ring in $n$ variables, and take the tame automorphism group to be that generated by all elementary automorphisms of the form
\[(x_1, \dots, x_i, \dots, x_n) \mapsto (x_1, \dots, \lambda x_i + f, \dots, x_n)\]
for $1 \leq i \leq n$, $\lambda \in k^{\times}$ and $f \in k[x_1, \dots, x_{i-1},x_{i+1},\dots, x_n]$; in two variables this coincides with the group defined in Example~\ref{ex:tame autos poly 2 vars} (see, e.g., \cite{NagataWild}).  In \cite{Nagata}, Nagata conjectured that the automorphism
\begin{equation}\label{eq:Nagata_automorphism}(x,y,z) \mapsto \big(x + (x^2-yz)z,\ y + 2(x^2-yz)x + (x^2-yz)^2z,\ z\big)\end{equation}
in $k[x,y,z]$ should be wild, a conjecture which remained open for over 30 years before being settled.  In \cite{NagataStablyTame}, it was shown that Nagata's automorphism is stably tame, i.e. becomes tame when new variables (upon which \eqref{eq:Nagata_automorphism} acts as the identity) are added.  However, it was not until 2003 that Shestakov and Umirbaev finally proved in \cite{NagataWild} that the Nagata automorphism \eqref{eq:Nagata_automorphism} is indeed wild, and hence the polynomial ring in three variables admits wild automorphisms.
\end{example}
Other examples of algebras with tame automorphism groups include the free algebra in two variables \cite{ML1}, the commutative field $k(x,y)$ in two variables \cite{Ishky} and the first Weyl algebra $A_1(k)$ \cite[\S8]{Dixmier}; examples of algebras with wild automorphisms include $\mathcal{U}(\mathfrak{sl_2})$, the enveloping algebra of the Lie algebra $\mathfrak{sl_2}$ \cite{JosephWild}.

\subsection{Automorphisms of $q$-commuting structures}\label{ss:autos of q-comm structures}

Since we will be interested in automorphisms of $q$-commuting algebras, let us examine what is already known about them in more detail.  The automorphism group of the quantum plane $k_q[x,y]$ is particularly easy to understand: for $q \neq \pm 1$ it admits only automorphisms of scalar multiplication, i.e. maps of the form 
\[x \mapsto \alpha x,\ y \mapsto \beta y, \quad (\alpha, \beta) \in (k^{\times})^2,\]
and hence $Aut(k_q[x,y]) \cong (k^{\times})^2$ \cite[Proposition~1.4.4]{AlevChamarie}.  This is far smaller than the automorphism group of the commutative polynomial ring $k[x,y]$, which is a result of the restrictions imposed by the lack of commutativity: the images of $x$ and $y$ must $q$-commute in $k_q[x,y]$, for example.  Since any homomorphism from $k_q[x,y]$ to itself must preserve the set of normal elements, and it is shown in \cite[Proposition~4.1.1]{dumas_invariants} that the only normal elements in $k_q[x,y]$ are the monomials, the possible images of $x$ and $y$ are immediately restricted to pairs of $q$-commuting monomials.  Of such pairs, the only ones ones defining an \textit{invertible} map of $k_q[x,y]$ are $x$ and $y$ themselves; this provides an elementary proof of the result in \cite{AlevChamarie}.

With this analysis in mind, upon moving up to the quantum torus $k_q[x^{\pm1},y^{\pm1}]$ we may define a new set of automorphisms of the form
\begin{equation}\label{eq:example of monomial auto}x \mapsto y^bx^d,\  y \mapsto  y^ax^c; \quad a,b,c,d \in \mathbb{Z}, \ ad - bc = 1.\end{equation}
We observe that since
\[(y^bx^d)(y^ax^c) = q^{ad}y^{ab}x^{cd} = q^{ad-bc}(y^ax^c)(y^bx^d),\]
the condition $ad-bc=1$ is both necessary and sufficient for the map defined in \eqref{eq:example of monomial auto} to be a well-defined homomorphism.  Since we may define the inverse map on $k_q[x^{\pm1},y^{\pm1}]$ by
\[x \mapsto q^m y^ax^{-b},\ y \mapsto q^n y^d x^{-c}\]
for some $m,n \in \mathbb{Z}$ depending on the values of $\{a,b,c,d\}$, the maps defined in \eqref{eq:example of monomial auto} define a set of automorphisms on $k_q[x^{\pm1},y^{\pm1}]$ which correspond to elements of the group $SL_2(\mathbb{Z})$.  

We will examine these automorphisms in more detail in Chapter~\ref{c:fixed_rings_chapter}, where we will refine the definition \eqref{eq:example of monomial auto} slightly in order to define an embedding of $SL_2(\mathbb{Z})$ into $Aut(k_q[x^{\pm1},y^{\pm1}])$.  For now, it suffices to observe that 
\[Aut(k_q[x^{\pm1},y^{\pm1}]) \cong (k^{\times})^2 \rtimes SL_2(\mathbb{Z}),\]
which is proved in \cite[\S4.1.1]{dumas_invariants}.

The full structure of the automorphism group of the $q$-division ring $D$ is not yet known; we outline existing results in this area in \S\ref{ss:autos of division ring background section}.  Many of these results make use of techniques originally developed for describing the automorphism group of a much larger division ring, namely the division ring of Laurent power series\nom{L@$L_q$}
\[L_q = k(\!(y)\!)(\!(x)\!) = \left \{\sum_{i \geq n}a_i x^i : n \in \mathbb{Z}, a_i \in k(\!(y)\!), xy=qyx\right\}.\]
This is a generalization of the ring $k_q(y)(\!(x)\!)$ defined in \eqref{eq:def of Laurent power series ring,notation section}, where here we allow coefficients in $k(\!(y)\!)$ instead of $k(y)$.

A key point when doing computations in $L_q$ and its subring $k_q(y)(\!(x)\!)$ is that one can often specify just the first term in a power series and then construct the rest of the coefficients recursively to satisfy a desired property.  For example, given an element
\begin{equation}\label{eq:standard form of g to construct f}g = \lambda y + \sum_{i \geq 1} g_i x^i \in L_q,\end{equation}
one may construct a second element $f = \sum_{i \geq n}f_i x^i$ which $q$-commutes with it by expanding out the expression $fg-qgf = 0$ and solving term-by-term for the coefficients of $f$.  Indeed, we see that
\begin{align}\label{eq:computation in Lq 1}
\nonumber 0&= fg-qgf \\
\nonumber &=\sum_{i \geq n}f_i x^i \left(\lambda y + \sum_{j \geq 1}g_j x^j\right) - q \left(\lambda y + \sum_{j \geq 1}g_j x^j\right)\sum_{i \geq n}f_i x^i \\
&= \sum_{i \geq n} \lambda(q^i-q)yf_i x^i + \sum_{i \geq n+1} \left(\sum_{k=n}^{i-1} f_k \alpha^k(g_{i-k}) \right)x^i,
\end{align}
where $\alpha$ denotes the map $y \mapsto qy$ on $k(\!(y)\!)$.

The coefficient of $x^n$ in \eqref{eq:computation in Lq 1} is $\lambda (q^n-q)yf_nx^n$, which is zero if and only if $n=1$ since $q$ is not a root of unity.  Having set $n=1$, we may choose $f_1 \in k(\!(y)\!)$ arbitrarily, provided it is non-zero.  Now by considering the coefficient of $x^m$ for any $m \geq 2$ in \eqref{eq:computation in Lq 1} and recalling that $n=1$, we can see that
\[f_m = -\sum_{k=1}^{m-1}f_k \alpha^k(g_{m-k})\lambda^{-1}(q^m-q)^{-1}y^{-1},\]
which is uniquely determined by $g$ and the choice of the coefficient $f_1$.  

In \cite{AD3}, Alev and Dumas use techniques of this form to describe the automorphism group of $L_q$.  They first show that if $\theta$ is an automorphism on $L_q$ then $\theta(x)$ and $\theta(y)$ must take the forms of the elements $f$ and $g$ described in the above discussion \cite[Lemme~2.6]{AD3}.  By considering the expansion of the equation $z\theta(y) = yz$ for some unknown $z \in L_q$, they are then able to describe necessary and sufficient conditions for $\theta$ to be an inner automorphism.  This result is recorded in the following lemma.
\begin{lemma}\label{res:Alev Dumas Lq lemma}\cite[Lemme~2.6]{AD3}
For all $\theta \in Aut(L_q)$, there exists some $\beta \in k^{\times}$ and two sequences $(a_i)_{i \geq 1}$, $(b_i)_{i \geq 1}$ of elements in $k(\!(y)\!)$ with $a_1 \neq 0$, such that the image of $\theta$ has the form
\begin{equation}\label{eq:standard form in Lq}\theta(x) = \sum_{i \geq 1} a_i x^i, \quad \theta(y) = \beta y + \sum_{i \geq 1} b_i x^i.\end{equation}
Further, $\theta$ is an inner automorphism if and only if it satisfies the following two conditions:
\begin{enumerate}
\item $\beta = q^n$ for some $n \in \mathbb{Z}$;
\item there exists some $u \in k(\!(y)\!)^{\times}$ such that $a_1 \alpha(u) = u$.
\end{enumerate}
\end{lemma}
Alev and Dumas define the set of elementary automorphisms on $L_q$ to be those of the form
\begin{equation}\label{eq:triangular autos on L_q}
\{\varphi_{\alpha,f}: x \mapsto f(y)x,\ y \mapsto \alpha y : \alpha \in k^{\times}, f \in k(\!(y)\!)^{\times}\},
\end{equation}
and observe that the automorphism group of $L_q$ is tame in the following theorem.
\begin{theorem}\cite[Th\'eor\`eme 2.7]{AD3}
The automorphism group of $L_q$ is generated by the elementary automorphisms and inner automorphisms, and hence $Aut(L_q)$ is tame.
\end{theorem}
The proof follows from Lemma~\ref{res:Alev Dumas Lq lemma} by observing that the image of $x$ and $y$ from \eqref{eq:standard form in Lq} can be transformed using an elementary automorphism to obtain elements satisfying $a_1 = 1$, $\beta = 1$.  These elements clearly satisfy the conditions required to be the image of an inner automorphism, and hence we have constructed the inverse of our original automorphism as the product of an elementary and an inner automorphism.

We can also easily obtain the same result for the slightly smaller ring $k_q(y)(\!(x)\!)$:
\begin{theorem}
Define the set of elementary automorphisms on $k_q(y)(\!(x)\!)$ to be those of the form
\[\{\varphi_{\alpha,f}: x \mapsto f(y)x,\ y \mapsto \alpha y: \alpha \in k^{\times}, f \in k(y)^{\times}\}.\]
Then the automorphism group of $k_q(y)(\!(x)\!)$ is generated by elementary and inner automorphisms, and hence $Aut\big(k_q(y)(\!(x)\!)\big)$ is tame.
\end{theorem}
We will not prove this here as it will not be used in this thesis.  However, it follows easily from the results in \cite{AD3} by observing that the proofs for $L_q$ up to and including Th\'eor\`eme 2.7 make no use of the properties of $k(\!(y)\!)$ except that it is a field, and hence work without modification for $k_q(y)(\!(x)\!)$ as well.

\subsection{The automorphisms of $k_q(x,y)$}\label{ss:autos of division ring background section}
One would hope that given our understanding of the automorphism groups of various subrings and overrings of the $q$-division ring $D$, the description of $Aut(D)$ would follow easily; unfortunately, this is not the case.  For example, in Proposition~\ref{res:z_automorphism_order_2} we will construct an example of a conjugation automorphism on $D$ which satisfies the conditions of Lemma~\ref{res:Alev Dumas Lq lemma} but is not an inner automorphism on $D$: the conjugating element $z$ is in $k_q(y)(\!(x)\!) \backslash D$, so the map is inner as an automorphism of $k_q(y)(\!(x)\!)$ but not as an automorphism of $D$.

Let $X$, $Y$ be a pair of $q$-commuting generators for $D$.  In this section we will outline the existing results which partially describe the structure of $Aut(D)$. 

In \cite{AD3} Alev and Dumas construct a set of generators for the tame automorphism group by analogy to the automorphism group of $k(x,y)$, while in \cite{AC1} Artamonov and Cohn define a different but possibly more natural set of elementary automorphisms.  We will use the definition from \cite{AC1} here; in Lemma~\ref{res:H1 and H2 are equal} we will show that the two definitions in fact coincide, thus justifying this choice.
\begin{definition}\label{def:elementary autos of D} The following automorphisms of $D$ are called \textit{elementary}:
\begin{align*}
\tau:&{}\quad X\mapsto X^{-1}, \ Y \mapsto Y^{-1} \\
h_X:&{}\quad X \mapsto b(Y)X, \ Y \mapsto Y,\quad b(Y) \in k(Y)^{\times}\\
h_Y:&{}\quad X \mapsto X,\ Y \mapsto a(X)Y,\quad a(X) \in k(X)^{\times} 
\end{align*}
Call an automorphism of $D$ \textit{tame} if it is in the group generated by the elementary automorphisms and the inner automorphisms.
\end{definition}
In \cite{AC1}, progress is made towards describing the automorphism group $Aut(D)$ in terms of the elementary automorphisms and certain types of conjugation maps.  Since we will build on this work in \S\ref{s:automorphism_consequences}, we give a brief outline of their results here.  

Let $\theta$ be a homomorphism from $D$ to itself.  As for $L_q$ in \S\ref{ss:autos of q-comm structures}, Artamonov and Cohn try to understand $\theta$ by applying elementary transformations and conjugation maps to the images $\theta(X)$ and $\theta(Y)$ until they arrive back at the original generators $X$ and $Y$.  

More generally, let $f$, $g$ be a pair of elements in $k_q(X,Y)$ such that $fg = qgf$ and identify them with their image in $k_q(Y)(\!(X)\!)$ as follows:
\[f = a_mX^m + \sum_{i > m} a_i X^i, \quad g = b_nX^n + \sum_{j >n} b_j X^i.\]
We may assume that $a_m$ and $b_n$ are both non-zero.  For any $r \in \mathbb{Z}$, we have
\begin{align*}
g^rf &= c_{m + rn}X^{m+rn} + \textrm{ [higher terms]}\\
f^rg &= d_{n + rm}X^{n+rm} + \textrm{ [higher terms]}
\end{align*}
where $c_{m+rn},d_{n+rm} \in k(Y)^{\times}$. As described in \cite[\S3]{AC1}, we may therefore apply a carefully chosen sequence of elementary transformations to $f$ and $g$ so that at each step the lowest $X$-degree of one element in the pair is closer to zero than before, while preserving the two properties that (i) the pair of elements $q$-commute, and  (ii) they generate the same ring as the original pair.  It is clear that this process must terminate in a finite number of steps, when the $X$-degree of one element reaches 0.

Using the fact that our pair of elements still $q$-commute, \cite[Proposition~3.2]{AC1} shows that these elements must have the form
\begin{equation}\label{eq:q comm standard form}F = f_sX^s + \sum_{i \geq s} f_i X^i, \ G = \lambda Y^s + \sum_{i \geq 1} g_i X^i,\end{equation}
where $s=\pm1$, $\lambda \in k^{\times}$ and $f_i,g_i \in k(Y)$ for all $i \geq s$.  In other words, for any $q$-commuting pair $(f,g)$ of elements in $D$, there exists a sequence of elementary transformations that reduces $(f,g)$ to a pair $(F,G)$ of the form \eqref{eq:q comm standard form}.  Further, we may apply two more elementary transformations to ensure that $f_s$ and $\lambda$ are both 1.

The next proposition completes the process by showing that we may always construct an element of $k_q(y)(\!(x)\!)$ that conjugates the pair $(F,G)$ back to $(X^s,Y^s)$.
\begin{proposition} \cite[Proposition~3.3]{AC1}\label{res:AC_construct_z} Let $F$, $G \in k_q(Y)(\!(X)\!)$ be $q$-commuting elements of the form \eqref{eq:q comm standard form}, where we may assume without loss of generality that $\lambda = 1$, $f_s = 1$.  Then there exists an element $z \in k_q(Y)(\!(X)\!)$ defined by
\begin{equation}\label{eq:def_z}\begin{gathered}z_0 = 1; \quad z_n = Y^{-s}(1-q^s)^{-1}\left(g_n + \sum_{\stackrel{i + j = n}{i,j>0}}z_j\alpha^j(g_i)\right) \textrm{ for } n\geq 1; \\
z: = \sum_{i \geq 0}z_i X^i. \end{gathered}
\end{equation}
such that
\[zFz^{-1} = f_sX^s, \quad zGz^{-1} = \lambda Y^s. \]
\end{proposition}
The element $z$ is constructed recursively by solving the equation $zG = Y^sz$ for coefficients of $z$, in a similar manner to the process described in \ref{ss:autos of q-comm structures}.  That we must also have $zFz^{-1} = f_sX^s$ in this case is a consequence of the fact that $zFz^{-1}$ must $q$-commute with $zGz^{-1}=\lambda Y^s$.  The main theorem of \cite{AC1} (which is stated next) is now an easy consequence.
\begin{theorem} \cite[Theorem~3.5]{AC1} \label{res:main AC theorem statement}
Let $\theta: D \rightarrow D$ be a homomorphism.  Then there exists a sequence of elementary automorphisms $\varphi_1, \dots, \varphi_n$, an element $z \in k_q(y)(\!(x)\!)$ constructed as in Proposition~\ref{res:AC_construct_z}, and $\epsilon \in \{0,1\}$ such that
\begin{equation}\label{eq:decomp}\theta = \varphi_1 \varphi_2 \dots \varphi_n c_{z^{-1}} \tau^{\epsilon},\end{equation}
where $c_{z^{-1}}$ denotes conjugation by $z^{-1}$ and $\tau: x \mapsto x^{-1}, \ y \mapsto y^{-1}$ is the elementary automorphism defined in Definition~\ref{def:elementary autos of D}.
\end{theorem}
This is not sufficient on its own to prove that $Aut(D)$ is tame, as it is not clear whether we must have $z \in D$ whenever $\theta$ is an automorphism.  Indeed, as we will see in \S\ref{s:automorphism_consequences}, it is possible to construct automorphisms of $D$ in this manner where $z \in k_q(y)(\!(x)\!) \backslash D$; it remains an open question whether automorphisms of this form can be decomposed further into a product of elementary and inner automorphisms, or whether $D$ admits wild automorphisms.

\section{Quantization-deformation}\label{s:background deformation}

Intuitively, if we set $q=1$ in $k_q(x,y)$ we recover the commutative field of rational functions in two variables, and so we would expect the structures of these two rings to be similar to some extent.  This type of example is the motivation for the theory of \textit{deformation-quantization}, which seeks to describe this relationship formally.  Quantization also has uses in many areas of physics, for example quantum mechanics: a classical system is often represented as families of smooth functions on a manifold while a quantum one involves certain non-commuting operators on a Hilbert space, but the two should be related in the sense that as the deforming parameter $t$ (often denoted by the Planck constant $\hbar$ in this context) tends to zero, we recover the original classical system (see, e.g. \cite[\S4]{quantization_ref} for more details on this).

The deformations of a commutative algebra $R$ are closely linked to the possible Poisson structures that can be defined on $R$, so we begin in \S\ref{ss:background poisson algebras def} by defining the notion of a Poisson algebra and elementary definitions relating to this.  In \S\ref{ss:background formal deformation} we define deformation-quantization formally in terms of star products on power series, and finally in \S\ref{ss:background examples deformation} we give several examples which will form a recurring theme in future chapters.

\subsection{Poisson algebras}\label{ss:background poisson algebras def}
A \textit{Poisson bracket} on a $k$-algebra $A$ is a skew-symmetric bilinear map $\{\cdot, \cdot\}: A \times A \rightarrow A$ which satisfies the conditions of a Lie bracket:
\begin{gather*}
\{a,a\} = 0 \quad \forall a \in A \\
\{a,\{b,c\}\} + \{b,\{c,a\}\} + \{c,\{a,b\}\} = 0 \quad \forall a,b,c \in A
\end{gather*}
and also satisfies the Leibniz identity:
\begin{equation}\label{eq:Leibniz identity}\{a,bc\} = \{a,b\}c + b\{a,c\} \quad \forall a,b,c \in A.\end{equation}
Intuitively, \eqref{eq:Leibniz identity} says that $\{a,\cdot\}$ and $\{\cdot, b\}$ are derivations of $A$ for any $a$ or $b$ in $A$.  If $A$ is an associative $k$-algebra with a Poisson bracket, we call $A$ a \textit{Poisson algebra}\nom{Poisson algebra}.  Although this definition makes sense for non-commutative algebras, for the purposes of this thesis we will always assume that our Poisson algebras are commutative.

Many of the standard algebraic concepts and definitions can be extended in a very natural way to the case of Poisson algebras.  We make the following definitions:
\begin{definition}\label{def:poisson ideals}
Let $A$ be a Poisson algebra.  We call an ideal $I  \subset A$ a \textit{Poisson ideal}\nom{Poisson ideal} if it is also closed under the Poisson bracket, that is $\{I,A\} \subseteq I$.  A Poisson ideal $I$ is called \textit{Poisson-prime}\nom{Poisson-prime} if whenever $J$, $K$ are Poisson ideals satisfying $JK \subseteq I$ then $J \subseteq I$ or $K \subseteq I$.
\end{definition}
In the case where $A$ is commutative Noetherian and $k$ has characteristic zero, the set of Poisson-prime ideals coincides with the set of Poisson ideals which are prime in the standard commutative sense \cite[Lemma~1.1]{Goodearl_Poisson}.  We may therefore use the terms ``Poisson-prime'' and ``prime Poisson'' interchangeably.

Given a Poisson algebra $A$ and a Poisson ideal $I$ we may form the quotient $A/I$, which by \cite[\S3.1.1]{dumas_invariants} is again a Poisson algebra.  The bracket on $A/I$ is induced from that of $A$ via the definition
\begin{equation}\label{eq:poisson bracket on quotient}\{a+I,b+I\} := \{a,b\} + I.\end{equation}
Similarly, if $A$ is a commutative domain and $S$ a multiplicatively-closed subset of $A$, then the Poisson bracket extends uniquely to the localization $AS^{-1}$ as follows:
\begin{equation}\label{eq:extend Poisson bracket to localization}
\{as^{-1},bt^{-1}\} = \{a,b\}s^{-1}t^{-1} - \{a,t\}bs^{-1}t^{-2} - \{s,b\}as^{-2}t^{-1} + \{s,t\}abs^{-2}t^{-2}.
\end{equation}
This formula is an easy consequence of the quotient rule for derivatives (see, for example, \cite[\S3.1.1]{dumas_invariants}).

A homomorphism $\varphi: A \rightarrow B$ is a \textit{homomorphism of Poisson algebras} if it respects the Poisson brackets of each structure, i.e. $\varphi(\{a_1,a_2\}_A) = \{\varphi(a_1),\varphi(a_2)\}_B$.  Using this definition it is easy to see that if $G$ is a group of Poisson automorphisms on a Poisson algebra $A$, then the fixed ring 
\[A^G = \{a \in A : g(a) = g \ \forall g \in G\}\]
is closed under the Poisson bracket: for $a,b \in A^G$, we have $g(\{a,b\}) = \{g(a),g(b)\} = \{a,b\}$.  Hence $A^G$ is again a Poisson algebra.

We may therefore formulate a Poisson version of Noether's problem as in \cite[\S5.5.1]{dumas_invariants}: 

\begin{question}\label{ques:Noether's problem for Poisson}If $F$ is a field equipped with a Poisson bracket and $G$ is a finite group of Poisson automorphisms, under what conditions is there an isomorphism of \textit{Poisson algebras} $F^G \cong F$?\end{question}

A full answer to this question is not known even in the case of fields of transcendence degree 2: while Castelnuovo's theorem (see \cite[\S5.1.1]{dumas_invariants}) guarantees the existence of an isomorphism of \textit{algebras}, the two fields need not have the same Poisson structure in general.  Existing results in this direction are summarised in \cite[\S3]{Dumas2}, including the following example where $F$ and $F^G$ have non-isomorphic Poisson structures:
\begin{theorem}\label{res:poisson fixed ring example} \cite[\S3]{Dumas2}
Let $F = k(x,y)$ with the Poisson bracket $\{y,x\} = yx$, and let $G$ be a finite group of Poisson automorphisms defined on the polynomial ring $k[x,y]$ and extended to $F$.  Let $u$ and $v$ be a pair of generators for $F^G$, i.e. $F^G = k(u,v)$, then the Poisson bracket on $F^G$ is given by $\{v,u\} = |G|.vu$, where $|G|$ denotes the order of the group $G$.
\end{theorem}
By \cite[Corollary~5.4]{GLa1}, the Poisson algebra $F^G$ in Theorem~\ref{res:poisson fixed ring example} cannot be Poisson-isomorphic to $F$ unless $|G| = 1$.  Since (for example) the group generated by the automorphism $x \mapsto -x,\ y \mapsto -y$ is non-trivial and satisfies the conditions of Theorem~\ref{res:poisson fixed ring example}, it is possible to answer the Poisson-Noether question in the negative even for fields of transcendence degree 2.

On the other hand, in Chapter \ref{c:deformation_chapter} we will show that for the field $k(x,y)$ and Poisson bracket $\{y,x\} = yx$ there is a Poisson isomorphism $k(x,y)^G \cong k(x,y)$ for all finite groups of Poisson monomial automorphisms on $k(x,y)$.

\subsection{Formal deformation}\label{ss:background formal deformation}

The theory of quantization-deformation can range from the extremely general and formal formulations (e.g. Kontsevich's Formality Theorem in \cite{Kontsevich}) to the informal notion given in Definition~\ref{def:poisson deformation} below, and the notation and terminology can vary wildly.  A common theme, however, is to construct the deformation of a ring $R$ by defining a new product (the ``star product'') on the power series ring $R[\![t]\!]$, and it is this approach we describe below.  For a more detailed description, see for example \cite{Gerstenhaber1} and \cite[\S7]{ginzburg_notes}, or for the case where we instead consider a smaller ring $R[t]$ or $R[t^{\pm1}]$ see e.g. \cite[\S3]{dumas_invariants}.

Let $R$ be an associative $k$-algebra, and form the ring of power series $R[\![t]\!]$ over $R$ in the central variable $t$.  Let
\[\pi_i: R \times R \rightarrow R, \quad i \geq 1\]
be a sequence of bilinear maps, and use these to define a new multiplication on $R[\![t]\!]$ (the \textit{star product}) by
\begin{equation}\begin{aligned}\label{eq:star product def}
a * b &= ab + \pi_1(a,b)t + \pi_2(a,b)t^2 + \pi_3(a,b)t^3 + \dots \quad \forall a,b \in R\\
a*t &= (a*1)t
\end{aligned}\end{equation}

We are interested in defining star products which are associative, or more generally are associative up to a certain degree.  Since $R[\![t]\!]$ is $\mathbb{N}$-graded we may solve the equality $a*(b*c) = (a*b)*c$ term by term, which at each step will allow us to impose restrictions on the $\pi_i$ to ensure that the product is associative up to that degree.

Using \eqref{eq:star product def} to expand out the products $a*(b*c)$ and $(a*b)*c$, we see that
\begin{align}\label{eq:star product associative 1}
\nonumber a*(b*c) &= a*(bc + \pi_1(b,c)t + \pi_2(b,c)t^2+\dots)\\
\nonumber &= abc + \pi_1(a,bc)t + \pi_2(a,bc)t^2 + \dots \\
\nonumber & \qquad + a\pi_1(b,c)t + \pi_1(a,\pi_1(b,c))t^2 + \dots \\
\nonumber & \qquad + a\pi_2(b,c)t^2 + \dots \\
&= abc + \Big(a\pi_1(b,c) + \pi_1(a,bc)\Big)t \\
\nonumber & \qquad + \Big(\pi_2(a,bc) + \pi_1(a,\pi_1(b,c)) + a\pi_2(b,c)\Big)t^2 + \dots,
\end{align}
and similarly,
\begin{equation}\label{eq:star product associative 2}(a*b)*c = abc + \Big(\pi_1(ab,c) + \pi_1(a,b)c\Big)t + \Big(\pi_2(ab,c) + \pi_1(\pi_1(a,b),c) + \pi_2(a,b)c\Big)t^2 + \dots .\end{equation}

The star product is always associative in degree 0 since our original algebra $R$ was assumed to be associative.  Using \eqref{eq:star product associative 1} and \eqref{eq:star product associative 2}, we see that the $t$ term in $a*(b*c) - (a*b)*c$ is zero if $\pi_1$ satisfies the following property:
\begin{equation}\label{eq:pi_1 condition}a\pi_1(b,c) - \pi_1(ab,c) + \pi_1(a,bc) - \pi_1(a,b)c = 0.\end{equation}
In other words, $\pi_1$ must be a 2-cocycle in the Hochschild cohomology of $R$, which we define next.

Let $A$ be an associative $k$-algebra.  Define a chain complex by
\[0 \longrightarrow A \stackrel{d_0}{\longrightarrow} Hom_k(A,A) \stackrel{d_1}{\longrightarrow} Hom_k(A^{\otimes 2}, A) \stackrel{d_2}{\longrightarrow} Hom_k(A^{\otimes 3},A) \stackrel{d_3}{\longrightarrow} \dots,\]
where the maps are defined by
\begin{equation}\label{eq:cochain complex for Hochschild def}\begin{aligned}
d_0a(b) &=  ba - ab, \\
d_nf(a_1, \dots, a_{n+1}) &= a_1f(a_2, \dots, a_{n+1}) + \sum_{i =1}^n (-1)^i f(a_1, \dots, a_ia_{i+1}, \dots, a_{n+1}) \\
& \qquad + (-1)^{n+1}f(a_1, \dots,a_n)a_{n+1}.
\end{aligned}\end{equation}
\begin{definition}\label{def:Hochschild cohomology}
The \textit{$n$th Hochschild cohomology of $A$ with coefficients in $A$} is defined to be\nom{H@$HH^n(A)$}
\[HH^n(A) := ker(d_n)/im(d_{n-1}),\]
where the $d_i$, $i \geq 0$ are defined as in \eqref{eq:cochain complex for Hochschild def}. (For more details on the Hochschild cohomology, see for example \cite[\S5]{ginzburg_notes}.)
\end{definition}
Applying this definition to $\pi_1 \in Hom_k(R^{\otimes 2},R)$, we see that 
\[d_2\pi_1(a,b,c) = a\pi_1(b,c) - \pi_1(ab,c) + \pi_1(a,bc) - \pi_1(a,b)c,\]
and hence \eqref{eq:pi_1 condition} is satisfied if and only if $\pi_1 \in ker(d_2)$.  In fact, a long but elementary calculation shows that if two such cocycles differ by a coboundary (i.e. they represent the same class in $HH^2(R)$) then they define the same star product modulo $t^2$ up to a change of variables (see, for example, \cite[\S3]{Gerstenhaber1}).  We may therefore view $\pi_1$ as an element of $HH^2(R)$.

Moving on to terms in $t^2$, we see that the star product is associative up to degree 2 if $\pi_1 \in HH^2(R)$ and there exists some $\pi_2 \in Hom_k(R^{\otimes2},R)$ satisfying the following equation:
\begin{equation}\label{eq:associative in degree 2}a\pi_2(b,c) - \pi_2(ab,c) + \pi_2(a,bc) - \pi_2(a,b)c = \pi_1(\pi_1(a,b),c) - \pi_1(a,\pi_1(b,c)).\end{equation}
The map defined by the RHS of \eqref{eq:associative in degree 2} (which we will denote by $f$) is in $ker(d_3)$ whenever $\pi_1 \in ker(d_2)$ \cite[\S2]{Gerstenhaber1}, which is true here by assumption.  Rewriting \eqref{eq:associative in degree 2} as $d_2(\pi_2) = f$, we see that a $\pi_2$ satisfying this equation  exists if and only if $f \in im(d_2)$ as well.  This says that a map $\pi_2$ making the star product associative up to degree 2 exists if and only if $f$ is zero in $HH^3(R)$.

The third Hochschild cohomology $HH^3(R)$ is therefore referred to as the \textit{obstruction} to extending the deformation: if it is trivial then there will always exist some $\pi_2$ satisfying \eqref{eq:associative in degree 2}, but if it is non-zero then it is possible that for certain choices of $\pi_1$ we will have $f \neq 0$ in $HH^3(R)$.  In fact this observation holds more generally as well: if the star product is associative up to degree $n-1$, then whether it will extend to an associative product in degree $n$ is controlled by the obstruction in $HH^3(R)$ \cite[\S5]{Gerstenhaber1}.

We may also consider the question of when the star product is commutative, or (since a commutative star product is in some sense the trivial one) how far from commutative it is.  By applying \eqref{eq:star product def} to the expression $a*b - b*a$, we see that
\[a * b - b * a = ab - ba + \Big(\pi_1(a,b) - \pi_1(b,a)\Big)t + \Big(\pi_2(a,b) - \pi_2(b,a)\Big)t^2 + \dots\]
If the original ring $R$ is commutative then the commutator $a*b - b*a$ is in the ideal $tR[\![t]\!]$, which allows us to make the following definition.
\begin{definition}\cite[\S7.1]{ginzburg_notes}\label{def:star product Poisson bracket}
Suppose that $R[\![t]\!]$ has a star product as defined in \eqref{eq:star product def}, and suppose further that $R$ is commutative and $\pi_1$ is not identically zero on $R$.  Then we can define a Poisson bracket on $R$ by
\[\{a,b\} = \frac{1}{t}(a*b- b*a) \quad \mod tR[\![t]\!]\]
for all $a,b \in R$.
\end{definition}
This Poisson bracket captures a first-order impression of the star product on $R[\![t]\!]$.  Conversely, we may also start with a commutative Poisson algebra $R$ and define a star product on $R[\![t]\!]$ by
\[a*b = ab + \{a,b\}t + \pi_2(a,b)t^2 + \dots\ .\]
Since the equality
\[a\{b,c\} - \{ab,c\} + \{a,bc\} - \{a,b\}c = 0\]
follows immediately from the Leibniz identity \eqref{eq:Leibniz identity}, a star product defined from a Poisson bracket in this manner will always be associative to at least degree 1.

This formal definition of deformation is often quite difficult to work with.  However, for certain nice rings and star products it may be that we do not require the whole power series ring, but can instead construct a $k[t]$ or $k[t^{\pm1}]$ algebra $\BB$ to play the part of $R[\![t]\!]$.  One advantage of this approach is that by constructing $\BB$ as an associative algebra directly we need not worry about the Hochschild cohomology at all; another is that we can now form ideals generated by polynomials of the form $t-\lambda$ for certain $\lambda \in k^{\times}$, and hence quotients of the form $\BB/(t-\lambda)\BB$.  This allows us to define a slightly more informal notion of quantization-deformation as follows, based on the convention in \cite[\S3.2.1]{dumas_invariants}.

\begin{definition}\label{def:poisson deformation}
Let $R$ be a commutative Poisson algebra, and $\BB$ an algebra containing a central, non-invertible, non-zero divisor element $h$ such that $\BB/h\BB \cong R$ as Poisson algebras (where the bracket on $\BB/h\BB$ is induced by the commutator in $\BB$ as in Definition~\ref{def:star product Poisson bracket}).  If $\lambda \in k^{\times}$ is such that $h-\lambda$ generates a proper, non-zero ideal in $\BB$ then we call the algebra $A_{\lambda}:=\BB/(h-\lambda)\BB$ a \textit{deformation} of $R$.  In the other direction, if $S$ is a subset of $k^{\times}$ such that $A_{\lambda}$ is defined for each $\lambda \in S$ then we call $R$ the \textit{semi-classical limit} of the family of algebras $\{A_{\lambda}:\lambda \in S\}$.
\end{definition}
This definition of deformation turns out to be sufficient for our purposes in this thesis, and it is this definition we will mean when we consider $D$ as a deformation of a commutative algebra in Chapter~\ref{c:deformation_chapter}.
\begin{remark}
It is often the case that certain choices of $\lambda$ in the above definitions will give rise to degenerate or undesirable deformations; to avoid this, we will always ensure that the polynomial $h-\lambda$ is invertible in $\BB$ for those choices of scalar.
\end{remark}
\subsection{Examples}\label{ss:background examples deformation}
We illustrate Definition~\ref{def:poisson deformation} with a few examples, one closely related to the $q$-division ring and one concerning quantum matrices.

\begin{example}\label{ex:deformation of torus}
This example is due to Baudry in \cite[\S5.4.3]{BaudryThesis}.

Let $\BB$ be the ring defined by\nom{B@$\mathfrak{B}$}
\begin{equation}\label{def:definition of BB}
\BB = k\langle x^{\pm1},y^{\pm1},z^{\pm1} \rangle / (xz-zx,yz-zy,xy-z^2yx)
\end{equation}
This is clearly a domain, and the element $h := 2(1-z)$ is central and non-invertible.  For $\lambda \in k^{\times}$, $\lambda \neq 2$, the quotient $\BB/(h-\lambda)\BB$ is isomorphic to the quantum torus $k_q[x^{\pm1},y^{\pm1}]$ for $q = (1-\frac{1}{2}\lambda)^2$.  We exclude $\lambda = 2$ because $h-2$ is invertible in $\BB$.

When $\lambda = 0$, the image of $z$ in the quotient $\BB/\lambda\BB$ is 1 and we recover the standard commutative Laurent polynomial ring $k[x^{\pm1},y^{\pm1}]$.  Further, we can compute the induced Poisson bracket as follows:
\begin{equation}
yx - xy = (1-z^2)yx = \frac{1}{2}(1+z)hyx
\end{equation}
and therefore $\{y,x\} = yx \mod h\BB$, since $z=1$ when $h=0$.

This shows that $k_q[x^{\pm1},y^{\pm1}]$ is a deformation of the commutative torus $k[x^{\pm1},y^{\pm1}]$ with respect to the Poisson bracket $\{y,x\} = yx$.
\end{example}

\begin{example}\label{ex:2x2 quantum matrices}
Recall that the ring of quantum $2 \times 2$ matrices $\QML{2}$ is given by the quotient of the free algebra $k\langle X_{11},X_{12},X_{21},X_{22} \rangle$ by the six relations
\begin{equation}\label{eq:relns for 2x2 matrices background section}\begin{aligned}
X_{11}X_{12} &- qX_{12}X_{11},\quad & X_{12}X_{22} &- qX_{22}X_{12}, \\
X_{11}X_{21} &- qX_{21}X_{11},\quad & X_{21}X_{22} &- qX_{22}X_{21}, \\
X_{12}X_{21}&-X_{21}X_{12},\quad & X_{11}X_{22}-X_{22}X_{11}&-(q-q^{-1})X_{12}X_{21};
\end{aligned}
\end{equation}
for some $q \in k^{\times}$.  We can observe that when $q=1$ we recover the commutative coordinate ring $\ML{2}$, and in \cite[Example~2.2(d)]{GoodearlSummary} the quantum matrices $\QML{2}$ are viewed as a deformation of $\ML{2}$ as follows.

Let $k[t^{\pm1}]$ be the Laurent polynomial ring in one variable $t$, and let $\BB$ be the algebra in four variables $Y_{11},Y_{12},Y_{21},Y_{22}$ over $k[t^{\pm1}]$ subject to the same relations as \eqref{eq:relns for 2x2 matrices background section} but with $q$ replaced by $t$.  Then $h:=t-1$ is clearly central, non-invertible and a non-zero-divisor in $\BB$, and 
\[\BB/h\BB \cong \ML{2}, \quad \BB/(h-\lambda)\BB \cong \QML{2}\]
for $q = 1+\lambda$.  This process induces a Poisson bracket on $\ML{2}$, which is defined by
\begin{equation}\label{eq:poisson relns for 2x2 matrices}\begin{aligned}
\{x_{11},x_{12}\} &= x_{11}x_{12}, & \{x_{12},x_{22}\} &= x_{12}x_{22}, \\
\{x_{11},x_{21}\} &= x_{11}x_{21}, & \{x_{21},x_{22}\} &= x_{21}x_{22}, \\
\{x_{12},x_{21}\} &= 0, & \{x_{11},x_{22}\} &= 2x_{12}x_{21}.
\end{aligned}\end{equation}
We will consider algebras with this Poisson structure in more detail in Chapter~\ref{c:H-primes}.
\end{example}

\section{Prime and primitive ideals of quantum algebras and their semi-classical limits}\label{s:background H primes}
Let $A$ be an algebra and $\varphi$ an automorphism defined on $A$.  Then $\varphi$ must preserve the structure of $A$ in certain ways: for example, it must map prime ideals to prime ideals and primitives to primitives.  \textit{Stratification theory}, which is due predominantly to Goodearl and Letzter in the case of quantum algebras \cite{Hprimes1,Hprimes2}, seeks to exploit this observation by using the action of a group $\HH$ on $A$ to partition $spec(A)$ into more readily understood strata based on orbits under the action of $\HH$.

In this section we set up the definitions and notation required to state the Stratification Theorem for both quantum algebras and Poisson algebras, and describe how we can use this result and the Dixmier-Moeglin equivalence to identify and understand the primitive (respectively Poisson-primitive) ideals in an algebra.  In \S\ref{ss:H action quantum version} we describe the quantum version of these results, and in \S\ref{ss:H action Poisson version} we give the corresponding Poisson formulation.  Finally, in \S\ref{ss:H primes examples} we give several examples of algebras to which these results can be applied, and discuss a conjecture made by Goodearl on the relationship between the prime and Poisson-prime ideals of these algebras.

\subsection{The Stratification Theorem and Dixmier-Moeglin Equivalence for quantum algebras}\label{ss:H action quantum version}

While much of the following theory has been developed in quite a general setting -- for example, many of the following results make no assumption on the field $k$ except that it be infinite --  we will quickly restrict our attention to a setting relevant to quantum algebras.

We begin by making some definitions.
\begin{definition}\label{def:eigenvalues etc}
Let $A$ be a $k$-algebra and $\HH$ a group acting on $A$ by $k$-algebra automorphisms.  If $a \in A$ is such that $h.a = \lambda_h a$ for all $h \in \HH$ (where $\lambda_h \in k^{\times}$ may depend on $h$) then we call $a$ an \textit{eigenvector}\nom{H@$\mathcal{H}$-eigenvector} for $\HH$.  The map $f_a: \HH \rightarrow k^{\times}: h \mapsto \lambda_h$ is called the \textit{eigenvalue}\nom{H@$\mathcal{H}$-eigenvalue} of $a$, and $A_{f_a} = \{x \in A : h.x = f_a(h)x\}$ the \textit{eigenspace}\nom{H@$\mathcal{H}$-eigenspace} associated to $a$.
\end{definition}
\begin{definition}\label{def:rational action}
Let $\HH$ be an affine algebraic group over $k$.  A homomorphism $f: \HH \rightarrow k^{\times}$ is called a \textit{rational character} if $f$ is also a morphism of algebraic varieties.  If $\HH = (k^{\times})^r$ is an algebraic torus acting by $k$-algebra automorphisms on $A$ and $k$ is an infinite field, we say $\HH$ is \textit{acting rationally} if $A$ is the direct sum of its eigenspaces with respect to $\HH$ and all its eigenvalues are rational characters. 
\end{definition}
\begin{remark}
This definition of rational action is a specific case of a more general definition; for further details and proof of the equivalence of the two definitions under the conditions of Definition~\ref{def:rational action}, see \cite[Definition~II.2.6, Theorem~II.2.7]{GBbook}.
\end{remark}
Following the example of \cite[\S3]{GBrown}, we will restrict our attention to algebras and actions satisfying the following set of conditions.  These conditions embody many of the desired characteristics of quantum algebras, and hence it makes sense to restrict our attention to algebras of this form.
\begin{conditions}\label{conditions 1} We will assume the following conditions throughout this section.
\begin{itemize}
\item $A$ is a Noetherian $k$-algebra, satisfying the non-commutative Nullstellensatz over $k$ (see Remark~\ref{rem:NSS});
\item $k$ is an algebraically closed field of characteristic 0;
\item $\HH = (k^{\times})^r$ is an algebraic torus acting rationally on $A$ by $k$-algebra automorphisms.
\end{itemize}
\end{conditions}
\begin{remark}\label{rem:NSS}
The precise statement of the non-commutative Nullstellensatz can be found in \cite[Definition~II.7.14]{GBbook}; informally stated, it requires that every prime ideal of $A$ is an intersection of primitive ideals, and that the endomorphism rings of irreducible $A$-modules are all algebraic over $k$.  As the next theorem demonstrates, every quantum algebra of interest to us satisfies the Nullstellensatz and hence it is not a restrictive condition to assume in this context.
\end{remark}
\begin{theorem}\cite[Corollary II.7.18]{GBbook} 
The following are all examples of algebras which satisfy the non-commutative Nullstellensatz over $k$.
\begin{enumerate}
\item $k_{\mathbf{q}}[x_1, \dots, x_n]$ and its localization $k_{\mathbf{q}}[x_1^{\pm1}, \dots, x_n^{\pm1}]$;
\item $\QML{n}$, $\QGL{n}$ and $\QSL{n}$;
\item $\mathcal{O}_{\lambda, \mathbf{p}}(M_n)$, $\mathcal{O}_{\lambda, \mathbf{p}}(GL_n)$, $\mathcal{O}_{\lambda, \mathbf{p}}(SL_n)$, i.e. the multiparameter versions of (2) (see \cite[\S I.2]{GBbook}).
\end{enumerate}
\end{theorem}

Our first aim is to pick out certain ideals which are stable under the action of $\HH$, and use these to break up $spec(A)$ into more manageable pieces.
\begin{definition}\label{def:H-prime ideals}
We call an ideal $I$ a \textit{$\HH$-stable} ideal if $h(I) = I$ for all $h \in H$, and say $I$ is \textit{$\HH$-prime}\nom{H@$\mathcal{H}$-prime} if whenever $J$, $K$ are $\HH$-stable ideals such that $JK \subseteq I$ then $J \subseteq I$ or $K \subseteq I$ as well.  An algebra is called \textit{$\HH$-simple}\nom{H@$\mathcal{H}$-simple} if it admits no non-trivial $\HH$-primes.
\end{definition}
Denote the set of $\HH$-primes in $A$ by $\Hspec{A}$\nom{H@$\mathcal{H}$-$spec(A)$}.  It is clear that any $\HH$-stable prime ideal of $A$ will be $\HH$-prime; the converse is not true in general but holds under the assumption of Conditions~\ref{conditions 1} \cite[Proposition II.2.9]{GBbook}.  We will therefore treat the concepts of ``$\HH$-prime'' and ``$\HH$-stable prime'' as interchangeable in what follows.

We define the \textit{rational character group} $X(\HH)$ to be the set of rational characters of $\HH$; in the case where $\HH = (k^{\times})^r$ this is the free abelian group $\mathbb{Z}^r$ \cite[Exercise II.2.E]{GBbook}.  By \cite[Lemma II.2.11]{GBbook}, rational actions of $\HH$ on $A$ correspond to gradings of $A$ by $X(\HH)$, a fact which is used heavily in the proof of the Stratification Theorem \cite[\S II.3]{GBbook}.  This also implies that an ideal $I$ is $\HH$-stable  with respect to a given $\HH$-action if and only if it is homogeneous with respect to the induced $X(\HH)$-grading \cite[Exercise II.2.I]{GBbook}, a fact which we shall make use of in Chapter~\ref{c:H-primes}.

The set of $\HH$-primes of an algebra may be used to stratify $spec(A)$ as follows.  Let $J$ be a $\HH$-stable ideal of $A$, and define the \textit{stratum associated to $J$} by\nom{S@$spec_J(A)$}
\begin{equation}\label{eq:specJ def}
spec_J(A) = \left\{I \in spec(A) : \bigcap_{h \in \HH} h(I) = J\right\}
\end{equation}
In other words, $spec_J(A)$ is the set of prime ideals of $A$ such that $J$ is the largest $\HH$-stable ideal contained in them.  It is clear from the definition in \eqref{eq:specJ def} that $J$ must be a $\HH$-prime and that the strata associated to different $\HH$-primes will be disjoint, and so we obtain a stratification of $spec(A)$ by the $\HH$-primes as follows:
\[spec(A) = \bigsqcup_{J \in \HH spec(A) } spec_J(A).\]
Similarly, we obtain a stratification of the primitive ideals by defining\nom{P@$prim_J(A)$}
\[prim_J(A) = spec_J(A) \cap prim(A)\]
for $J \in \Hspec{A}$.

We may now state the Stratification Theorem, which gives us a way to understand the prime ideals in each stratum $spec_J(A)$.
\begin{theorem}[Stratification Theorem]\label{res:Stratification Theorem, quantum version} \cite[Theorem~II.2.13]{GBbook}

Assume Conditions~\ref{conditions 1}, and let $J \in \Hspec{A}$.  Then
\begin{enumerate}[(i)]
\item The set $\mathcal{E}_J$ of all regular $\HH$-eigenvectors in $A/J$ is a denominator set in $A/J$ (see \S\ref{s:notation} for the definition of denominator set), and the localization $A_J:=\qr{A}{J}\big[\mathcal{E}_J^{-1}\big]$ is $\HH$-simple (with respect to the induced $\HH$-action).
\item $spec_J(A)$ is homeomorphic to $spec(A_J)$ via localization and contraction, and $spec(A_J)$ is homeomorphic to $spec(Z(A_J))$ via contraction and extension.
\item The centre $Z(A_J)$ is a Laurent polynomial ring in at most $r$ variables over the fixed field $Z(A_J)^{\HH} = Z(Fract(A/J))^{\HH}$.  The inteterminates can be chosen to be $\HH$-eigenvectors with linearly independent $\HH$-eigenvalues.
\end{enumerate}
\end{theorem}
In \cite[\S3.2]{GL1}, Goodearl and Lenagan observe that the denominator set $\mathcal{E}_J$ can be replaced with a smaller subset $E_J$ without affecting the conclusions of the theorem, provided $E_J$ is also a denominator set such that the localization $\qr{A}{J}\big[E_J^{-1}\big]$ is $\HH$-simple.  For sufficiently nice algebras (such as those considered in \cite{GL1}), this allows us to compute the localizations $\qr{A}{J}\big[E_J^{-1}\big]$ explicitly by chosing denominator sets generated by finitely many normal $\HH$-eigenvectors.  In many cases (see \S\ref{ss:H primes examples} below) we will also see that $Z(A_J)^{\HH} = k$, in which case $spec_J(A)$ is homeomorphic to an affine scheme.

Using the Stratification Theorem we may describe the prime ideals of $A$ up to localization, but on its own this tells us very little about which primes are primitive.  The Dixmier-Moeglin equivalence, which was formulated originally by Dixmier and Moeglin in the context of enveloping algebras and extended to quantum algebras by Goodearl and Letzter in \cite{Hprimes2}, gives us a number of equivalent criteria for a prime ideal to be primitive: one algebraic criterion, one topological, and one formulated in terms of $\HH$-strata.  

Before we can state the Dixmier-Moeglin equivalence for quantum algebras, we require one more set of definitions.
\begin{definition}\label{def:defs for DM equiv}
Let $A$ be a Noetherian $k$-algebra.  A prime ideal $P$ in $A$ is called \textit{rational} if $Z(Fract(A/P))$ is algebraic over $k$, where $Fract(A/P)$ denotes the simple Artinian Goldie quotient ring of $A/P$.  Meanwhile, we say that $P$ is \textit{locally closed} if the singleton $\{P\}$ is is the intersection of an open set and a closed set in $spec(A)$ with respect to the Zariski toplogy.
\end{definition}
\begin{theorem}[Dixmier-Moeglin Equivalence] \label{res:DM equiv, quantum version} \cite[Theorem~II.8.4]{GBbook}

Apply the assumptions of Conditions~\ref{conditions 1}, and further assume that $\Hspec{A}$ is finite.  Let $J$ be a $\HH$-prime, and $P \in spec_J(A)$.  Then the following are equivalent:
\begin{enumerate}[(i)]
\item $P$ is a primitive ideal of $A$;
\item $P$ is locally closed in $A$;
\item $P$ is rational in $A$;
\item $P$ is a maximal element of $spec_J(A)$.
\end{enumerate}
\end{theorem}
The final condition of this theorem is of most interest to us: combined with Theorem~\ref{res:Stratification Theorem, quantum version} above, this says that $prim_J(A)$ is homeomorphic to the set of maximal ideals of a commutative Laurent polynomial ring $Z(A_J)^{\HH}[x_1^{\pm1}, \dots, x_n^{\pm1}]$ for some $n \geq 0$.  When $Z(A_J)^{\HH} = k$, this will allow us to describe the elements of $prim_J(A)$ explicitly as the pullbacks to $A$ of ideals of the form $\langle x_1 - \lambda_1, \dots, x_n - \lambda_n\rangle$, where $\lambda_i \in k^{\times}$ for $1 \leq i \leq n$.

\subsection{Stratification of Poisson algebras}\label{ss:H action Poisson version}
Since quantum algebras are often constructed as deformations of commutative coordinate rings, we might expect that the Poisson bracket induced as in Definition~\ref{def:star product Poisson bracket} on the semi-classical limit will give rise to a Poisson ideal structure mirroring the ideal structure in the quantum algebra.  And indeed, it turns out that once we define a suitable analogue of primitive ideal we can obtain a Poisson version of all the results described in the previous section, which we will summarise here.

We will assume in this section that $k$ is (as always) a field of characteristic zero and $R$ is a commutative Poisson algebra.

Recall from \S\ref{ss:background poisson algebras def} that we observed  ``prime Poisson'' and ``Poisson-prime'' ideals were equivalent notions for a Noetherian algebra in characteristic zero.  On the other hand, it is noted in \cite[Definition~1.6]{JordanOh1} that the maximal Poisson ideals (maximal ideals which are closed under the Poisson bracket) need not coincide with the Poisson-maximal ideals (ideals maximal in the set of Poisson ideals).  Since primitive ideals in a commutative algebra are precisely the maximal ones, this suggests we should take a different approach to defining a Poisson analogue of primitive ideals.  The following definition is originally due to Oh \cite[Definition~1.2]{OhSymplectic}:
\begin{definition}\label{def:poisson primitive}
Let $I$ be an ideal in a commutative Poisson algebra $R$.  Define the \textit{Poisson core} of $I$ to be the largest Poisson ideal contained in $I$; since the sum of two Poisson ideals is again a Poisson ideal, the Poisson core is uniquely defined.  We call an ideal \textit{Poisson primitive}\nom{Poisson-primitive} if it is the Poisson core of a maximal ideal.
\end{definition}
Clearly every maximal Poisson ideal is Poisson primitive, but the set of Poisson-primitive ideals in $R$ is strictly greater than the set of Poisson-maximal ideals whenever $R$ admits a maximal ideal which isn't a Poisson ideal.  By \cite[Lemma~1.3]{OhSymplectic}, every Poisson-primitive ideal is Poisson-prime.

We may also define a Poisson equivalent of the centre, namely the \textit{Poisson centre}\nom{Poisson centre}
\begin{equation}\label{eq:poisson centre def}
PZ(R) := \big\{r \in R : \{r,s\} = 0 \ \forall s \in R\big\}.
\end{equation}
Let $\HH = (k^{\times})^r$ be an algebraic torus acting on a commutative Noetherian Poisson algebra $R$ by Poisson automorphisms.  Then we may make many of the same definitions and observations for $\HH$ as we did in the quantum case above:
\begin{itemize}
\item We say $\HH$ \textit{acts rationally} if $R$ is the direct sum of its eigenspaces and all the eigenvalues of $\HH$ are rational, i.e. morphisms of algebraic varieties. 
\item By replacing ``prime'' with ``Poisson prime'' in Definition~\ref{def:H-prime ideals} we obtain the notion of \textit{Poisson $\HH$-prime}\nom{Poisson $\mathcal{H}$-prime}; if $\HH$ acts rationally on a commutative Noetherian Poisson algebra $R$, then Poisson $\HH$-primes are equivalent to Poisson ideals which are stable under the action of $\HH$ and prime in the conventional commutative sense (see \cite[Lemma~3.1]{Goodearl_Poisson}).
\item Denote the set of Poisson prime ideals in $R$ by $Pspec(R)$\nom{P@$Pspec(R)$} and the set of Poisson primitive ideals by $PPrim(R)$\nom{P@$PPrim(R)$}, and equip both with the Zariski topology as in \S\ref{s:notation}.
\item We call $R$ a \textit{Poisson $\HH$-simple}\nom{Poisson $\mathcal{H}$-simple} algebra if it has no non-trivial Poisson $\HH$-primes.
\end{itemize}
As in \S\ref{ss:H action quantum version}, we may use the Poisson $\HH$-primes to stratify $Pspec(R)$ and $PPrim(R)$ by making the following definitions: for $J \in \PHspec{R}$, we define\nom{P@$Pspec_J(R)$}\nom{P@$Pprim_J(R)$}
\begin{gather*}
Pspec_J(R) = \{P \in Pspec(R) : \bigcap_{h \in \HH} h(P) = J\}, \\
Pprim_J(R) = Pspec_J(R) \cap Pprim(R),
\end{gather*}
so that we obtain partitions of $Pspec(R)$ and $Pprim(R)$ as follows:
\begin{equation*}
Pspec(R) = \bigsqcup_{J \in \HH Pspec(R)} Pspec_J(R), \qquad Pprim(R) = \bigsqcup_{J \in \HH Pspec(R)} Pprim_J(R).
\end{equation*}
We may now state the Poisson versions of the Stratification Theorem and Dixmier-Moeglin equivalence.
\begin{theorem}[Stratification Theorem for Poisson algebras]\label{res: stratification theorem, Poisson version} \cite[Theorem~4.2]{Goodearl_Poisson}

Let $R$ be a Noetherian Poisson $k$-algebra, with $\HH = (k^{\times})^r$ an algebraic torus acting rationally on $R$ by Poisson automorphisms.  Let $J$ be a Poisson $\HH$-prime of $R$, and let $\mathcal{E}_J$ be the set of all regular $\HH$-eigenvectors in $R/J$.  Then
\begin{enumerate}[(i)]
\item $Pspec_J(R)$ is homeomorphic to $Pspec(R_J)$ via localization and contraction, where $R_J := \qr{R}{J}\big[\mathcal{E}_J^{-1}\big]$;
\item $Pspec(R_J)$ is homeomorphic to $spec(PZ(R_J))$ via contraction and extension;
\item $PZ(R_J)$ is a Laurent polynomial ring in at most $r$ indeterminates over the fixed field $PZ(R_J)^{\HH} = PZ(Fract(R/J))^{\HH}$.  The indeterminates  can be chosen to be $\HH$-eigenvectors with $\mathbb{Z}$-linearly independent $\HH$-eigenvalues.
\end{enumerate}
\end{theorem}
As in the quantum version of the Stratification Theorem, we may replace $\mathcal{E}_J$ by a subset $E_J$ provided the localization $\qr{R}{J}\big[E_J^{-1}\big]$ remains Poisson $\HH$-simple (a proof of this is given in \S\ref{s:poisson primitive ideals}).

\begin{theorem}[Dixmier-Moeglin equivalence for Poisson algebras]\label{res: dixmier moeglin, Poisson version} \cite[Theorem~4.3]{Goodearl_Poisson}

Let $R$ be an affine Poisson $k$-algebra, and $\HH = (k^{\times})^r$ acting rationally on $R$ by Poisson automorphisms.  Assume that $R$ has only finitely many Poisson $\HH$-primes, and let $J$ be one of them.  For $P \in Pspec_J(R)$, the following conditions are equivalent:
\begin{enumerate}[(i)]
\item $P$ is locally closed in $Pspec(R)$;
\item $P$ is Poisson primitive;
\item $PZ(Fract(R/P))$ is algebraic over $k$;
\item $P$ is maximal in $Pspec_J(R)$.
\end{enumerate}
\end{theorem}
As we will see in the next section, for appropriately-chosen pairs of quantum and Poisson algebras the similarities between the two theories can in fact extend even further than this.

\subsection{Examples and a conjecture}\label{ss:H primes examples}
We begin this section by defining some specific examples of group actions on quantum and Poisson algebras which are covered by the framework of \S\ref{ss:H action quantum version} and \S\ref{ss:H action Poisson version}.

Let $\mathcal{O}_{\mathbf{q}}(k^n) = k_{\mathbf{q}}[x_1, \dots, x_n]$ be the quantum affine space defined in \S\ref{s:notation}.  Then we may define an action of $\HH = (k^{\times})^r$ on $\mathcal{O}_{\mathbf{q}}(k^n)$ by
\begin{equation}\label{eq:action of H on quantum affine space}h = (\alpha_1, \dots, \alpha_n) \in \HH, \quad h.x_i = \alpha_i x_i.\end{equation}
On $\QML{n}$ we may define an action of $\HH = (k^{\times})^{2n}$ by
\begin{equation}\label{eq:action of H on M_n}h = (\alpha_1, \dots, \alpha_n, \beta_1, \dots, \beta_n) \in \HH, \quad h.X_{ij} = \alpha_i\beta_jX_{ij}.\end{equation}
It is clear from the definition of the quantum determinant in \eqref{eq:quantum det def} that $Det_q$ is an eigenvector for this action, and hence \eqref{eq:action of H on M_n} extends uniquely to an action on $\QGL{n}$ as well \cite[II.1.15]{GBbook}.  The action \eqref{eq:action of H on M_n} does not descend immediately to an action on $\QSL{n}$, however, since $Det_q - 1$ is not an eigenvector.  Instead, we define an action of the subgroup
\[\HP = \{(\alpha_1, \dots, \alpha_n, \beta_1, \dots, \beta_n) \in (k^{\times})^{2n}: \alpha_1\alpha_2\dots\beta_n = 1\}\]
on $\QSL{n}$ in the natural way, by defining
\[h = (\alpha_1, \dots, \alpha_n, \beta_1, \dots, \beta_n) \in \HP, \quad h.X_{ij} = \alpha_i\beta_jX_{ij}\]
as before, where $X_{ij}$ now denotes generators in $\QSL{n}$.  By \cite[II.1.16]{GBbook}, $\HP$ is isomorphic to the torus $(k^{\times})^{2n-1}$.

There are of course many more quantum algebras than the four types considered here, including multiparameter versions of each of the above algebras, and these are discussed further in \cite{GBbook}.  We now focus in more detail on the properties of the above algebras and their $\HH$-actions.

Let $A$ denote any one of the algebras $\mathcal{O}_{\mathbf{q}}(k^n)$, $\QML{n}$, $\QGL{n}$ and $\QSL{n}$, and let $\HH$ be the corresponding algebraic torus acting on $A$ as defined above.  Then $A$ satisfies a number of useful properties:
\begin{itemize}
\item $A$ is Noetherian and satisfies the non-commutative Nullstellensatz over $k$ \cite[Corollary II.7.18]{GBbook};
\item The action of $\HH$  defined above is a rational action on $A$ \cite[II.2.12]{GBbook};
\item $A$ has finitely many $\HH$-primes \cite[Theorem~II.5.14, II.5.17]{GBbook};
\item All prime ideals of $A$ are completely prime \cite[Corollary~II.6.10]{GBbook};
\item $Z(A_J)^{\HH} = k$ for any $J \in \Hspec{A}$ \cite[Corollary~II.6.5, II.6.6]{GBbook}.
\end{itemize}
These algebras are therefore covered by the framework of the Stratification Theorem and Dixmier-Moeglin equivalence outlined in \S\ref{ss:H action quantum version}.

We may also consider the semi-classical limits of these algebras.  The semi-classical limit of $k_{\mathbf{q}}[x_1, \dots, x_n]$ is the polynomial ring in $n$ variables, which we denote by $\mathcal{O}(k^n)$.  It has a Poisson bracket given by
\[\{x_i,x_j\} = a_{ij}x_ix_j,\]
where $a_{ij}$ is the $(i,j)$th entry of the matrix $\mathbf{q}$.  Moreover the action of $\HH = (k^{\times})^n$ on $\mathcal{O}_{\mathbf{q}}[x_1, \dots, x_n]$ defined in \eqref{eq:action of H on quantum affine space} induces an action of $\HH$ by Poisson automorphisms on $\mathcal{O}(k^n)$ \cite[\S2.2]{GLa1}.

The semi-classical limit of $\QML{n}$ is $\ML{n}$, which is isomorphic to the polynomial ring on the $n^2$ variables $\{x_{ij}: 1 \leq i,j \leq n\}$.  The Poisson bracket induced on $\ML{n}$ by this process satisfies the property that for any set of four variables $\{x_{ij},x_{im},x_{lj},x_{lm}\}$, the subalgebra of $\ML{n}$ generated by them is Poisson-isomorphic to $\ML{2}$ (as defined in \eqref{eq:poisson relns for 2x2 matrices}).  By \cite[\S2.3]{GLa1}, the action of $\HH = (k^{\times})^{2n}$ defined in \eqref{eq:action of H on M_n} induces an action of $\HH$ by Poisson automorphisms on $\ML{n}$.  

Finally, since the Poisson bracket on $\ML{n}$ induces a unique Poisson bracket on a localization or quotient as described in \S\ref{ss:background poisson algebras def}, we obtain the semi-classical limits $\GL{n}$ and $\SL{n}$ as the localization $\GL{n} = \ML{n}[Det^{-1}]$ (where the determinant $Det$ may be obtained by setting $q=1$ in \eqref{eq:quantum det def}) and the quotient $\SL{n} = \ML{n}/\langle Det-1 \rangle$.  The actions of algebraic tori on $\QGL{n}$ and $\QSL{n}$ defined above induce actions by Poisson automorphisms on $\GL{n}$ and $\SL{n}$.

\begin{example}\label{ex:gl2 example}
Let $\HH = (k^{\times})^4$ act rationally on $\QGL{2}$ by $k$-algebra automorphisms as defined in \eqref{eq:action of H on M_n}, and let $J$ be a $\HH$-prime ideal in $\QGL{2}$.  If $X_{11} \in J$ or $X_{22} \in J$, then since $X_{11}X_{22}-X_{22}X_{11} = (q-q^{-1})X_{12}X_{21}$ we have $X_{12}X_{21}\in J$ as well; as noted above, all prime ideals of $\QGL{2}$ are completely prime and hence $X_{12} \in J$ or $X_{21} \in J$.  Similarly, if for example $\langle X_{11},X_{12}\rangle \subseteq J$ then $Det_q = X_{11}X_{22} - qX_{12}X_{21} \in J$ and so $J = \QGL{2}$.  Continuing in this manner (for details, see \cite[Example~II.2.14]{GBbook}) we find that $\QGL{2}$ admits only four $\HH$-primes, namely:
\begin{equation*}\label{eq:list of H primes in qGL2}
0, \quad \langle X_{12} \rangle, \quad \langle X_{21} \rangle, \quad \langle X_{12},X_{21} \rangle.
\end{equation*}
Meanwhile, let $\HH = (k^{\times})^4$ act on $\GL{2}$ by Poisson automorphisms as described above.  Since $\{x_{11},x_{22}\} = 2x_{12}x_{21}$, a similar line of reasoning to the above yields precisely four Poisson $\HH$-primes in $\GL{2}$:
\[0, \quad \langle x_{12} \rangle, \quad \langle x_{21} \rangle, \quad \langle x_{12},x_{21} \rangle.\]
Goodearl observes in \cite[\S9.8]{GoodearlSummary} that with a bit more work we in fact obtain a homeomorphism from $spec(\QGL{2})$ to $Pspec(\GL{2})$, which restricts to a homeomorphism of primitive/Poisson-primitive ideals as well.
\end{example}

In \cite{GoodearlSummary}, Goodearl makes the following conjecture:
\begin{conjecture}\label{conj:goodearl} \cite[Conjecture~9.1]{GoodearlSummary}
Let $k$ be an algebraically closed field of characteristic zero, and $A$ a generic quantized coordinate ring of an affine variety $V$ over $k$.  Then $A$ should be a member of a flat family of $k$-algebras with semiclassical limit $\mathcal{O}(V)$, and there should be compatible homeomorphisms $prim(A) \rightarrow P.prim(\mathcal{O}(V))$ and $spec(A) \rightarrow P.spec(\mathcal{O}(V))$.
\end{conjecture}
This conjecture has so far only been verified for a few algebras, foremost among them the quantum affine spaces $k_{\mathbf{q}}[x_1, \dots, x_n]$ and similar spaces \cite{GLz2}.  For quantum matrices, the conjecture can be verified by direct computation for $\QGL{2}$ (as in Example~\ref{ex:gl2 example} above) and $\QSL{2}$ (e.g. \cite{HL}); meanwhile, in \cite[Theorem~2.12]{OhSymplectic} Oh constructs an explicit bijection between $spec(\QML{2})$ and $Pspec(\ML{2})$ but does not verify that it is a homeomorphism.  A general technique to handle $n \times n$ matrices has not yet been discovered.

In \cite{GL2}, Goodearl and Lenagan describe generating sets for the 230 $\HH$-primes of $\QML{3}$, of which 36 continue to be $\HH$-primes in $\QGL{3}$ and $\QSL{3}$ \cite[\S3]{GL2}.  In \cite{GL1}, they build on this to give explicit generators for all of the primitive ideals in $\QGL{3}$ and $\QSL{3}$; further, these generating sets are described in the algebra itself rather than generators in some localization.

In Chapter~\ref{c:H-primes} we perform a similar analysis to \cite{GL1} for the Poisson algebras $\GL{3}$ and $\SL{3}$, thus laying the groundwork for verifying Conjecture~\ref{conj:goodearl} in these cases in the future.

\chapter{The $\lowercase{q}$-Division Ring: Fixed Rings and Automorphism Group}\label{c:fixed_rings_chapter}
\chaptermark{The $\lowercase{q}$-Division Ring}
As described in Chapter~\ref{c:introduction}, the $q$-division ring $D = k_q(x,y)$ features in Artin's conjectured classification of finitely-generated division rings of transcendence degree 2, and hence so do subrings of finite index within $D$.  One way of constructing such subrings is to consider fixed rings of $D$ under finite groups of automorphisms, and in this chapter we describe the structure of many rings of this type: in \S\ref{s:fixed ring} we consider the fixed ring of an automorphism of order 2 which does not restrict to an automorphism of $k_q[x^{\pm1},y^{\pm1}]$, and in \S\ref{s:more fixed rings} we describe the fixed rings under finite groups of monomial automorphisms.

While the fixed ring $D^{\tau}$ for $\tau: x \mapsto x^{-1},\ y \mapsto y^{-1}$ was originally described in \cite[Example~13.6]{SV1} using techniques from non-commutative algebraic geometry, we take a more ring-theoretic approach and construct pairs of $q$-commuting generators for each of the fixed rings under consideration.  This approach is possible since we can reduce the question to  three specific cases: the automorphism $\tau$ above, a monomial automorphism of order 3 and a non-monomial automorphism of order 2 defined in \eqref{eq:def_varphi} below.  In each case we will describe an explicit isomorphism between $D$ and its fixed ring.

The study of these fixed rings leads naturally on to an examination of the automorphisms and endomorphisms of $D$, since both involve looking for $q$-commuting pairs of elements satisfying certain properties.  In \S\ref{s:automorphism_consequences} we show that the structure of $Aut(D)$ is far from easy to understand: indeed, as we show in Proposition~\ref{res:z_endomorphism} we can no longer rely on conjugation maps to necessarily even be bijective!  We also construct examples of conjugation automorphisms on $D$ by elements $z \in k_q(y)(\!(x)\!) \backslash D$ (see Proposition~\ref{res:z_automorphism_order_2} and Proposition~\ref{res:z_automorphism_inf_order}), thus raising the possibility that $D$ may admit wild automorphisms.

These examples of conjugation maps were constructed in order to answer several open questions posed by Artamonov and Cohn in \cite{AC1} concerning the structure of $Aut(D)$; in particular, we show that their conjectured set of generators for the automorphism group in fact generate the whole endomorphism group $End(D)$.  These results indicate that the structure of $Aut(D)$ is an interesting question to study in its own right, and we list a number of new open questions arising from the results of this chapter in \S\ref{s:open questions on D}.

Excluding Proposition~\ref{res:action of SL2, q version} and Lemma~\ref{res:H1 and H2 are equal}, this chapter has also appeared in the Journal of Algebra as \textit{The $q$-Division Ring and its Fixed Rings}; see \cite{ME1}.  As per Global Convention~\ref{q is not a root of unity for now and forever}, we continue to assume that $k$ has characteristic zero and $q \in k^{\times}$ is not a root of unity.

\section{An automorphism of $k_q(x,y)$ and its fixed ring}\label{s:fixed ring}
In Theorem~\ref{res:AD_prop_intro} we saw that for finite groups of automorphisms $G$ defined on the quantum plane $k_q[x,y]$ and extended to $D$, the fixed ring $D^G$ is isomorphic to $k_p(x,y)$ for $p = q^{|G|}$, while for the automorphism $\tau$ above the fixed ring $D^{\tau}$ is isomorphic to $D$ with the same value of $q$ (Theorem~\ref{res:order_2_monomial_result}).  This suggests that the structure of the fixed ring may be related to which subrings of $D$ the automorphisms can be restricted to, and raises the natural question: can we define an automorphism of finite order on $D$ which does not restrict to an automorphism of $k_q[x^{\pm1},y^{\pm1}]$, and what is the structure of its fixed ring?

This section provides the answer to this question for one such automorphism of $D$, which is defined as follows:
\begin{equation}\label{eq:def_varphi}\varphi: x \mapsto (y^{-1}-q^{-1}y)x^{-1}, \quad y \mapsto -y^{-1}.\end{equation}
Since $x$ only appears once in the image, it is easy to see that these images $q$-commute and so this is a homomorphism.  We can also easily check that it has order 2, since
\begin{align*}
\varphi^2(x) &= \varphi\Big((y^{-1}-q^{-1}y)x^{-1}\Big) \\
&= (-y+q^{-1}y^{-1})x(y^{-1}-q^{-1}y)^{-1} \\
&= (q^{-1}y^{-1}-y)(q^{-1}y^{-1}-y)^{-1}x \\
&= x
\end{align*}
and it is therefore an automorphism on $D$. The aim of this section is to prove the following result.
\begin{theorem}\label{res:epic_theorem} Let $G$ be the group generated by $\varphi$.  Then $D^G \cong D$ as $k$-algebras.\end{theorem}

Before tackling the proof of this theorem, we will need some subsidiary results.

Recall that the algebra generated by two elements $u$, $v$ subject to the relation $uv-qvu = \lambda$ (for some $\lambda \in k^{\times}$) is called the \textit{quantum Weyl algebra}.  This ring also has a full ring of fractions, which can be seen to be equal to $D$ by sending $u$ to the commutator $uv-vu$ \cite[Proposition~3.2]{AD2}.

We will construct a pair of elements in $D^G$ which satisfy a quantum Weyl relation and show that they generate the fixed ring.  A simple change of variables then yields the desired isomorphism.

In order to simplify the notation, set $\Lambda = y^{-1} - q^{-1}y$.  Inspired by \cite{division} and \cite[\S13.6]{SV1}, we define our generators using a few simple building blocks.  We set
\begin{equation}\begin{gathered}\label{eq:abc_defs}
a= x - \Lambda x^{-1},\qquad  b= y+y^{-1}, \qquad c= xy + \Lambda x^{-1}y^{-1} \\
 h = b^{-1}a, \qquad g = b^{-1}c\end{gathered}
\end{equation}
and verify that $h$ and $g$ satisfy the required properties.
\begin{lemma} 
The elements $h$ and $g$ are fixed by $\varphi$ and satisfy the relation 
\[hg-qgh = 1-q.\]\end{lemma}

\begin{proof}
The first statement is trivial, since $\varphi$ acts on $a$, $b$ and $c$ as multiplication by $-1$.

After multiplying through by $b$, we see that the equality $hg - qgh = 1-q$ is equivalent to 
\[ab^{-1}c - qcb^{-1}a = (1-q)b\]
which allows us to verify it by direct computation.  Indeed,
\begin{equation}\label{eq:abc}\begin{aligned}
 ab^{-1}c &= (x-\Lambda x^{-1})(y+y^{-1})^{-1}(xy + \Lambda x^{-1}y^{-1}) \\
&= \Big((qy + q^{-1}y^{-1})^{-1}x - \Lambda(q^{-1}y + qy^{-1})^{-1}x^{-1}\Big)\big(xy + \Lambda x^{-1}y^{-1}\big) \\
&= q^2y(qy + q^{-1}y^{-1})^{-1}x^2 + \alpha(\Lambda)(qy + q^{-1}y^{-1})^{-1}y^{-1} \\
& \quad \qquad - \Lambda(q^{-1}y + qy^{-1})^{-1}y - q^2y^{-1}\Lambda \alpha^{-1}(\Lambda)(q^{-1}y + qy^{-1})^{-1}x^{-2}
\end{aligned}\end{equation}
\begin{equation}\label{eq:qcba}\begin{aligned}
qcb^{-1}a &= q(xy + \Lambda x^{-1}y^{-1})(y+y^{-1})^{-1}(x-\Lambda x^{-1}) \\
&= q\Big(qy(qy + q^{-1}y^{-1})^{-1}x + q\Lambda y^{-1}(q^{-1}y + qy^{-1})^{-1}x^{-1}\Big)\big(x-\Lambda x^{-1}\big) \\
&= q^2y(qy+q^{-1}y^{-1})^{-1}x^2 - q^2y\alpha(\Lambda)(qy + q^{-1}y^{-1})^{-1} \\
&\quad \qquad + q^2\Lambda y^{-1}(q^{-1}y + qy^{-1})^{-1} - q^2y^{-1}\Lambda \alpha^{-1}(\Lambda)(q^{-1}y + qy^{-1})^{-1} x^{-2}
\end{aligned}\end{equation}
Putting these together, we see that the terms in $x^2$ and $x^{-2}$ cancel out, leaving us with
\begin{align*}
 ab^{-1}c - qcb^{-1}a &= \alpha(\Lambda)(qy + q^{-1}y^{-1})^{-1}(y^{-1} + q^2y) \\
& \quad \qquad - \Lambda(q^{-1}y + qy^{-1})^{-1}(y+q^2y^{-1}) \\
&= \alpha(\Lambda)q - \Lambda q \\
&= (q^{-1}y^{-1} - y)q - (y^{-1}-q^{-1}y)q\\
&=(1-q)b \qedhere
\end{align*}\end{proof}
Let $R$ be the division ring generated by $h$ and $g$; it is a subring of $D^G$, and the next step is to show that these two rings are actually equal.  We can do this by checking that $[D : R] = 2$, since $R \subseteq D^G \subsetneq D$ will then imply $R = D^G$.  

\begin{lemma}\label{res:computation_lemma}
\begin{enumerate}[(i)]
 \item The following elements are all in $R$:
 \[x + \Lambda x^{-1}, \quad y - y^{-1}, \quad xy - \Lambda x^{-1}y^{-1}.\]
\item Let $b = y+y^{-1}$ as in \eqref{eq:abc_defs}.  Then $b^2 \in R$ and $R\langle b \rangle = D$.
\end{enumerate}
\end{lemma}
\begin{proof} $(i)$ 
 We begin by proving directly that $y-y^{-1} \in R$, as the others will follow easily from this.  Indeed, we will show that
\begin{equation*}y-y^{-1} = (hg-1)^{-1}(qg^2-h^2).\end{equation*}
Using the definitions in \eqref{eq:abc_defs} this is equivalent to checking that
\begin{equation}\label{eq:computation_2}(ab^{-1}c - b)(y-y^{-1}) = qcb^{-1}c - ab^{-1}a.\end{equation}
Expanding out the components on the right in \eqref{eq:computation_2}, we get
\begin{equation*}\begin{aligned}qcb^{-1}c &= q(xy + \Lambda x^{-1}y^{-1})(y+y^{-1})^{-1}(xy + \Lambda x^{-1}y^{-1}) \\
&= \Big (q^2y(qy + q^{-1}y^{-1})^{-1}x + q^2y^{-1}\Lambda (qy^{-1} + q^{-1}y)^{-1}x^{-1}\Big)(xy + \Lambda x^{-1}y^{-1}) \\
&= q^2(qy + q^{-1}y^{-1})^{-1}yx^2y + q^2\alpha(\Lambda)(qy + q^{-1}y^{-1})^{-1} \\
& \quad \qquad+ q^2\Lambda(qy^{-1} + q^{-1}y)^{-1} + q^2\Lambda \alpha^{-1}(\Lambda)(qy^{-1} + q^{-1}y)^{-1}y^{-1}x^{-2}y^{-1} \\
& \\
ab^{-1}a &= (x-\Lambda x^{-1})(y+y^{-1})^{-1}(x-\Lambda x^{-1}) \\
&= \Big((qy + q^{-1}y^{-1})^{-1}x - \Lambda (q^{-1}y + qy^{-1})^{-1}x^{-1}\Big)(x-\Lambda x^{-1}) \\
&= q^2(qy + q^{-1}y^{-1})^{-1}yx^2y^{-1} - \alpha(\Lambda)(qy + q^{-1}y^{-1})^{-1} \\
&\quad \qquad - \Lambda(q^{-1}y + qy^{-1})^{-1} + q^2\Lambda \alpha^{-1}(\Lambda)(q^{-1}y + qy^{-1})^{-1}y^{-1} x^{-2}y
\end{aligned}\end{equation*}
so that the difference $qcb^{-1}c - ab^{-1}a$ is
\begin{equation}\begin{aligned}\label{eq:computation_RHS}
qcb^{-1}c - ab^{-1}a &= q^2y(qy + q^{-1}y^{-1})^{-1}x^2(y-y^{-1}) \\
&\quad \qquad  -q^2\Lambda\alpha^{-1}(\Lambda)y^{-1}(qy^{-1} + q^{-1}y)^{-1}x^{-2}(y-y^{-1})\\
&\quad \qquad + (q^2+1)\Big(\alpha(\Lambda)(qy + q^{-1}y^{-1})^{-1} + \Lambda(qy^{-1} + q^{-1}y)^{-1}\Big)
\end{aligned}\end{equation}
Meanwhile, using the expression for $ab^{-1}c$ obtained in \eqref{eq:abc}, we find that
\begin{equation}\begin{aligned}\label{eq:computation_LHS}
& (ab^{-1}c - b)(y-y^{-1}) =\\
&\qquad \qquad q^2y(qy + q^{-1}y^{-1})^{-1}x^2(y-y^{-1}) \\
& \qquad \qquad \quad - q^2\Lambda \alpha^{-1}(\Lambda)y^{-1}(q^{-1}y + qy^{-1})^{-1}x^{-2}(y-y^{-1}) \\
&\, \qquad \qquad \quad+ \Big(\alpha(\Lambda)(qy + q^{-1}y^{-1})^{-1}y^{-1} - \Lambda(q^{-1}y + qy^{-1})^{-1}y - y - y^{-1}\Big)(y-y^{-1}) \\
\end{aligned}\end{equation}
Comparing the expressions in \eqref{eq:computation_RHS} and \eqref{eq:computation_LHS} it is immediately clear that the terms involving $x^2$ and $x^{-2}$ are equal.  This leaves just the terms involving only powers of $y$ to check; each of these are elements of $k(y)$ and therefore commutative, so it is now a simple computation to check that both expressions reduce to the form
\begin{equation*}
 \frac{(1+q)(y-y^{-1})(y+y^{-1})(q+q^{-1})}{(qy + q^{-1}y^{-1})(qy^{-1} + q^{-1}y)}.
\end{equation*}
Thus $(ab^{-1}c - b)(y-y^{-1}) = qcb^{-1}c - ab^{-1}a$ as required.  This proves that $y-y^{-1} \in R$.

Inside $D$ we can notice that
\begin{equation}\begin{aligned}\label{eq:formula_for_x}
 y^{-1}h + q^{-1}g &= (y+y^{-1})^{-1}\left(y^{-1}(x-\Lambda x^{-1}) + q^{-1}(xy + \Lambda x^{-1}y^{-1})\right) \\ 
&= (y+y^{-1})^{-1}\left(y^{-1}x - \Lambda y^{-1}x^{-1} + yx + \Lambda y^{-1}x^{-1}\right) \\ 
&= (y+y^{-1})^{-1}(y^{-1} + y)x \\
&= x.\end{aligned}
\end{equation}
and so $\Lambda x^{-1} = \varphi(y^{-1}h + q^{-1}g) = q^{-1}g - yh$.  Putting these together we get
\[x+\Lambda x^{-1} = (y^{-1}-y)h + 2q^{-1}g \in R\]
and similarly,
\[xy - \Lambda x^{-1}y^{-1} = 2qh + (y-y^{-1})g \in R.\]

$(ii)$ It's clear that $b = y+y^{-1} \not\in R$ since $R$ is a subring of $D^{\varphi}$ and $b$ is not fixed by $\varphi$.  However, $(y+y^{-1})^2 = (y-y^{-1})^2 + 4$, hence $b^2 \in R$ by (i).

To prove that $R\langle b \rangle  = D$ it is enough to show that $x, y \in R\langle b \rangle$.  This is now clear, however, since $y = \frac{1}{2}(y-y^{-1}) + \frac{1}{2}(y+y^{-1}) \in R \langle b \rangle$, hence by \eqref{eq:formula_for_x}, $x = y^{-1}h + q^{-1}g \in R\langle b \rangle$ as well.
\end{proof}
Since we are working with fixed rings, the language of Galois theory is a natural choice to use here, and in \cite[\S3.6]{skewfields} we find conditions for a quotient of a general Ore extension $R[u;\gamma, \delta]/(u^2+\lambda u  + \mu)$ to be a quadratic division ring extension of $R$.  (Note that the language of \cite{skewfields} is that of \textit{right} Ore extensions, so we make the necessary adjustments below to apply the results to left extensions.)

When char $k$ $\neq 2$, such an extension will be Galois if and only if $\delta$ is inner \cite[Theorem~3.6.4(i)]{skewfields} so here it is sufficient to only consider the case when $\delta = 0$.  Further, since $b^2 \in R$ by Lemma~\ref{res:computation_lemma} (ii), we see that $b$ satisfies a quadratic equation over $R$ with $\lambda = 0$, which allows us to simplify matters even further.

The next result is a special case of \cite[Theorem~3.6.1]{skewfields}, which by the above discussion is sufficient for our purposes.
\begin{proposition}\label{res:gth_prop}
Let $K$ be a division ring, $\gamma$ an endomorphism on $K$ and $\mu \in K^{\times}$.  The ring $T:= K[u; \gamma]/(u^2+\mu)$ is a quadratic division ring extension of $K$ if and only if T has no zero-divisors and $\mu$, $\gamma$ satisfy the following two conditions:
\begin{enumerate}
\item $\mu r =  \gamma^2(r)\mu$ for all $r \in K$;
\item $\gamma(\mu) = \mu$.
\end{enumerate}
\end{proposition} 
\begin{proof}
By \cite[Theorem~3.6.1]{skewfields} and replacing right Ore extensions with left, the ring $K[u;\gamma, \delta]/(u^2 + \lambda u + \mu)$ is a quadratic division ring extension of $K$ if and only if it contains no zero divisors and $\gamma$, $\delta$, $\lambda$ and $\mu$ satisfy the equalities
\begin{equation*}\label{eq:gth_conditions_1}\begin{aligned}
\gamma\delta(r) + \delta\gamma(r) &= \gamma^2(r)\lambda - \lambda\gamma(r),\\
\delta^2(r) +\lambda\delta(r) &=  \gamma^2(r)\mu - \mu r, \\
\delta(\lambda) &= \mu - \gamma(\mu) - (\lambda - \gamma(\lambda))\lambda, \\
\delta(\mu) &= (\lambda - \gamma(\lambda))\mu.
\end{aligned}\end{equation*}
Once we impose the conditions $\delta = 0$, $\lambda = 0$ the result follows immediately.
\end{proof}
Viewing $R$ as a subring of $D$, we can set $u = b$, $\mu = -b^2$.  The following choice of $\gamma$ is suggested by \cite{division}.
\begin{lemma}\label{res:gamma_def}
Let $b$, $h$ and $g$ be as defined in \eqref{eq:abc_defs}, and $R$ the division ring generated by $h$ and $g$ inside $D$.  Then the conjugation map defined by
\begin{equation*}
\gamma(r) = brb^{-1} \quad \forall r \in R
\end{equation*}
is a well-defined automorphism on $R$.
\end{lemma}
\begin{proof}
It is sufficient to check that the images of the generators of $R$ under $\gamma$ and $\gamma^{-1}$ are themselves in $R$, i.e. that
\begin{align*}
\gamma(h) &= (ab)b^{-2} &\gamma(g) = (cb)b^{-2} \\
 \gamma^{-1}(h) &= b^{-2}(ab) &\gamma^{-1}(g) = b^{-2}(cb)
\end{align*}
are all in $R$.

By Lemma \ref{res:computation_lemma} (ii) we already know that $b^2  \in R$.  As for $ab$ and $cb$, they decompose into elements of $R$ as follows:
\begin{equation}\begin{aligned}\label{eq:formulas_cb}
 ab &= (x-\Lambda x^{-1})(y+y^{-1}) \\
&= xy + xy^{-1} - \Lambda x^{-1}y - \Lambda x^{-1}y^{-1} \\
&= 2(xy - \Lambda x^{-1}y^{-1}) - (x+\Lambda x^{-1})(y-y^{-1})\in R\\
&\\
 cb &= (xy + \Lambda x^{-1}y^{-1})(y+y^{-1}) \\
&= xy^2 + x + \Lambda x^{-1} + \Lambda x^{-1}y^{-2} \\
&= (xy - \Lambda x^{-1}y^{-1})(y-y^{-1}) + 2(x+\Lambda x^{-1}) \in R
\end{aligned}\end{equation}
by Lemma~\ref{res:computation_lemma} (i).  Therefore $\gamma$ is a well-defined bijection on $R$, and since conjugation respects the relation $hg - qgh = 1-q$, it is an automorphism on $R$.\end{proof}
We are now in a position to prove Theorem~\ref{res:epic_theorem}.  

Recall that $R \subseteq D^G$ is a division ring with generators $h$ and $g$, which satisfy a quantum Weyl relation $hg - qgh = 1-q$.  We can make a change of variables $h \mapsto \frac{1}{1-q}(hg-gh)$ so that $R$ has the structure of a $q$-division ring \cite[Proposition~3.2]{AD2}.  (The only exception is when $q=1$, where this change of variables does not make sense; however, since $h$ and $g$ already ``$q$-commute'' in this case we can simply set $f:=h$.)

Define the automorphism $\gamma$ as in Lemma~\ref{res:gamma_def} and set $\mu := -b^2 \in R$.  The extension $L :=R[b;\gamma]/(b^2+\mu)$ is a subring of the division ring $D$, and therefore has no zero divisors.  Further,
\[ \gamma^2(r) \mu = -(b^2rb^{-2})b^2 = -b^2r = \mu r \quad \forall r \in R\]
and similarly $\gamma(\mu) = \mu$.  Therefore by Proposition~\ref{res:gth_prop}, $L$ is a quadratic division ring extension of $R$.  Since it is a subring of $D$ containing both $R$ and $b$, by Lemma~\ref{res:computation_lemma} (ii) we can conclude that $L = D$.

Now since $R \subseteq D^G \subsetneq D = L$, and the extension $R \subset L$ has degree 2, we must have $R = D^G$ and Theorem~\ref{res:epic_theorem} is proved.

\section{Fixed rings of monomial automorphisms}\label{s:more fixed rings}
Theorem~\ref{res:epic_theorem} came about as a result of a related question, namely: if we take an automorphism of finite order defined on $k_q[x^{\pm1},y^{\pm1}]$ and extend it to $D$, what does its fixed ring look like?  

As discussed in \cite[\S4.1.1]{dumas_invariants}, the automorphism group of $k_q[x^{\pm1},y^{\pm1}]$ is generated by automorphisms of scalar multiplication and the monomial automorphisms (see Definition~\ref{def:def_monomial_autos} below).  Since the case of scalar multiplication has been covered in Theorem~\ref{res:AD_prop_intro}, in this section we will focus on monomial automorphisms with the aim of proving Theorem~\ref{res:thm_monomial_results}.

For the remainder of this section we will assume that $k$ contains a square root of $q$, denoted by $\hat{q}$\nom{Q@$\hat{q}$}.  The following result appeared originally in \cite[\S5.4.2]{BaudryThesis} for the case of $k_q[x^{\pm1},y^{\pm1}]$; since we have exchanged the roles of $x$ and $y$ and extended the result to $D$ we provide a full proof of the result here. 
\begin{notation}\label{not:exponents of q}
In order to make the computations more readable in the following proposition, we define the notation
\[\expq{m} := \hat{q}^m.\]
\end{notation}
\begin{proposition}\label{res:action of SL2, q version}
The group $SL_2(\mathbb{Z})$ acts by algebra automorphisms on the $q$-division ring $D$.  The action is defined by
\begin{equation}\label{eq:q action of SL2 def}g.y=\expq{ac}y^ax^c, \quad g.x = \expq{bd}y^bx^d, \quad g =\begin{pmatrix}a&b\\c&d\end{pmatrix} \in SL_2(\mathbb{Z}),\end{equation}
or more generally for any $m,n \in \mathbb{Z}$:
\begin{equation}\label{eq:q action of SL2, general form}g.(y^mx^n) = \expq{(am+bn)(cm+dn)-mn}y^{am+bn}x^{cm+dn}.\end{equation}
\end{proposition}
\begin{proof}
Let $g, g' \in SL_2(\mathbb{Z})$, which we write as follows:
\[g = \begin{pmatrix}a&b\\c&d\end{pmatrix}, \quad g' = \begin{pmatrix}a'&b'\\c'&d'\end{pmatrix}.\]
The following equality will be useful for computations.
\begin{equation}\label{eq:q action of SL2, product g'g}g'g = \begin{pmatrix}a'a + b'c & a'b + b'd \\ ac' + d'c & c'b + d'd\end{pmatrix}.\end{equation}
In order to show that $SL_2(\mathbb{Z})$ acts by algebra automorphisms on $D$, we need to check that
\begin{enumerate}
 \item $g.(xy-qyx) = 0$ in $D$, so $g$ is an algebra automorphism on $D$;
 \item $g'.(g.x) = (g'g).x$ and $g'.(g.x) = (g'g).y$ in $D$.
\end{enumerate}
The first equality can be verified by direct computation as follows:
\begin{align*}
 g.(xy-qyx) &= \expq{bd}y^bx^d\expq{ac}y^ax^c - \expq{2}\expq{ac}y^ax^c\expq{bd}y^bx^d \\
&= \expq{ac+bd}\Big(\expq{2ad}y^{a+b}x^{c+d} - \expq{2+2cb}y^{a+b}x^{c+d}\Big) \\
&= 0
\end{align*}
since $ad = bc+1$.

It will be useful to verify \eqref{eq:q action of SL2, general form} before tackling condition (2) above.  Indeed,
\begin{equation}\label{eq:q action of SL2, computation}\begin{aligned}
 g.(y^mx^n) &= \big(\expq{ac}y^ax^c\big)^m\big(\expq{bd}y^bx^d\big)^n \\
&= \expq{acm + bdn}\expq{2acm(m-1)/2}y^{am}x^{cm}\expq{2bdn(n-1)/2}y^{bn}x^{dn} \\
&= \expq{acm^2 + bdn^2 + 2cbmn}y^{am+bn}x^{cm+dn} 
\end{aligned}\end{equation}
Recalling that $ad-bc=1$, we can observe that
\[acm^2+bdn^2+2cbmn =(am+bn)(cm+dn)-mn \]
Substituting this into \eqref{eq:q action of SL2, computation}, we obtain the equality \eqref{eq:q action of SL2, general form}.

This simplifies the computations involved in (2) considerably: using \eqref{eq:q action of SL2, general form} and \eqref{eq:q action of SL2, product g'g} we can now see that
\begin{align*}
g'.(g.x) &= g'.\big(\expq{bd}y^bx^d\big) = \expq{bd}\expq{(a'b + b'd)(c'b + d'd) - bd}y^{a'b + b'd}x^{c'b + d'd},  \\
(g'g).x &= \expq{(a'b + b'd)(c'b + d'd)}y^{a'b + b'd}x^{c'b + d'd},
\end{align*}
and
\begin{align*}
g'.(g.y) &= g'\big(\expq{ac}y^ax^c\big) = \expq{ac}\expq{(a'a + b'c)(c'a + d'c)-ac}y^{a'a + b'c}x^{c'a + d'c}, \\
(g'g).y &= \expq{(a'a + b'c)(ac' + d'c)}y^{a'a + b'c}x^{ac' + d'c}.
\end{align*}
From this we can conclude that $g'.(g.x) = (g'g).x$ and $g'.(g.x) = (g'g).y$, and hence that the definition in \eqref{eq:q action of SL2 def} does indeed define an action of $SL_2(\mathbb{Z})$ on $D$.
\end{proof}
Using this result, we may refine the definition of a monomial automorphism given in \S\ref{ss:autos of q-comm structures} as follows:
\begin{definition}\label{def:def_monomial_autos} We call an automorphism of $k_q[x^{\pm1},y^{\pm1}]$ or $D$ a \textit{monomial automorphism} if it is defined by an element of $SL_2(\mathbb{Z})$ as in \eqref{eq:q action of SL2 def}.  
\end{definition}
It is well known that up to conjugation, $SL_2(\mathbb{Z})$ has only four non-trivial finite subgroups: the cyclic groups of orders 2, 3, 4 and 6 (see, for example, \cite[\S1.10.1]{Lorenz}).   Table~\ref{fig:table_of_maps} lists conjugacy class representatives for each of these groups, and we will use the same symbols to refer to both these automorphisms and their extensions to $D$.
\begin{table}[h]
\centering
\begin{tabular}{c|rl}
Order & \multicolumn{2}{c}{Automorphism} \\ \hline
2 & $\tau:$&$ x \mapsto x^{-1},\  y \mapsto y^{-1}$ \\
3 & $\sigma:$&$ x \mapsto y, \ y \mapsto \hat{q}y^{-1}x^{-1}$ \\
4 & $\rho:$&$ x \mapsto y^{-1}, \ y \mapsto x$ \\
6 & $\eta:$&$ x \mapsto y^{-1}, \ y \mapsto \hat{q}yx$
\end{tabular}
\caption{Conjugacy class representatives of finite order monomial automorphisms on $k_q[x^{\pm1},y^{\pm1}]$.}\label{fig:table_of_maps}
\end{table}

As noted in \cite[\S1.3]{Baudry}, it is sufficient to consider the fixed rings for one representative of each conjugacy class. We will therefore approach Theorem~\ref{res:thm_monomial_results} by examining the fixed rings of $D$ under each of the automorphisms in Table~\ref{fig:table_of_maps} in turn.

By Theorem~\ref{res:order_2_monomial_result}, we already know that $D^{\tau} \cong D$.  This is proved by methods from noncommutative algebraic geometry in \cite[\S13.6]{SV1}, but the authors also provide a pair of $q$-commuting generators for $D^{\tau}$, namely
\begin{equation}\label{eq:order_2_gens}
 u = (x-x^{-1})(y^{-1}-y)^{-1}, \quad v = (xy - x^{-1}y^{-1})(y^{-1}-y)^{-1}.
\end{equation}
We can use this and Theorem~\ref{res:epic_theorem} to check that the fixed ring of $D$ under an order 4 monomial automorphism is again isomorphic to $D$.
\begin{theorem}\label{res:order_4_thm}
 Let $\rho$ be the order 4 automorphism on $D$ defined by
\[\rho: x \mapsto y^{-1}, \quad y \mapsto x.\]
Then $D^{\rho} \cong D$ as $k$-algebras.
\end{theorem}
\begin{proof}
We can first notice that $\rho^2 = \tau$, so the fixed ring $D^{\rho}$ is a subring of $D^{\tau}$.  By \cite[\S13.6]{SV1}, $D^{\tau} = k_q(u,v)$ with $u,v$ as in \eqref{eq:order_2_gens}, so it is sufficient to consider the action of $\rho$ on $u$ and $v$.  By direct computation, we find that
\begin{equation*}
 \begin{aligned}
  \rho(u) &= (y^{-1} - y)(x^{-1} - x)^{-1} = -u^{-1}\\
\rho(v) &= (y^{-1}x - yx^{-1})(x^{-1} - x)^{-1} = (u^{-1} - qu)v^{-1}
 \end{aligned}
\end{equation*}
i.e. $\rho$ acts as $\varphi$ from \eqref{eq:def_varphi} on $k_{q^{-1}}(v,u)$, which by Theorem~\ref{res:AD_prop_intro} is isomorphic to $k_q(u,v)$.  Now by Theorem~\ref{res:epic_theorem}, $D^{\rho} \cong D^{\varphi} \cong D$.\end{proof}

We now turn our attention to the fixed ring of $D$ under the order 3 automorphism $\sigma$ defined in Table~\ref{fig:table_of_maps}, where matters become significantly more complicated.  Attempting to construct generators by direct analogy to the previous cases fails, and computations become far more difficult as both $x$ and $y$ appear in the denominator of any potential generator.  While the same theorem can be proved for this case, our chosen generators are unfortunately quite unintuitive.

For the following results, we will assume that $k$ contains a primitive third root of unity, denoted $\omega$\nom{W@$\omega$ (root of unity)}.  As with Theorem~\ref{res:epic_theorem}, we define certain elements which are fixed by $\sigma$ or are acted upon as multiplication by a power of $\omega$.  We set
\begin{equation}\label{eq:order_3_building_blocks}\begin{aligned}
 a &= x + \omega y + \omega^2 \hat{q}y^{-1}x^{-1} \\
b &= x^{-1} + \omega y^{-1} + \omega^2 \hat{q}yx \\
c &= y^{-1}x + \omega \hat{q}^3y^2x + \omega^2\hat{q}^3y^{-1}x^{-2} \\
\pa &= x + y + \hat{q}y^{-1}x^{-1} \\
\pb &= x^{-1} + y^{-1} + \hat{q} yx \\
\pc &= y^{-1}x + \hat{q}^3y^2x + \hat{q}^3y^{-1}x^{-2}
\end{aligned}\end{equation}
The elements $\pa$, $\pb$ and $\pc$ are fixed by $\sigma$, while $\sigma$ acts on $a$, $b$ and $c$ as multiplication by $\omega^2$.  We can further define
\begin{equation}\label{eq:order_3_gens}\begin{aligned}
 g &= a^{-1}b \\
 f &= \pb - \omega^2\pa g + (\omega^2-\omega)\hat{q}^{-1}(\omega^2g^2 + \hat{q}^2g^{-1})
\end{aligned}\end{equation}
\begin{proposition}\label{res:order_3_q_comm_proof}
Let $k$ be a field of characteristic 0 that contains a primitive third root of unity $\omega$ and a square root of $q$, denoted by $\hat{q}$.  The elements $f$ and $g$ in \eqref{eq:order_3_gens} are fixed by $\sigma$ and satisfy $fg = qgf$.
\end{proposition}
\begin{proof}
 As always the first statement is clear: $\sigma$ acts on $a$ and $b$ by $\omega^2$ and therefore fixes $g$, and since $\pa$ and $\pb$ are already fixed by $\sigma$ we can now see that $\sigma(f) = f$.

To verify the second statement, we need to understand how $g$ interacts with $\pa$ and $\pb$.  Simple multiplication of polynomials yields the identities
\begin{align*}
 a\pa &=\pa a + (\omega-\omega^2)(\hat{q}-\hat{q}^{-1})b \\
a \pb  &= \hat{q}^2\pb a + (\hat{q}^{-2} - \hat{q}^2)c \\
\pa b &= \hat{q}^2b\pa  + \omega(\hat{q}^{-2} - \hat{q}^2)c \\
\pb b &= b\pb  + (\omega^2-\omega)(\hat{q}-\hat{q}^{-1})a
\end{align*}
and hence
\begin{align*}
 g\pa  &= \hat{q}^{-2}\pa g - \hat{q}^{-2}\omega(\hat{q}^{-2}-\hat{q}^2)a^{-1}c - (\omega-\omega^2)\hat{q}^{-2}(\hat{q}-\hat{q}^{-1})g^2 \\
g \pb  &= \hat{q}^{-2}\pb g - (\omega^2-\omega)(\hat{q}-\hat{q}^{-1}) - \hat{q}^{-2}(\hat{q}^{-2} - \hat{q}^2)a^{-1}cg
\end{align*}
since $g = a^{-1}b$.

Now by direct computation, we find that
\begin{align*}
 \hat{q}^2gf &= \hat{q}^2g\pb - \hat{q}^2\omega^2g\pa g + (w^2-w)\hat{q}(\omega^2g^3 + \hat{q}^2)\\
&= \pb g - (\omega^2-\omega)(\hat{q}-\hat{q}^{-1})\hat{q}^2 - (\hat{q}^{-2} - \hat{q}^2)a^{-1}cg \\
& \qquad - \omega^2\pa g^2 + (\hat{q}^{-2}-\hat{q}^2)a^{-1}cg + \omega^2(\omega-\omega^2)(\hat{q}-\hat{q}^{-1})g^3 \\
& \qquad + (\omega^2-\omega)\hat{q}(\omega^2g^3 + \hat{q}^2) \\
&= \pb g - \omega^2 \pa g^2 + (\omega^2-\omega)\hat{q}^{-1}(\omega^2g^3 + \hat{q}^2)\\
&=fg\qedhere
\end{align*}
\end{proof}
\begin{theorem}\label{res:epic_theorem_2}
Let $k$, $f$ and $g$ be as in Proposition~\ref{res:order_3_q_comm_proof}.  Then the division ring $k_q(f,g)$ generated by $f$ and $g$ over $k$ is equal to the fixed ring $D^{\sigma}$, and hence $D^{\sigma} \cong D$ as $k$-algebras.
\end{theorem}
\begin{proof}
We claim that it suffices to prove $k_q[x^{\pm1},y^{\pm1}]^{\sigma} \subset k_q(f,g)$.  Indeed, $k_q[x^{\pm1},y^{\pm1}]$ is a Noetherian domain, and therefore both left and right Ore, while $\langle \sigma \rangle$ is a finite group.  We can therefore apply \cite[Theorem~1]{faith} to see that
\[Q\big(k_q[x^{\pm1},y^{\pm1}]^{\sigma}\big) = D^{\sigma}\]
where $Q(R)$ denotes the full ring of fractions of a ring $R$.  Hence if $k_q[x^{\pm1},y^{\pm1}]^{\sigma} \subset k_q(f,g)$, we see that
\[Q(k_q[x^{\pm1},y^{\pm1}]^{\sigma}) \subseteq k_q(f,g) \subseteq D^{\sigma} \Rightarrow k_q(f,g) = D^{\sigma}.\]

We will show that $k_q[x^{\pm1}, y^{\pm1}]^{\sigma}$ is generated as an algebra by the elements $\pa$, $\pb$ and $\pc$ from \eqref{eq:order_3_building_blocks}, and then check that these three elements are in $k_q(f,g)$. 

By \cite[Th\'eor\`eme~2.1]{Baudry}, $k_q[x^{\pm1}, y^{\pm1}]^{\sigma}$ is generated as a Lie algebra with respect to the commutation bracket by seven elements:
\begin{gather*}
R_{0,0} = 1, \quad R_{1,0} = x+ y + \hat{q}y^{-1}x^{-1}, \quad R_{1,1} = x^{-1} + y^{-1} + \hat{q}yx, \\
R_{1,2} = y^{-1}x + \hat{q}^3y^2x + \hat{q}^3y^{-1}x^{-2}, \quad R_{1,3} = y^{-1}x^2 + \hat{q}^5y^3x + \hat{q}^8y^{-2}x^{-3}, \\
R_{2,0} = x^2+y^2+\hat{q}^4y^{-2}x^{-2}, \quad R_{3,0} = x^3 + y^3 + \hat{q}^9y^{-3}x^{-3}.
\end{gather*}
and so it is also generated as a $k$-algebra by these elements.  $R_{1,0}$, $R_{1,1}$ and $R_{1,2}$ are precisely the aforementioned elements $\pa$, $\pb$ and $\pc$, and it is a simple computation to verify that $R_{1,3}$, $R_{2,0}$ and $R_{3,0}$ are in the algebra generated by these three. 

It is clear from the definition of $f$ that once we have found either $\pa$ or $\pb$ in $k_q(f,g)$ we get the other one for free, and we can also observe that
\[\pa \pb - \hat{q}^2 \pb \pa = (\hat{q}^{-2} - \hat{q}^2) \pc - 3\hat{q}^2 + 3 \in k_q(f,g)\]
so $\pc \in k_q(f,g)$ follows from $\pa$, $\pb \in k_q(f,g)$.  Unfortunately there seems to be no easy way to make the first step, i.e. verify that either $\pa$ or $\pb$ is in $k_q(f,g)$. 

In fact, the element $\pa$ can be written in terms of $f$ and $g$ as in the following equality; this is the result of a long and tedious calculation, and was verified using the computer algebra system Magma (v2.18) and the methods described in Appendix~\ref{c:appendix}.  We find that
\begin{gather*}
\pa = (\omega - \omega^2)^{-1}\hat{q}^{-2}g^{-1}f + (\omega^2\hat{q} + \omega\hat{q}^{-1})g + (\hat{q} + \hat{q}^{-1})g^{-2}  \\
\qquad \qquad    + (\omega - \omega^2)\Big(\hat{q}^{-2}g^3 + (\hat{q}^2+1) + \hat{q}^4g^{-3}\Big)f^{-1}.
\end{gather*}
Therefore $\pa \in k_q(f,g)$, and the result now follows.
\end{proof}
\begin{remark}\label{rem:order 3 snark}
By analogy to the pairs of generators in \eqref{eq:order_2_gens} and Theorem~\ref{res:epic_theorem}, we might hope to find similarly intuitive generators for $D^{\sigma}$.  Having set $g := a^{-1}b$ as in the proof above, computation in Magma shows that there does exist a left fraction $f' \in D^{\sigma}$ such that $f'g = qgf'$; unfortunately, $f'$ takes 9 pages to write down.  In the interest of brevity, we chose to use here the less intuitive $f$ defined in \eqref{eq:order_3_gens}; the original $f'$ can be found in Appendix~\ref{s:magma_example}.
\end{remark}

In a similar manner to Theorem~\ref{res:order_4_thm}, we can now describe the one remaining fixed ring $D^{\eta}$ using our knowledge of the fixed rings with respect to monomial maps of order 2 and 3.
\begin{theorem}\label{res:order_6_thm}
Let $\eta$ be the order 6 map defined in Table~\ref{fig:table_of_maps}, and suppose $k$ contains a primitive third root of unity and both a second and third root of $q$.  Then $ D^{\eta} \cong D$ as $k$-algebras.
\end{theorem}
\begin{proof}
We first note that $\eta^3 = \tau$, so $D^{\eta} = (D^{\tau})^{\eta}$.  Take $u,v$ in \eqref{eq:order_2_gens} as our generators of $D^{\tau}$, and now we can observe that the action of $\eta$ on $u$ and $v$ is as follows:
\begin{align*}
\eta(u) &= (y^{-1} - y)(\hat{q}^{-1}x^{-1}y^{-1} - \hat{q}yx)^{-1} \\
&=-\hat{q}(y^{-1} - y)(xy -x^{-1}y^{-1})^{-1} \\
&=-\hat{q}v^{-1} \\
\eta(v) &= (\hat{q}y^{-1}yx - \hat{q}^{-1}yx^{-1}y^{-1})(\hat{q}^{-1}x^{-1}y^{-1} - \hat{q}yx)^{-1} \\
&= -q(x-x^{-1})(xy-x^{-1}y^{-1})^{-1} \\
&= -v^{-1}u 
\end{align*}
Let $p = \sqrt[3]{q^{-1}}$. By making a change of variables $u_1= -p^{-1}\hat{q}^{-1}u$, $v_1 =  pv$ in $D^{\tau} = k_q(u,v)$, we see that $\eta$ acts on $D^{\tau}$ as
\[\eta(u_1) = v_1^{-1}, \quad \eta(v_1) = \hat{q}^{-1}v_1^{-1}u_1\]
This is a monomial map of order 3 and so its fixed ring is isomorphic to $D^{\sigma}$, as noted in \cite[\S1.3]{Baudry}.  Now by Theorem~\ref{res:epic_theorem_2} and Theorem~\ref{res:order_2_monomial_result}, $D^{\eta} = (D^{\tau})^{\eta} \cong (D^{\tau})^{\sigma} \cong D^{\tau} \cong D$.
\end{proof}
\begin{theorem}\label{res:thm_monomial_results_2}
Let $k$ be a field of characteristic zero, containing a primitive third root of unity $\omega$ and both a second and a third root of $q$.  If $G$ is a finite group of monomial automorphisms of $D$ then $D^G \cong D$ as $k$-algebras.
\end{theorem}
\begin{proof}
Theorems~\ref{res:order_4_thm}, \ref{res:epic_theorem_2}, \ref{res:order_6_thm} and \cite[\S13.6]{SV1}.
\end{proof}
Finally, by combining Theorem~\ref{res:epic_theorem} and Theorem~\ref{res:thm_monomial_results_2}, we obtain Theorem~\ref{res:thm_monomial_results} as promised in Chapter~\ref{c:introduction}.

\section{Consequences for the automorphism group of $D$}\label{s:automorphism_consequences}

The construction of $q$-commuting pairs of elements is closely linked to questions about the automorphisms and endomorphisms of the $q$-division ring: such maps are defined precisely by where they send the two generators of $D$, and naturally these images must $q$-commute.  Despite similarities to the commutative field $k(x,y)$ a full description of the automorphism group $Aut(D)$ remains unknown, with a major stumbling block being understanding the role played by conjugation maps.

Intuition suggests that ``inner automorphism'' and ``conjugation'' should be synonymous; certainly all conjugation maps should be bijective, at the very least.  Here we challenge this intuition by showing that the conjugation maps described in Proposition~\ref{res:AC_construct_z} not only gives rise to conjugations which are not inner, but also conjugation maps which are well-defined \textit{endomorphisms} (not automorphisms) on $D$.  This provides answers to several of the questions posed at the end of \cite{AC1} (outlined in Questions~\ref{ques:AC_questions} below), while also raising several new ones.

Let $X$ and $Y$ be a pair of $q$-commuting generators for $D$, or a pair of commutative generators for $k(x,y)$, as appropriate.  We continue to assume that $q$ is not a root of unity.

The first question that must be answered when considering the automorphism group of $D$ is how to define the subgroup of tame automorphisms.  As noted in \S\ref{ss:autos of division ring background section}, this is approached in different ways by Alev and Dumas in \cite{AD3} and Artamonov and Cohn in \cite{AC1}; our initial aim is to show that these two approaches in fact define the same group of automorphisms.

Alev and Dumas proceed by analogy to the commutative case $k(X,Y)$, where the automorphism group is known to be generated by the \textit{fractional linear transformations}:
\begin{equation}\label{eq:fractional linear transformations comm version}
X \mapsto \frac{\alpha X + \beta Y + \gamma}{\alpha'' X + \beta'' Y +\gamma''}, \ Y \mapsto \frac{\alpha' X + \beta' Y +\gamma'}{\alpha'' X + \beta'' Y +\gamma''}, \quad \textrm{for }\begin{bmatrix}\alpha & \beta & \gamma \\ \alpha' & \beta' & \gamma' \\ \alpha'' & \beta'' & \gamma''\end{bmatrix} \in PGL_3(k)
\end{equation}
and \textit{triangular automorphisms} which preserve the embedding $k(Y) \subset k(X,Y)$:
\begin{equation}\label{eq:triangular autos comm version}\begin{gathered}
X \mapsto \frac{a(Y)X + b(Y)}{c(Y)X + d(Y)}, \ Y \mapsto \frac{\alpha Y + \beta}{\gamma Y + \delta} \\
\begin{bmatrix}a(Y)&b(Y)\\c(Y)&d(Y)\end{bmatrix} \in PGL_2(k(Y)),\ \begin{bmatrix}\alpha & \beta \\ \gamma & \delta\end{bmatrix} \in PGL_2(k).
\end{gathered}\end{equation}
If we try to view these as maps on $D$ instead, the images of $X$ and $Y$ face the additional restriction of being required to $q$-commute; as demonstrated in \cite[Propositions~1.4, 1.5]{AD3}, this severely restricts what forms the automorphisms can take.  It is shown in \cite{AD3} that a map given by \eqref{eq:fractional linear transformations comm version} only defines an automorphism on $D$ if it takes one of the following three forms:
\begin{equation}\label{eq:q comm fractional linear maps}\begin{gathered}
X \mapsto \lambda X, \ Y \mapsto \mu Y, \\
X \mapsto \lambda Y^{-1}, \ Y \mapsto \mu Y^{-1}X, \\
X \mapsto \lambda YX^{-1}, \ Y \mapsto \mu X^{-1},
\end{gathered}\end{equation}
where $\lambda, \mu \in k^{\times}$.  This corresponds to the subgroup $(k^{\times})^2 \rtimes C_3$ of $(k^{\times})^2 \rtimes SL_2(\mathbb{Z})$, where $C_3$ is the group of order 3 generated by the monomial automorphism corresponding to the matrix $\left(\begin{smallmatrix} -1 & -1 \\ 1 & 0 \end{smallmatrix}\right)$.  Meanwhile, the automorphisms of $D$ corresponding to those of the form \eqref{eq:triangular autos comm version} are precisely the ones generated by the following automorphisms:
\begin{equation}\label{eq:q comm triangular automorphisms}\begin{aligned}
\psi_X&:X \mapsto a(Y)X, \quad Y \mapsto \alpha Y, & a(Y) \in k(Y)^{\times}, \ \alpha \in k^{\times}, \\
\tau&: X \mapsto X^{-1}, \quad Y \mapsto Y^{-1} &
\end{aligned}\end{equation}
Let $H_1$ denote the subgroup of $Aut(D)$ generated by automorphisms of the form \eqref{eq:q comm fractional linear maps} and \eqref{eq:q comm triangular automorphisms}.

Recall from Definition~\ref{def:elementary autos of D} that Artanomov and Cohn defined the elementary automorphisms of $D$ to be those of the form
\begin{equation}\label{eq:AC elementary autos}\begin{aligned}
\tau:&{}\quad X\mapsto X^{-1}, \ Y \mapsto Y^{-1}, \\
\varphi_X:&{}\quad X \mapsto b(Y)X, \ Y \mapsto Y,\quad b(Y) \in k(Y)^{\times},\\
\varphi_Y:&{}\quad X \mapsto X,\ Y \mapsto a(X)Y,\quad a(X) \in k(X)^{\times}. 
\end{aligned}\end{equation}
Let $H_2$ denote the group generated by automorphisms of the form \eqref{eq:AC elementary autos}.

\begin{lemma}\label{res:H1 and H2 are equal}
The groups $H_1$ and $H_2$ defined above are equal.
\end{lemma}
\begin{proof}
We first show that $H_2 \subseteq H_1$.  By taking $\alpha = 1$ in \eqref{eq:q comm triangular automorphisms} we immediately obtain $\tau$ and all automorphisms of the form $\varphi_X$, so we need only show that we can construct all automorphisms of the form $\varphi_Y$ from elements of $H_1$.  We start by defining
\begin{align*}
\sigma&: X \mapsto Y^{-1}, \ Y \mapsto Y^{-1}X \ \in H_1 \\
\psi_1&: X \mapsto YX, \ Y \mapsto Y\  \in H_1,
\end{align*}
and observe that $\rho:= \psi_1 \circ \sigma X: \mapsto Y^{-1},\ Y \mapsto X \in H_1$.  Now for an arbitrary automorphism $\psi_X$ of the form \eqref{eq:q comm triangular automorphisms} we may combine it with $\rho$ to obtain
\[\rho^{-1} \circ \psi_X \circ \rho: X \mapsto X, \ Y \mapsto a(X^{-1})Y \ \in H_1;\]
since $a(Y) \in k(Y)^{\times}$ was arbitrary, we see that $H_2 \subseteq H_1$.  Conversely, to show $H_1 \subseteq H_2$ we need only obtain the automorphisms from \eqref{eq:q comm fractional linear maps}, which can be decomposed as follows: define
\begin{align*}
\varphi_1&: X \mapsto \mu^{-1} \lambda YX, \ Y \mapsto Y \in H_2, \\
\varphi_2&: X \mapsto X, \ Y \mapsto \mu^{-1}X^{-1}Y \in H_2,
\end{align*}
so that
\[\varphi_2 \circ \tau \circ \varphi_1: X \mapsto \lambda Y^{-1}, \ Y \mapsto \mu Y^{-1}X\]
is also in $H_2$ as required.
\end{proof}
This justifies the choice to call an automorphism of $D$ \textit{elementary} if it is of the form \eqref{eq:AC elementary autos}, and \textit{tame} if it is in the group generated by the elementary automorphisms and the inner automorphisms on $D$.

Recall from Theorem~\ref{res:main AC theorem statement} that any homomorphism from $D$ to itself can be decomposed as a product of elementary automorphisms and a conjugation map $c_z$, where $z$ is constructed as in Proposition~\ref{res:AC_construct_z}.  The following questions are posed by Artamonov and Cohn in \cite{AC1}.
\begin{questions}\label{ques:AC_questions}\ 
\begin{enumerate}
 \item Does there exist an element $z$ satisfying the recursive definition \eqref{eq:def_z}, such that $z \not\in k_q(X,Y)$?  What if $z^n \in k_q(X,Y)$ for some positive integer $n$?
 \item Does there exist an element $z$ from \eqref{eq:def_z} such that $z^{-1}k_q(X,Y)z \subsetneq k_q(X,Y)$?
 \item The group of automorphisms of $k_q(X,Y)$ is generated by elementary automorphisms, conjugation by some elements of the form $z$, and $\tau$.  Find a set of defining relations for this generating set.
\end{enumerate}
\end{questions}
We first note that (3) needs rephrasing, since we can provide affirmative answers for both (1) and (2).  Indeed, we will construct examples of conjugation automorphisms $c_z$ satisfying $z^2 \in D$ (Proposition~\ref{res:z_automorphism_order_2}) and $z^n \not\in D$ for all $n \geq 1$ (Proposition~\ref{res:z_automorphism_inf_order}), and also a conjugation endomorphism such that $z^{-1}Dz \subsetneq D$ (Proposition~\ref{res:z_endomorphism}).

In light of this, (3) should be modified to read:
\begin{enumerate}\setcounter{enumi}{3}
\item Under what conditions is $c_z$ an automorphism of $D$ rather than an endomorphism?  Using this, give a set of generators and relations for $Aut(D)$.
\end{enumerate}
For each of our examples below, we start by defining a homomorphism $\psi$ on $D$ and then verify that the image of the generators of $D$ under $\psi$ have the form \eqref{eq:q comm standard form}.  This allows us to use Proposition~\ref{res:AC_construct_z} to construct $z$ as in \eqref{eq:def_z} such that $\psi = c_{z^{-1}}$ (possibly after a change of variables in $D$ to ensure the leading coefficients are both 1).  The final step in each proof is checking whether $z \in D$, for which we will use the following lemmas.
\begin{lemma}\label{res:two_conjugations}
 Let $z_1$, $z_2 \in k_q(Y)(\!(X)\!)$.  If the conjugation maps $c_{z_1}$ and $c_{z_2}$  have the same action on $D = k_q(X,Y)$, then $z_1$ and $z_2$ differ only by a scalar.
\end{lemma}
\begin{proof}
 If $c_{z_1} = c_{z_2}$, then $z_1Yz_1^{-1} = z_2Yz_2^{-1}$, i.e. $z_2^{-1}z_1Y = Yz_2^{-1}z_1$.  Similarly $z_2^{-1}z_1$ commutes with $X$, so $z_2^{-1}z_1$ is in the centralizer of $D$ in $k_q(Y)(\!(X)\!)$, which we write as $C(D)$.

We now verify that $C(D) = k$.  Indeed, if $u = \sum_{i \geq n}a_i X^i \in C(D)$, then $u$ must commute with $Y$, i.e.
\[Y \sum_{i \geq n} a_i X^i = \sum_{i \geq n} a_i X^i Y = \sum_{ i\geq n}q^ia_i Y X^i\]
Since $q$ is not a root of unity, we must have $a_i =0$ for all $i \neq 0$, i.e. $u = a_0 \in k(Y)$.  Since $u$ is now in $D$ and must commute with both $X$ and $Y$, $u \in Z(D) = k$.  The result now follows.
\end{proof}
Recall that for $r \in k_q[X,Y]$, $deg_X(r)$ denotes the degree of $r$ as a polynomial in $X$.  This extends naturally to a notion of degree on $k_q(X,Y)$ by defining 
\[deg_X(t^{-1}s) := deg_X(s) - deg_X(t),\]
where $s, t \in k_q[X,Y]$.  We note that this definition is multiplicative.
\begin{lemma}\label{res:matching_degrees} If $c_z$ is an inner automorphism on $k_q(X,Y)$, then $c_z(Y) = v^{-1}u$ satisfies $deg_X(u) = deg_X(v)$.\end{lemma}
\begin{proof} We can write the commutation relation in $D$ as $YX = \beta(X)Y$, where $\beta$ is the automorphism $X \mapsto q^{-1}X, \ Y \mapsto Y$.  Let $c_z$ be an inner automorphism on $D$, so that $z = t^{-1}s \in D$ for $s, t \in k_q[X,Y]$.  Thus the image of $Y$ under $c_z$ is
 \begin{equation}\label{eq:ore_1}c_z(Y) = t^{-1}s\beta(s)^{-1}\beta(t) Y\end{equation}
Since $\beta$ does not affect the $X$-degree of a polynomial and $deg_X$ is multiplicative, it is clear from \eqref{eq:ore_1} that $deg_X(c_z(y)) = 0$ and hence $deg_X(u) = deg_X(v)$ as required.
\end{proof}

We are now in a position to answer Questions~\ref{ques:AC_questions} (1) and (2).

\begin{proposition}\label{res:z_automorphism_order_2}
Let $D$ and $\varphi$ be as in Theorem~\ref{res:epic_theorem}, and set $E = D^{\varphi}$.  With an appropriate choice of generators for $E$, the map
\[\gamma: \ E \rightarrow E:\  r \mapsto (y^{-1} + y)r(y^{-1} + y)^{-1}\]
defined in Lemma~\ref{res:gamma_def} is an automorphism of the form $c_{z^{-1}}$, with $z$ defined as in \eqref{eq:def_z}. Further, we have $z\not\in E$ while $z^2 \in E$.  This provides an affirmative answer to Question~\ref{ques:AC_questions} (1).
\end{proposition}
\begin{proof}
$E$ is a $q$-division ring by Theorem~\ref{res:epic_theorem}, and $\gamma$ is an automorphism by Lemma~\ref{res:gamma_def}.  Write $E = k_q(f,g)$, where $f:=\frac{1}{1-q}(hg-gh)$ as in the proof of Theorem~\ref{res:epic_theorem}, which allows us to use the methods of \cite{AC1}.

In order to check that $\gamma$ has the form described by \eqref{eq:def_z}, by Proposition~\ref{res:AC_construct_z} it is sufficient to check that $\gamma(f)$, $\gamma(g)$ are of the following form:
\begin{equation}\label{eq:form for image of gamma}\gamma(f) = f^s + \sum_{i> s}a_if^i, \quad \gamma(g) = g^s + \sum_{i \geq 1}b_i f^i, \quad s \in \{\pm1\},\ a_i, b_i \in k(g).\end{equation}
Recall that $\gamma(g) = (cb)b^{-2}$, which can be written in terms of $h$ and $g$ using Lemma~\ref{res:computation_lemma} and \eqref{eq:formulas_cb}.  From the definition $f = \frac{1}{1-q}(hg - gh)$ we find that $h = g^{-1}(1-f)$, and after some rearranging we obtain
\begin{align*}
 b^2 &= (q^3 - g^4)(q^7-g^4)q^{-6}g^{-3}f^{-2} + [\textrm{higher terms in }f] \\
 cb &= (q^3 - g^4)(q^7-g^4)q^{-7}g^{-3}f^{-2} + [\textrm{higher terms in }f]
\end{align*}
Therefore the lowest term of $\gamma(g) = (cb)b^{-2}$ is $qg$, which can easily be transformed into the form given in \eqref{eq:form for image of gamma} by a simple change of variables.

By \cite[Proposition~3.2]{AC1}, it now follows that $\gamma(f)$ must have the form
\[\gamma(f) = b_1f + \sum_{i > 1} b_i f^i, \quad b_1, b_i \in k(g).\]
Again, we can make a change of variables in $k_q(f,g)$ using elementary automorphisms to obtain $b_1= 1$ while simultaneously ensuring that $\gamma(g)$ remains in the form \eqref{eq:form for image of gamma}. Now by \cite[Theorem~3.5]{AC1}, $\gamma = c_{z^{-1}}$ with $z$ constructed as in Proposition~\ref{res:AC_construct_z}.  

By Lemma \ref{res:two_conjugations}, $y+y^{-1}$ and $z$ differ by at most a scalar.  Since $y + y^{-1} \not\in E$, $\gamma = c_{z^{-1}}$ defines an automorphism of $E$ with $z \not\in E$.

Finally, we have already noted in Lemma~\ref{res:computation_lemma} (ii) that $b^2 = (y-y^{-1})^2 + 4 \in E$, and so $z^2 \in E$ as well.
\end{proof}

\begin{remark}It is worth noting that this phenomenon of non-inner conjugations cannot happen when $q$ is a root of unity.  Indeed, if $q^n = 1$ for some $n$, then $D$ is a finite dimensional central simple algebra over its centre and by the Skolem-Noether theorem every automorphism of $D$ should be inner.  The automorphism $\gamma$ is still a well-defined automorphism on $D$ in this case, but the difference is that now $D$ has a non-trivial centre: since $y^{n} - y^{-n}$ is a central element we can replace $y + y^{-1}$ with $(y+y^{-1})(y^n - y^{-n})$ in the definition of $\gamma$ without affecting the map at all.  We now have
\[(y+y^{-1})(y^{n} - y^{-n}) = y^{n+1} - y^{-(n+1)} + y^{n-1} - y^{-(n-1)} \in k_q(x,y)^G\]
and so $\gamma$ is indeed an inner automorphism in this case. \end{remark}

\begin{proposition}\label{res:z_endomorphism}
Let $D = k_q(x,y)$.  Then there exists $z \in k_q(y)(\!(x)\!)$ such that $z^{-1}Dz \subsetneq D$.  This provides an affirmative answer to Question~\ref{ques:AC_questions}(2).
\end{proposition}
\begin{proof}
We can view the isomorphism $\theta: D \stackrel{\sim}{\longrightarrow} D^G$ from Theorem \ref{res:epic_theorem} as an endomorphism on $D$, and so $\theta$ decomposes into the form \eqref{eq:decomp} with $z$ constructed as in \eqref{eq:def_z}.  Since $\theta$ is not surjective, and $c_{z^{-1}}$ is the only map in the decomposition not already known to be an automorphism, we must have $z^{-1}Dz \subsetneq D$.
\end{proof}

Propositions \ref{res:z_automorphism_order_2} and \ref{res:z_endomorphism} illustrate some of the difficulties involved in giving a set of relations for $Aut(D)$: not only is it possible for both endomorphisms and automorphisms to arise as conjugations $c_z$ for $z \not\in D$, but as we show next, it turns out to be quite easy to define an automorphism $\sigma$ on $D$ which is a product of elementary automorphisms, but also satisfies $\psi = c_{z^{-1}}$ with $z^n \not\in D$ for any $n \geq 1$.

\begin{example}\label{res:auto_example}We define maps by
\begin{align*}
h_1&: x \mapsto (1+y)x, \ y \mapsto y,\\
h_2&: x \mapsto x, \ y\mapsto (1+x)y,\\
h_3&: x \mapsto \frac{1}{1+y}x, \ y \mapsto y,
\end{align*}
which are all elementary automorphisms on $k_q(x,y)$.  Let $\psi = h_3\circ h_2 \circ h_1$, so that
\begin{align*}
\psi(x) &=  x + \frac{qy}{(1+y)(1+qy)}x^2,\\
\psi(y) &=  y + \frac{qy}{1+y}x.
\end{align*}
is an automorphism on $k_q(x,y)$.  These have been chosen so that $\psi(x)$, $\psi(y)$ are already in the form \eqref{eq:q comm standard form}, so there exists $z \in k_q(y)(\!(x)\!)$ defined by \eqref{eq:def_z} such that $\psi = c_{z^{-1}}$.  Since $\psi(y)$ is a polynomial in $x$ of non-zero degree, $\psi$ is not an inner automorphism by Lemma \ref{res:matching_degrees}.

In fact, $\psi^n(y)$ is a polynomial in $x$ of degree $n$.  The key observation in proving this is to notice that $(1+y)^{-1}x$ is fixed by $\psi$.  Indeed,
\begin{align*}
 \psi((1+y)^{-1}x) &= \left[1 + y + \frac{qy}{1+y}x\right]^{-1}\left[x + \frac{qy}{(1+y)(1+qy)}x^2\right] \\
&= \left[\left(1+\frac{qy}{(1+y)(1+qy)}x\right)(1+y)\right]^{-1}\left[1 + \frac{qy}{(1+y)(1+qy)}x\right]x \\
&= (1+y)^{-1}x.
\end{align*}
If we write $\psi(y) = y(1+q(1+y)^{-1}x)$, it is now clear by induction that 
\begin{equation}\label{eq:psi_induction}\psi^n(y) = \psi^{n-1}(y)(1+q(1+y)^{-1}x)\end{equation}
is polynomial in $x$ of degree $n$.  
\end{example}
\begin{proposition}\label{res:z_automorphism_inf_order}With $\psi$ as in Example~\ref{res:auto_example} and $D = k_q(x,y)$, $\psi = c_{z^{-1}}$ is an example of a conjugation automorphism satisfying $z^n \not\in D$ for all $n \geq 1$.
\end{proposition}
\begin{proof}
By \eqref{eq:psi_induction}, $\psi^n(y) = z^{-n}yz^n$ is a polynomial in $x$ of degree $n$ so by Lemma~\ref{res:matching_degrees} we have $z^n \not\in D$ for all $n \geq 1$.
\end{proof}
Combining these results, we obtain the theorem promised in the introduction.
\begin{theorem}\label{res:thm_AC_results_2}
Let $k$ be a field of characteristic zero and $q \in k^{\times}$ not a root of unity.  Then:
\begin{enumerate}[(i)]
\item The $q$-division ring $D$ admits examples of bijective conjugation maps by elements $z \not\in D$; these include examples satisfying $z^n \in D$ for some positive $n$, and also those such that $z^n \not\in D$ for all $n \geq 1$.
\item $D$ also admits an endomorphism which is not an automorphism, which can be represented in the form of a conjugation map.  
\end{enumerate}
\end{theorem}
\begin{proof}
Propositions~\ref{res:z_automorphism_order_2}, \ref{res:z_endomorphism}, \ref{res:z_automorphism_inf_order}.
\end{proof}
We have seen that the set of generators for $Aut(D)$ proposed by Artamonov and Cohn in \cite{AC1} in fact generate the whole endomorphism group $End(D)$.  On the other hand, Theorem~\ref{res:thm_AC_results_2} also suggests that if we restrict our attention to the group generated by the elementary and \textit{inner} automorphisms only, this may generate a proper subgroup of $Aut(D)$ rather than the whole group.

A good test case for this question would be the automorphism $\gamma$ in Proposition~\ref{res:z_automorphism_order_2}: can it be decomposed into a product of elementary automorphisms and inner automorphisms?  The next proposition indicates one way of approaching this question.
\begin{proposition}\label{res:does the automorphism decompose?}
 Let $c_z$ be a bijective conjugation map on $k_q(x,y)$. Suppose that $c_z$ fixes some element $r \in k_q(x,y)\backslash k$, and that there is a product $\varphi$ of elementary automorphisms such that $r$ is the image of $x$ or $y$ under $\varphi$.  Then $c_z$ decomposes as a product of elementary automorphisms.
\end{proposition}
\begin{proof}
 Suppose first that $\varphi(x) = r$.  Define $u := \varphi(x)$, $v := \varphi(y)$; since $\varphi$ is a product of elementary automorphisms, this gives rise to a change of variables in $k_q(x,y)$, i.e. $k_q(x,y) = k_q(u,v)$.  

We would like to show that $c_z$ acts as an elementary automorphism on $u$ and $v$.  While $z$ is an element of $k_q(y)(\!(x)\!)$ and is not necessarily in $k_q(v)(\!(u)\!)$, $c_z$ is still a well-defined automorphism on $k_q(u,v)$ and so $c_z(v) = zvz^{-1} \in k_q(u,v)$.

Meanwhile $c_z$ fixes $u$, which $q$-commutes with both $v$ and $c_z(v)$; it is easy to see that $u$ must therefore commute with $c_z(v)v^{-1}$.  The centralizer of $u$ in $k_q(u,v)$ is precisely $k(u)$, so $c_z(v)v^{-1} = a(u) \in k(u)$.  Now
\[c_z(u) = u, \quad c_z(v) = c_z(v)v^{-1}v = a(u)v\]
is elementary as required.

Let $a(x) \in k(x)$ be the element obtained by replacing every occurrence of $u$ in $a(u)$ by $x$.  We can define an elementary automorphism on $k_q(x,y)$ by $h: x \mapsto x, \ y \mapsto a(x)y$, which allows us to write the action of $c_z$ as follows:
\[c_z(u) = u = \varphi \circ h(x), \quad c_z(v) = a(u)v = \varphi(a(x)y) = \varphi \circ h (y).\]
Hence $c_z \circ \varphi = \varphi \circ h$, and so $c_z = \varphi \circ h \circ \varphi^{-1}$ is a product of elementary automorphisms as required.  The case $\varphi(y) = r$ follows by a symmetric argument.
\end{proof}
We note that the automorphism $\gamma$ from Proposition~\ref{res:z_automorphism_order_2} fixes $y-y^{-1} \in E$, but it is not clear whether $y-y^{-1}$ satisfies the hypotheses of Proposition~\ref{res:does the automorphism decompose?}.

\section{Open Questions}\label{s:open questions on D}
While the results of \S\ref{s:automorphism_consequences} answer the questions raised by Artamonov and Cohn at the end of \cite{AC1}, they trigger a number of interesting new open questions.  We finish this chapter by listing some of these questions.
\begin{questions}\ 
\begin{enumerate}
 \item Is there an algorithm to identify elements fixed by a given conjugation map $c_z$ and establish whether they satisfy the hypotheses of Proposition~\ref{res:does the automorphism decompose?}?
\item Does every conjugation automorphism $c_z$ with $z \not\in D$ decompose into a product of elementary automorphisms and an inner automorphism?  In particular, does $\gamma$ from Lemma~\ref{res:gamma_def} decompose in this fashion?
\item An automorphism of order 5 is defined on $D$ in \cite[\S3.3.2]{Dumas2} by
\[\pi:\ x \mapsto y, \quad y \mapsto x^{-1}(y + q^{-1}).\]
Based on preliminary computations we conjecture that $D^{\pi} \cong D$ as well.  Does $D$ admit any other automorphisms of finite order, and in particular any such automorphisms with fixed rings which are not $q$-division?
\item The ``non-bijective conjugation'' map in Proposition~\ref{res:z_endomorphism} gives rise to a doubly-infinite chain of $q$-division rings
\[\ldots \subsetneq z^2Dz^{-2} \subsetneq zDz^{-1} \subsetneq D \subsetneq z^{-1}Dz \subsetneq z^{-2}Dz^2 \subsetneq \dots\]
 What can be said about the limits
\[
 \bigcup_{i \geq 0} z^{-i}Dz^i\textrm{ and }\bigcap_{i \geq 0} z^iDz^{-i} \ \ ?
\]
\end{enumerate}
\end{questions}


\chapter{Poisson Deformations and Fixed Rings}\label{c:deformation_chapter}

The results of Chapter~\ref{c:fixed_rings_chapter} demonstrate that while it is possible to understand certain classes of fixed rings of the $q$-division ring $D$ by direct computation, this approach rapidly grows more difficult as the complexity of the automorphisms increase and it is highly unlikely to lead to a general theorem on the possible structures of $D^G$ for arbitrary finite $G$.  In this chapter we discuss a different approach to this question, which uses deformation theory to reframe the problem in terms of Poisson structures on certain commutative algebras and the possible deformations of these structures.  This approach is inspired by the work of \cite{BaudryThesis}, which considered a similar problem on the $q$-commuting Laurent polynomial ring $k_q[x^{\pm1},y^{\pm1}]$.

In \S\ref{s:q-division ring as deformation} we prove that the $q$-division ring is a deformation (in the sense of Definition~\ref{def:poisson deformation}) of the commutative field of fractions in two variables $k(x,y)$, with respect to the Poisson bracket $\{y,x\} = yx$.  Further, we show that for finite groups $G$ of monomial automorphisms the fixed ring $D^G$ is in turn a deformation of $k(x,y)^G$.  This allows us to break down the problem of describing fixed rings $D^G$ into two sub-problems: describing the Poisson structure of the commutative fixed rings $k(x,y)^G$, and describing the possible deformations of these fixed rings.

The first of these is precisely the Poisson equivalent of the Noether problem for $k(x,y)$: given a field of fractions $k(x,y)$ with an associated Poisson structure, and a finite group of Poisson automorphisms $G$, does there exist an isomorphism of Poisson algebras $k(x,y)^G \cong k(x,y)$?  In \S\ref{s:fixed Poisson rings} we will show that when $G$ is any finite group of Poisson monomial automorphisms (see Definition~\ref{def:action of SL2, Poisson def} for the precise definition) and the bracket on $k(x,y)$ is given by $\{y,x\} = yx$ then we can always find such an isomorphism.  

While we cannot yet describe all the possible Poisson deformations in this situation, and hence these results do not yet provide a full alternative proof to the results of Chapter~\ref{c:fixed_rings_chapter}, we will see that the proofs in \S\ref{s:fixed Poisson rings} are shorter and more intuitive than their $q$-commuting counterparts.  This suggests that Poisson deformation may be a more fruitful avenue to explore in seeking a general classification of the structure of fixed rings of $D$ under finite groups of automorphisms.

\section{A new perspective on $D$}\label{s:q-division ring as deformation}
Recall the setup of Example~\ref{ex:deformation of torus}, where we saw that the ring $k_q[x^{\pm1},y^{\pm1}]$ can be viewed quite concretely as a deformation of the commutative Poisson algebra $k[x^{\pm1},y^{\pm1}]$ via the ring 
\begin{equation}\label{eq:def of ring BB}\BB = k\langle x^{\pm1},y^{\pm1},z^{\pm1}\rangle/(xz-zx,yz-zy,xy-z^2yx).\end{equation}
By localizing $\BB$ at an appropriate Ore set, we may construct a larger ring $\DD$ such that  both $D = k_q(x,y)$ and the commutative field $k(x,y)$ can be realised as factor rings of $\DD$.  The next section is devoted to the proof of this result.

\subsection{The $q$-division ring as a deformation of $k(x,y)$}\label{ss:constructing the deformation}
We will begin by proving that $\BB$ is a Noetherian UFD (see Definition~\ref{def:NC UFD} below), and then use certain properties of this class of rings to construct an appropriate Ore set $\CC$ of $\BB$ to localize at.  This will give rise to the ring $\DD:= \BB\CC^{-1}$, which we will use to show that $D$ is a deformation of $k(x,y)$.

\begin{definition}\label{def:nc prime element}
An element $p$ of a ring $R$ is called \textit{prime} if $pR=Rp$ is a height 1 completely prime ideal of $R$.
\end{definition}
Let $\CC(P)$ denote the set of elements which are regular mod $P$, and set $\CC = \bigcap \CC(P)$ where $P$ runs over all height 1 primes of $R$.  We will sometimes write $\CC(R)$ for $\CC$ when we need to specify which ring it comes from.\nom{Noetherian UFD}
\begin{definition}\label{def:NC UFD}
A prime Noetherian ring $R$ is called a \textit{Noetherian UFD} if it has at least one prime ideal of height 1 and satisfies the following equivalent conditions:
\begin{enumerate}
\item Every height 1 prime of $R$ is of the form $pR$ for some prime element $p$ of $R$;
\item $R$ is a domain and every non-zero element of $R$ can be written in the form $cp_1\dots p_n$, where $p_1, \dots, p_n$ are prime elements of $R$ and $c \in \CC$.
\end{enumerate}
\end{definition}
This definition was introduced by Chatters in \cite{Chatters1}.  In particular, commutative rings which are UFDs in the conventional sense satisfy this definition \cite[Corollary~2.4]{Chatters1}; in this case, the elements of $\CC$ are the units in $R$.  This need not be true for non-commutative UFDs: instead, the set $\CC$ forms an Ore set in $R$ \cite[Proposition~2.5]{Chatters1} and the localization $R\CC^{-1}$ is both a Noetherian UFD \cite[Theorem~2.7]{Chatters1} and a principal ideal domain \cite[Corollary~1]{GS}.

Returning to the ring $\BB$ defined in \eqref{eq:def of ring BB}, our first aim will be to show that it is a Noetherian UFD.  We will show that it satisfies condition (1) of Definition~\ref{def:NC UFD}.  It is without a doubt a prime Noetherian ring, so we need only describe its height 1 prime ideals.

By standard localization theory (see for example \cite[Theorem~10.20]{GW1}), the prime ideals of $\BB$ are in 1-1 correspondence with those prime ideals of the ring
\[S:=k\langle x,y,z^{\pm1}\rangle /(xz-zx,yz-zy,xy-z^2yx)\]
which do not contain $x$ or $y$.  Hence we will restrict our attention to the prime ideals of $S$.

If we define $R = k[y,z^{\pm1}]$, then $S$ can be viewed as an Ore extension
\begin{equation}\label{eq:def of Ore extension S}S = R[x;\alpha], \quad \alpha: R \rightarrow R: y \mapsto z^2y,\ z \mapsto z.\end{equation}
The primes of $S$ which do not contain $x$ can be understood using the following results.
\begin{definition}
Let $R$ be a commutative Notherian ring, $I$ an ideal of $R$ and $\alpha$ an endomorphism.  Call $I$ an \textit{$\alpha$-invariant ideal} if $\alpha^{-1}(I) = I$, and \textit{$\alpha$-prime} if $I$ is $\alpha$-invariant and whenever $J$ and $K$ are two other ideals satisfying $\alpha(J) \subseteq J$ and $JK \subseteq I$, then $J \subseteq I$ or $K \subseteq I$ as well.  Finally, we call $R$ an \textit{$\alpha$-prime ring} if 0 is an $\alpha$-prime ideal.
\end{definition}
The following theorem is stated in slightly greater generality in \cite{Irving}; we quote here only the part required for our analysis of the primes of $S$.
\begin{theorem}\label{res:irving th1}\cite[Theorem~4.1,4.2]{Irving} Let $R$ be a commutative ring, $\alpha$ an endomorphism of $R$, and $S = R[x; \alpha]$.  If $I$ is a prime ideal of $S$ not containing $x$,  then $J = I \cap R$ is an $\alpha$-prime ideal of $R$.  Conversely, if $J$ is an $\alpha$-prime ideal of $R$, then $SJS$ is a prime ideal of $S$.\end{theorem}
\begin{theorem}\label{res:irving th2}\cite[Theorem~4.3]{Irving} Let $R$ be an $\alpha$-prime Noetherian ring, where $\alpha$ is an endomorphism of infinite order.  Then the only prime of $S = R[x;\alpha]$ not containing $x$ which lies over $(0)$ in $R$ is $(0)$.\end{theorem}
The following lemma uses these results to show that all non-zero primes of $S$ must contain (at least) one of $x$, $y$ or an irreducible polynomial in $z$.
\begin{lemma}\label{res:irving consequence lemma}
With $S = R[x;\alpha]$ as in \eqref{eq:def of Ore extension S}, let $P$ be a non-zero prime ideal of $S$ which does not contain $x$.  Then $P$ must contain $y$ or some element $p(z)$ which is irreducible in $k[z^{\pm1}]$.
\end{lemma}
\begin{proof}By Theorems~\ref{res:irving th1} and \ref{res:irving th2}, $I = P \cap R$ is a non-zero $\alpha$-prime ideal of $R$; we will show that $I$ must contain one of the required elements, and hence so will $P$.

Let $f \in I$ be an element of minimal $y$-degree; we will first show that $f = g(z)y^n$ for some $g(z) \in k[z^{\pm1}]$ and $n \geq 0$.  If $deg_y(f) = 0$ then this is clear, so suppose that $f$ has the form
\[f = g_n(z)y^n + \dots + g_1(z)y + g_0(z), \quad g_i(z) \in k[z]\]
for some $n \geq 1$ and $g_n(z) \neq 0$.  Since $I$ is $\alpha$-invariant and $\alpha$ is an automorphism we must have $\alpha(f) - z^{2n}f \in I$ as well, which has degree $<n$ and so must be zero by the minimality of $n$.  Comparing coefficients, we see that
\[g_i(z)(z^{2i} - z^{2n}) = 0, \quad 0 \leq i < n\]
which implies that $g_i(z) = 0$ for all $i < n$.  Thus $f = g_n(z)y^n$ as required. 

Let $g_n(z) = up_1(z)p_2(z)\dots p_r(z)$ be the factorization of $g_n(z)$ into irreducible polynomials, with $p_i(z) \in k[z^{\pm1}]$ irreducible and $u \in k[z^{\pm1}]^{\times}$.  All that remains is to observe that the ideals $p_i(z)R$ and $yR$ are all $\alpha$-invariant (indeed, they are $\alpha$-prime) and hence the $\alpha$-prime ideal $I$ must by definition contain one of these ideals.
\end{proof}
\begin{proposition}\label{res:BB is UFD}
The ring $\BB$ defined in \eqref{eq:def of ring BB} is a Noetherian UFD, and its height 1 prime ideals are precisely those generated by irreducible polynomials in $k[z^{\pm}]$.  
\end{proposition}
\begin{proof}
Combining Theorem~\ref{res:irving th1}, Theorem~\ref{res:irving th2} and Lemma~\ref{res:irving consequence lemma}, we see that every non-zero prime ideal of $S = R[x;\alpha]$ must contain (at least) one of $x$, $y$ or some irreducible $p(z) \in k[z^{\pm1}]$.  Since prime ideals of $\BB$ are in 1-1 correspondence with prime ideals of $S$ which do not contain $x$ or $y$ (where the correspondence is the natural one given by localization and contraction as in \cite[Theorem~10.20]{GW1}) it must follow that every prime ideal of $\BB$ contains some irreducible polynomial in $z$.  On the other hand, we can easily check that every irreducible polynomial $p(z) \in k[z^{\pm1}]$ generates a height 1 prime ideal in $\BB$.

Since every non-zero prime in $\BB$ must contain an irreducible polynomial $p(z) \in k[z^{\pm1}]$, and $p(z)\BB$ is prime for any such $p(z)$, it follows that the height 1 primes of $\BB$ are precisely the set $\{p(z)\BB : p(z) \textrm{ is irreducible in }k[z^{\pm1}]\}$.  Finally, since $z$ is central in $\BB$ these are all completely prime ideals, and therefore the $p(z)$ are prime elements in the sense given in Definition~\ref{def:nc prime element}.  It is now clear that $\BB$ satisfies condition (1) of Definition~\ref{def:NC UFD}, and hence it is a Noetherian UFD.
\end{proof}

By \cite[Proposition~2.5]{Chatters1} we can now form the localization\nom{D@$\mathfrak{D}$}
\begin{equation}\label{eq:def of DD}\DD := \BB\CC^{-1},\end{equation}
where $\CC = \bigcap\CC(P)$ for $P$ running through all height 1 primes of $\BB$.  By \cite[Corollary~1]{GS} every left or right ideal in $\DD$ is two-sided and principal; more precisely, we can see that every left or right ideal is generated by a polynomial in $z$.  In particular, it is clear that for any $\lambda \in k^{\times}$ the factor ring $\DD/(z-\lambda)\DD$ must be a division ring.


We will restrict our attention to the case where $q$ is not a root of unity and the base field $k$ admits an element $\hat{q}$ such that $\hat{q}^2 = q$.  As in Example~\ref{ex:deformation of torus}, we will make $k(x,y)$ into a Poisson algebra by defining $\{y,x\} = yx$; as in \cite[Equation 0-3]{GLa1}, this extends to a general formula for the bracket of two polynomials as follows:
\begin{equation}\label{eq:Poisson bracket def for polynomials}
\{a,b\} = yx \frac{\partial a}{\partial y} \frac{\partial b}{\partial x} - yx\frac{\partial a}{\partial x} \frac{\partial b}{\partial y}
\end{equation}
and this can be extended to the whole field $k(x,y)$ using the formula in \eqref{eq:extend Poisson bracket to localization}.

We are now in a position to prove one of the main results of this section.  The proof is based on the corresponding result in \cite[\S5.4.3]{BaudryThesis} for $k[x^{\pm1},y^{\pm1}]$.
\begin{proposition}\label{res:F is a deformation of D}
Let $k(x,y)$ be the field of rational functions in two commuting variables with Poisson bracket defined by $\{y,x\} = yx$.  Then $D$ is a deformation of $k(x,y)$ via the ring $\DD$ from \eqref{eq:def of DD}.
\end{proposition}
\begin{proof}
We need to show that $\DD$ contains some central, non-invertible, non-zero-divisor element $h$ such that $\DD/h\DD \cong k(x,y)$ as Poisson algebras (where the Poisson bracket on $\DD/h\DD$ is induced as in Definition~\ref{def:star product Poisson bracket}), and $\DD/(h-\lambda)\DD \cong D$ as algebras for appropriate values of $q$ and $\lambda$.

As in \cite[\S5.4.3]{BaudryThesis}, we set $h =2(1-z)$.  It is clear that $\DD$ is a domain and $h$ is central.  By Proposition~\ref{res:BB is UFD}, the polynomial $z-\lambda$ generates a height 1 completely prime ideal of $\DD$ for any $\lambda \in k^{\times}$; in particular, $\langle h \rangle = \langle z-1 \rangle$ is a proper ideal and so $h$ is non-invertible. 

The set of ideals $\{\langle z-\lambda \rangle: \lambda \in k^{\times}\}$ is equal to the set $\{\langle h-\mu\rangle: \mu \in k\backslash \{2\}\}$.  We have already noted that the quotient $\DD/(h-\mu)\DD$ for $\mu \neq 2$ must be a division ring; we can further observe that since $x$ and $y$ satisfy $xy=(1-\frac{1}{2}\mu)^2yx$ in this ring, we have a sequence of embeddings
\[k_q[x^{\pm1},y^{\pm1}] \hookrightarrow \DD/(h-\mu)\DD \hookrightarrow D\]
for $q = (1-\frac{1}{2}\mu)^2$.  By the universality of localization this second embedding must be an isomorphism, that is $\DD/(h-\mu)\DD \cong D$. In particular, when $\mu = 0$ there is an isomorphism of algebras $\DD/h\DD \cong k(x,y)$.

All that remains is to check that this process induces the correct Poisson bracket on $\DD/h\DD$, and to do this it suffices to check that we obtain the correct Poisson bracket on the generators $x$ and $y$.  As elements of $\DD$, we have
\begin{equation*}
yx - xy = (1-z^2)yx = \frac{1}{2}(1+z)hyx
\end{equation*}
and therefore according to the formula in Definition~\ref{def:star product Poisson bracket},
\begin{align*}
\{y,x\} &= \frac{1}{2}(1+z)yx \mod h\DD \\
&= yx 
\end{align*}
since $h=0$ in $\DD/h\DD$ implies $z=1$.  We therefore have an isomorphism of Poisson algebras between $k(x,y)$ with the multiplicative Poisson bracket $\{y,x\} = yx$ and $\DD/h\DD$ with induced Poisson bracket, and the result is proved.
\end{proof}

\subsection{The fixed ring $D^G$ as a deformation of $k(x,y)^G$}\label{ss:constructing deformation of fixed ring}
The results of \S\ref{ss:constructing the deformation} allow us to understand the $q$-division ring $D$ as a deformation of the Poisson algebra $k(x,y)$ with bracket defined by $\{y,x\} = yx$.  However, for certain finite groups of automorphisms $G$ we can extend this result to obtain further information, namely by using a subring of the ring $\DD$ from \eqref{eq:def of DD} to describe the fixed ring $D^G$ as a deformation of $k(x,y)^G$.

We will be interested in \textit{monomial actions} on $k(x,y)$ and $D$, which have already been considered in \S\ref{s:more fixed rings} in the case of the $q$-division ring.  For the Poisson algebra $k(x,y)$ these are defined using the following proposition, which is the Poisson equivalent of Proposition~\ref{res:action of SL2, q version}.
\begin{proposition}\label{res:action of SL2, Poisson version}
The group $SL_2(\mathbb{Z})$ acts by Poisson automorphisms on the commutative Poisson field $k(x,y)$ with bracket $\{y,x\} = yx$, where the action is defined by
\begin{equation}\label{eq:action of SL2 Poisson def}g.y = y^ax^c, \quad g.x = y^bx^d, \quad g =\begin{pmatrix}a&b\\c&d\end{pmatrix} \in SL_2(\mathbb{Z}),\end{equation}
or more generally for any $m,n \in \mathbb{Z}$:
\[g.(y^mx^n) = y^{am+bn}x^{cm+dn}\]
\end{proposition}
\begin{proof}
We define
\[g = \begin{pmatrix}a&b\\c&d\end{pmatrix}, \quad g' = \begin{pmatrix}a'&b'\\c'&d'\end{pmatrix} \in SL_2(\mathbb{Z})\]
and we are required to prove that
\begin{enumerate}
 \item $g.\{y,x\} = \{g.y,g.x\}$, i.e. $g$ is a Poisson automorphism on $k(x,y)$;
 \item $g'.(g.x) = (g'g).x$ and $g'.(g.x) = (g'g).y$ in $k(x,y)$, i.e. this defines an action of $SL_2(\mathbb{Z})$ on $k(x,y)$.
\end{enumerate}
Using \eqref{eq:Poisson bracket def for polynomials}, we can observe that the action of our Poisson bracket on monomials is $\{y^ax^c,y^bx^d\} = (ad-bc)y^{a+b}x^{c+d}$.  Now
\begin{align*}
 g.\{y,x\} = g.yx &=y^ax^cy^bx^d, \\
 \{g.y,g.x\} =\{y^ax^c,y^bx^d\} &= (ad-bc)y^{a+b}x^{c+d},
\end{align*}
and hence $g.\{y,x\} = \{g.y,g.x\}$ since $ad-bc=1$.  Thus $g \in SL_2(\mathbb{Z})$ defines a Poisson automorphism on $k(x,y)$.

For the reader's convenience, we record again the product of the matrices $g$ and $g'$:
\[g'g = \begin{pmatrix}a'a + b'c & a'b + b'd \\ ac' + d'c & c'b + d'd\end{pmatrix}.\]
The computation to verify condition (2) is now a simple one.  Indeed,
\begin{equation*}\begin{gathered}g'.(g.x) = g'.(y^bx^d) = y^{ba'}x^{bc'}y^{db'}x^{dd'} = y^{ba' + db'}x^{bc' + dd'} = (g'g).x\\
g'.(g.y) = g'.(y^ax^c) = y^{aa'}x^{ac'}y^{cb'}x^{cd'} = y^{aa' + cb'}x^{ac' + cd'} = (g'g).y
\end{gathered}\end{equation*}
\end{proof}
\begin{definition}\label{def:action of SL2, Poisson def}
Let $\theta$ be a Poisson automorphism on $k(x,y)$.  We call $\theta$ a \textit{Poisson monomial automorphism} if it can be represented by an element of $SL_2(\mathbb{Z})$ with the action defined in Proposition~\ref{res:action of SL2, Poisson version}.
\end{definition}

The corresponding action of $SL_2(\mathbb{Z})$ on $\DD$ can be defined in a very similar way to that of the action on $D$.  Here $z$ takes on the role of $\hat{q}$ and we define the action to be
\begin{equation}\label{eq:action of SL2 on DD}g.y = z^{ac}y^ax^c, \quad g.x = z^{bd}y^bx^d, \quad g.z = z, \quad \textrm{ where }g = \begin{pmatrix}a&b\\c&d\end{pmatrix}\in SL_2(\mathbb{Z}).\end{equation}
Since $z$ is central and invertible in $\DD$, the proof that this defines an action of $SL_2(\mathbb{Z})$ on $\DD$ follows identically to that of Proposition~\ref{res:action of SL2, q version}.

Given that $z$ is fixed by the action of $\SLZ$, the ideals $(h-\lambda)\DD$ are stable under this action and so the definition in \eqref{eq:action of SL2 on DD} induces an action by Poisson automorphisms on $\DD/h\DD$ and an action by algebra automorphisms on $\DD/(h-\lambda)\DD$.  It is easy to see that these actions agree with those defined in Proposition~\ref{res:action of SL2, q version} and Proposition~\ref{res:action of SL2, Poisson version}.  Therefore if $G$ is a finite subgroup of $\SLZ$ we will assume that it acts on each of the rings $k(x,y)$, $D$ and $\DD$ according to definitions \eqref{eq:action of SL2 Poisson def}, \eqref{eq:q action of SL2 def}, \eqref{eq:action of SL2 on DD} respectively, without distinguishing between them unnecessarily.

Before proving our main result of this section (Theorem~\ref{res:fixed ring is deformation theorem}) we state one additional technical lemma which will be used in the proof of the theorem.  The proof for this result can be found in \cite{dumas_invariants}.
\begin{lemma}\label{res:sublemma}\cite[\S3.2.3]{dumas_invariants} Let $G$ be a finite group.  If
\[0 \longrightarrow A \stackrel{\alpha}{\longrightarrow} B \stackrel{\beta}{\longrightarrow} C \longrightarrow 0\]
is an exact sequence of $G$-modules, then the induced sequence
\[0 \longrightarrow A^G \stackrel{\alpha'}{\longrightarrow} B^G \stackrel{\beta'}{\longrightarrow} C^G \longrightarrow 0\]
is exact.
\end{lemma}
\begin{theorem}\label{res:fixed ring is deformation theorem}
Let $G$ be a finite subgroup of $\SLZ$.  Then the fixed ring $D^G$ is a deformation of $k(x,y)^G$, where the Poisson bracket on $k(x,y)^G$ is induced by the bracket $\{y,x\} = yx$ on $k(x,y)$.
\end{theorem}
\begin{proof}
Let $G$ be a finite subgroup of $\SLZ$ acting on $k(x,y)$, $D$ and $\DD$ by monomial automorphisms.  Then $0 \rightarrow h\DD \rightarrow \DD \rightarrow \DD/h\DD \rightarrow 0$ is an exact sequence of $G$-modules, and hence by Lemma~\ref{res:sublemma} we have another exact sequence
\begin{equation}\label{eq:short exact Poisson sequence}0 \longrightarrow h\DD^G \longrightarrow \DD^G \longrightarrow (\DD/h\DD)^G \longrightarrow 0.\end{equation}
This gives rise to an isomorphism of rings $\DD^G/h\DD^G \cong (\DD/h\DD)^G$, which will be an isomorphism of Poisson algebras if the brackets on $\DD^G/h\DD^G$ and $(\DD/h\DD)^G$ agree.  This is easy to see, however, since both brackets are induced by commutators which can each be computed in the same ring $\DD$.

Since $\DD/h\DD \cong k(x,y)$ by Proposition~\ref{res:F is a deformation of D} and this isomorphism is clearly $G$-equivariant, we obtain isomorphisms of Poisson algebras $\DD^G/h\DD^G \cong (\DD/h\DD)^G \cong k(x,y)^G$.

By a similar argument, we obtain an isomorphism of algebras $\DD^G/(h-\lambda)\DD^G \cong D^G$, for $\lambda = 2(1-\hat{q})$.  Thus the fixed ring of the deformation is a deformation of the fixed ring, as required.
\end{proof}

\section{Fixed rings of Poisson fields}\label{s:fixed Poisson rings}
The results of \S\ref{s:q-division ring as deformation} translate the problem of understanding the fixed ring $D^G$ into two sub-problems:
\begin{enumerate}
\item Understanding the Poisson structure of the fixed ring $k(x,y)^G$;
\item Describing all the possible deformations of this ring.
\end{enumerate}
In this section we will focus on the first of these problems.

It is standard that $k(x,y)^G \cong k(x,y)$ as \textit{algebras} for any finite group $G$, but this does not guarantee that their Poisson structures will also agree.  Even for the case of finite groups of monomial automorphisms, until now only the Poisson structure of $k(x,y)^{\tau}$ was known.  In this section we will extend this to a description of the fixed rings of all finite groups of monomial Poisson automorphisms on $k(x,y)$; in addition to being an interesting result in its own right, this will demonstrate that with the right techniques it is a genuine simplification to consider the structure of commutative Poisson fixed rings rather than their $q$-commuting equivalents.

The aim of this section will be to prove the following theorem, which we will approach on a case by case basis as in Chapter~\ref{c:fixed_rings_chapter}.
\begin{theorem}\label{res:Poisson fixed rings big theorem}
Let $k$ be a field of characteristic zero which contains a primitive third root of unity $\omega$, and let $G$ be a finite subgroup of $SL_2(\mathbb{Z})$ which acts on $k(x,y)$ by Poisson monomial automorphisms as defined in Definition~\ref{def:action of SL2, Poisson def}.  Then there exists an isomorphism of Poisson algebras $k(x,y)^G \cong k(x,y)$.
\end{theorem}
At first glance it is not clear that describing the Poisson structure of $k(x,y)^G$ should be any easier than describing the algebra structure of $D^G$: all we have done is replace the requirement to find two elements $f, g \in D^G$ such that $fg = qgf$ with the requirement that we find two elements in $k(x,y)^G$ such that $\{g,f\} = gf$.  However, by exploiting both the $q$-commuting structure of $D$ and the ease of computation in $k(x,y)$ we can develop a method which often produces suitable Poisson generators for the fixed rings.

The key idea is that by using a technique inspired by the work of Alev and Dumas in \cite{AD3} and Artamonov and Cohn in \cite{AC1}, we can construct potential $q$-commuting generators for $D^G$ by constructing them term by term in $k_q(y)(\!(x)\!)$, and then replace $\hat{q}$ by 1 throughout to obtain elements of $k(x,y)$ with the desired properties. We describe this approach in more detail next.

In Appendix~\ref{s:magma_code} we define the Magma procedure \texttt{qelement}, which accepts as input an element of the form
\begin{equation}\label{eq:form of g required to use qelement}g = \lambda y + \sum_{i \geq 1}a_i x^i \in k_q(y)(\!(x)\!), \quad \lambda \in k^{\times},\ a_i \in k(y)\end{equation}
and constructs another power series $f \in k_q(y)(\!(x)\!)$ such that $fg = qgf$.  We note that $f$ need not represent an element of $D$ even if $g$ does.  Appendix~\ref{s:magma_theory} also describes results which allow us to test if $f \in D$ and if so, writes it as a left fraction $f = v^{-1}u$; however, as demonstrated by the example in Appendix~\ref{s:magma_example} even quite simple products of non-commutative fractions become unmanageably complicated when forced into the form of a single left fraction.  Verifying that $f \in D^G$ or proving that $k_q(f,g) = D^G$ is essentially impossible in this situation.

On the other hand, commutative fractions are far easier to multiply and factorize, and elements which were unmanageably large in $D$ often reduce to quite simple elements of $k(x,y)$ upon replacing $\hat{q}$ by 1 (recall that $\hat{q}$ denotes a square root of $q$).  Further, if $f$, $g \in D$ satisfy $fg=qgf$ and it makes sense to replace $\hat{q}$ by 1 in these elements (denoted here by $\overline{f}$ and $\overline{g}$) then the construction of the Poisson bracket as the image of a commutator in $\DD$ guarantees that $\{\overline{g},\overline{f}\} = \overline{g}\overline{f}$ (this claim is illustrated more rigorously in Lemma~\ref{res:order 2 SVdB gens Poisson}).

Therefore if $G$ is a finite subgroup of $SL_2(\mathbb{Z})$ acting on $k(x,y)$ by Poisson monomial automorphisms and we expect that the fixed ring $k(x,y)^G$ will be Poisson-isomorphic to $k(x,y)$, we may apply the following procedure to attempt to construct generators for $k(x,y)^G$.
\begin{enumerate}
\item Choose a fraction $g \in D^G$ of the form \eqref{eq:form of g required to use qelement}.
\item Apply procedure \texttt{qelement} in Magma to construct $f \in k_q(y)(\!(x)\!)$ such that $fg = \hat{q}^2gf$.
\item Use procedure \texttt{checkrationalL} to check whether $f \in D$ (within the limits of the computer's computational power); if true, use \texttt{findrationalL} to write $f = v^{-1}u$ for $v,u \in k_q[x,y]$.
\item If possible, replace $\hat{q}$ by 1 in $f$ and $g$ and check whether $f \in k(x,y)^G$.
\end{enumerate}
We note that having already proved the $q$-commuting version of Theorem~\ref{res:Poisson fixed rings big theorem}, we could simply take the $q$-commuting generators obtained for the corresponding results in Chapter~\ref{c:fixed_rings_chapter} and set $\hat{q} = 1$ in order to obtain Poisson generators.  However, since the motivation for studying these Poisson fixed rings is to demonstrate that we can understand subrings of $D$ via the Poisson structure of subrings of $k(x,y)$, for the purposes of this section we will (mostly) ignore the results of \S\ref{s:more fixed rings} and proceed using \texttt{qelement} and the approach outlined above.

Recall that up to conjugation, the group $SL_2(\mathbb{Z})$ admits only four non-trivial finite subgroups: the cyclic groups of order 2, 3, 4 and 6.  As in Chapter~\ref{c:fixed_rings_chapter}, it therefore suffices to describe the fixed rings of $k(x,y)$ with respect to one Poisson automorphism of each conjugacy class, which are listed in Table~\ref{fig:table_of_maps_poisson} below.
\begin{table}[h]
\centering
\begin{tabular}{c|rl}
Order & \multicolumn{2}{c}{Automorphism} \\ \hline
2 & $\tau:$&$ x \mapsto x^{-1},\  y \mapsto y^{-1}$ \\
3 & $\sigma:$&$ x \mapsto y, \ y \mapsto (xy)^{-1}$ \\
4 & $\rho:$&$ x \mapsto y^{-1}, \ y \mapsto x$ \\
6 & $\eta:$&$ x \mapsto y^{-1}, \ y \mapsto xy$
\end{tabular}
\caption{Conjugacy class representatives of finite order Poisson monomial automorphisms on $k(x,y)$.}\label{fig:table_of_maps_poisson}
\end{table}

We have already noted that the fixed ring under the automorphism $\tau$ of order 2 has been described in \cite{BaudryThesis}; the proof involves certain clever factorizations and manipulations of equalities in $k(x,y)^{\tau}$ and does not generalize easily to automorphisms of higher order.  A simpler description of $k(x,y)^{\tau}$ may be obtained by using the pair of elements defined in \eqref{eq:order_2_gens}; since we will use this description of $k(x,y)^{\tau}$ in later results, the next lemma provides a proof of this statement.
\begin{lemma}\label{res:order 2 SVdB gens Poisson}
Let $u,\ v \in k(x,y)$ be defined by
\begin{equation}\label{eq:order 2 poisson generators}u = \frac{x-x^{-1}}{y^{-1}-y}, \quad v = \frac{xy-(xy)^{-1}}{y^{-1}-y},\end{equation}
and let $\tau$ be as in Table~\ref{fig:table_of_maps_poisson}.  Then $k(u,v) = k(x,y)^{\tau}$ and $\{v,u\} = vu$.
\end{lemma}
\begin{proof}
It is clear that $u,v \in k(x,y)^{\tau}$.  The claim that $\{v,u\} = vu$ may be verified computationally using the formula \eqref{eq:Poisson bracket def for polynomials}, but since no polynomials in $q$ appear in the denominator of $u$ or $v$ we may also prove this claim as follows.  Let
\[u_q = (x-x^{-1})(y^{-1}-y)^{-1}, \ v_q = (xy-x^{-1}y^{-1})(y^{-1}-y)^{-1}\]
be elements of $D$; by \cite[\S 13.6]{SV1} we know that $u_qv_q = qv_qu_q$.  These lift without modification to elements $u_z$ and $v_z$ of the ring $\DD$ from Theorem~\ref{res:F is a deformation of D}, where $u_zv_z = z^2v_zu_z$.  Recall that we defined $h = 2(1-z)$ and that $\DD/h\DD \cong k(x,y)$ as Poisson algebras; hence
\begin{align*}
\{v,u\} &= \frac{1}{h}(v_zu_z - u_zv_z) \quad \mod h\DD \\
& = \frac{1}{2(1-z)}(1-z^2)v_zu_z \quad \mod h\DD \\
& = \frac{1}{2}(1+z) v_zu_z \quad \mod h\DD\\
& = vu,
\end{align*}
since $h = 0$ implies $z=1$ in $\DD/h\DD$.

Finally, we need to prove that $k(u,v) = k(x,y)^{\tau}$.  Since $k(u,v) \subseteq k(x,y)^{\tau} \subsetneq k(x,y)$ and $[k(x,y):k(x,y)^{\tau}] = 2$, if we can show that $[k(x,y):k(u,v)]\leq 2$ as well then it must follow that $k(u,v) = k(x,y)^{\tau}$.  We define a polynomial in $k(u,v)[t]$ by
\begin{equation}\label{eq:min poly for x order 2}m_x(t) = vt^2 + (v^2-u^2+1)t +v,\end{equation}
which has $x$ as a root (this can be seen by direct computation in $k(x,y)$).  Since $k(u,v)$ is a subring of $k(x,y)^{\tau}$ and $x$ is not fixed by $\tau$, we cannot have $x \in k(u,v)$ and so \eqref{eq:min poly for x order 2} must be irreducible, i.e. it is the minimal polynomial for $x$ over $k(u,v)$.

Therefore the Galois extension $k(u,v)(x)$ has order 2 over $k(u,v)$.  Observe further that $k(u,v)(x) = k(x,y)$ since
\begin{align*}
\frac{1+vx}{ux} &= \frac{y^{-1}-y + x^2y-y^{-1}}{x^2-1} \\
&= \frac{(x^2-1)y}{x^2-1}\\
&=y,
\end{align*}
and hence $y \in k(u,v)(x)$.  This implies that $[k(x,y):k(u,v)] = 2$ and so $k(u,v) = k(x,y)^{\tau}$ as required.  Finally, since $u$ and $v$ satisfy $\{v,u\} = vu$ the isomorphism $k(x,y) \cong k(x,y)^{\tau}$ is in fact an isomorphism of Poisson algebras.
\end{proof}
We now turn our attention to the order 3 case, which caused such problems in the $q$-division ring.  As observed in Remark~\ref{rem:order 3 snark}, the unintuitive generator $f$ used in Theorem~\ref{res:epic_theorem_2} had its roots in a single left fraction constructed using \texttt{qelement}; the full definition of this element is given in Appendix~\ref{s:magma_example} and takes 9 pages to write down fully.  However, since the Poisson bracket captures only a first-order impression of the non-commutative structure in $D$ and multiplication of fractions is far less complicated in $k(x,y)$, it is perhaps unsurprising that upon replacing $\hat{q}$ with 1 in this 9 page element we obtain a far simpler element which satisfies our requirements in the Poisson case.

Having set $\hat{q}=1$ in the elements appearing in Appendix~\ref{s:magma_example}, we obtain two elements in $k(x,y)$ of the form 
\begin{equation}\label{eq:order 3 Poisson generators}f = a^2b/c^2, \quad g = b/a,\end{equation}
where
\begin{equation}\label{eq:order 3 poisson building blocks}\begin{aligned}
a &= x + \omega y + \omega^2 (xy)^{-1}, \\
b &= x^{-1} + \omega y^{-1} + \omega^2 xy, \\
c &= xy^{-1} + xy^2 + x^{-2}y^{-1} - 3,
\end{aligned}\end{equation}
in a similar manner to the $q$-commuting case.  Observe that $\sigma$ acts on $a$ and $b$ as multiplication by $\omega^2$ and fixes $c$.  

We note that since our Magma functions can only \textit{approximate} computations in $k_q(y)(\!(x)\!)$, the above on its own is not a proof: we still need to verify that $f$ and $g$ from \eqref{eq:order 3 Poisson generators} do indeed generate the fixed ring $k(x,y)^{\sigma}$.  This is the purpose of the next result.
\begin{proposition}\label{res:order 3 poisson result}
Let $k$ be a field of characteristic 0 containing a primitive third root of unity $\omega$, and let $f$ and $g$ be defined as in \eqref{eq:order 3 Poisson generators}.  Then the Poisson subalgebra of $k(x,y)$ generated by $f$ and $g$ is equal to $k(x,y)^{\sigma}$, and there is an isomorphism of Poisson algebras $k(x,y)^{\sigma} \cong k(x,y)$.
\end{proposition}
\begin{proof}
Since $\sigma$ acts on $a$ and $b$ as multiplication by $\omega^2$ and fixes $c$, it is clear that $\sigma(f) = f$ and $\sigma(g)= g$.  Using the formula for the bracket of two elements given in \eqref{eq:Poisson bracket def for polynomials} it follows by a long yet elementary computation (which we do not reproduce here) that $\{g,f\} = gf$.

The proof that $k(f,g) = k(x,y)^{\sigma}$ follows in a similar manner to the corresponding $q$-commuting case.  Indeed, we find that the fixed ring $k[x^{\pm1},y^{\pm1}]^{\sigma}$ is generated as an algebra by the three standard generators
\begin{align*}
p_1&:= x + y + (xy)^{-1} \\
p_2&:= x^{-1} + y^{-1} + xy\\
p_3&:= y^{-1}x + y^2x + y^{-1}x^2 +6
\end{align*}
(see, for example, \cite[\S4.2.2]{dumas_invariants}) and hence it suffices to show that $p_1$, $p_2$ and $p_3$ are in the Poisson algebra $k(f,g)$.  This can now be observed by direct computation, however, since
\begin{align*}
p_1 &= \frac{\omega (g^3+1)^2f^2 + g(2-g^3)f + \omega^2 g^2}{fg^3}\\
p_2 &= \frac{\omega^2 (g^3+1)^2f^2 + g(2g^3-1) + \omega g^2}{fg^2}\\
p_3 &= \frac{1}{2}(p_1p_2 - \{p_2,p_1\}+9)
\end{align*}
are all in the Poisson algebra $k(f,g)$, as required.
\end{proof}

\begin{corollary}\label{res:order 6 poisson result}
Let the field $k$ be as in Proposition~\ref{res:order 3 poisson result}, and $\eta$ the Poisson monomial automorphism of order 6 in Table~\ref{fig:table_of_maps_poisson}.  Then the fixed ring $k(x,y)^{\eta}$ is isomorphic to $k(x,y)$ as Poisson algebras.
\end{corollary}
\begin{proof}
As in Theorem~\ref{res:order_6_thm}, we observe that $\eta^3=\tau$ and so $k(x,y)^{\eta} = (k(x,y)^{\tau})^{\eta}$.  Thus it suffices to consider $k(u,v)^{\eta}$, where $k(u,v) = k(x,y)^{\tau}$ as in Lemma~\ref{res:order 2 SVdB gens Poisson}.  We may make a change of variables $u':=-u^{-1}$ without affecting the structure of $k(u,v)$: using the formula in \eqref{eq:Poisson bracket def for polynomials} we can easily see that $\{v,u'\} = -vu'$ and hence $k(u,v)$ and $k(v,u')$ are Poisson-isomorphic.  Now the action of $\eta$ on $u'$ and $v$ is as follows:
\begin{align*}
\eta(u') &= (xy - (xy)^{-1})/(y^{-1} - y)\\
& = v; \\
\eta(v) &= (y^{-1}yx - y(yx)^{-1})/(xy - (xy)^{-1}) \\
&= v^{-1}u^{-1}.
\end{align*}
Hence $\eta$ acts on $k(u',v)$ as the order 3 map $\sigma$, and therefore by Lemma~\ref{res:order 2 SVdB gens Poisson} and Proposition~\ref{res:order 3 poisson result} we have isomorphisms of Poisson algebras $k(x,y)^{\eta} = (k(x,y)^{\tau})^{\eta} \cong k(u',v)^{\sigma} \cong k(x,y)$.
\end{proof}

We have only one case left to consider: the Poisson monomial automorphisms of order 4.  As in Proposition~\ref{res:order 3 poisson result}, we proceed by first considering the corresponding map on the $q$-division ring, and then formally constructing a pair of $q$-commuting elements in Magma and replacing $\hat{q}$ by 1 throughout to obtain appropriate generators for the fixed ring.

Let $\rho$ be the automorphism of order 4 defined in Table~\ref{fig:table_of_maps_poisson}, that is
\[\rho: x \mapsto y^{-1}, \ y \mapsto x.\]  
As in Corollary~\ref{res:order 6 poisson result}, we may begin by observing that $\rho^2 = \tau$ and hence restrict our attention to the action of $\rho$ on the elements $u$ and $v$ from Lemma~\ref{res:order 2 SVdB gens Poisson}.  We find that
\begin{align*}
 \rho(u) &= \frac{y-y^{-1}}{x^{-1}-x} = -u^{-1} \\
 \rho(v) &= \frac{yx^{-1}-y^{-1}x}{x-x^{-1}} = (u^{-1}-u)v^{-1}.
\end{align*}
Let $\varphi$ be the map defined on $k(u,v)$ by
\begin{equation}\label{eq:def phi Poisson}\varphi: u \mapsto -u^{-1}, \ v \mapsto (u^{-1}-u)v^{-1}.\end{equation}
This must be a Poisson homomorphism since it is induced by the action of the Poisson automorphism $\rho$ on $k(u,v)$, and an easy computation shows that $\varphi^2 = id$; hence $\varphi$ is an automorphism of order 2 on $k(u,v)$, and $k(x,y)^{\rho} = k(u,v)^{\varphi}$.  We may also define an automorphism corresponding to $\varphi$ on the $q$-division ring $k_q(u,v)$, namely
\[\varphi_q: u \mapsto -u^{-1},\ v \mapsto (u^{-1}-qu)v^{-1};\]
note that up to a change of variables this is precisely the automorphism considered in \S\ref{s:fixed ring}.  

Define an element in $k_q(u,v)$ by
\[g_q = (u-\varphi_q(u))(v-\varphi_q(v))^{-1} = (u + u^{-1})(v - (u^{-1}-qu)v^{-1})^{-1},\]
which as always is fixed by $\varphi_q$ since the map acts on each component as multiplication by -1.  The element $g_q$ has been chosen precisely because it has the form \eqref{eq:form of g required to use qelement} when embedded into $k_q(v)(\!(u)\!)$, so we may use \texttt{qelement} to construct some $f_q \in k_q(v)(\!(u)\!)$ such that $f_qg_q = qg_qf_q$.  Finally, upon setting $q=1$, we obtain the elements 
\begin{equation}\label{eq:order 4 gens Poisson}f = \frac{(u^2 + uv^2 - 1)^2u}{(u^2v^2 - u^2 - 2u -1)(u^2v^2 + u^2 - 2u+1)}, \quad g = \frac{(u^2+1)v}{u^2 + uv^2 - 1},\end{equation}
which satisfy the required properties as demonstrated by the following lemma.
\begin{lemma}\label{res:order 4 Poisson lemma}The elements $f$ and $g$ in \eqref{eq:order 4 gens Poisson} are fixed by $\varphi$ and satisfy $\{g,f\} = gf$.
\end{lemma}
\begin{proof}
We begin by computing the action of $\varphi$ on the various polynomials appearing in $f$ and $g$.
\begin{align*}
 \varphi((u^2+1)v) &= (u^{-2}+1)(u^{-1}-u)v^{-1}\\
&= u^{-3}v^{-1}(1-u^2)(1+u^2)\\
 \varphi(u^2 +uv^2 -1) &= u^{-2} - u^{-1}(u^{-1}-u)^2v^{-2} - 1 \\
 &= u^{-3}v^{-2}(uv^2 - 1 - u^4 + 2u^2 - u^3v^2) \\ 
 &= u^{-3}v^{-2}(1-u^2)(u^2 + uv^2 - 1)\\
\varphi(u^2v^2 - u^2 - 2u -1) &= u^{-2}(u^{-1}-u)^2v^{-2}-u^{-2} + 2u^{-1}-1 \\
 &= u^{-4}v^{-2}(1+u^4 - 2u^2 - 2^2v^2 + 2u^3v^2 - u^4v^2) \\
 &= u^{-4}v^{-2}(-(u-1)^2(u^2v^2-u^2-2u-1))
\end{align*}
Similarly, $\varphi(u^2v^2+u^2-2u+1) = u^{-4}v^{-2}((u+1)^2(u^2v^2+u^2-2u+1))$.  

Putting these together, it is now easy to see that
\begin{align*}
 \varphi(f) &= \frac{-u^{-1}u^{-6}v^{-4}(1-u^2)^2(u^2+uv^2-1)^2 }{-u^{-8}v^{-4}(1+u)^2(1-u)^2(u^2v^2-u^2-2u-1)(u^2v^2+u^2-2u+1)} \\
&= \frac{(u^2 + uv^2 - 1)^2u}{(u^2v^2 - u^2 - 2u -1)(u^2v^2 + u^2 - 2u+1)}\\
&= f,
\end{align*}
and similarly,
\begin{align*}
 \varphi(g) &= \frac{u^{-3}v^{-1}(1-u^2)(1+u^2)}{u^{-3}v^{-2}(1-u^2)(u^2 + uv^2 - 1)}\\
&= \frac{(1+u^2)v}{(u^2 + uv^2 - 1)}\\
&= g.
\end{align*}
Direct computation (e.g. in Magma) using the formula in \eqref{eq:Poisson bracket def for polynomials} demonstrates that $\{g,f\} = gf$ as required.
\end{proof}
All that remains is to show that $k(f,g) = k(x,y)^{\rho}$.  Since $k(x,y)^{\rho} = k(u,v)^{\varphi}$ and $\varphi$ has order 2, it suffices to verify that $[k(u,v):k(f,g)] = 2$, which can be done using (commutative) Galois theory.
\begin{proposition}\label{res:order 4 poisson result}
Let $k$ be a field of characteristic zero, and $\rho$ the Poisson monomial automorphism of order 4 in Table~\ref{fig:table_of_maps_poisson}.  Then $k(x,y)^{\rho}\cong k(x,y)$ as Poisson algebras.
\end{proposition}
\begin{proof}
By the preceding discussion and Lemma~\ref{res:order 4 Poisson lemma}, all that remains to show is that $[k(u,v):k(f,g)] = 2$, where $f$ and $g$ are the elements defined in \eqref{eq:order 4 gens Poisson}, $u$ and $v$ are from Lemma~\ref{res:order 2 SVdB gens Poisson} and $\varphi$ is the Poisson automorphism of order 2 defined in \eqref{eq:def phi Poisson}.

We define a polynomial in $k(f,g)[t]$ by
\[m_{u}(t) = ft^2 - (fg^2-f+1)(fg^2+f+1)t -f,\]
which has $u$ as a root.  Since we cannot have $u \in k(f,g)$, we conclude in the same manner as Lemma~\ref{res:order 2 SVdB gens Poisson} that $m_u(t)$ must be the minimal polynomial for $u$ over $k(f,g)$.

Now we may see that the Galois extension $k(f,g)(u)$ is equal to $k(u,v)$, since direct computation shows that
\[v = (fg^3 +fgu + g)/(fu-1).\]
Thus $k(f,g) \subseteq k(u,v)^{\varphi} \subsetneqq k(u,v)$ with $[k(u,v):k(f,g)] = 2$, and since there can be no intermediate extension we must have $k(f,g) = k(u,v)^{\varphi}$, as required.
\end{proof}
Finally, by combining Lemma~\ref{res:order 2 SVdB gens Poisson}, Proposition~\ref{res:order 3 poisson result}, Proposition~\ref{res:order 4 poisson result} and Corollary~\ref{res:order 6 poisson result} the proof of Theorem~\ref{res:Poisson fixed rings big theorem} is complete.

As a result of the computations done to understand the fixed ring $k(x,y)^{\rho}$, we also obtain the following corollary of Proposition~\ref{res:order 4 poisson result}, which is the Poisson analogue of Theorem~\ref{res:epic_theorem}:
\begin{corollary}
There is an isomorphism of Poisson algebras $k(u,v)^{\varphi} \cong k(u,v)$.
\end{corollary}
This suggests that, as in the $q$-commuting case, we should look for a Poisson isomorphism from a general fixed ring $k(x,y)^G$ to $k(x,y)$ whenever $G$ is a finite group of Poisson automorphisms which do not restrict to $k[x,y]$.  Since the Poisson automorphism group of $k(x,y)$ with respect to the bracket $\{y,x\} = yx$ is known (see \cite{blanc}), proving a theorem of this form for the Poisson case may be a more attractive problem to tackle than the corresponding $q$-commuting one.

\chapter{Poisson Primitive Ideals in $\GL{3}$ and $\SL{3}$}\label{c:H-primes}

In Chapter~\ref{c:deformation_chapter} we viewed the $q$-division ring $D$ as a deformation of the commutative Poisson field $k(x,y)$ with the aim of learning more about the structure of $D$.  In this chapter we will take the opposite view: starting with a non-commutative algebra, we will use the language of deformation to better understand the structure of its semi-classical limit.

Much of the work in this chapter is based on the corresponding results for the quantum algebras $\QML{3}$, $\QGL{3}$ and $\QSL{3}$ in $\cite{GL2,GL1}$.  In the first of these papers, Goodearl and Lenagan define a rational action of an algebraic torus $\HH$ on $\QML{3}$ and construct generating sets of quantum minors for each of the 230 $\HH$-prime ideals.  In \cite{GL1} they focus on $\QGL{3}$, which admits a much more manageable 36 $\HH$-primes, and use this and the Stratification Theorem to find generating sets for all of the primitive ideals of $\GL{3}$.  Finally, these results are extended to $\QSL{3}$ by use of the isomorphism $\QGL{3} \cong \QSL{3}[z^{\pm1}]$ from \cite{LS1}.

Our aim will be to perform a similar analysis for the Poisson algebras $\GL{3}$ and $\SL{3}$, with a view to eventually verifying Conjecture~\ref{conj:goodearl} for the case of $GL_3$ and $SL_3$.  We will find that the Poisson structure of $\GL{3}$ and $\SL{3}$ matches up very closely with the non-commutative structure of $\QGL{3}$ and $\QSL{3}$ in almost all respects, although we will see in \S\ref{ss:UFDs} that occasionally we will need to apply quite different techniques to the quantum case to prove the corresponding Poisson result.

\begin{ImpNot}Throughout this chapter, we will assume that $k$ is an algebraically closed field of characteristic zero.  The assumption that $q \in k^{\times}$ is not a root of unity remains in force.\end{ImpNot}

This chapter references several large figures, which have been collected together in Appendix~\ref{c:H-prime figures} for convenience.  Figures of this type are referenced as Figure~\ref{c:H-prime figures}.$n$.  Note that there is also a List of Figures on page~\pageref{key for list of figures page}.

\section{Background and initial results}\label{s:H primes background,defs}
We begin by making formal the view of $\ML{n}$ as the semi-classical limit of the quantum matrices $\QML{n}$ (for the definition of $\QML{n}$, see \S\ref{s:notation}).
\begin{definition}\label{def:algebra B for quanum matrices}
Define $\RR_n$\nom{R@$\mathfrak{R}_n$} to be the $k[t^{\pm1}]$-algebra in $n^2$ variables $\{Y_{ij}: 1 \leq i,j \leq n\}$, subject to the same relations as $\QML{n}$ but with every occurence of $q$ replaced by the variable $t$.
\end{definition}
This is a generalization of the setup from \S\ref{ss:background examples deformation}, and it is easy to see that in this case we obtain an isomorphism of $k$-algebras
\[\QML{n} \cong \RR_n/(t-q)\RR_n.\]
Meanwhile, when we quotient out the ideal $(t-1)\RR_n$ we obtain the commutative coordinate ring $\ML{n}$.  Using the semi-classical limit process defined in Definition~\ref{def:star product Poisson bracket}, this induces a Poisson bracket on $\ML{n}$, which we will take as our definition of the Poisson structure on $\ML{n}$.  By direct computation, we find that for any set of four generators $\{x_{ij},x_{im},x_{lj},x_{lm}\}$ with $i<l$ and $j<m$ the Poisson bracket is defined by\nom{O@$(\mathcal{O}(M_n), \{\cdot, \cdot\})$}
\begin{equation}\label{eq:poisson relns for nxn matrices}\begin{aligned}
\{x_{ij},x_{im}\} &= x_{ij}x_{im}, & \{x_{im},x_{lm}\} &= x_{im}x_{lm}, \\
\{x_{im},x_{lj}\} &= 0, & \{x_{ij},x_{lm}\} &= 2x_{im}x_{lj}.
\end{aligned}\end{equation}
Recall from \S\ref{s:notation} that $[I|J]_q$ denotes a \textit{quantum minor} in $\QML{n}$: here $I$ and $J$ are ordered subsets of $\{1,\dots, n\}$ of equal cardinality, and $[I|J]$ is defined to be the quantum determinant on the subalgebra of $\QML{n}$ generated by $\{X_{ij}:i \in I, j \in J\}$.  Recall also that $\wt{I}$ denotes the complement of the set $I$ in $\{1, \dots, n\}$ and that we will often drop the set notation for ease of notation: for example, the minor $[\{1,2\}|\{2,3\}]$ could be denoted by $[12|23]$ or $[\wt{3}|\wt{1}]$.

\begin{notation}
In order to easily distinguish between elements of the different types of algebra, the generators of $\RR_n$ will be denoted by $Y_{ij}$\nom{Y@$Y_{ij}$}, the generators of $\QML{n}$ by $X_{ij}$\nom{X@$X_{ij}$} and those of $\ML{n}$ by $x_{ij}$\nom{X@$x_{ij}$}.  Minors in each algebra  will be denoted by $[I|J]_t$, $[I|J]_q$ and $[I|J]$ respectively.  Finally, elements of $\QGL{n}$ or $\QSL{n}$ will use the same notation as that of $\QML{n}$, where they will always be understood to mean ``the image of this element in the appropriate algebra'', and similarly for $\GL{n}$ and $\SL{n}$.
\end{notation}
\begin{notation}
Since most of this chapter is concerned specifically with $3 \times 3$ matrices, we will often drop the subscript and simply write $\RR$ for $\RR_3$ in order to simplify the notation.
\end{notation}

The algebras $\ML{n}$ and $\QML{n}$ admit a number of automorphisms and anti-automorphisms, which will allow us to reduce the number of cases we check.  We first define on $\QML{n}$ the maps\nom{T@$\tau$}\nom{R@$\rho$}\nom{S@$S$}
\begin{align*}
\tau&: X_{ij} \mapsto X_{ji}, \\
\rho&: X_{ij} \mapsto X_{n+1-j,n+1-i}.
\end{align*}
The map $\tau$ defines an automorphism of $\QML{n}$ corresponding to the transpose operation on matrices, while $\rho$ defines an anti-automorphism corresponding to transposition along the reverse diagonal.  Both of these maps have order 2.  By \cite{GL1}, the action of these maps on minors and on $Det_q$ is as follows:
\begin{gather*}
\tau([I|J]_q) = [J|I]_q; \quad\tau(Det_q) = Det_q;\\
\rho([I|J]_q) = [w_0(J)|w_0(I)]_q; \quad \rho(Det_q) = Det_q;
\end{gather*}
where $w_0$ denotes the permutation $\left(\begin{smallmatrix}1&2&\dots&n \\ n & n-1 & \dots &1\end{smallmatrix}\right) \in S_n$, i.e. the ``longest element'' of $S_n$.

$\QML{n}$ admits the structure of a bialgebra but does not have an antipode map; on $\QGL{n}$ and $\QSL{n}$ we obtain a genuine Hopf algebra structure by defining the antipode:
\[S: X_{ij} \mapsto (-q)^{i-j}[\wt{j}|\wt{i}]_qDet_q^{-1}.\]
By \cite{GL1}, the action of $S$ on minors is as follows:
\[S([I|J]_q) = (-q)^{\sum I - \sum J}[\wt{J}|\wt{I}]Det_q^{-1}, \quad S(Det_q) = Det_q^{-1}.\]
These maps induce (anti-)automorphisms of Poisson algebras on $\ML{n}$, $\GL{n}$ and $\SL{n}$ as appropriate (just replace $X_{ij}$ in the definitions by $x_{ij}$, and $q$ by 1), and we denote these maps by the same symbols as the quantum case.  We note that when we ignore the Poisson structure on the semi-classical limits and simply view them as commutative algebras, the distinction between automorphism and anti-automorphism disappears and each of the maps $\tau$, $\rho$ and $S$ are simply automorphisms of commutative algebras.  

Finally, we observe that while $S$ has infinite order as a map on $\QGL{n}$ or $\QSL{n}$, the antipode of any commutative Hopf algebra has order 2 \cite[Corollary~1.5.12]{montgomery_hopf}.

\subsection{Commutation relations and interactions for minors}

To save us from excessive computation in future sections, it will be useful to obtain some identities concerning how certain $(n-1)\times(n-1)$ minors interact with the generators $x_{ij}$ under the Poisson bracket.  In \cite[\S1.3]{GL1} a number of identities for $\QML{n}$ are listed, and we will use these to derive Poisson versions of these equalities.

We are interested in computing the bracket $\{x_{ij}, [\wt{l}|\wt{m}]\}$ for $1 \leq i,j,l,m \leq n$.  Suppose first that $j=m$ and $i \neq l$; by \cite[E1.3c]{GL1}, we have the following equality in $\QML{n}$:
\begin{equation}\label{eq:one of the quantum equalities}X_{ij}[\wt{l}|\wt{j}]_q = q[\wt{l}|\wt{j}]_qX_{ij} + (q-q^{-1})\sum_{s <j}(-q)^{s-j}[\wt{l}|\wt{s}]_qX_{is} \quad (i \neq l).\end{equation}
The key point here is that we may replace $q$ by $t$ and $X_{ij}$ by $Y_{ij}$ in \eqref{eq:one of the quantum equalities} and obtain an equality which is valid in $\RR_n$.  We may then use Definition~\ref{def:star product Poisson bracket} and \eqref{eq:one of the quantum equalities} to compute $\{x_{ij},[\wt{l}|\wt{j}]\}$ for $i \neq l$ as follows:
\begin{align*}
\{x_{ij},[\wt{l}|\wt{j}]\} &= (t-1)^{-1}\left(X_{ij}[\wt{l}|\wt{j}]_t - [\wt{l}|\wt{j}]_tX_{ij}\right) \quad &\mod t-1 \\
&= (t-1)^{-1}\left((t-1)[\wt{l}|\wt{j}]_tX_{ij} + t^{-1}(t^2-1)\sum_{s <j}(-t)^{s-j}[\wt{l}|\wt{s}]_tX_{is}\right) \quad &\mod t-1 \\
&= [\wt{l}|\wt{j}]x_{ij} + 2\sum_{s<j}(-1)^{s-j}[\wt{l}|\wt{s}]x_{is}. &
\end{align*}
Applying a similar process to the other equalities in \cite[\S1.3]{GL1}, we obtain the following list of relations:
\begin{align}
\label{eq:Preln1}\{x_{ij},[\wt{l}|\wt{m}]\} &= 0 &(i \neq l, j \neq m) \\
\label{eq:Preln2}\{x_{ij},[\wt{i}|\wt{m}]\} &= -2\sum_{s>i}(-1)^{s-i}[\wt{s}|\wt{m}]x_{sj} - [\wt{i}|\wt{m}]x_{ij} &(j \neq m) &\\
&= 2\sum_{s<j}(-1)^{s-j}[\wt{l}|\wt{s}]x_{is} + [\wt{l}|\wt{j}]x_{ij} \nonumber \\
\label{eq:Preln3}\{x_{ij},[\wt{l}|\wt{j}]\} &= 2\sum_{s<j}(-1)^{s-j}[\wt{l}|\wt{s}]x_{is} + [\wt{l}|\wt{j}]x_{ij} &(i \neq l)\\
&= -2\sum_{s>j}(-1)^{s-j}[\wt{l}|\wt{s}]x_{is} - [\wt{l}|\wt{j}]x_{ij} \quad & \nonumber\\
\label{eq:Preln4}\{x_{ij},[\wt{i}|\wt{j}]\} &= 2\left(\sum_{s < i}(-1)^{s-i}x_{sj}[\wt{s}|\wt{j}] - \sum_{t > j}(-1)^{j-t}x_{it}[\wt{i}|\wt{t}]\right) &\\
&= 2\left(\sum_{t < j}(-1)^{j-t}x_{it}[\wt{i}|\wt{t}] - \sum_{s > i}(-1)^{s-i}x_{sj}[\wt{s}|\wt{j}] \right). & \nonumber
\end{align}

\begin{definition}\label{def:Poisson central, normal}
Let $R$ be a commutative Poisson algebra.  We call an element $r \in R$ \textit{Poisson central}\nom{Poisson central} if $\{r,s\} = 0$ for all $s \in R$, and \textit{Poisson normal}\nom{Poisson normal} if $\{r,s\} \in rR$ for all $s \in R$.
\end{definition}
When $i \neq l$ and $j \neq k$, the variable $x_{ij}$ appears as part of the expansion of the minor $[\wt{l}|\wt{m}]$ and so we may view \eqref{eq:Preln1} as a relation in a subalgebra of $\ML{n}$ isomorphic to $\ML{n-1}$.  The minor $[\wt{l}|\wt{m}]$ plays the role of the $(n-1)\times(n-1)$ determinant $Det$ in this copy of $\ML{n-1}$, and by \eqref{eq:Preln1} its bracket with any generator $x_{ij}$ in $\ML{n-1}$ is zero.  Hence we may conclude that the determinant $Det$ is Poisson central in $\ML{n}$, and therefore in $\GL{n}$ as well.

We now specialise to the case $n=3$.  Since we will mostly be interested in ideals of $\GL{3}$ and $\SL{3}$ (and their quantum counterparts) it will be useful to have some results on when Poisson-prime (resp. prime) ideals contain the $3\times3$ determinant $Det$ (resp. the $3\times3$ quantum determinant $Det_q$).
\begin{lemma}\label{res:Det in prime ideal, Poisson}
If $P$ is a Poisson-prime ideal in $\ML{3}$ and $P$ contains at least one of: 
\[x_{11},x_{22},x_{33},[\wt{1}|\wt{1}],[\wt{2}|\wt{2}],[\wt{3}|\wt{3}]\]
then $Det \in P$ as well.
\end{lemma}
\begin{proof}
Using the identities in \eqref{eq:Preln4}, we see that
\begin{equation*}
\{x_{11},[\wt{1}|\wt{1}]\} = 2x_{21}[\wt{2}|\wt{1}] -2x_{31}[\wt{3}|\wt{1}]
\end{equation*}
and hence
\[Det = x_{11}[\wt{1}|\wt{1}] - \frac{1}{2}\{x_{11},[\wt{1}|\wt{1}]\}\]
is in the Poisson-prime ideal $P$ whenever $x_{11}$ or $[\wt{1}|\wt{1}]$ is.  
Since $\tau$ fixes $Det$, by applying $\tau$ to the above equalities we immediately obtain the same conclusion for $x_{33}$ and $[\wt{3}|\wt{3}]$.
Next, we can observe that
\[x_{22}x_{33} - \frac{1}{2}\{x_{22},x_{33}\}= [\wt{1}|\wt{1}], \]
and so $x_{22} \in P$ implies $Det \in P$ as well.

Finally, suppose $[\wt{2}|\wt{2}] \in P$.  Applying the identities in \eqref{eq:Preln3}, we see that
\[\{x_{12},[\wt{2}|\wt{2}]\} = -2[\wt{2}|\wt{1}]x_{11} + [\wt{2}|\wt{2}]x_{12}, \quad \{x_{32},[\wt{2}|\wt{2}]\} = 2[\wt{2}|\wt{3}]x_{33} - [\wt{2}|\wt{2}]x_{32},\]
and hence both $[\wt{2}|\wt{1}]x_{11}$ and $[\wt{2}|\wt{3}]x_{33}$ are in $P$ as well.  Since $P$ is also prime, we must have $[\wt{2}|\wt{1}] \in P$ or $x_{11} \in P$, and similarly for $[\wt{2}|\wt{3}]$ and $x_{33}$.  If $x_{11}$ or $x_{33} \in P$ then $Det \in P$ as well, so suppose that $[\wt{2}|\wt{1}]$ and $[\wt{2}|\wt{3}]$ are in $P$ instead.  Since our initial hypothesis was that $[\wt{2}|\wt{2}] \in P$, we once again obtain
\[Det = x_{21}[\wt{2}|\wt{1}] - x_{22}[\wt{2}|\wt{2}] + x_{23}[\wt{2}|\wt{3}] \in P.\]
\end{proof}

\begin{lemma}\label{res:Det in Prime ideal, quantum}
If $Q$ is a prime ideal in $\QML{3}$ and $Q$ contains at least one of: 
\[X_{11}, X_{22}, X_{33}, [\wt{1}|\wt{1}]_q, [\wt{2}|\wt{2}]_q, [\wt{3}|\wt{3}]_q,\]
 then $Det_q \in Q$ as well.
\end{lemma}
\begin{proof}
The quantum proof follows in a very similar manner to the Poisson proof.  From the equalities in \cite[E1.3a]{GL1}, we obtain
\begin{align*}
Det_q &= [\wt{1}|\wt{1}]_qX_{11} + q^{-1}(q-q^{-1})^{-1}(X_{11}[\wt{1}|\wt{1}]_q - [\wt{1}|\wt{1}]_qX_{11})\\
&=  [\wt{3}|\wt{3}]_qX_{33} + q(q-q^{-1})^{-1}(X_{33}[\wt{3}|\wt{3}]_q - [\wt{3}|\wt{3}]_qX_{33}),
\end{align*}
while from the definition of quantum minor and the defining relations of $\QML{3}$ we have
\begin{align*}
[\wt{1}|\wt{1}]_q &= X_{22}X_{33} - qX_{23}X_{32} \\
&= X_{22}X_{33} - q(q-q^{-1})^{-1}(X_{22}X_{33}-X_{33}X_{22}).
\end{align*}
Hence $Det_q$ is in the prime ideal $Q$ whenever any one of $X_{11}$, $X_{22}$, $X_{33}$, $[\wt{1}|\wt{1}]_q$ or $[\wt{3}|\wt{3}]_q$ is.  

Now suppose that $[\wt{2}|\wt{2}]_q \in Q$ and recall from \S\ref{ss:H primes examples} that all prime ideals in $\QML{3}$ are completely prime. Using the identities from \cite[\S1.3]{GL1}, we find that
\begin{align*}
X_{12}[\wt{2}|\wt{2}]_q - q[\wt{2}|\wt{2}]_qX_{12} &= -q^{-1}(q-q^{-1})[\wt{2}|\wt{1}]_qX_{11}\\
X_{32}[\wt{2}|\wt{2}]_q - q^{-1}[\wt{2}|\wt{2}]_qX_{32} &= q(q-q^{-1})[\wt{2}|\wt{3}]_qX_{33},
\end{align*}
and so $[\wt{2}|\wt{1}]_qX_{11}, [\wt{2}|\wt{3}]_qX_{33} \in Q$.  If $X_{11}$ or $X_{33}$ are in $Q$ then $Det_q \in Q$ by the above; if not, then both $[\wt{2}|\wt{1}]_q$ and $[\wt{3}|\wt{2}]_q$ are in $Q$ and hence by \cite[E1.3a]{GL1},
\[Det_q = -q^{-1}X_{21}[\wt{2}|\wt{1}]_q + X_{22}[\wt{2}|\wt{2}]_q - qX_{32}[\wt{3}|\wt{2}]_q \in Q. \qedhere\]
\end{proof}
It is noted in \cite[\S2.4]{GL1} that $[\wt{3}|\wt{1}]_q$ and $[\wt{1}|\wt{3}]_q$ are normal in $\QML{3}$.  Since we will want to use $[\wt{3}|\wt{1}]$ and $[\wt{1}|\wt{3}]$ as generators of Poisson ideals, we prove the corresponding result for $\ML{3}$.
\begin{lemma}\label{res:Poisson normal elements}
The minors $[\wt{1}|\wt{3}]$ and $[\wt{3}|\wt{1}]$ are Poisson-normal in $\ML{3}$, and hence in $\GL{3}$ and  $\SL{3}$ as well.
\end{lemma}
\begin{proof}
We will prove this for $[\wt{3}|\wt{1}]$, since the corresponding result for $[\wt{1}|\wt{3}]$ will then follow by applying $\tau$. 

We first need to check that $\{x_{ij},[\wt{3}|\wt{1}]\} \in [\wt{3}|\wt{1}]\ML{3}$ for $1 \leq i,j \leq 3$, which is simple to verify using \eqref{eq:Preln1} - \eqref{eq:Preln4}; indeed:
\begin{align*}
\{x_{ij},[\wt{3}|\wt{1}]\} &= 0 &i \neq 3, j \neq 1 \\
\{x_{3j},[\wt{3}|\wt{1}]\} &= x_{3j}[\wt{3}|\wt{1}] & j \neq 1 \\
\{x_{i1},[\wt{3}|\wt{1}]\} &= x_{i1}[\wt{3}|\wt{1}] & i \neq 3 \\
\{x_{31},[\wt{3}|\wt{1}]\} &= 0. &
\end{align*}
Thus $[\wt{3}|\wt{1}]$ is Poisson normal in $\ML{3}$, and hence in $\SL{3}$ as well.  Further, for any $a \in \ML{3}$ we have
\[\{aDet^{-1},[\wt{3}|\wt{1}]\} = \{a,[\wt{3}|\wt{1}]\}Det^{-1} - \{Det,[\wt{3}|\wt{1}]\}aDet^{-2} =  \{a,[\wt{3}|\wt{1}]\}Det^{-1}\]
by \eqref{eq:extend Poisson bracket to localization} and the fact that $Det$ is Poisson central.  It therefore follows that $[\wt{3}|\wt{1}]$ is Poisson normal in $\GL{3}$ as well.  
\end{proof}
In many cases, we will want to take the existing analysis done in \cite{GL1} and transfer it directly to the semi-classical limits.  The following results will show that the process of taking semi-classical limits commutes with both localization and taking quotients.  This will be useful when we apply the stratification theory described in \S\ref{ss:H action quantum version}-\ref{ss:H action Poisson version} to $\QGL{3}$ and $\GL{3}$.
\begin{proposition}\label{res:localization and quotient commute}
Let $R$ be a ring (possibly non-commutative), $I$ an ideal of $R$ and $X$ a right denominator set in $R$.  Then there is an isomorphism of rings
\[(R/I)[X^{-1}] \cong RX^{-1}/IX^{-1},\]
i.e. the processes of localizing and taking quotients commute.
\end{proposition}
\begin{proof}
By \cite[Corollary~10.13]{GW1}, $RX^{-1}$ is a flat left $R$-module, i.e. if 
\[0 \longrightarrow A \longrightarrow B \longrightarrow C \longrightarrow 0\]
is an exact sequence of right $R$-modules, then the localizations also form an exact sequence
\[0 \longrightarrow AX^{-1} \longrightarrow BX^{-1} \longrightarrow CX^{-1} \longrightarrow 0\]
In particular, if we choose $A = I$, $B = R$ and $C = (R/I)$ and equip each with the natural $(R,R)$-bimodule structure and natural maps between them, then we obtain an isomorphism of $R$-modules
\[(R/I)X^{-1} \cong RX^{-1}/IX^{-1}\]
Since the natural module homomorphisms defined above are simultaneously ring homomorphisms, this is in fact an isomorphism of rings, thus proving the result.
\end{proof}
Note that if $I \cap X \neq \emptyset$, this reduces to the statement that the zero ring is isomorphic to itself.  This is reassuring but unhelpful, so we will always ensure that $I \cap X = \emptyset$ when applying this result (in particular, the following theorem).

\begin{proposition}\label{res:localization and SCL commute}
Let $B$ be a $k[t^{\pm1}]$-algebra and $S \subset k\backslash \{0,1\}$ a set of scalars such that none of the elements $\{t-q: q \in S \cup \{1\}\}$ are invertible in $B$.  Suppose further that $B/(t-1)B$ is commutative, and write $A:= B/(t-1)B$, $A_q := B/(t-q)B$ for $q \in S$.  Finally, let $X$ be an Ore set of regular elements in $B$ such that $X \cap (t-q)B = \emptyset$ for all $q \in S \cup \{1\}$, and let $X_q$ denote the image of $X$ in $B/(t-q)B$ for $S \cup \{1\}$.  Under these conditions, localizing $A$ at $X_1$ is equivalent to localizing $A_q$ at $X_q$ and then taking the semi-classical limit.
\end{proposition}
\begin{proof}
Applying Proposition~\ref{res:localization and quotient commute}, we obtain isomorphisms of rings 
\[A_q[X_q^{-1}] \cong B[X^{-1}]/(t-q)B[X^{-1}] \quad\textrm{ and } \quad A[X_1^{-1}] \cong B[X^{-1}]/(t-1)B[X^{-1}].\]
In order to establish the result, we just need to check that the Poisson bracket $\{\cdot,\cdot\}_1$ induced on $B[X^{-1}]/(t-1)B[X^{-1}]$ from the commutator in $B[X^{-1}]$ (as in Definition~\ref{def:star product Poisson bracket}) agrees with the Poisson bracket $\{\cdot,\cdot\}_2$ induced on $A$ and extended to $A[X_1]^{-1}$ by \eqref{eq:extend Poisson bracket to localization}.  

We will use the fact that $\{uv^{-1}, \cdot\}$ and $\{\cdot,uv^{-1}\}$ are always derivations for any $uv^{-1}$ to show that both Poisson brackets are defined by their restriction to $A$.  In particular, for any derivation $\delta$ on a commutative ring in which some elements are invertible, it follows easily from the definition that $\delta$ must satisfy the equality
\[\delta(rs^{-1}) = \delta(r)s^{-1} - \delta(s)rs^{-2}.\]
Hence, to compute $\{ab^{-1},cd^{-1}\}_1$ for any $ab^{-1}, cd^{-1} \in B[X^{-1}]/(t-1)B[X^{-1}]$, we may view it as a derivation in first one then the other variable to obtain
\begin{align*}
\{ab^{-1},cd^{-1}\}_1 &= \{ab^{-1},c\}_1d^{-1} - \{ab^{-1}, d\}_1cd^{-2} \\
&= \{a,c\}_1b^{-1}d^{-1} - \{a,d\}_1cb^{-1}d^{-2} - \{b,c\}_1ab^{-2}d^{-1} + \{b,d\}_1acb^{-2}d^{-2}.
\end{align*}
This agrees precisely with the definition of $\{ab^{-1},cd^{-1}\}_2$ obtained using the formula \eqref{eq:extend Poisson bracket to localization} to extend a Poisson bracket to a localization, and hence it suffices to check that $\{a,c\}_1 = \{a,c\}_2$ for all $a, c \in A$.  This follows trivially from the definition of the two brackets in terms of commutators on $B$ and $B[X^{-1}]$, however.
\end{proof}

\begin{proposition}\label{res:quotient and scl commute}
Let $A$, $A_q$ and $B$ be as in Proposition~\ref{res:localization and SCL commute}, and let $I$ be an ideal of $B$ such that $t-q \not\in I$ for any $q \in S \cup \{1\}$.  Denote by $I_q$ the image of $I + \langle t-q\rangle$ in $B/(t-q)B$ for $q \in S \cup \{1\}$.  Then the semi-classical limit of the quotient $A_q/I_q$ is the same as the quotient of the semi-classical limit $A$ by the ideal $I_1$, i.e. taking quotients and semi-classical limits commute.
\end{proposition}
\begin{proof}
Using the Third Isomorphism Theorem we easily obtain the isomorphisms of rings
\[\bigslant{B/I}{(t-1)B/I} \cong A/I_1 \textrm{\quad and \quad}\bigslant{B/I}{(t-q)B/I} \cong A_q/I_q.\]
As above, we just need to check that the two Poisson brackets induced  on $A/I_1$ agree.  

By taking the quotient first, a Poisson bracket is induced directly on $A/I_1$ from the semi-classical limit process as follows: for $a+I,b+I \in B/I$, we have
\begin{align*}
\{\overline{a+I},\overline{b+I}\}_1 &:= \frac{1}{t-1}\Big((a+I)(b+I) - (b+I)(a+I)\Big) \quad \mod t-1 \\
&= \frac{1}{t-1}(ab-ba + I) \quad \mod t-1 \\
&= \frac{1}{t-1}(ab-ba) + (I+\langle t-1 \rangle) \quad \mod t-1
\end{align*}
Meanwhile, $A$ already has a Poisson bracket $\{\cdot,\cdot\}_2$ induced from $A_q$, and this induces a unique Poisson bracket on the quotient $A/I_1$ using the formula from \eqref{eq:poisson bracket on quotient}: 
\begin{align*}
\{\overline{a} + I_1,\overline{b} + I_1\}_2 &= \{\overline{a},\overline{b}\}_2 + I_1 \\
&= \Big(\frac{1}{t-1}(ab-ba) \mod t-1\Big) \ + I_1.
\end{align*}
Since the image of the ideal $I + \langle t-1 \rangle$ modulo $t-1$ is $I_1$, we see that $\{\cdot,\cdot\}_1$ and $\{\cdot,\cdot\}_2$ are equal on $A/I_1$.
\end{proof}

\section{$\HH$-primes}\label{s:36 H primes, quantum}

As described in the introduction to this chapter, our aim is to apply the Poisson Stratification Theorem to $\GL{3}$ and $\SL{3}$ in a similar manner to the quantum algebras in \cite{GL1}.  As in the quantum case, we will make use of the fact that these two algebras can be related via an isomorphism $\GL{n} \cong \SL{n}[z^{\pm1}]$; it will turn out that some results are easier to prove in $\GL{3}$ and others in $\SL{3}$, so it will be useful to be able to move between the two as required.  We prove the existence of a Poisson version of this isomorphism in Lemma~\ref{res:poisson iso sln to gln} below.

Our first aim will be to define a rational action of a torus $\HH$ on $\GL{3}$ and $\SL{3}$ and to identify the Poisson $\HH$-primes: Poisson prime ideals which are stable under the action of $\HH$.  This is complicated slightly by the fact that the standard action of $\HH = (k^{\times})^{2n}$ on $\GL{n}$ does \textit{not} restrict directly to an action on $\SL{n}$; however, we will show in \S\ref{ss:translating H primes to SL3} that for an appropriate action of a torus $\HP \cong (k^{\times})^{2n-1}$ there is a natural bijection from the $\HH$-primes of $\GL{n}$ to the $\HP$-primes of $\SL{n}$.

\subsection{$\HH$-primes of $\GL{3}$}

As described in \cite[II.1.15, II.2.6]{GBbook}, the torus $\HH = (k^{\times})^{2n}$\nom{H@$\mathcal{H}$} acts rationally on $\QGL{n}$ by
\begin{equation}\label{eq:H-action on GLn}
h.X_{ij} = \alpha_i\beta_jX_{ij}, \quad \textrm{where }h = (\alpha_1,\dots,\alpha_n,\beta_1,\dots,\beta_n) \in \HH.
\end{equation}
This also defines an action of $\HH$ on $\GL{n}$, by replacing $X_{ij}$ with $x_{ij}$ in \eqref{eq:H-action on GLn} above.  By \cite[\S2.2]{GLL1}, this defines a rational action of $\HH$ on $\GL{n}$.

We will now restrict our attention to the case where $n=3$.  

Since we are using the same action of $\HH$ on $\QGL{3}$ and $\GL{3}$, we would expect that $\HH$-primes of $\QGL{3}$ should match up bijectively with Poisson $\HH$-primes in $\GL{3}$.  In this section we will show that every $\HH$-prime of $\QGL{3}$ defines a distinct Poisson $\HH$-prime when its generators are viewed as elements of $\GL{3}$, and in Theorem~\ref{res:gl3 has only 36 h primes} we will show that $\GL{3}$ admits no other Poisson $\HH$-primes.

The 36 $\HH$-primes in $\QGL{3}$ are described in \cite[Figure~1]{GL1}, which we reproduce in Figure~\ref{fig:H_primes_gens} in Appendix~\ref{c:H-prime figures}.  Each ideal is represented pictorially by a $3 \times 3$ grid of dots: a black dot in position $(i,j)$ denotes the element $X_{ij}$, and a square represents a $2\times2$ quantum minor in the natural way.  For example, the ideal in position $(231,231)$ denotes the ideal generated by $X_{31}$ and $[\wt{3}|\wt{1}]_q$.

We will adopt the indexing convention used in \cite{GL1} for these ideals. $\HH$-primes are indexed by elements $\omega = (\omega_{+},\omega_{-}) \in S_3 \times S_3$,\nom{W@$\omega = (\omega_{+},\omega_{-})$} where we write permutations in $S_3$ using an abbreviated form of 2-line cycle notation, e.g. $321$ represents the permutation $1 \mapsto 3, \ 2 \mapsto 2,\ 3 \mapsto 1$.  The ideal $I_{\omega}$\nom{I@$I_{\omega}$} will denote the ideal generated by the elements in position $\omega = (\omega_{+},\omega_{-})$ of Figure~\ref{fig:H_primes_gens}, where it will always be clear from context whether we mean an ideal in $\QGL{3}$ or $\GL{3}$.

We may first observe that each of these ideals are generated only by (quantum) minors, which are eigenvectors for the action of $\HH$, and hence the corresponding generator sets in $\GL{3}$ also generate $\HH$-stable ideals.  Further, as the next lemma verifies, the resulting ideals are all closed under the Poisson bracket of $\GL{3}$ as well.
\begin{lemma}\label{res:Poisson H ideals}
Let $\omega \in S_3 \times S_3$, and let $I_{\omega}$ be the ideal of $\ML{3}$ generated by the elements in position $\omega$ from Figure~\ref{fig:H_primes_gens}.  Then $I_{\omega}$ is a Poisson ideal in $\ML{3}$, and hence induces a Poisson ideal in $\GL{3}$ as well.
\end{lemma}
\begin{proof}
We write $I_{\omega} = \langle f_1, \dots, f_n\rangle$, where the $f_i$ are the minors depicted in position $\omega$ of Figure~\ref{fig:H_primes_gens}; it suffices to check that $\{f_r,\ML{3}\} \in I_{\omega}$ for $1 \leq r \leq n$.  We first note that this is immediate for $f_r = [\wt{1}|\wt{3}]$ or $[\wt{3}|\wt{1}]$, since by Lemma~\ref{res:Poisson normal elements} these elements are Poisson-normal in $\ML{3}$ and $\GL{3}$.

Now consider the case where $f_r = x_{ij}$ for some $1 \leq i,j \leq 3$ and $1 \leq r \leq n$.  We need to check that $\{x_{ij}, x_{kl}\} \in I$ for $1 \leq k,l\leq 3$, which is easy to see when $i=k$, or $j=l$, or $i<k$ and $j>l$ (or vice versa): in these cases the Poisson bracket is either multiplicative or zero on the given elements. The only remaining cases are when $i<k$ and $j<l$, or $i>k$ and $j>l$, i.e. $x_{kl}$ is diagonally below and to the right or above and to the left of $x_{ij}$.  Since the Poisson bracket is anti-symmetric, we may assume that $i < k$ and $j < l$.  In this case,
\[\{x_{ij},x_{kl}\} = 2x_{il}x_{kj}\]
and from Figure~\ref{fig:H_primes_gens} we can observe directly that whenever this situation occurs for a generator of $I_{\omega}$, we have $x_{il}$ or $x_{kj} \in I$ as well and the ideal is closed under Poisson bracket as required.

This shows that each of the 36 ideals listed in Figure~\ref{fig:H_primes_gens} are Poisson ideals in $\ML{3}$, and by using the formula \eqref{eq:extend Poisson bracket to localization} for the unique extension of the bracket to a localization we see that the induced ideals in $\GL{3}$ are also Poisson ideals.
\end{proof}
It is observed in \cite[\S1.5]{GL1} that the maps $\tau$, $\rho$ and $S$ preserve $\HH$-stable subsets of $\QGL{3}$ with respect to the action defined in \eqref{eq:H-action on GLn}; since this follows purely from considering the action of $\HH$ on generators, the same observation holds true for $\GL{3}$ since the action of $\HH$ is the same.  In particular, if $\varphi$ is some combination of $\tau$, $\rho$ and $S$, and $I$, $J$ are $\HH$-primes (respectively Poisson $\HH$-primes) such that $\varphi(I) = J$ then this induces an (anti-)isomorphism of algebras (resp. Poisson algebras)
\[\varphi: \QGL{3}/I \rightarrow \QGL{3}/J, \quad \textrm{resp. } \GL{3}/I \rightarrow \GL{3}/J.\]
By direct computation, we find that the 36 ideals in Figure~\ref{fig:H_primes_gens} form 12 orbits under combinations of $\tau$, $\rho$ and $S$, and hence it often suffices only to consider the structure or properties of $\GL{3}/I_{\omega}$ or $\QGL{3}/I_{\omega}$ for one example from each orbit.  Since we will regularly use this fact to simplify case-by-case analyses in various proofs, in Figure~\ref{fig:H_primes_nice} we present a diagram of these orbits.  Note that we will always use the first ideal listed in Figure~\ref{fig:H_primes_nice} when we require a representative for a given orbit.

Most of the arrows in Figure~\ref{fig:H_primes_nice} are immediately clear from the definitions of $\tau$, $\rho$ and $S$; the five which are not clear are justified in the following lemma.
\begin{lemma}
We have the following equalities in $\QGL{3}$ and $\GL{3}$:
\begin{align*}
(S \circ \rho)(I_{132,312}) &= I_{213,231} \\
(S \circ \rho)(I_{231,132}) &= I_{312,213} \\
(S \circ \rho)(I_{231,213}) &= I_{312,132} \\
(S \circ \rho)(I_{213,312}) &= I_{132,231} \\
S(I_{231,123}) &= I_{312,123}
\end{align*}
\end{lemma}
\begin{proof}
We will prove the first equality, as the others follow by almost identical arguments.  Consider first the case of $\QGL{3}$; note that $\rho(I_{132,312}) = \langle X_{13},X_{31},X_{32}\rangle$, so what we need to prove is that
\[S(\langle X_{13},X_{31},X_{32} \rangle) = I_{213,231}.\]
In other words, we need to check that
\begin{equation}\label{eq:equality to prove at some point 1}\langle [\wt{3}|\wt{1}]_q, [\wt{1}|\wt{3}]_q, [\wt{2}|\wt{3}]_q \rangle = \langle X_{31},X_{32},[\wt{3}|\wt{1}]_q\rangle.\end{equation}
Since $[\wt{1}|\wt{3}]_q = X_{21}X_{32} - qX_{22}X_{31}$ and $[\wt{2}|\wt{3}]_q = X_{11}X_{32}-qX_{12}X_{31}$, the $\subseteq$ direction is clear.  Conversely, using the formulas from \cite[\S1.3]{GL1} we can observe that
\begin{equation}\label{eq:some technical eqs, quantum}\begin{aligned}
q^2[\wt{1}|\wt{3}]_qX_{12} - q[\wt{2}|\wt{3}]_qX_{22} &= -[\wt{3}|\wt{3}]_qX_{32}, \\
q^2[\wt{1}|\wt{3}]_qX_{11} - q[\wt{2}|\wt{3}]_qX_{21} &= -[\wt{3}|\wt{3}]_qX_{31}. 
\end{aligned}\end{equation}
$[\wt{3}|\wt{3}]_qX_{32}$ and $[\wt{3}|\wt{3}]_qX_{31}$ are therefore in $\langle [\wt{3}|\wt{1}]_q, [\wt{1}|\wt{3}]_q, [\wt{2}|\wt{3}]_q \rangle$.  This is a prime ideal since it is the image of a prime ideal under the automorphism $S \circ \rho$, and hence is completely prime since all primes of $\QGL{3}$ are completely prime by \cite[Corollary~II.6.10]{GBbook}.  By Lemma~\ref{res:Det in Prime ideal, quantum} no non-trivial prime in $\QGL{3}$ can contain $[\wt{3}|\wt{3}]_q$, and so $X_{32}$, $X_{31}$ are both in $\langle [\wt{3}|\wt{1}]_q, [\wt{1}|\wt{3}]_q, [\wt{2}|\wt{3}]_q \rangle.$  The equality \eqref{eq:equality to prove at some point 1} is now proved, and the other four equalities follow by similar arguments.

Finally, the Poisson proof is almost unchanged from the quantum version, except that we use Lemma~\ref{res:Det in prime ideal, Poisson} instead of Lemma~\ref{res:Det in Prime ideal, quantum} and instead of the equalities in \eqref{eq:some technical eqs, quantum}, we observe that
\begin{align*}
x_{31} &= Det^{-1}([\wt{1}|\wt{2}][\wt{2}|\wt{3}] - [\wt{2}|\wt{2}][\wt{1}|\wt{3}]),\\
x_{32} &= Det^{-1}([\wt{1}|\wt{1}][\wt{2}|\wt{3}] - [\wt{2}|\wt{1}][\wt{1}|\wt{3}]),
\end{align*}
which can be easily seen by applying $S$ to the equalities $[\wt{1}|\wt{3}] = x_{21}x_{32} - x_{22}x_{31}$, $[\wt{2}|\wt{3}] = x_{11}x_{32} - x_{12}x_{31}$.
\end{proof}

We will now proceed to check that the ideals appearing in Figure~\ref{fig:H_primes_gens}/Figure~\ref{fig:H_primes_nice} are indeed distinct Poisson prime ideals which are invariant under the action of $\HH$.  We have already checked that they are Poisson $\HH$-ideals, so all that remains is to verify that they are prime (in the standard commutative sense) and distinct.
\begin{lemma}\label{res:distinct H ideals}
The ideals generated in $\ML{3}$ by the sets of generators listed in Figure~\ref{fig:H_primes_gens} are pairwise distinct and do not contain the $3\times3$ determinant $Det$.  They therefore generate 36 pairwise distinct $\HH$-ideals in $\GL{3}$.
\end{lemma}
\begin{proof}
The proof that the ideals are distinct closely follows the corresponding quantum proof from \cite[\S3.6]{GL2}.  Indeed, we can define two projections
\[\theta_1: \ML{3} \rightarrow B_1 = k[x_{12},x_{13},x_{22},x_{23}], \quad \theta_2: \ML{3} \rightarrow B_2 = k[x_{21},x_{22},x_{31},x_{32}]\]
where each map fixes $x_{ij}$ if it exists in the target ring and maps it to zero otherwise.  It suffices to view these as maps of commutative algebras rather than Poisson algebras.

If we consider the images of the ideals from Figure~\ref{fig:H_primes_gens} under $\theta_1$ and $\theta_2$, it is clear that $\theta_1$ sends all ideals in a given column to the same ideal in $B_1$, and the images of ideals from different columns are distinct in $B_1$.  Similarly $\theta_2$ sends every ideal from a given row of Figure~\ref{fig:H_primes_gens} to one ideal in $B_2$, and ideals from different rows are mapped to distinct ideals in $B_2$.  Hence two ideals in different columns or different rows must be distinct in $\ML{3}$, and therefore all 36 of the ideals in the table are distinct.

By observation, we can see that all of the ideals in Figure~\ref{fig:H_primes_gens} are contained inside the ideal
\begin{center}\begin{tabular}{c}
\protect\begin{smallarray}{m321321}
{
 \circ & \bullet & \bullet \\
 \bullet & \circ & \bullet \\
 \bullet & \bullet & \circ \\
};
\end{smallarray} \\
$I_{123,123}$
\end{tabular}\end{center}
and hence it suffices to check that $Det \not\in I_{123,123}$.  This is equivalent to checking that $Det \neq 0$ in $\ML{3}/I_{123,123}$, but since $\ML{3}/I_{123,123} \cong k[x_{11},x_{22},x_{33}]$ we have $Det = x_{11}x_{22}x_{33} \neq 0$ in this ring.  The result now follows.
\end{proof}

\begin{lemma}\label{res:prime H ideals}
Each of the 36 ideals whose generators are listed in Figure~\ref{fig:H_primes_gens} are prime ideals in both $\ML{3}$ and $\GL{3}$.
\end{lemma}
\begin{proof}
By Lemma~\ref{res:distinct H ideals} none of the ideals in Figure~\ref{fig:H_primes_gens} contain the determinant $Det$, so by \cite[Theorem~10.20]{GW1} they will be prime in $\ML{3}$ if and only if they are prime in $\GL{3}$.  We can therefore immediately observe that the 25 ideals generated only by $1\times1$ minors are prime in $\ML{3}$ since the quotient $\ML{3}/I_{\omega}$ is simply a polynomial ring in fewer variables, and hence the extensions of these ideals to $\GL{3}$ are prime as well.

Next consider the ideal
\begin{center}\begin{tabular}{c}
\protect\begin{smallarray}{m321321}
{
 \circ &  &  \\
\circ &  &  \\
\bullet & \circ & \circ \\
};
\MyZ(1,2)
\end{smallarray} \\
$I_{231,231}$
\end{tabular}\end{center}
in $\ML{3}$.  In the subalgebra of $\ML{3}$ generated by $\{x_{12},x_{13},x_{22},x_{23}\}$ the minor $[\wt{3}|\wt{1}]$ plays the role of the $2 \times 2$ determinant, which we will denote temporarily by $Det_2$.  This generates a prime ideal in $\ML{2}$, and hence as commutative algebras we have an isomorphism
\[\ML{3}/I_{231,231} \cong \left(\ML{2}/Det_2\right)[x_{11},x_{21},x_{32},x_{33}],\]
where the latter ring is a polynomial extension of a domain.  The ideal $I_{231,231}$ is therefore prime in $\ML{3}$ and hence in $\GL{3}$ as well.

Let $I$ now be one of the 10 remaining ideals from Figure~\ref{fig:H_primes_gens}; using the symmetries listed in Figure~\ref{fig:H_primes_nice} there is always another ideal $J$ among the 26 already considered such that $\GL{3}/I$ is isomorphic as commutative algebras to $\GL{3}/J$.  The ideal $I$ is therefore prime in $\GL{3}$ and hence in $\ML{3}$ as well.
\end{proof}
We have constructed here 36 examples of Poisson $\HH$-primes in $\GL{3}$, but we postpone the proof that $\GL{3}$ admits no more such primes until \S\ref{ss:quantum to Poisson}.

\subsection{$\HH$-primes in $\SL{3}$}\label{ss:translating H primes to SL3}
As noted above, when working with $\SL{n}$ for any $n$ we cannot use the action of $\HH = (k^{\times})^{2n}$ defined in \eqref{eq:H-action on GLn} for $\ML{n}$ and $\GL{n}$ as it does not give rise to an action on $\SL{n}$: in the notation of \eqref{eq:H-action on GLn}, we would have $h.Det = \alpha_1\dots\alpha_n\beta_1\dots\beta_n Det \neq h.1$ in general.  Instead, we restrict our attention to a subset\nom{H@$\mathcal{H}'$}
\[\HP = \{h \in \HH : \alpha_1\dots\alpha_n\beta_1\dots\beta_n = 1\} \subset \HH\]
and take the induced action of $\HP$ on $\SL{n}$, that is:
\begin{equation}\label{eq:action of HP on SLn}
h.x_{ij} = \alpha_i \beta_j x_{ij}, \quad h = (\alpha_1, \dots, \alpha_n, \beta_1, \dots, \beta_n) \in \HP.
\end{equation}
The problem with this definition is it is not immediately clear how to connect the $\HH$-primes of $\ML{n}$ or $\GL{n}$ with the $\HP$-primes of $\SL{n}$.  In \cite[Lemma~II.5.16]{GBbook} and \cite[\S2]{quantumUFDs}, it is shown that by applying the natural projection map to the $\HH$-primes of $\QGL{n}$ we obtain precisely the $\HP$-primes of $\QSL{n}$, and we adapt their argument to the Poisson case here.

We begin by establishing a Poisson version of the isomorphism $\QSL{n}[z^{\pm1}] \cong \QGL{n}$ from \cite{LS1}.
\begin{lemma}\label{res:poisson iso sln to gln}
Let $\SL{n}$ and $\GL{n}$ be Poisson algebras, where the Poisson bracket in each case is the one induced by \eqref{eq:poisson relns for nxn matrices}.  Define $\SL{n}[z^{\pm1}] = \SL{n} \otimes k[z^{\pm1}]$, and extend to it the Poisson bracket from $\SL{n}$ by defining $z$ to be Poisson central: $\{z,a\} = 0$ for all $a \in \SL{n}$.  Then there is an isomorphism of Poisson algebras $\SL{n}[z^{\pm1}] \rightarrow \GL{n}$, defined by
\begin{align*}
\theta:  \SL{n}[z^{\pm1}] &\longrightarrow \GL{n} \\
x_{1j}\ &\mapsto \ x_{1j}Det^{-1} \\
x_{ij}\ &\mapsto \ x_{ij} \qquad \qquad (i \neq 1) \\
z\ &\mapsto \ Det.
\end{align*}
\end{lemma}
\begin{proof}
By taking $q = \lambda = 1$ in \cite{LS1}, we immediately get that $\theta$ is an isomorphism of commutative algebras, and we need only check that it respects the Poisson bracket.  Let $\delta_{ij}$ denote the Kronecker delta, i.e. $\delta_{ij} = 1$ when $i=j$ and $\delta_{ij} = 0$ otherwise, and recall that the determinant $Det$ is Poisson central in $\GL{n}$.

For $x_{ij}$, $x_{lm} \in \SL{n}$ we have
\begin{align}\label{eq:poisson iso eq 1}
\nonumber\theta\{x_{ij},x_{lm}\} &= \left\{\begin{array}{cc} \theta(x_{ij}x_{lm}) & (i=l \textrm{ or } j=m)\\ 0 & (i > l, j < m \textrm{ or }i < l,j > m)\\ 2\theta(x_{im}x_{lj})  & (i < l, j <m \textrm{ or }i>l,j>m)\end{array}\right. \\
\nonumber& \ \\
&= \left\{\begin{array}{cc} x_{ij}x_{lm}Det^{-\delta_{1i}}Det^{-\delta_{il}} & (i=l \textrm{ or } j=m)\\ 0 & (i > l, j < m \textrm{ or }i < l,j > m)\\ 2x_{im}x_{lj}Det^{-\delta_{1i}}Det^{-\delta_{il}}  & (i < l, j <m \textrm{ or }i>l,j>m)\end{array}\right. ,
\end{align}
while
\begin{align}\label{eq:poisson iso eq 2}
\nonumber\{\theta(x_{ij}), \theta(x_{lm})\} &= \{x_{ij}Det^{-\delta_{1i}}, x_{lm}Det^{-\delta_{1l}}\}\\
&= \{x_{ij},x_{lm}\}Det^{-\delta_{1i}}Det^{-\delta_{1l}},
\end{align}
since $Det$ is Poisson central.  Using the definition of the Poisson bracket in $\GL{n}$, it is clear that \eqref{eq:poisson iso eq 1} and \eqref{eq:poisson iso eq 2} are equal.  Finally, we see that 
\[\theta\{x_{ij},z\} = \theta(0) = 0 = \{x_{ij}Det^{-\delta_{1i}},Det\} = \{\theta(x_{ij}),\theta(z)\},\]
for all $x_{ij} \in \SL{n}$, and $\theta$ is therefore an isomorphism of Poisson algebras as required.
\end{proof}
We can now define an action of $\HH$ on $\SL{n}[z^{\pm1}]$ by conjugating the standard action of $\HH$ on $\GL{n}$ with $\theta$, i.e. for $h \in \HH$ we define
\begin{equation}\label{eq: action of HH on SLnZ}h.f = \theta^{-1}\circ h \circ \theta(f) \quad \forall f \in \SL{n}[z^{\pm1}].\end{equation}
This also restricts to an action of $\HH$ on $\SL{n}$.  Indeed, by working through the definition in \eqref{eq: action of HH on SLnZ}, we find that
\begin{equation}\label{eq:action of HH on SLn}\begin{aligned}
h.x_{1j} &= \alpha_1\beta_j(\alpha_1\dots\alpha_n\beta_1\dots\beta_n)^{-1}x_{1j}, \\
h.x_{ij} &= \alpha_i\beta_j x_{ij}\qquad  (i \neq 1).
\end{aligned}\end{equation}
The next two lemmas show that the set of $\HP$-primes in $\SL{n}$ coincides with the set of $\HH$-primes, which in turn coincides with the set of $\HH$-primes of $\SL{n}[z^{\pm1}]$.  This approach is based on the corresponding quantum result outlined in \cite[Lemma II.5.16, Exercise II.5.H]{GBbook}.
\begin{lemma}\label{res:bijection_lemma_1}
Let $\rho_1: \HH \rightarrow Aut(\SL{n})$ be the homomorphism of groups induced by the action defined in \eqref{eq:action of HH on SLn} above, and $\rho_2: \HP \rightarrow Aut(\SL{n})$ be the homomorphism induced by the standard action defined in \eqref{eq:action of HP on SLn}.  Then $im(\rho_1) = im(\rho_2)$ and hence $\HH$-$Pspec(\SL{n}) = \HP$-$Pspec(\SL{n})$.
\end{lemma}
\begin{proof}
Since $\HP\subset \HH$ and for $h \in \HP$ we have $\alpha_1\dots\alpha_n\beta_1\dots\beta_n =1$, it is easy to see that $\rho_1(h) = \rho_2(h)$ for all $h \in \HP$ and hence $im(\rho_2) \subseteq im(\rho_1)$.  Conversely, if $h \in \HH\backslash \HP$ then the action of $h$ on $\SL{n}$ is the same as the action of
\[h' = ((\alpha_2\dots\alpha_n\beta_1\dots\beta_n)^{-1},\alpha_2,\dots,\beta_n) \in \HP\]
and so $im(\rho_1) \subseteq im(\rho_2)$ as well.  Thus an ideal of $\SL{n}$ is fixed by $\HP$ if and only if it is fixed by the action of $\HH$ given in \eqref{eq: action of HH on SLnZ}.
\end{proof}

\begin{lemma}\label{res:bijection_lemma_2}
The mapping $\varphi: P \mapsto P[z^{\pm1}]$ defines a bijection between $\HH$-$Pspec(\SL{n})$ and $\HH$-$Pspec(\SL{n}[z^{\pm1}])$.
\end{lemma}
\begin{proof}
This proof is based on the non-commutative argument in \cite[Lemma~2.2]{quantumUFDs}.  Since $z$ is Poisson central and a $\HH$-eigenvector it is clear that $\varphi$ sends Poisson $\HH$-primes to Poisson $\HH$-primes and so $\varphi$ is well-defined.  

We claim that the inverse map is $Q \mapsto Q \cap \SL{n}$.  To prove this, we need to show that
\begin{equation}\label{eq:two equalitites of ideals for checking}\begin{aligned}
P[z^{\pm1}]\cap\SL{n} &= P \quad \forall P \in \HH\textrm{-}Pspec(\SL{n}) \\
(Q\cap \SL{n})[z^{\pm1}] &= Q \quad \forall Q \in \HH\textrm{-}Pspec(\SL{n}[z^{\pm1}])
\end{aligned}\end{equation}
and for this it will suffice to check the following statement: 
\begin{quote}For all $Q \in \HH$-$Pspec(\SL{n}[z^{\pm1}])$ and for all $f = f_1z^{k_1} + \dots + f_nz^{k_n} \in Q$, then $f_i \in Q \cap \SL{n}$ for all $i$.\end{quote}
The statement is clear when $n=1$, since $z$ is invertible in $\SL{n}[z^{\pm1}]$.  Now assume it is true for all sums of length $n-1$, and let
\[f = f_1z^{k_1} + \dots + f_nz^{k_n}\]
where we may assume without loss of generality that the $k_i$ are distinct and the $f_i \in \SL{n}$ for all $i$.  Let $h = (2,1,\dots,1) \in \HH$; observe from \eqref{eq: action of HH on SLnZ} and \eqref{eq:action of HH on SLn} that $h$ fixes all of $\SL{n}$ but acts on $z$ as multiplication by 2.  Since $Q$ is a $\HH$-stable ideal, we have
\[f - 2^{-k_n}h.f = \sum_{i=1}^n f_i(1-2^{k_i-k_n})z^{k_i} \in Q.\]
The final term in this sum is zero, leaving us with a sum of length $n-1$; by the inductive assumption, we therefore have $(1-2^{k_i-k_n})f_i \in Q$ for $1 \leq i \leq n-1$.  Since the $k_i$ are distinct and the $f_i$ are in $\SL{n}$, we can conclude that $f_i \in Q \cap \SL{n}$ for $1 \leq i \leq n-1$.

Now we have $f_nz^{k_n} = f - f_1z^{k_1} - \dots - f_{n-1}z^{k_{n-1}} \in Q$ and so $f_n \in Q \cap \SL{n}$ as required.

It is now easy to verify that the two equalities in \eqref{eq:two equalitites of ideals for checking} above are true, and the result follows.
\end{proof}
Now we are in a position to compare the $\HH$-primes of $\GL{n}$ and the $\HP$-primes of $\SL{n}$ directly, which we address in Proposition~\ref{res:bijection of Hprimes} next.  

Let $\pi: \ML{n} \rightarrow \SL{n}$ be the natural quotient map, which extends uniquely by localization theory to a Poisson homomorphism $\GL{n} \rightarrow \SL{n}$;  we will denote this map by $\pi$ as well.  The map $\theta$ continues to denote the isomorphism of Poisson algebras $\SL{n}[z^{\pm1}] \rightarrow \GL{n}$ from Lemma~\ref{res:poisson iso sln to gln}.

\begin{proposition}\label{res:bijection of Hprimes}
The mapping $P \mapsto \theta(P[z^{\pm1}])$ is a bijection of sets from the Poisson $\HP$-primes of $\SL{n}$ to the Poisson $\HH$-primes of $\GL{n}$, and the inverse map is given by $Q \mapsto \pi(Q)$.
\end{proposition}
\begin{proof}
By Lemmas~\ref{res:bijection_lemma_1} and \ref{res:bijection_lemma_2}, the $\HP$-primes of $\SL{n}$ are in bijection with the $\HH$-primes of $\SL{n}[z^{\pm1}]$.  Further, it is clear from the definition of the $\HH$-action on $\SL{n}$ in \eqref{eq: action of HH on SLnZ} that $\theta$ commutes with the action of $\HH$, and so we easily obtain the promised bijection $\HP$-$Pspec(\SL{n}\rightarrow \HH$-$Pspec(\GL{n})$.  

In \cite[Proposition~2.5]{GL1}, it is proved that $Q \mapsto \pi(Q)$ is the inverse mapping to $P \mapsto \theta(P[z^{\pm1}])$ in the case of quantum $GL_n$ and $SL_n$; however, since their proof relies only on looking at the action of $\theta$ and $h \in \HH$ on monomials and makes no use of the $q$-commuting structure, we can observe that the same proof works without modification for the Poisson case.
\end{proof}

\section{Poisson primitive ideals}\label{s:poisson primitive ideals}
Once we have identified all of the $\HH$-primes in an algebra, the Stratification Theorem (Theorem~\ref{res:Stratification Theorem, quantum version}) gives us a way of understanding its prime and primitive ideals -- up to localization, at least.  By the Stratification Theorem we know that if $I_{\omega}$ is a $\HH$-prime in $\QGL{3}$, then the prime ideals in the stratum\nom{S@$spec_{\omega}(A)$}
\[spec_{\omega}(\QGL{3}) = \left\{P \in spec(\QGL{3}) : \bigcap_{h \in \HH} h(P) = I_{\omega}\right\}\]
correspond homeomorphically to the prime ideals in $Z\big(\qr{\QGL{3}}{I_{\omega}}\big[\mathcal{E}_{\omega}^{-1}\big]\big)$, where $\mathcal{E}_{\omega}$ denotes the set of all regular $\HH$-eigenvectors in $\QGL{3}/I_{\omega}$.

\begin{notation}
While the notation $\big(\qr{R}{I}\big)\big[E^{-1}\big]$ eliminates any possible ambiguity, the brackets are cumbersome and we will often simply write $\qr{R}{I}\big[E^{-1}\big]$ instead; this will always denote the localization of $R/I$ at the set $E \subset R/I$.
\end{notation}

Goodearl and Lenagan prove in \cite[\S3.2]{GL1} that we may replace $\mathcal{E}_{\omega}$ with a subset $E_{\omega}$, provided that $E_{\omega}$ is still an Ore set (in a Noetherian ring, this is equivalent to being a denominator set by \cite[Proposition~10.7]{GW1}) such that the localization is $\HH$-simple.  For each $\HH$-prime $I_{\omega}$, they construct an Ore set $E_{\omega}$ satisfying these properties which is generated by finitely many normal elements.  This allows them to compute the localizations and their centres explicitly, and hence pull back the generators of the primitive ideals in the localizations to generators in $\QGL{3}$ itself.

Our aim in this section is to build on the work of \cite{GL1} to obtain a situation where we can develop the quantum and Poisson results simultaneously.  We start by modifying the Ore sets of \cite{GL1} so that our localizations $\qr{\QGL{3}}{I_{\omega}}\big[E_{\omega}^{-1}\big]$ are always quantum tori of the form $k_{\mathbf{q}}[z_1^{\pm1},z_2^{\pm1},\dots, z_n^{\pm1}]$, i.e. localizations of quantum affine spaces at the set of all their monomials.  The correspondence between prime and primitive ideals of a quantum torus and Poisson prime/primitive ideals of its semi-classical limit is already well understood (see for example \cite{OhSymplectic,GLz2}), and combined with the following slight generalizations of the Poisson Stratification Theorem this allows us to easily pull back the results to $\GL{3}$.

\begin{proposition}\label{res:poisson strat generalization 1}
Let $R$ be a commutative Noetherian Poisson algebra upon which an algebraic torus $\HH = (k^{\times})^r$ acts rationally by Poisson automorphisms, and let $J$ be a Poisson $\HH$-prime in $R$.  Suppose that $E_J$ is a multiplicative set generated by $\HH$-eigenvectors in $R/J$ such that the localization $R_J := \qr{R}{J}\big[E_J^{-1}\big]$ is Poisson $\HH$-simple.  Then the stratum $Pspec_J(R) = \{P \in Pspec(R) : \bigcap_{h\in \HH} h(P) = J\}$ is homeomorphic to $Pspec(R_J)$ via localization and contraction.
\end{proposition}
\begin{proof}
By standard ring theory (e.g. \cite[Theorem~10.20]{GW1}) there is an inclusion-preserving bijection given by extension and contraction between $\{P/J \in spec(R/J): P/J \cap E_J = \emptyset\}$ and $spec(R_J)$, and using the definition in \eqref{eq:extend Poisson bracket to localization} for the extension of a Poisson bracket to a localization it is easy to see that this restricts to a bijection on \textit{Poisson} primes.

We therefore need to prove that whenever $E_J$ satisfies the conditions of the proposition, we have an equality of sets
\begin{equation}\label{eq:equality of sets in stratification generalization}\{P/J \in Pspec(R/J) : \bigcap_{h \in \HH}h(P/J) = 0\} = \{P/J \in Pspec(R/J): P/J \cap E_J = \emptyset\};\end{equation}
this will be sufficient, since $Pspec_J(R)$ corresponds precisely to the first set in \eqref{eq:equality of sets in stratification generalization}.

If $P/J \in Pspec(R/J)$ satisfies $P/J \cap E_J \neq \emptyset$, then $P/J$ contains a $\HH$-eigenvector and it is clear that $\bigcap_{h \in \HH}h(P/J) \neq 0$.  Conversely, if $\bigcap_{h \in \HH}h(P/J) \neq 0$, then $P/J$ contains a non-trivial $\HH$-prime in $R/J$.  The ideal $P/J$ must therefore become trivial upon extension to $R_J = \qr{R}{J}\big[E_J^{-1}\big]$ since $R_J$ is Poisson $\HH$-simple, and so $P/J \cap E_J \neq \emptyset$.
\end{proof}
The Poisson Stratification Theorem also describes a homeomorphism between the Poisson primes of $\qr{R}{J}\big[\mathcal{E}_J^{-1}\big]$ (where $\mathcal{E}_J$ is the multiplicative set generated by all $\HH$-eigenvectors in $R/J$) and the primes of the Poisson centre $PZ\left(\qr{R}{J}\big[\mathcal{E}_J^{-1}\big]\right)$.  While it is routine to modify existing quantum proofs to replace $\mathcal{E}_J$ by $E_J$ in this result as well, we will not need this level of generality in this chapter.  As noted above, our localizations $\qr{R}{J}\big[E_J^{-1}\big]$ will always be semi-classical limits of quantum tori and so it suffices to use the following result by Oh.
\begin{proposition}\label{res:poisson strat generalization 2}
Let $R = k[z_1^{\pm1}, \dots, z_n^{\pm1}]$ be a commutative Laurent polynomial ring with a multiplicative Poisson bracket, i.e.
\[\{x_i. x_j\} = \lambda_{ij}x_ix_j \quad \lambda_{ij} \in k \textrm{ for all }i,j.\]
Then there is a homeomorphism between $Pspec(R)$ and $spec(PZ(R))$ given by contraction and extension, and this restricts to a homeomorphism $Pprim(R) \approx max(PZ(R))$.
\end{proposition}
\begin{proof}
\cite[Lemma~2.2, Corollary~2.3]{OhSymplectic}.
\end{proof}
Now suppose that $R$ is a commutative affine Noetherian Poisson $k$-algebra which has a rational $\HH$-action and only finitely many Poisson $\HH$-primes, and suppose further that for a $\HH$-prime $J$ there is a multiplicative set of $\HH$-eigenvectors in $R/J$ such that the localization $\qr{R}{J}\big[E_J^{-1}\big]$ has the form given in Proposition~\ref{res:poisson strat generalization 2}.  The Dixmier-Moeglin equivalence (Theorem~\ref{res: dixmier moeglin, Poisson version}) applies to algebras of this type, and so we also obtain a homeomorphic correspondence between the Poisson primitive ideals in the stratum corresponding to $J$ and the maximal ideals of $PZ\left(\qr{R}{J}\big[E_J^{-1}\big]\right)$.

\subsection{The quantum case}\label{ss:ore sets, quantum,etc}
We begin by summarising the work of Goodearl and Lenagan in \cite{GL1}, which allows us to set up the appropriate notation and present the results in a convenient form for transferring to the Poisson case.

The original Ore sets from \cite[Figure~3]{GL1} are reproduced in Table~\ref{fig:Ore set gens,GL version} for the reader's convenience.  Let $\omega = (\omega_{+},\omega_{-}) \in S_3 \times S_3$; the Ore set $E_{\omega}$ corresponding to the ideal $I_{\omega}$ is generated by $E_{\omega_{+}} \cup E_{\omega_{-}}$ from Table~\ref{fig:Ore set gens,GL version}.  The elements in each $E_{\omega_{+}}$ are viewed as coset representatives in the factor ring $\QGL{3}/I_{\omega_{+},321}$, since for any $\omega_{-} \in S_3$ we have $I_{\omega_{+},321} \subseteq I_{\omega_{+},\omega_{-}}$; similarly, the elements of $E_{\omega_{-}}$ are viewed as coset representatives in $\QGL{3}/I_{321,\omega_{-}}$.

\begin{table}
\centering
\begin{tabular}{c|cc}
$\omega$ & $E_{\omega_{+}}$ & $E_{\omega_{-}}$ \\
\hline
321 & \ $X_{31}$, $[\wt{1}|\wt{3}]_q$\vphantom{$\int^{\int}$}& $[\wt{3}|\wt{1}]_q$, $X_{13}$ \\
231 & $X_{21}$, $X_{32}$& $[\wt{2}|\wt{1}]_q$, $X_{13}$\\
312 & $X_{31}$, $[\wt{2}|\wt{3}]_q$& $X_{23}$, $X_{12}$\\
132 & $X_{11}$, $X_{32}$, $[\wt{1}|\wt{1}]_q$& $[\wt{1}|\wt{1}]_q$, $X_{23}$, $X_{11}$\\
213 & $X_{21}$, $[\wt{3}|\wt{3}]_q$, $X_{33}$& $X_{33}$, $X_{12}$, $[\wt{3}|\wt{3}]_q$\\
123 & $X_{11}$, $X_{22}$, $X_{33}$& $X_{11}$, $X_{22}$, $X_{33}$
\end{tabular}
\caption{Original generators for Ore sets in $\QGL{3}$ (see \cite[Figure~3]{GL1}).}\label{fig:Ore set gens,GL version}
\end{table}

These Ore sets satisfy all of the required properties: the induced action of $\HH$ on the localization is rational, the localization map $\QGL{3}/I_{\omega} \rightarrow \qr{\QGL{3}}{I_{\omega}}\big[E_{\omega}^{-1}\big]$ is always injective, and the localization is $\HH$-simple (see \cite[\S3.2]{GL1}).

The generators in Table~\ref{fig:Ore set gens,GL version} have also been chosen to exploit the symmetries induced $\tau$, $\rho$ and $S$: as discussed in \cite[\S3.3]{GL1}, in most cases it is immediately clear that a map $I_{\omega_1} \rightarrow I_{\omega_2}$ in Figure~\ref{fig:H_primes_nice} will also map the corresponding Ore set $E_{\omega_1}$ to $E_{\omega_2}$, and hence induce (anti-)isomorphisms of the localizations
\[\qr{\QGL{3}}{I_{\omega_1}}\big[E_{\omega_1}^{-1}\big] \longrightarrow \qr{\QGL{3}}{I_{\omega_2}}\big[E_{\omega_2}^{-1}\big].\]
These symmetries become less obvious when $\omega_{\pm} = 231$ or 312, so our first aim will be to modify the generators of these sets slightly (without changing the overall Ore set) in order to make it clear that these symmetries do actually induce (anti-)isomorphisms in these cases.

First we observe by direct calculation that
\begin{align}\label{eq:eqn1 for Ore set changes}[\wt{2}|\wt{3}]_q[\wt{1}|\wt{2}]_q &= Det_qX_{31} - X_{11}[\wt{1}|\wt{3}]_qX_{33}\\
\label{eq:eqn 2 for Ore set changes}[\wt{3}|\wt{2}]_q[\wt{2}|\wt{1}]_q &= Det_qX_{13} - X_{11}[\wt{3}|\wt{1}]_qX_{33}\end{align}
We would like to replace $X_{31}$ by $[\wt{1}|\wt{2}]_q$ in the set of generators for $E_{312_{+}}$ from Table~\ref{fig:Ore set gens,GL version}.  Since we are viewing elements of $E_{312_{+}}$ as coset representatives modulo $I_{312,321} = \langle [\wt{1}|\wt{3}]_q\rangle$ as explained above, by reducing \eqref{eq:eqn1 for Ore set changes} mod $[\wt{1}|\wt{3}]_q$ it is clear that we may substitute $[\wt{1}|\wt{2}]_q$ for $X_{31}$ in the generating set for $E_{312_{+}}$ without changing the Ore set at all.  Similarly, we may replace $X_{13}$ by $[\wt{3}|\wt{2}]_q$ in $E_{231_{-}}$.  

We may therefore take the elements in Table~\ref{fig:Ore set gens} as the generators for our Ore sets instead of those in Table~\ref{fig:Ore set gens,GL version}, where $E_{\omega}$ is the multiplicative set generated by $E_{\omega_{+}}\cup E_{\omega_{-}}$ as before.

\begin{table}
\centering

\begin{tabular}{c|cc}
$\omega$ & $E_{\omega_{+}}$ & $E_{\omega{-}}$ \\
\hline
321 & \ $X_{31}$, $[\wt{1}|\wt{3}]_q$ \vphantom{$\int^{\int}$}& $[\wt{3}|\wt{1}]_q$, $X_{13}$ \\
231 & $X_{21}$, $X_{32}$& $[\wt{2}|\wt{1}]_q$, $[\wt{3}|\wt{2}]_q$\\
312 & $[\wt{1}|\wt{2}]_q$, $[\wt{2}|\wt{3}]_q$& $X_{23}$, $X_{12}$\\
132 & $X_{11}$, $X_{32}$, $[\wt{1}|\wt{1}]_q$& $[\wt{1}|\wt{1}]_q$, $X_{23}$, $X_{11}$\\
213 & $X_{21}$, $[\wt{3}|\wt{3}]_q$, $X_{33}$& $X_{33}$, $X_{12}$, $[\wt{3}|\wt{3}]_q$\\
123 & $X_{11}$, $X_{22}$, $X_{33}$& $X_{11}$, $X_{22}$, $X_{33}$
\end{tabular}
\caption{Modified generators for Ore sets in $\QGL{3}$.}\label{fig:Ore set gens}
\end{table}

We now obtain the following equalities (based on \cite[\S3.3]{GL1}) with no restriction on $\omega_{+}$ or $\omega_{-}$:
\begin{equation}\label{eq:maps between Ore sets, quantum}\begin{aligned}
\tau(E_{y,z}) &= E_{z^{-1},y^{-1}} \\
S(E_{y,z}) &= E_{y^{-1},z^{-1}} \\
\rho(E_{y,z}) &= E_{w_0 y^{-1} w_0,w_0 z^{-1} w_0}
\end{aligned}\end{equation}
As before, $w_0$ denotes the transposition $(13) \in S_3$.

The arrangement of maps between $\HH$-primes in Figure~\ref{fig:H_primes_nice} have been chosen to be compatible with \eqref{eq:maps between Ore sets, quantum}, so whenever there is a map from $\omega_1$ to $\omega_2$ in Figure~\ref{fig:H_primes_nice} this induces an isomorphism or anti-isomorphism 
\[\qr{\QGL{3}}{I_{\omega_1}}\big[E_{\omega_1}^{-1}\big] \rightarrow \qr{\QGL{3}}{I_{\omega_2}}\big[E_{\omega_2}^{-1}\big].\]

When considering the structure of the localization $\qr{\QGL{3}}{I_{\omega}}\big[E_{\omega}^{-1}\big]$ and its centre, it now suffices to consider one example from each orbit in Figure~\ref{fig:H_primes_nice} since the other cases in the same orbit can easily be obtained by applying the appropriate (anti-)isomorphisms.

We will now make two final changes to these Ore sets, to ensure that the generating sets are as simple as possible and that the localization we obtain is a quantum torus.  First, whenever the determinant $Det_q$ decomposes as a product $X_{11}[\wt{1}|\wt{1}]_q$ or $X_{11}X_{22}X_{33}$ modulo a $\HH$-prime $I_{\omega}$, it is redundant to include these factors in the Ore set since they are already invertible, so we remove them from our generating sets for simplicity.  Second, when computing the centres of each localization in \cite[\S4]{GL1}, Goodearl and Lenagan first invert up to 4 additional elements in order to obtain a quantum torus and hence simplify the computation of the centres; we will add these elements to our Ore sets as well.

These changes are summarised in Figure~\ref{fig:Ore_sets_for_tori}.

\begin{notation}
For the remainder of the chapter, $E_{\omega}$\nom{E@$E_{\omega}$ (quantum)} will denote the multiplicative set generated by the elements in Figure~\ref{fig:Ore_sets_for_tori} in the row corresponding to $I_{\omega}$, which are viewed as elements in the factor ring $\QGL{3}/I_{\omega}$.  Ore sets for the remaining 24 $\HH$-primes can be obtained by applying the appropriate combination of $\tau$, $\rho$ and $S$ from Figure~\ref{fig:H_primes_nice}.  In order to simplify the notation, we define\nom{A@$A_{\omega}$}
\begin{equation}\label{eq:notation for localization, quantum case}
A_{\omega}:=\qr{\QGL{3}}{I_{\omega}}\big[E_{\omega}^{-1}\big].
\end{equation} 
\end{notation}

Based on the computations in \cite[\S4]{GL1}, Figure~\ref{fig:localizations} lists the generators of the quantum torus $A_{\omega}$ for one example of $\omega$ from each orbit defined in Figure~\ref{fig:H_primes_nice}.  As always, generators for the algebras $A_{\omega}$ not listed in this figure can be obtained using $\tau$, $\rho$ and $S$ as appropriate, and the $q$-commuting relations between pairs of generators in a given ring $A_{\omega}$ can easily be computed using the relations in $\QML{3}$ and deleting any terms which appear in the ideal $I_{\omega}$.  

We also reproduce in Figure~\ref{fig:centres} the generators for the centres $Z(A_{\omega})$, which appear in \cite[Figure~5]{GL1}.  Observe that for $\omega = (123,123)$, the image of $Det_q$ in $\QGL{3}/I_{\omega}$ is $Det_q = X_{11}X_{22}X_{33}$ and the centre is generated by $X_{11}$, $X_{22}$ and $X_{33}$; we can therefore replace (for example) $X_{33}$ by $Det_q$ in the list of generators, and this we shall do.  We make a similar change when $\omega = (132,132)$ or $(123,132)$, so that $Det_q$ appears as a generator of the centre in all 36 cases; this will make it simpler to transfer our results to $\SL{3}$ in future sections.

It is the description of the centres $Z(A_{\omega})$ which are of the most use to us: they are commutative Laurent polynomial rings, and so when $k$ is algebraically closed the maximal ideals of a given algebra $Z(A_{\omega}) = k[Z_1^{\pm1}, \dots, Z_n^{\pm1}]$ are precisely those of the form
\begin{equation}\label{eq:maximal ideal in centre, quantum}\mathfrak{m}_{\lambda} = \langle Z_1 - \lambda_1, \dots, Z_n - \lambda_n\rangle, \quad \lambda = (\lambda_1, \dots, \lambda_n) \in k^{\times}, \ 1 \leq i \leq n.\end{equation}
From Figure~\ref{fig:centres}, we can observe that each $Z_i$ has the form $E_iF_i^{-1}$, where $E_i$ and $F_i$ are both normal elements of $\QGL{3}/I_{\omega}$.  The key result of \cite{GL1} is that for each $\omega \in S_3 \times S_3$ and each maximal ideal \eqref{eq:maximal ideal in centre, quantum} in $Z(A_{\omega})$, we have
\[\mathfrak{m}_{\lambda} \cap \qr{\QGL{3}}{I_{\omega}} = \langle E_1 - \lambda_1F_1, \dots, E_n - \lambda_nF_n\rangle,\]
(see \cite[\S5]{GL1}).  By the Stratification Theorem these describe all of the primitive ideals in $\QGL{3}$.

\subsection{From quantum to Poisson}\label{ss:quantum to Poisson}

Our eventual aim is to show that there is a natural bijection between $prim(\QGL{3})$ and $Pprim(\GL{3})$, and similarly for $SL_3$.  The next step is therefore to obtain a description of the Poisson primitive ideals in algebra $\GL{3}$; however, rather than simply repeat the analysis of \cite{GL1} and replace ``quantum'' by ``Poisson'' throughout, we will take a shortcut using Proposition~\ref{res:localization and SCL commute} and the close relationship between quantum and Poisson tori originally described in \cite{OhSymplectic}.

We will start by checking that the Ore sets $E_{\omega}$ lift to Ore sets in the formal $k[t^{\pm1}]$-algebra that governs the deformation process.  Using this, we will show that by taking the semi-classical limit of the quantum tori appearing in Figure~\ref{fig:localizations}, we obtain the same algebras as if we had localized the Poisson algebras $\GL{3}/I_{\omega}$ at the sets in Figure~\ref{fig:Ore_sets_for_tori} (now viewed as elements of the corresponding Poisson algebra).  This will give us Poisson $\HH$-simple localizations $\qr{\GL{3}}{I_{\omega}}\big[E_{\omega}^{-1}\big]$ with a structure which is already well understood from the quantum case, and we may use these algebras to describe the Poisson-prime and Poisson-primitive ideals of $\GL{3}$.

Recall from Definition~\ref{def:algebra B for quanum matrices} that $\RR$ is the $k[t^{\pm1}]$-algebra on 9 generators $Y_{ij}$, $1 \leq i,j \leq 3$ such that $\RR/(t-q)\RR \cong \QML{3}$ and $\RR/(t-1)\RR \cong \ML{3}$.  By a slight abuse of notation we will also denote by $I_{\omega}$ the ideals in $\RR$ corresponding to the 36 ideals in Figure~\ref{fig:H_primes_gens}, obtained by replacing $X_{ij}$ by $Y_{ij}$ and $[\wt{i}|\wt{j}]_q$ by $[\wt{i}|\wt{j}]_t$ in the generating sets.  

Since we have expanded our Ore sets $E_{\omega}$ to include elements which are not normal in $\GL{3}/I_{\omega}$, we have to work slightly harder to verify that the corresponding multiplicative sets in $\RR/I_{\omega}$ are also Ore sets.  We will approach this in a roundabout manner, by constructing iterated Ore extensions with exactly the properties that $\RR/I_{\omega}[E_{\omega}^{-1}]$ would have if it exists; hence by the universality of localization this algebra \textit{is} $\RR/I_{\omega}[E_{\omega}^{-1}]$ and $E_{\omega}$ must be an Ore set.

The following two results encapsulate the process we will use.

\begin{lemma}\label{res:extending maps to localizations}
Let $\alpha$ be an endomorphism and $\delta$ an $\alpha$-derivation on a ring $R$, and suppose that $X$ is an Ore set in $R$.  If $\alpha$ extends to an endomorphism of $R[X^{-1}]$ then $\delta$ extends to an $\alpha$-derivation of $R[X^{-1}]$.
\end{lemma}
\begin{proof}
If $\alpha$ extends to $R[X^{-1}]$, then $\alpha(x)^{-1}$ is defined for all $x \in X$.  If $\delta$ also extended to $R[X^{-1}]$ then it would have to satisfy the following equality for any $x \in X$:
\[0 = \delta(1) = \delta(xx^{-1}) = \alpha(x)\delta(x^{-1}) + \delta(x)x^{-1},\]
and hence 
\[\delta(x^{-1}) = -\alpha(x)^{-1}\delta(x)x^{-1}\]
is uniquely determined.  Since $\alpha(x)^{-1}$ exists by assumption, this is well-defined and $\delta$ extends as required.
\end{proof}
\begin{corollary}\label{res:extending ore extensions to localizations}
If $R[z;\alpha,\delta]$ is an Ore extension of a ring $R$ and $X$ is an Ore set in $R$, then the extension $R[X^{-1}][z;\alpha,\delta]$ exists (with the natural extension of $\alpha$ and $\delta$ to $R[X^{-1}]$) if and only if $\alpha(x)^{-1}$ is defined for each $x \in X$.
\end{corollary}
\begin{proof}
The extension $R[X^{-1}][z;\alpha,\delta]$ exists if and only if $\alpha$ is an endomorphism of $R[X^{-1}]$ and $\delta$ is an $\alpha$-derivation of $R[X^{-1}]$.  By Lemma~\ref{res:extending maps to localizations}, this happens if and only if $\alpha$ is defined on $X^{-1}$.
\end{proof}

Let $F_{\omega}$ denote the multiplicative set generated in $\RR/I_{\omega}$ by taking the corresponding set of generators from column 3 of Figure~\ref{fig:Ore_sets_for_tori} and applying the rewriting rule $X_{ij} \mapsto Y_{ij}$, $[\wt{i}|\wt{j}]_q \mapsto [\wt{i}|\wt{j}]_t$.  We are now in a position to verify that the $F_{\omega}$ are indeed Ore sets in $\RR/I_{\omega}$.  We begin by considering our usual 12 cases, that is the ones listed explicitly in Figure~\ref{fig:Ore_sets_for_tori}.

\begin{proposition}\label{res:Ore sets lift to formal deformation ring}
For the 12 examples of $F_{\omega}$ induced by the elements listed in Figure~\ref{fig:Ore_sets_for_tori}, $F_{\omega}$ is an Ore set in $\RR/I_{\omega}$.
\end{proposition}
\begin{proof}
As described above, our approach will be to construct $k[t^{\pm1}]$-algebras with precisely the properties that $\RR/I_{\omega}[F_{\omega}^{-1}]$ will have if it exists; by universality the localization therefore must exist, which is possible if and only if $F_{\omega}$ is an Ore set.

We begin with the case $\omega = (321,321)$, and we will compute this case in detail as all of the others follow by a very similar method.  Note that $I_{\omega} = (0)$, so we identify $\RR/(0)$ with $\RR$.  We need to show that the localization $\RR[Y_{11}^{-1},Y_{12}^{-1},Y_{21}^{-1},[\wt{3}|\wt{3}]_t^{-1}]$ exists.

We start by defining the following algebra:
\begin{equation*}R_1 := k[t^{\pm1},Y_{11}^{\pm1}][Y_{12}^{\pm1};\alpha_0][Y_{21}^{\pm1};\alpha_1],
\end{equation*}
where the $k[t^{\pm1}]$-linear automorphisms $\alpha_0$ and $\alpha_1$ are defined by
\begin{align*}
\alpha_0&: Y_{11} \mapsto t^{-1}Y_{11}, \\
\alpha_1&: Y_{11} \mapsto t^{-1} Y_{11},\ Y_{12} \mapsto Y_{12}.
\end{align*}
We next define
\begin{gather*}
R_2 := R_1[Y_{22};\alpha_2,\delta_2],\\
\alpha_2: Y_{11} \mapsto Y_{11},\ Y_{12} \mapsto t^{-1}Y_{12}, \ Y_{21} \mapsto t^{-1}Y_{22}; \\
\delta_2: Y_{11} \mapsto (t^{-1}-t)Y_{12}Y_{21}, \ Y_{12} \mapsto 0, \ Y_{21} \mapsto 0.
\end{gather*}
It is easy to see that $R_2 \cong \RR_2[Y_{11}^{-1},Y_{12}^{-1},Y_{21}^{-1}]$ (recall $\RR_2$ is the $k[t^{\pm1}]$-algebra giving rise to the deformation $\QML{2}$).  Since $[\wt{3}|\wt{3}]_t$ plays the role of the $2\times 2$ determinant in $\RR_2$, it is central and therefore invertible, and we define
\[R_3 = R_2[[\wt{3}|\wt{3}]_t^{-1}].\]
We have now inverted all the required elements; the next step is to verify that we can adjoin the remaining generators $\{Y_{13},Y_{31},Y_{23},Y_{32},Y_{33}\}$ to $R_3$ via Ore extensions in the appropriate way, checking at each step that the Ore extension makes sense on the inverted elements.  To do this we will rely heavily on Corollary~\ref{res:extending ore extensions to localizations}, which tells us that if $\alpha$ is an endomorphism and $\delta$ an $\alpha$-derivation on a ring $R$ and $X$ is an Ore set, then it suffices to check that $\alpha$ is defined on $X^{-1}$ in order to construct $R[X^{-1}][z;\alpha,\delta]$.

We define
\begin{equation}\label{eq:more endless ore extensions}\begin{gathered}
R_4 := R_3[Y_{13};\alpha_3][Y_{31};\alpha_4];\\
\alpha_3: Y_{11} \mapsto t^{-1}Y_{11}, \ Y_{12} \mapsto t^{-1}Y_{12}; \\
\alpha_4: Y_{11} \mapsto t^{-1}Y_{11}, \ Y_{21} \mapsto t^{-1}Y_{21};
\end{gathered}\end{equation}
where each $\alpha_i$ acts as the identity on any generators not listed in \eqref{eq:more endless ore extensions}.  These automorphisms are clearly defined on $Y_{11}^{-1}$, $Y_{12}^{-1}$ and $Y_{21}^{-1}$, and $\alpha_3([\wt{3}|\wt{3}]_t) = \alpha_4([\wt{3}|\wt{3}]_t) = t^{-1}[\wt{3}|\wt{3}]_t$ also poses no problems.  The algebra $R_4$ therefore makes sense, and we proceed to adjoin $Y_{23}$:
\begin{equation}\label{eq:endless ore extensions}\begin{gathered}
R_5:=R_4[Y_{23};\alpha_5,\delta_5];\\
\alpha_5: Y_{13} \mapsto t^{-1}Y_{13}, \ Y_{21} \mapsto t^{-1}Y_{21},\ Y_{22} \mapsto t^{-1}Y_{22}; \\
\delta_5: Y_{11} \mapsto (t^{-1}-t)Y_{13}Y_{21}, \ Y_{12} \mapsto (t^{-1}-t)Y_{13}Y_{22};
\end{gathered}\end{equation}
where $\alpha_5$ acts as the identity on any generators not listed in \eqref{eq:endless ore extensions}, and $\delta_5$ acts as 0 on any generators not listed.  We observe that $\alpha_5([\wt{3}|\wt{3}])^{-1} = t[\wt{3}|\wt{3}]_t^{-1}$ is defined, so $R_5$ is a genuine Ore extension.  Similarly, we set
\begin{gather*}
R_6:=R_5[Y_{32};\alpha_6,\delta_6];\\
\alpha_6: Y_{31} \mapsto t^{-1}Y_{31}, \ Y_{22} \mapsto t^{-1}Y_{22}, \ Y_{12} \mapsto t^{-1}Y_{12};\\
\delta_6: Y_{21} \mapsto (t^{-1}-t)Y_{22}Y_{31},\ Y_{11} \mapsto (t^{-1}-t)Y_{12}Y_{31}.
\end{gather*}
The only remaining variable to adjoin is $Y_{33}$, which proceeds in a very similar manner: define
\begin{gather*}
R_7:=R_6[Y_{33};\alpha_7,\delta_7];\\
\alpha_7: Y_{23} \mapsto t^{-1}Y_{23}, \ Y_{13}\mapsto t^{-1}Y_{13}, \ Y_{32}\mapsto t^{-1}Y_{32}, \ Y_{31}\mapsto t^{-1}Y_{31};\\
\delta_7: Y_{11} \mapsto (t^{-1}-t)Y_{13}Y_{31}, \ Y_{12}\mapsto (t^{-1}-t)Y_{13}Y_{32}, \ Y_{21}\mapsto (t^{-1}-t)Y_{23}Y_{31}, \\
Y_{22}\mapsto (t^{-1}-t)Y_{23}Y_{32}.
\end{gather*}
Here $\alpha_7$ acts as the identity on each of our elements of interest $Y_{11}$, $Y_{12}$, $Y_{21}$ and $[\wt{3}|\wt{3}]_t$, and so again by Corollary~\ref{res:extending ore extensions to localizations} the definition of $R_7$ makes sense.

Observe that the variables $\{Y_{ij}: 1 \leq i,j\leq 3\}$ in $R_7$ satisfy exactly the same relations as those in $\RR$, and so we have $R_7 = \RR[F_{321,321}^{-1}]$ as required.  Since the localization $\RR[F_{321,321}^{-1}]$ exists if and only if $F_{321,321}$ is an Ore set in $\RR$, the result is proved for this case.

The six cases $\omega = (321,312), (231,312), (321,132), (321,123), (132,312)$ and $(132,132)$ follow by almost identical methods: in each case $\RR/I_{\omega}$ can be identified with an iterated Ore extension in $\leq 8$ variables, and by choosing the order of the variables carefully the generators of $F_{\omega}$ can easily be inverted early on in the process to construct $\qr{\RR}{I_{\omega}}[F_{\omega}^{-1}]$.

In four more cases, namely $\omega = (123,312), (213,132), (123,132)$ and $(123,123)$, the set $F_{\omega}$ is empty and there is nothing to prove.  This leaves us with only the case $\omega = (231,231)$ to consider.  

The ideal $I_{231,231}$ is generated by $Y_{31}$ and $[\wt{3}|\wt{1}]_t$, and $F_{231,231}$ is the multiplicative set generated by $Y_{33}$ and $[\wt{1}|\wt{1}]_t$.  By a similar method to the above we may easily construct $\RR[F_{231,231}^{-1}]$, and by Proposition~\ref{res:localization and quotient commute} we have
\[\big(\RR[F_{231,231}^{-1}]\big)/I_{231,231} \cong \qr{\RR}{I_{231,231}}[F_{231,231}^{-1}].\]
All that remains is to check that we have not constructed the zero ring, i.e. that $I_{231,231} \cap F_{231,231} = \emptyset$, but this is easy to check using the grading on $\RR$.  The ring $\qr{\RR}{I_{231,231}}[F_{231,231}^{-1}]$ therefore exists, and $F_{231,231}$ is an Ore set.
\end{proof}
Let $Det_t$ denote the $3\times3$ determinant in $\RR$.  This is central in $\RR$, since the computations involved in verifying its centrality in $\QML{3}$ continue to be valid if we replace $q$ by $t$.  Similarly, $Det_t$ remains central and non-zero modulo each $I_{\omega}$, and we may form the algebras
\[\RR/I_{\omega}[F_{\omega}^{-1},Det_t^{-1}] \cong \big(\RR[Det_t^{-1}]\big)/I_{\omega}[F_{\omega}^{-1}].\]
Having now inverted $Det_t$, we may now use the (anti-)isomorphisms of Figure~\ref{fig:H_primes_nice} once again: this tells us that $F_{\omega}$ is in fact an Ore set in $\RR[Det_t^{-1}]/I_{\omega}$ for all $\omega \in S_3 \times S_3$.  We can now obtain the result we have been working towards:
\begin{corollary}\label{res:ore sets ACTUALLY DO LIFT}
Let $\RR$ and $F_{\omega}$ be as above, and let $F_{\omega}'$ denote the multiplicative set in $R:=\RR[Det_t^{-1}]/I_{\omega}$ generated by the corresponding elements in the \textit{second} column of Figure~\ref{fig:Ore_sets_for_tori}.  Then $R[F_{\omega}^{-1},F_{\omega}'^{-1}]$ exists, and $F_{\omega} \cup F_{\omega}'$ is an Ore set in $R$.
\end{corollary}
\begin{proof}
Since $F_{\omega}$ is an Ore set in $R = \RR[Det_t^{-1}]/I_{\omega}$ by Proposition~\ref{res:Ore sets lift to formal deformation ring} and comments following the proof, we need only check that $F_{\omega}'$ is an Ore set in $R[F_{\omega}^{-1}]$.  This is easy to check, however, since the generators of $F_{\omega}'$ are normal in $R$ (this can be seen by using the relations in $\RR$ and deleting any terms which are in $I_{\omega}$) and therefore it automatically forms an Ore set.  That the union of two Ore sets is an Ore set follows by the universality of localization.
\end{proof}

\begin{proposition}\label{res:the poisson localizations are scls}
For each $\omega \in S_3\times S_3$, let $E_{\omega}'$ be the multiplicative set in $\GL{3}/I_{\omega}$ generated by the set of elements indicated in Figure~\ref{fig:Ore_sets_for_tori}, viewed as elements in the Poisson algebra $\GL{3}/I_{\omega}$ rather than $\QGL{3}/I_{\omega}$.  Then the localization of $\GL{3}/I_{\omega}$ at the set $E_{\omega}'$ is precisely the semi-classical limit of the corresponding quantum torus $A_{\omega}$ in Figure~\ref{fig:localizations}.
\end{proposition}
\begin{proof}
By lifting the generators of the $\HH$-prime $I_{\omega}$ to $\RR$, which is possible since the generators are always quantum minors, we can observe that $t-q$ and $t-1$ are not in the resulting ideal (here $q$ can be any non-zero non-root of unity in $k$).  By Proposition~\ref{res:quotient and scl commute} we must have that $\GL{3}/I_{\omega}$ is the semi-classical limit of $\QGL{3}/I_{\omega}$.

Now consider the elements of $E_{\omega}$ lifted to $\RR/I_{\omega}$; these are still Ore sets by Proposition~\ref{res:ore sets ACTUALLY DO LIFT} and we can easily check that the conditions of Proposition~\ref{res:localization and SCL commute} are satisfied.  By Proposition~\ref{res:localization and SCL commute}, the localization of the semi-classical limit $\GL{3}/I_{\omega}$ at the set $E_{\omega}'$ is therefore the same as the semi-classical limit of the localized algebra $\qr{\QGL{3}}{I_{\omega}}\big[E_{\omega}^{-1}\big]$.
\end{proof}
\begin{corollary}
The localization $\qr{\GL{3}}{I_{\omega}}\big[E_{\omega}^{-1}\big]$ is Poisson $\HH$-simple.
\end{corollary}
\begin{proof}
By Proposition~\ref{res:the poisson localizations are scls}, the algebra $\qr{\GL{3}}{I_{\omega}}\big[E_{\omega}^{-1}\big]$ is a commutative Laurent polynomial ring $k[z_1^{\pm1}, \dots, z_m^{\pm1}]$ with the multiplicative Poisson bracket $\{z_i,z_j\} = \pi_{ij}z_iz_j$ for some appropriate set of scalars $\{\pi_{ij}\}$.  (The precise values of the $\pi_{ij}$ do not matter here, but can be computed easily from the $q$-commuting structure of the corresponding $A_{\omega}$.)  By \cite[Example~4.5]{Goodearl_Poisson}, a Poisson algebra of this form is Poisson $\HH$-simple.
\end{proof}
\begin{notation}
From now on, we will use the notation $E_{\omega}$\nom{E@$E_{\omega}$ (Poisson)} interchangeably for the Ore set in $\QGL{3}/I_{\omega}$ and for the multiplicative set $E_{\omega}'$ defined in Proposition~\ref{res:the poisson localizations are scls}; it should always be clear from context whether we mean $E_{\omega}$ as a set in $\QGL{3}$ or $\GL{3}$.  We will retain the notation $A_{\omega} = \qr{\QGL{3}}{I_{\omega}}\big[E_{\omega}^{-1}\big]$, and define $B_{\omega} := \qr{\GL{3}}{I_{\omega}}\big[E_{\omega}^{-1}\big]$\nom{B@$B_{\omega}$} for the corresponding Poisson algebra.
\end{notation}

We may now tackle the proof that $\GL{3}$ admits only the 36 Poisson $\HH$-primes displayed in Figure~\ref{fig:H_primes_gens}, which has been postponed until now because it uses the Poisson $\HH$-simplicity of the localizations $B_{\omega}$.  We will first require one more lemma, which we prove next.
\begin{lemma}\label{res:small lemma for 36 H primes}
Let $P$ be a non-trivial Poisson $\HH$-prime in $\GL{3}$.  Then:
\begin{enumerate}[(i)]
\item If $X_{12}$ or $X_{23} \in P$ then $X_{13} \in P$ as well;
\item If $X_{21}$ or $X_{32} \in P$ then $X_{31} \in P$ as well.
\end{enumerate}
\end{lemma}
\begin{proof}
Since $P$ is closed under Poisson brackets, if $X_{12}$ or $X_{23} \in P$ then $\{X_{12},X_{23}\} = 2X_{13}X_{22} \in P$ as well.  The ideal $P$ is also assumed to be prime in the commutative sense and so $X_{13}$ or $X_{22} \in P$ as well, but by Lemma~\ref{res:Det in prime ideal, Poisson}, any Poisson prime in $\GL{3}$ containing $X_{22}$ also contains $Det$.  Since $P$ is non-trivial by assumption, we can conclude $X_{13} \in P$.  The statement $(ii)$ follows by a similar argument.
\end{proof}

\begin{theorem}\label{res:gl3 has only 36 h primes}
The Poisson algebra $\GL{3}$ admits only 36 Poisson $\HH$-primes, and these are the 36 ideals appearing in Figure~\ref{fig:H_primes_gens}.
\end{theorem}
\begin{proof}
Suppose $P$ is a non-trivial Poisson $\HH$-prime in $\GL{3}$ which is not one of the 36 appearing in Figure~\ref{fig:H_primes_gens}; we will use the Ore sets from Figure~\ref{fig:Ore_sets_for_tori} to show that this is a contradiction.  The key observation is that if $J$ is a Poisson $\HH$-prime appearing in Figure~\ref{fig:H_primes_gens} such that $J \subset P$, then $P$ is a non-trivial Poisson $\HH$-prime in $\GL{3}/J$ and therefore must contain one of the elements in the Ore set associated to $J$ since the localization is Poisson $\HH$-simple.

We start the process by noting that $0 \subset P$, and so by the observation above $P$ must contain one of the elements in the first row of Figure~\ref{fig:Ore_sets_for_tori}.  By Lemma~\ref{res:Det in prime ideal, Poisson} $P$ cannot contain either $X_{11}$ or $[\wt{3}|\wt{3}]$ since it is non-trivial, while by Lemma~\ref{res:small lemma for 36 H primes} if $P$ contains $X_{12}$ or $X_{21}$ then it must also contain $X_{13}$ or $X_{31}$ respectively.  Therefore $P$ must contain one of $X_{13}$, $X_{31}$, $[\wt{1}|\wt{3}]$ or $[\wt{3}|\wt{1}]$, each of which generate a Poisson $\HH$-prime appearing in Figure~\ref{fig:H_primes_gens}.

We may now iterate this argument: suppose $P$ contains a known $\HH$-prime $J$ of height $n$ (i.e. one from Figure~\ref{fig:H_primes_gens}).  If $J$ appears in Figure~\ref{fig:Ore_sets_for_tori} then $P$ must contain one of the elements listed in the row corresponding to $J$, and by applying Lemmas~\ref{res:Det in prime ideal, Poisson} and \ref{res:small lemma for 36 H primes} as above we find that $P$ must contain a known $\HH$-prime of height $n+1$ as well.

On the other hand, if $J$ does not appear in Figure~\ref{fig:Ore_sets_for_tori} then there exists some combination $\varphi$ of the maps $\tau$, $\rho$ and $S$ from Figure~\ref{fig:H_primes_nice} such that $\varphi(J)$ \textit{does} appear in Figure~\ref{fig:Ore_sets_for_tori} and $\varphi(J) \subset \varphi(P)$.  The ideal $\varphi(P)$ must be a Poisson $\HH$-prime, since $\tau$, $\rho$ and $S$ are Poisson morphisms that preserve $\HH$-stable subsets, and $\varphi(P)$ cannot appear in Figure~\ref{fig:H_primes_gens} otherwise $P$ would as well.

Now, by a similar argument $\varphi(P)$ must contain a known $\HH$-prime $K$ of height $n+1$ as above, and so $\varphi^{-1}(K) \subset P$ where $K$ is a known $\HH$-prime of height $n+1$.  

This process must terminate since $\GL{3}$ has finite Krull dimension; this is a contradiction and so the unknown Poisson $\HH$-prime $P$ does not exist.
\end{proof}

Combining Proposition~\ref{res:the poisson localizations are scls} and Theorem~\ref{res:gl3 has only 36 h primes}, we are now well on our way towards understanding the Poisson primitive ideals of $\GL{3}$.  The next step is to understand the similarity between $Z(A_{\omega})$ and $PZ(B_{\omega})$.
\begin{proposition}\label{res:centres of gl and qgl localizations agree}
For each $\omega \in S_3 \times S_3$, the Poisson centre of $B_{\omega}$ is equal to the algebra obtained by taking the centre of the corresponding quantum algebra $A_{\omega}$ and renaming the generators by applying the rule $X_{ij} \mapsto x_{ij}$, $[\wt{i}|\wt{j}]_q \mapsto [\wt{i}|\wt{j}]$.
\end{proposition}
\begin{proof}
Fix an $\omega \in S_3 \times S_3$.  From Figure~\ref{fig:localizations}, $A_{\omega} = k_{\mathbf{q}}[W_1^{\pm1}, \dots, W_n^{\pm1}]$ is a quantum torus, where $\mathbf{q} = (a_{ij})$ is an additively antisymmetric matrix and $W_iW_j = q^{a_{ij}}W_jW_i$.  (The values of the $a_{ij}$ can easily be calculated as in \cite{GL1}, but since their precise values have no impact on the proof and will not be used subsequently we do not define them here.)  By Proposition~\ref{res:the poisson localizations are scls}, $B_{\omega}$ is the semi-classical limit of $A_{\omega}$, and we can therefore easily compute its Poisson bracket as follows.  Write $B_{\omega} = k[w_1^{\pm1}, \dots, w_n^{\pm1}]$, where $w_i$ is the image of $W_i$ under the rewriting map $X_{ij} \mapsto x_{ij}$, and now we can observe that
\begin{align*}
\{w_i,w_j\} &= \frac{1}{t-1}(W_iW_j - W_jW_i) \quad \mod (t-1) \\
&= \frac{1}{t-1}(t^{a_{ij}}-1)W_jW_i \quad \mod (t-1)\\
&= (1 + t + \dots + t^{a_{ij}-1})W_iW_j \quad \mod (t-1) \\
&= a_{ij}w_iw_j.
\end{align*}
Let $\mathbb{Z}^n$ be the free abelian group of rank $n$ with basis $\{e_i\}_{i=1}^n$.  Then we may define two maps as follows:
\begin{align*}
\sigma&: \mathbb{Z}^n \times \mathbb{Z}^n \rightarrow k^{\times}: (e_i, e_j) \mapsto q^{a_{ij}}, \\
u&: \mathbb{Z}^n \times \mathbb{Z}^n \rightarrow k: (e_i, e_j) \mapsto a_{ij}.
\end{align*}
These define an alternating bicharacter and an antisymmetric biadditive map respectively (for definitions, see \cite[\S2]{OhSymplectic}; that $\sigma$ and $u$ satisfy the required properties follows directly from the $q$-commuting structure of $A_{\omega}$ and the Poisson structure of $B_{\omega}$). We define two subsets of $\mathbb{Z}^n$ as follows:
\begin{align*}
\mathbb{Z}^n_{\sigma} = \{\lambda \in \mathbb{Z}^n: \sigma(\lambda, \mu) = 1 \ \forall \mu \in \mathbb{Z}^n\},\\
\mathbb{Z}^n_u = \{\lambda \in \mathbb{Z}^n: u(\lambda, \mu) = 0 \ \forall \mu \in \mathbb{Z}^n\}.
\end{align*}
By \cite[\S2.5]{OhSymplectic}, the centre of $A_{\omega}$ is generated by the monomials $\{W^{\lambda}: \lambda \in \mathbb{Z}^n_{\sigma}\}$ (where we use the standard multi-index notation for monomials), while by \cite[Lemma~2.1]{OhSymplectic} the Poisson centre of $B_{\omega}$ is generated by the monomials $\{w^{\lambda}: \lambda \in \mathbb{Z}^n_u\}$.

It is clear from our definitions of $\sigma$ and $u$ that $\mathbb{Z}^n_{\sigma} = \mathbb{Z}^n_u$ in this case, and the result now follows.
\end{proof}
With this description for the Poisson centres $PZ(B_{\omega})$ in hand, we are now in a position to apply the Poisson Stratification theorem and obtain a description of the Poisson-primitive ideals of $\GL{3}$.
\begin{theorem}\label{res:thm, desc of poisson primitives up to localization, gl3}
Let $\omega \in S_3 \times S_3$, and let $I_{\omega}$ be the corresponding Poisson $\HH$-prime of $\GL{3}$ listed in Figure~\ref{fig:H_primes_gens}.  Then the Poisson-primitive ideals in the stratum 
\[Pprim_{\omega}(\GL{3}) = \left\{P \in Pprim(\GL{3}) : \bigcap_{h \in \HH} h(P) = I_{\omega}\right\}\]
correspond precisely to ideals of the form
\begin{equation}\label{eq:first description of poisson primitives}\left(\sum_{i=1}^n (z_i - \lambda_i)B_{\omega}\right) \cap \GL{3}/I_{\omega}, \quad (\lambda_1, \dots, \lambda_n) \in (k^{\times})^n,\end{equation}
where $PZ(B_{\omega}) = k[z_1^{\pm1}, \dots, z_n^{\pm1}]$ is the Poisson centre of $B_{\omega}$, and the generators $z_i$ ($1 \leq i \leq n$) are given in Figure~\ref{fig:centres}.

Conversely, every Poisson primitive ideal of $\GL{3}$ has this form (for an appropriate choice of $\omega$).
\end{theorem}
\begin{proof}
By Propositions~\ref{res:poisson strat generalization 1} and \ref{res:poisson strat generalization 2} there are homeomorphisms
\begin{equation}\label{eq:reminder of homeomorphisms}Pspec_{\omega}(\GL{3}) \approx Pspec(B_{\omega}) \approx spec(PZ(B_{\omega})),\end{equation}
given by localization/contraction and contraction/extension respectively.  By the Poisson Dixmier-Moeglin equivalence (Theorem~\ref{res: dixmier moeglin, Poisson version}), this restricts to a homeomorphism 
\[Pprim_{\omega}(\GL{3}) \approx max(PZ(B_{\omega})).\]
In Proposition~\ref{res:centres of gl and qgl localizations agree} we have described the Poisson centre $PZ(B_{\omega})$: it is the Laurent polynomial ring $k[z_1^{\pm1}, \dots, z_n^{\pm1}]$ in the generators $z_i$ listed in Figure~\ref{fig:centres} (viewed as elements of $\GL{3}$ rather than $\QGL{3}$).  Since $k$ is algebraically closed, the maximal ideals of $PZ(B_{\omega})$ are precisely those of the form
\[\sum_{i=1}^n(z_i-\lambda_i)PZ(B_{\omega}), \quad (\lambda_1, \dots, \lambda_n) \in (k^{\times})^n.\]
By applying the homeomorphisms in \eqref{eq:reminder of homeomorphisms} we therefore obtain the description of the Poisson-primitive ideals given in \eqref{eq:first description of poisson primitives}.
\end{proof}
In \S\ref{s:pulling back to gens in the ring} we will build on this result to obtain generating sets in $\GL{3}$ (rather than in a localization) for the Poisson-primitive ideals.  First, however, we will turn our attention briefly to $\SL{3}$ and use our existing results to say something about the properties of the quotients $\SL{3}/I_{\omega}$.

\subsection{$\SL{3}/I_{\omega}$ is a UFD for each $\omega$}\label{ss:UFDs}

One consequence of the previous section is that we can use the localizations $\qr{\SL{3}}{I_{\omega}}\big[E_{\omega}^{-1}\big]$ to learn more about the structure of various factor rings of $\SL{3}$.  In particular, we will show that $\SL{3}/I_{\omega}$ is a UFD for each of the 36 Poisson $\HH$-primes $I_{\omega}$.  

Recall from Definition~\ref{def:NC UFD} that a Noetherian UFD is a prime Noetherian ring $A$ such that every height 1 prime ideal is generated by a single \textit{prime element}, i.e. some $p \in A$ such that $pA = Ap$ and $A/pA$ is a domain.  When $A$ is commutative this coincides with the standard definition of UFD \cite[Corollary~2.4]{Chatters1}.

We will make extensive use of the following two results, which are generalizations of Nagata's lemma \cite[Lemma~19.20]{Eisenbud1}.
\begin{lemma}\label{res:nc_nagata} \cite[Lemma~1.4]{quantumUFDs} \ 
Let $A$ be a prime Noetherian ring and $x$ a nonzero, non-unit, normal element of $A$ such that $\langle x \rangle$ is a completely prime ideal of $A$.  Denote by $Ax^{-1}$ the localization of $A$ at powers of $x$.  Then:
\begin{enumerate}[(i)]
\item If $P$ is a prime ideal of $A$ not containing $x$ and such that the prime ideal $PAx^{-1}$ of $Ax^{-1}$ is principal, then $P$ is principal.
\item If $Ax^{-1}$ is a Noetherian UFD, then so is $A$.
\end{enumerate}
\end{lemma}
\begin{proposition}\label{res:nc_iterated_nagata} \cite[Proposition~1.6]{quantumUFDs} \ 
Let $A$ be a prime Noetherian ring and suppose that $d_1, \dots, d_t$ are nonzero normal elements of $A$ such that the ideals $d_1A,\dots, d_tA$ are completely prime and pairwise distinct.  Denote by $T$ the right quotient ring of $A$ with respect to the right denominator set generated by $d_1, \dots, d_t$.  Then if $T$ is a Noetherian UFD, so is $A$.
\end{proposition}
Note that when $A$ is a commutative ring the conditions of Proposition~\ref{res:nc_iterated_nagata} reduce to requiring $A$ to be a Noetherian domain, and from Lemma~\ref{res:nc_nagata} we recover the standard statement of Nagata's lemma.

In \cite[Theorem~5.2]{GBrown}, Brown and Goodearl prove that $\QSL{3}$ is a Noetherian UFD.  However, their proof does \textit{not} generalize directly to the commutative case as it makes use of stratification theory, which cannot be used to understand the commutative ring structure of $\SL{3}$ as it does not ``see'' the non-Poisson prime ideals.

To illustrate this, we begin by proving the following general proposition for quantum algebras; this underpins the proof that $\QSL{3}$ is a Noetherian UFD but is not expanded upon in \cite{GBrown}.
\begin{proposition}\label{res:quantum UFD prop}
Let $A$ be a prime Noetherian ring with a $\HH$-action, which satisfies the conditions of the Stratification Theorem (Theorem~\ref{res:Stratification Theorem, quantum version}), has only finitely many $\HH$-primes and such that all prime ideals are completely prime.  Let $I$ be a $\HH$-prime in $A$.  Then the quotient $A/I$ is a Noetherian UFD if and only if each height 1 $\HH$-prime in $A/I$ is generated by a single normal element. 
\end{proposition}
\begin{proof}
Since prime ideals are completely prime, it suffices to check that every height 1 prime is principal; the condition that each height 1 $H$-primes in $A/I$ is generated by a single normal element is therefore clearly necessary. 

However, we will now show that this condition is also sufficient.  Indeed, assume that all height 1 $\HH$-primes in $A/I$ (of which there are only finitely many) are principally generated, and let $\{u_1, \dots, u_n\}$ be a set of normal generators for them.  We therefore only need to focus on the height 1 primes which do not contain a non-zero $\HH$-prime in $A/I$, i.e. we need to show that all of the height 1 primes in the set
\begin{equation}\label{eq:ht1 primes with trivial H prime}
X = \{P/I: P \textrm{ is prime and }\bigcap_{h \in \HH}h(P/I) = 0\}
\end{equation}
are principally generated. 

By the Stratification Theorem, there are homeomorphisms
\begin{equation}\label{eq:homeos from Stratification}X \approx spec\left(\qr{A}{I}\big[E^{-1}\big]\right) \approx spec\left(Z\big(\qr{A}{I}\big[E^{-1}\big]\big)\right) \\
\end{equation}
where the first homeomorphism is given by localization and contraction, and the second is given by contraction and extension.  Here $E$ is any denominator set of regular $\HH$-eigenvectors in $A/I$ such that the localization $\qr{A}{I}\big[E^{-1}\big]$ is $\HH$-simple.  Further, by the Stratification Theorem $Z\big(\qr{A}{I}\big[E^{-1}\big]\big)$ is always a Laurent polynomial ring in finitely many variables.

Let $E_I$ be the multiplicative set in $A/I$ generated by $\{u_1, \dots, u_n\}$; since the $u_i$ are normal and $A/I$ is Noetherian this is automatically a denominator set, and by definition the localization $\qr{A}{I}\big[E_I^{-1}\big]$ must be $\HH$-simple.  However, it now follows immediately that every height 1 prime in $\qr{A}{I}\big[E_I^{-1}\big]$ must be principally generated, since by \eqref{eq:homeos from Stratification} all primes of $\qr{A}{I}\big[E_I^{-1}\big]$ are centrally generated and the centre is a commutative UFD.

Finally, since $E_I$ was generated by the set $\{u_1, \dots, u_n\}$ which satisfies the conditions of Proposition~\ref{res:nc_iterated_nagata}, we can apply this result to conclude that all height 1 primes in the set $X$ must be principal as well.
\end{proof}
In \cite{GBrown}, Brown and Goodearl prove that each height 1 $\HH$-prime in $\QSL{3}/I_{\omega}$ is generated by a single normal normal element, thus proving that $\QSL{3}/I_{\omega}$ is a Noetherian UFD for each $\HH$-prime $I_{\omega}$ \cite[Theorem~5.2]{GBrown}.  It is now clear from the proof that a Poisson version of Proposition~\ref{res:quantum UFD prop} is not sufficient to verify that $\GL{3}/I$ or $\SL{3}/I$ are commutative UFDs for Poisson $\HH$-primes $I$, since this approach can only tell us about the height 1 \textit{Poisson} primes.  

Instead, for each Poisson $\HH$-prime $I_{\omega}$ of $\SL{3}$, we will show that the generators of the corresponding Ore set $E_{\omega}$ from Figure~\ref{fig:Ore_sets_for_tori} satisfies the conditions of Nagata's lemma (Proposition~\ref{res:nc_iterated_nagata}).  Since the localizations $\qr{\SL{3}}{I_{\omega}}\big[E_{\omega}^{-1}\big]$ are isomorphic to Laurent polynomial rings over $k$ by Proposition~\ref{res:the poisson localizations are scls}, it will then follow that $\SL{3}/I_{\omega}$ must be a UFD for each $\omega$ as well.

\begin{proposition}\label{res:UFD Poisson SL3 quotients}
For any Poisson $\HH$-prime $I_{\omega}$ in $\SL{3}$, the quotient $\SL{3}/I_{\omega}$ is a commutative UFD.
\end{proposition}
\begin{proof}
Let $\omega \in S_3 \times S_3$, and let $E_{\omega}$ denote the multiplicative set of $\HH$-eigenvectors for $\GL{3}$ defined in Figure~\ref{fig:Ore_sets_for_tori}.  Let $\pi$ be the natural map $\GL{3} \rightarrow \SL{3}$, and observe that $\pi(E_{\omega})$ defines a set of $\HP$-eigenvectors in $\SL{3}$.  Using the fact from Proposition~\ref{res:localization and quotient commute} that quotient and localization commute, we see that
\begin{align*}\qr{\SL{3}}{I_{\omega}}\big[\pi(E_{\omega})^{-1}\big] &\cong \Big(\GL{3}/(I_{\omega},Det-1)\Big)\big[\pi(E_{\omega})^{-1}\big]\\
& \cong \Big(\qr{\GL{3}}{I_{\omega}}\big[E_{\omega}^{-1}\big]\Big)/(Det-1)\\
& \cong B_{\omega}/(Det-1)
\end{align*}
Now consider the generators for $B_{\omega}$ given in Figure~\ref{fig:localizations}.  In each case, $Det$ appears as a generator or can be obtained by a change of variables: for example, when $\omega = (321,132)$ the image of $Det$ is $x_{11}[\wt{1}|\wt{1}]$ and both $x_{11}$ and $[\wt{1}|\wt{1}]$ appear as generators, so we may replace either $x_{11}$ or $[\wt{1}|\wt{1}]$ by $Det$ without affecting the structure of $B_{\omega}$.

It is now clear that $B_{\omega}/(Det-1)$ is a Laurent polynomial ring in $n-1$ variables whenever $B_{\omega}$ is a Laurent polynomial ring in $n$ variables.  These are commutative UFDs, and we apply the generalization of Nagata's lemma in Proposition~\ref{res:nc_iterated_nagata}: if we can show that the generators of the Ore set $\pi(E_{\omega})$ each generate distinct prime ideals in $\SL{3}/I_{\omega}$, then $\SL{3}/I_{\omega}$ will itself be a UFD.

Since we are only concerned with the commutative algebra structure of $\SL{3}/I_{\omega}$ we may ignore the Poisson structure and view the Poisson (anti-)isomorphisms in Figure~\ref{fig:H_primes_nice} just as isomorphisms of commutative rings; hence up to isomorphism there are only 12 cases to consider.  

Let $I_{\omega}$ be one of the 12 ideals in Figure~\ref{fig:Ore_sets_for_tori} (which correspond precisely to the 12 isomorphism classes).  First, when $\omega = (321,321)$ or $(123,123)$, we have $\SL{3}/I_{\omega} \cong \SL{3}$ or $k[x_{11}^{\pm1},x_{22}^{\pm1}]$ respectively, and these are clearly both UFDs.  For the remaining 10 cases, let $X_{\omega}$ be the set of generators for the corresponding Ore set from Figure~\ref{fig:Ore_sets_for_tori}; we will show that the elements of $X_{\omega}$ generate pairwise distinct prime ideals.

We first consider the elements appearing in the second column of Figure~\ref{fig:Ore_sets_for_tori}; we will show that they generate distinct $\HP$-primes.  Whenever $\omega \neq (231,231)$, if $g_{\omega}$ appears in column 2 and row $I_{\omega}$ then $I_{\omega} + \langle g_{\omega} \rangle$ is easily seen to be a $\HP$-prime in $\SL{3}$; hence $\langle g_{\omega} \rangle$ is prime in $\SL{3}/I_{\omega}$.  Similarly, any two such ideals are distinct in $\SL{3}$ and hence distinct in $\SL{3}/I_{\omega}$ as well.

For $\omega = (231,231)$, we need to check that $Q_1:=I_{\omega} + \langle [\wt{2}|\wt{1}] \rangle$ and $Q_2:= I_{\omega} + \langle [\wt{3}|\wt{2}] \rangle$ are genuine $\HP$-primes in $\SL{3}$.  We note that
\[Q_1 = \langle [\wt{3}|\wt{1}],[\wt{2}|\wt{1}],x_{13}\rangle = S(P),\]
where $P:= \langle x_{12},x_{13},[\wt{1}|\wt{3}]\rangle$ is a $\HP$-prime, and $S$ is as always the antipode map.  As observed in \S\ref{s:H primes background,defs}, we have $\rho^2 = S^2 = id$ since $\SL{3}$ is commutative; using Figure~\ref{fig:H_primes_nice}, it is now clear that $S(P) = \rho^{-2} \circ S^{-1}(P) = I_{231,132}$ and hence $Q_1 = I_{231,132}$.  By a similar argument, $Q_2 = I_{231,213}$.

Thus in each of the 10 cases of interest to us, the elements in the second column of Figure~\ref{fig:Ore_sets_for_tori} generate distinct height 1 $\HP$-primes in $\SL{3}/I_{\omega}$.  We now consider the elements in the third column of Figure~\ref{fig:Ore_sets_for_tori} for each case; these split into four broad groups, which we treat separately.

\textbf{Case I}: $\omega$ = (123,132), (213,132) or (123,312). 

In each of these cases third column of Figure~\ref{fig:Ore_sets_for_tori} is empty and there is nothing to check.

\textbf{Case II}: $\omega$ = (132,132), (132,312), (321,123), (321,132) or (231,312).

The relevant information from Figure~\ref{fig:Ore_sets_for_tori} can be summarised as follows:
\begin{center}\begin{tabular}{c|ccccc}
\multirow{2}{*}{$\omega$}
&
\begin{smallarray}{m132132}
{
 \circ & \bullet & \bullet \\
 \bullet & \circ & \circ \\
 \bullet & \circ & \circ \\
};
\end{smallarray}
&
\begin{smallarray}{m132312}
{
 \circ & \circ & \bullet \\
 \bullet & \circ & \circ \\
 \bullet & \circ & \circ \\
};
\end{smallarray}
&
\begin{smallarray}{m321123}
{
 \circ & \bullet & \bullet \\
\circ & \circ & \bullet \\
\circ & \circ & \circ \\
};
\end{smallarray}
&
\begin{smallarray}{m321132}
{
 \circ & \bullet & \bullet \\
\circ & \circ & \circ \\
\circ & \circ & \circ \\
};
\end{smallarray}
&
\begin{smallarray}{m231312}
{
 \circ & \circ & \bullet \\
\circ & \circ & \circ \\
\bullet & \circ & \circ \\
};
\end{smallarray}
\\
 & $(132,132)$ & $(132,312)$ & $(321,123)$ & $(321,132)$ & $(231,312)$ \\ \hline
Elements to check & $x_{33}$ & $x_{33}$ & $x_{21}$ & $x_{32}$, $x_{33}$ & $x_{33}$, $[\wt{1}|\wt{1}]$
\end{tabular}\end{center} 

Consider first $\omega$ = (132,132).  We need to show that $x_{33}$ is prime in $\SL{3}/I_{132,132}$; this is equivalent to checking that $B := \SL{3}/(I_{132,132}, x_{33})$ is a domain.  Observe that $Det = x_{11}x_{23}x_{32} = 1$ in $B$, and so
\[B \cong k[x_{11}^{\pm1},x_{22},x_{23}^{\pm1}]\]
is easily seen to be a domain.

The other cases proceed similarly: in each case, we observe that $Det$ becomes a monomial modulo the element we would like to check is prime, so the quotient is simply a localization of a polynomial ring.  The only case requiring some care is checking that $[\wt{1}|\wt{1}]$ is prime in (231,312): here $Det = -x_{12}x_{21}x_{33} = 1$  and so
\[B:=\SL{3}/(I_{231,312},[\wt{1}|\wt{1}]) \cong k[x_{11},x_{21}^{\pm1}, x_{22},x_{23}, x_{32},x_{33}^{\pm1}]/(x_{22}x_{33}-x_{23}x_{32}).\]
However, since $x_{33}$ is invertible we can observe that $x_{22} = x_{33}^{-1}x_{23}x_{32}$ and hence
\[B \cong k[x_{11},x_{12}^{\pm1},x_{21}^{\pm1},x_{23}, x_{32},x_{33}^{\pm1}]\]
is a domain.

\textbf{Case III}: $\omega$ = (231,231).

\begin{center}\begin{tabular}{c}
\begin{smallarray}{m231231}
{
 \circ &  &  \\
\circ &  &  \\
\bullet & \circ & \circ \\
};
\MyZ(1,2)
\end{smallarray}
\\
$(231,231)$
\end{tabular}\end{center} 
We need to verify that $x_{33}$ and $[\wt{1}|\wt{1}]$ are prime in $\SL{3}/I_{\omega}$, i.e. that
\[B_1:=\SL{3}/(I_{231,231},x_{33}), \quad B_2 := \SL{3}/(I_{231,231},[\wt{1}|\wt{1}])\]
are both domains.  For $B_1$, we observe that the image of the determinant $Det$ in this quotient ring is $-x_{32}[\wt{3}|\wt{2}]$; this will allow us to identify a subalgebra of $B_1$ with $\GL{2}$ as follows.  Write $\GL{2} = k[a,b,c,d][(ad-bc)^{-1}]$, then identify $\{x_{11},x_{13},x_{21},x_{23}\}$ in $B_1$ with $\{a,b,c,d\}$ in $\GL{2}$, and $-x_{32}$ with $(ad-bc)^{-1}$.

Under this identification we obtain an isomorphism of commutative algebras
\[B_1 \cong \GL{2}[x_{12},x_{22}]/(x_{11}d-bx_{22}).\]
As a polynomial extension of a UFD, $\GL{2}[x_{12},x_{22}]$ is a UFD itself, so we can apply Eisenstein's criterion to see that $(x_{11}d-bx_{22})$ generates a prime ideal in $\GL{2}[x_{12},x_{22}]$.  Hence $B_1$ is a domain as required.

Similarly, we obtain the isomorphism
\[B_2 \cong \GL{2}[x_{11},x_{22},x_{23}]/(ax_{23}-bx_{22}, x_{22}d - x_{23}c).\]
In order to simplify this quotient, we can observe that 
\begin{equation}\label{eq:simplifying an ideal in ufd proof}\begin{aligned}
d(ax_{23} - bx_{22}) + b(x_{22}d-x_{23}c) &= (ad-bc)x_{23}, \\
c(ax_{23} - bx_{22}) + a(x_{22}d-x_{23}c) &= (ad-bc)x_{22}.
\end{aligned}\end{equation}
Since $ad-bc$ is invertible, we see from \eqref{eq:simplifying an ideal in ufd proof} that $(ax_{23}-bx_{22},x_{22}d-x_{23}c)$ is nothing but the ideal $(x_{22},x_{23})$.  It is now clear that $B_2 \cong \GL{2}[x_{11}]$ is a domain, as required.

\textbf{Case IV}: $\omega$ = (321,312). 

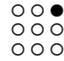
\begin{center}\begin{tabular}{c}
\begin{smallarray}{m321312}
{
 \circ & \circ & \bullet \\
\circ & \circ & \circ \\
\circ & \circ & \circ \\
};
\end{smallarray}
\\
$(321,312)$
\end{tabular}\end{center} 

There are three elements that we need to check are prime: $x_{11}$, $x_{21}$ and $[\wt{3}|\wt{3}]$.  By a similar analysis to Case III, we obtain
\begin{align*}
\SL{3}/(I_{321,312},x_{11}) &\cong \GL{2}[x_{22},x_{32}], \\
\SL{3}/(I_{321,312},[\wt{3}|\wt{3}]) &\cong \GL{2}[x_{21},x_{22},x_{33}]/(ax_{22}-bx_{21}),
\end{align*}
which we have already observed are domains.  Finally, we consider
\begin{align*}B&:=\SL{3}/(I_{321,312},x_{21}) \\
&\cong \ML{3}/(Det-1,x_{13},x_{21})\\
& \cong k[x_{11},x_{12},x_{22},x_{23},x_{31},x_{32},x_{33}]/(x_{11}[\wt{1}|\wt{1}]  + x_{12}x_{23}x_{31} - 1).\end{align*}
We are factoring by an element which is linear as a polynomial in $x_{11}$, with coefficients in $k[x_{12},x_{22},x_{23},x_{31},x_{32},x_{33}]$.  It will therefore be irreducible (and hence prime since we are working in a commutative polynomial ring) if $[\wt{1}|\wt{1}] = x_{22}x_{33} - x_{23}x_{32}$ and $x_{12}x_{23}x_{31}-1$ have no non-invertible factors in common, but this is immediately clear.  $B$ is therefore a domain, as required.

This covers all of the 10 isomorphism classes of algebras $\SL{3}/I_{\omega}$ under consideration: in each case, the set $X_{\omega}$ satisfies the conditions of Proposition~\ref{res:nc_iterated_nagata} as required.  We therefore conclude that $\SL{3}/I_{\omega}$ is a UFD for any $\omega \in S_3 \times S_3$.
\end{proof}
It is quite easy to prove a similar result for $\GL{3}/I_{\omega}$, although we will not do so here.  While it has so far been easier to work with $\GL{3}$ rather than $\SL{3}$, the advantages of $\SL{3}$ would become clear if we moved on to consider the Poisson-prime ideals rather than the Poisson-primitives: since the Poisson centres $PZ\left(\qr{\SL{3}}{I_{\omega}}\big[E_{\omega}^{-1}\big]\right)$ are Laurent polynomial rings on at most \textit{two} variables (rather than three variables for $GL_3$) the non-maximal prime ideals will have height $\leq 1$ and can therefore be understood (to some extent at least).

Our first aim, however, is to understand the Poisson-primitive ideals in terms of generators within $\GL{3}$ or $\SL{3}$ themselves rather than a localization of these rings.  This is the focus of the next section.

\section{Pulling back to generators in the ring}\label{s:pulling back to gens in the ring}

While Theorem~\ref{res:thm, desc of poisson primitives up to localization, gl3} gives a full description of the Poisson-primitive ideals of $\GL{3}$, it only tells us the generators of each ideal up to localization.  The aim of this section is to obtain a description for the Poisson-primitive ideals in terms of generating sets in $\GL{3}$ itself, in a similar manner to the quantum case covered in \cite{GL1}.

For $\omega \in S_3\times S_3$ write $PZ(B_{\omega}) = k[z_1^{\pm1},\dots, z_n^{\pm1}]$, where the $z_i$ have the form listed in Figure~\ref{fig:centres}.  Each $z_i$ is written in the form $e_if_i^{-1}$, where $e_i$ and $f_i$ are both elements of the Ore set $E_{\omega}$.  Since $k$ is algebraically closed, the maximal ideals of $PZ(B_{\omega})$ are precisely those of the form
\[M_{\lambda} = \langle z_1 - \lambda_1, \dots, z_n - \lambda_n \rangle, \quad \lambda = (\lambda_1, \dots, \lambda_n) \in (k^{\times})^n.\]
and so by Proposition~\ref{res:poisson strat generalization 2} the Poisson primitive ideals of $B_{\omega}$ are simply the extensions of these ideals to $B_{\omega}$, i.e.
\[P_{\lambda} := M_{\lambda}B_{\omega}, \quad \lambda = (\lambda_1, \dots, \lambda_n) \in (k^{\times})^n.\]
By Proposition~\ref{res:poisson strat generalization 1} and Theorem~\ref{res: dixmier moeglin, Poisson version}, these in turn correspond homeomorphically to the Poisson primitive ideals
\[Pprim_{\omega}(\GL{3}) = \{P/I_{\omega} : P \in Pprim(\GL{3}), P \supseteq I_{\omega}, \bigcap_{h \in \HH}h(P/I_{\omega}) = 0\}\]
Our aim is therefore to find generators for the Poisson primitive ideals $P_{\lambda} \cap \GL{3}/I_{\omega}$ for $\lambda \in (k^{\times})^n$ and for each $\omega \in S_3\times S_3$.  Similarly to \cite{GL1}, we will show that
\begin{equation}\label{eq:prediction of generators of primitives}P_{\lambda} \cap \qr{\GL{3}}{I_{\omega}} = \langle e_1 - \lambda_1 f_1, \dots, e_n - \lambda_n f_n\rangle.\end{equation}
The first step, quite naturally, is to check that the ideals on the RHS of \eqref{eq:prediction of generators of primitives} are closed under the Poisson bracket in $\GL{3}/I_{\omega}$.
\begin{lemma}\label{res:predicted primitives are Poisson ideals}
Let $PZ(B_{\omega}) = k[(e_1f_1^{-1})^{\pm1}, \dots, (e_nf_n^{-1})^{\pm1}]$, for $\omega \in S_3 \times S_3$ and the choices of $e_if_i^{-1}$ given in Figure~\ref{fig:centres}.  Then for any $\lambda = (\lambda_1, \dots, \lambda_n) \in (k^{\times})^n$, the ideal
\begin{equation}\label{eq:def of ideal, pullback of primitive}Q_{\lambda}:=\langle e_1 - \lambda_1 f_1, \dots, e_n - \lambda_n f_n\rangle\end{equation}
is a Poisson ideal in $\GL{3}/I_{\omega}$.
\end{lemma}
\begin{proof}
We first observe that each $f_i$ is Poisson normal in $\GL{3}/I_{\omega}$ (see Definition~\ref{def:Poisson central, normal} for the definition of Poisson normal).  This is clear when $f_i = x_{13}$ or $x_{31}$, since these both generate Poisson primes in $\GL{3}$.  By direct computation we see that $x_{21}$ and $x_{32}$ are Poisson normal modulo $x_{31}$, while $x_{12}$ and $x_{23}$ are Poisson normal modulo $x_{13}$.  This covers all the denominators appearing in Figure~\ref{fig:centres}, and since Poisson isomorphisms and anti-isomorphisms must map Poisson normal elements to Poisson normal elements, the same conclusion follows for the other 24 cases.

Let $e - \lambda f$ be one of the generators appearing in \eqref{eq:def of ideal, pullback of primitive}; we will not need to work with more than one generator at once so we may dispense with the subscripts.  Since $ef^{-1}$ is Poisson central in $B_{\omega} = \qr{\GL{3}}{I_{\omega}}\big[E_{\omega}^{-1}\big]$, for all $a \in \GL{3}/I_{\omega}$ we have
\[0= \{ef^{-1},a\} = \{e,a\}f^{-1} - \{f,a\}ef^{-2},\]
and hence
\[\{e,a\}f = \{f,a\}e.\]
Combining this with the fact that $f$ is Poisson normal in $\GL{3}/I_{\omega}$, we see that
\begin{align*}
\{e-\lambda f, a\}f &= \{e,a\}f - \lambda\{f,a\}f\\
&= \{f,a\}e - \lambda \{f,a\}f \\
&= (e-\lambda f)\{f,a\} \\
&= (e-\lambda f)r_af
\end{align*}
for some element $r_a \in \GL{3}/I_{\omega}$ which depends on $a$.  Since $\GL{3}/I_{\omega}$ is a domain, we thus obtain $\{e-\lambda f,a\} = (e-\lambda f)r_a$ for any $a \in \GL{3}/I_{\omega}$.  It is now clear that the ideal in \eqref{eq:def of ideal, pullback of primitive} is always a Poisson ideal in $\GL{3}/I_{\omega}$.
\end{proof}

\begin{proposition}\label{res:checking the generators of primitives in GL3}
Let $\omega \in S_3 \times S_3$ and write $PZ(B_{\omega}) \cong k[(e_1f_1^{-1})^{\pm1}, \dots, (e_nf_n^{-1})^{\pm1}]$, where the values of $e_if_i^{-1}$ are from Figure~\ref{fig:centres} as before.  Then $P_{\lambda} \cap \qr{\GL{3}}{I_{\omega}} = Q_{\lambda}$, where $Q_{\lambda}$ is defined as in \eqref{eq:def of ideal, pullback of primitive}.
\end{proposition}
\begin{proof}
Note that when we extend the ideal $Q_{\lambda}$ to $B_{\omega}$ by localization, we get $Q_{\lambda}B_{\omega} = P_{\lambda}$.  Our aim is therefore to prove that
\[Q_{\lambda}B_{\omega} \cap \qr{\GL{3}}{I_{\omega}} = Q_{\lambda},\]
which is equivalent by \cite[Theorem~10.18]{GW1} to verifying that the elements of $E_{\omega} \subset \GL{3}/I_{\omega}$ are regular modulo $Q_{\lambda}$.  Since $\GL{3}$ is commutative, it suffices to show that $Q_{\lambda}$ is a prime ideal in $\GL{3}/I_{\omega}$.

We will consider the 12 cases listed in Figure~\ref{fig:centres}; the other 24 follow will then follow by the isomorphisms and anti-isomorphisms $\tau$, $\rho$ and $S$.

We deal first with the cases where $Q_{\lambda} = \langle Det - \lambda_1\rangle$, that is the four cases
\[\omega = (231,312), \ (321,132), \ (123,312), \ (213,132),\]
Note that $Q_{\lambda}$ is a prime ideal if and only if $h(Q_{\lambda})$ is, where $h = (\lambda_1,1,1,1,1,1) \in \HH$.  We have $h(Q_{\lambda}) = \langle Det-1\rangle$, and $\GL{3}/(I_{\omega},h(Q_{\lambda})) \cong \SL{3}/I_{\omega}$; since (the image of) the $\HH$-prime $I_{\omega}$ is prime in $\SL{3}$ by Proposition~\ref{res:bijection of Hprimes}, this is a domain and we are done.

We next consider the four cases
\begin{center}\begin{tabular}{cccc}
\begin{smallarray}{m321123}
{
 \circ & \bullet & \bullet \\
\circ & \circ & \bullet \\
\circ & \circ & \circ \\
};
\end{smallarray}
&
\begin{smallarray}{m132132}
{
 \circ & \bullet & \bullet \\
 \bullet & \circ & \circ \\
 \bullet & \circ & \circ \\
};
\end{smallarray}
&
\begin{smallarray}{m123132}
{
 \circ & \bullet & \bullet \\
 \bullet & \circ & \circ \\
 \bullet & \bullet & \circ \\
};
\end{smallarray}
&
\begin{smallarray}{m123123}
{
 \circ & \bullet & \bullet \\
 \bullet & \circ & \bullet \\
 \bullet & \bullet & \circ \\
};
\end{smallarray}
\\
$(321,123)$ & $(132,132)$ & $(123,132)$ & $(123,123)$
\end{tabular}\end{center}
where in each case we can describe the structure of $\GL{3}/(I_{\omega},Q_{\lambda})$ explicitly and verify that it is a domain.

(321,123): We have $Q_{\lambda} = \langle Det - \lambda_1, x_{22}[\wt{1}|\wt{3}] - \lambda_2x_{31} \rangle$, and we may assume that $\lambda_1, \lambda_2 = 1$ by applying the automorphism $h = (1,\sqrt{\lambda_2},\lambda_1\sqrt{\lambda_2}^{-1},1,1,1)$ to $Q_{\lambda}$.  The image of $Det$ modulo $I_{321,123}$ is $x_{11}x_{22}x_{33}$, so we have 
\begin{equation}\label{eq:random ring that needs to be a domain}\begin{aligned}\GL{3}/(I_{321,123},h(Q_{\lambda})) &\cong \ML{3}/(I_{321,123},x_{11}x_{22}x_{33}-1, x_{22}[\wt{1}|\wt{3}] - x_{31}) \\
& \cong k[x_{11}^{\pm1},x_{21},x_{22}^{\pm1},x_{31},x_{32}]/(x_{22}[\wt{1}|\wt{3}] - x_{31}).\end{aligned}\end{equation}
We can write the generator of the quotient ideal as
\begin{equation}\label{eq:random gen of ideal somewhere}x_{22}[\wt{1}|\wt{3}] - x_{31} = -x_{31}(x_{22}^2 + 1) + x_{22}x_{21}x_{32},\end{equation}
which is linear as a polynomial in $x_{31}$ over the commutative UFD $k[x_{11}^{\pm1},x_{21},x_{22}^{\pm1},x_{32}]$.  Since it is clear that the coefficients $(x_{22}^2 + 1) = (x_{22} + i)(x_{22} - i)$ and $x_{22}x_{21}x_{32}$ have no non-invertible factors in common, \eqref{eq:random gen of ideal somewhere} is irreducible and hence prime.  The ring \eqref{eq:random ring that needs to be a domain} is therefore a domain, as required.

(132,132): In this case $Q_{\lambda} = \langle Det - \lambda_1, x_{11} - \lambda_2, x_{23}-\lambda_3x_{32}\rangle$, which we replace by $h(Q_{\lambda}) = \langle Det - 1, x_{11} - 1, x_{23} - \lambda_3 x_{32} \rangle$ under the action of $h = (\lambda_2,1,\lambda_1\lambda_2^{-1},1,1,1)$.  We also observe that $\SL{3}/I_{132,132} \cong \GL{2}$ as commutative algebras, where we identify $(x_{22},x_{23},x_{32},x_{33})$ with $(a,b,c,d)$ and $x_{11}$  with $(ad-bc)^{-1}$.  Now
\begin{align*}\GL{3}/(I_{132,132},h(Q_{\lambda})) & \cong \SL{3}/(I_{132,132},x_{11} - 1, x_{23} - \lambda_3 x_{32})\\
&\cong \GL{2}/((ad-bc)^{-1} - 1, b - \lambda_3 c) \\
& \cong \SL{2}/(b-\lambda_3 c),
\end{align*}
and this is a domain since $b-\mu c$ is a prime ideal in $\SL{2}$ for all $\mu \in k^{\times}$.

(123,132): We have $Q_{\lambda} = \langle Det-\lambda_1, x_{11} - \lambda_2\rangle$, and as always we may assume without loss of generality that $\lambda_1 = 1$.  Since the image of $Det$ in $\GL{3}/I_{123,132}$ is $x_{11}x_{22}x_{33}$, it is now clear that
\begin{align*}
\GL{3}/(I_{123,132},Q_{\lambda}) \cong \SL{3}/(I_{123,132}, x_{11} - \lambda_2) \cong k[x_{11}^{\pm1},x_{22}^{\pm1},x_{23}]/(x_{11}- \lambda_2)
\end{align*}
is a domain.

(123,123): With $Q_{\lambda} = \langle Det-\lambda_1, x_{11} - \lambda_2, x_{22} - \lambda_3 \rangle$, we obtain
\[\GL{3}/(I_{123,123},Q_{\lambda}) \cong k[x_{11}^{\pm1},x_{22}^{\pm1}]/(x_{11} - \lambda_2, x_{22} - \lambda_3)\]
in the same manner as the previous case.

Four cases remain:
\begin{center}\begin{tabular}{cccc}
\begin{smallarray}{m321321}
{
 \circ & \circ & \circ \\
\circ & \circ & \circ \\
\circ & \circ & \circ \\
};
\end{smallarray}
&
\begin{smallarray}{m321312}
{
 \circ & \circ & \bullet \\
\circ & \circ & \circ \\
\circ & \circ & \circ \\
};
\end{smallarray}
&
\begin{smallarray}{m231231}
{
 \circ &  &  \\
\circ &  &  \\
\bullet & \circ & \circ \\
};
\MyZ(1,2)
\end{smallarray}
&
\begin{smallarray}{m132312}
{
 \circ & \circ & \bullet \\
 \bullet & \circ & \circ \\
 \bullet & \circ & \circ \\
};
\end{smallarray}
\\
$(321,321)$ & $(321,312)$ & $(231,231)$ & $(132,312)$
\end{tabular}\end{center}
In \cite{GL1} these cases are dealt with by showing that the generators of the Ore set $E_{\omega}$ are regular modulo $Q_{\lambda}'$, where they make clever use of the maps $\tau$, $\rho$ and $S$ to verify that certain elements are not in various ideals, and repeatedly apply the observation that if an element $x$ is regular modulo $\langle y \rangle$ then $y$ is regular modulo $\langle x \rangle$.  This proof is almost entirely independent of the specific quantum algebra setting and can be applied to $\GL{3}$ and $\SL{3}$ with minimal modification.  

Verifying this is long and tedious and not especially illuminating, however, so we will simply observe that in our commutative setting we can check that the ideals $Q_{\lambda}'$ are indeed prime using the Magma computer algebra system.  The relevant code is reproduced in Appendix~\ref{s:magma_gl3}.

Hence $Q_{\lambda}$ is prime in each of the 12 cases from Figure~\ref{fig:centres}, and the remaining 24 cases are handled by the (anti-)isomorphisms displayed in Figure~\ref{fig:H_primes_nice}.
\end{proof}
Proposition~\ref{res:checking the generators of primitives in GL3} gives us explicit sets of generators for the ideals $P_{\lambda} \cap \GL{3}/I_{\omega}$, which by the Poisson Stratification Theorem correspond precisely to the Poisson-primitive ideals in the stratum $Pprim_{\omega}(\GL{3})$.  Since we have chosen our elements carefully to ensure that the first generator is always $Det - \lambda_1$, we also obtain the corresponding description of Poisson-primitive ideals in $\SL{3}$.  This is summarised in the following theorem.
\begin{theorem}\label{res:Poisson primitives of gl3 and sl3}
Let $\omega \in S_3 \times S_3$, and let $PZ(B_{\omega}) = k[(e_1f_1^{-1})^{\pm1}, \dots, (e_nf_n^{-1})^{\pm1}]$ as in Figure~\ref{fig:centres}.  Then
\begin{enumerate}[(i)]
\item The Poisson-primitive ideals in the stratum $Pprim_{\omega}(\GL{3})$ corresponding to the $\HH$-prime $I_{\omega}$ have the form
\begin{equation}\label{eq:gens of primitive in gl3, corollary statement}I_{\omega} + \langle e_1 - \lambda_1 f_1, \dots, e_n - \lambda_n f_n \rangle, \quad (\lambda_1, \dots, \lambda_n) \in (k^{\times})^n,\end{equation}
and these are all the Poisson-primitive ideals in this stratum.
\item The Poisson-primitive ideals in the stratum $Pprim_{\omega}(\SL{3})$ corresponding to the $\HP$-prime $I_{\omega}$ (now viewed as an ideal in $\SL{3}$) are precisely those of the form
\[I_{\omega} + \langle e_2 - \lambda_2 f_2, \dots, e_n - \lambda_n f_n\rangle, \quad (\lambda_2, \dots, \lambda_n) \in (k^{\times})^{n-1}.\] 
\end{enumerate}
\end{theorem}
\begin{proof}
Combine Proposition~\ref{res:checking the generators of primitives in GL3} and Theorem~\ref{res:thm, desc of poisson primitives up to localization, gl3}.
\end{proof}
Comparing this result to the quantum case in \cite{GL1}, the following corollary is now immediate:
\begin{corollary}\label{res:bijection between primitives,H-prime section}
Let $k$ be algebraically closed and $q \in k^{\times}$ not a root of unity.  Let $A$ denote $\QGL{3}$ or $\QSL{3}$, and let $B$ denote the semi-classical limit of $A$.  Then there is a bijection of sets between $prim(A)$ and $Pprim(B)$, which is induced by the ``preservation of notation map'' 
\[A \rightarrow B: \ X_{ij} \mapsto x_{ij}, \ [\wt{i}|\wt{j}]_q \mapsto [\wt{i}|\wt{j}].\]
\end{corollary}
\begin{proof}
\cite[Theorem~5.5]{GL1}, Theorem~\ref{res:Poisson primitives of gl3 and sl3}.
\end{proof}
Corollary~\ref{res:bijection between primitives,H-prime section} strongly suggests that Conjecture~\ref{conj:goodearl} should be true for $\QGL{3}$ and $\QSL{3}$.  The remaining step would be to prove that the bijection in Corollary~\ref{res:bijection between primitives,H-prime section} in fact defines a homeomorphism (with respect to the Zariski topology) between $prim(A)$ and $Pprim(B)$.

This could be accomplished in two ways: one is to simply prove directly that the bijection on primitives is a homeomorphism.  It is not at all clear how to go about doing this, however, so an alternative approach would be to first extend the bijection in Corollary~\ref{res:bijection between primitives,H-prime section} to a bijection $spec(A) \rightarrow Pspec(B)$; by \cite[Lemma~9.4]{GoodearlSummary} it would then suffice to check that this bijection and its inverse preserved inclusions of primes.  

Unfortunately the intermediate step of extending the bijection from primitives to primes would be a necessary part of this approach, since the statement of \cite[Lemma~9.4]{GoodearlSummary} is no longer true if we replace ``primes'' by ``primitives''.  An elementary illustration of this has been observed by Goodearl: if we take $R$ to be a commutative noetherian $k$-algebra with trivial Poisson bracket, then we have $prim(R) = Pprim(R) = max(R)$ and any bijection $prim(R) \rightarrow Pprim(R)$ will (trivially) preserve inclusions.  It is clear that not all such bijections will be homeomorphisms, however.

Instead, the first step towards proving Conjecture~\ref{conj:goodearl} for $SL_3$ would be to obtain generators for the prime ideals of $\QSL{3}$ (respectively Poisson-prime generators of $\SL{3}$).  Assuming (as seems quite likely) that the bijection of Corollary~\ref{res:bijection between primitives,H-prime section} extends to a bijection on primes, we would then need to check that this map and its inverse both preserve inclusions among primes -- not a simple task to approach directly, with no easy way to tell if a prime from one stratum is contained within a prime from another stratum and 36 distinct strata to consider!  We focus first on $SL_3$ here since the (Poisson-)primes within a given strata will always have height at most 2; hence most of the (Poisson-)prime ideals are already described in Theorem~\ref{res:Poisson primitives of gl3 and sl3} and those that remain are known to be principally generated by Proposition~\ref{res:UFD Poisson SL3 quotients} (resp. \cite[Theorem~5.2]{GBrown}).

In future work we hope to develop a Poisson version of the results of \cite{GBrown}, which would allow us to ``patch together'' the topologies of each stratum $Pspec_{\omega}(\SL{3})$ (which are well-understood) into a picture of the Zariski topology on the whole space $Pspec(\SL{3})$.  We then hope to use these results to tackle the question of whether this (as yet only conjectured) bijection $spec(\QSL{3}) \rightarrow Pspec(\SL{3})$ preserves inclusions, although this approach will still involve significant amounts of computation.  This approach will not generalise easily even to other fairly low-dimensional examples such as $\QML{3}$ (230 $\HH$-primes) or $\QML{4}$ (6902 $\HH$-primes), and new techniques will clearly be required to tackle the general case.

\appendix

\chapter{Computations in Magma}\label{c:appendix}
As illustrated quite neatly by Chapter~\ref{c:fixed_rings_chapter}, computation with non-commutative fractions is often difficult and messy.  This appendix details the techniques used to make some of these computations feasible on a computer: we describe both the theory that makes it possible and the code written for the computer algebra system Magma by the author to implement these techniques.  We also provide an example to illustrate some of the limitations of this approach.

As in Chapter~\ref{c:fixed_rings_chapter}, our main tool is the embedding of the $q$-division ring $D$ into the larger division ring of Laurent power series: recall that this is the ring of Laurent power series of the form
\[k_q(y)(\!(x)\!) = \left\{\sum_{i \geq n} a_i(y)x^i : n \in \mathbb{Z}, a_i(y) \in k(y)\right\},\]
subject to the relation $xy = qyx$.  We will continue to assume that $q$ is not a root of unity.

In this larger ring, operations such as the multiplication of two elements or finding the inverse of an element can be performed term by term on the coefficients.  In particular, we can compute that the product of two elements is
\begin{equation}\label{eq:multiplication coefficients}
\sum_{i \geq n}a_i x^i \sum_{j \geq m}b_j x^j = \sum_{k \geq 0}c_k x^{m+n+k} \textrm{ where }c_k = \sum_{r =0}^k a_{n+r}\alpha^{n+r}(b_{m+k-r}).
\end{equation}
Similarly, we find that for an element of the form $1 + \sum_{i \geq 1} b_i x^i$, the inverse is
\begin{equation}\label{eq:inverse coefficients}\left(1 + \sum_{i \geq 1} b_i x^i\right)^{-1} = 1 + \sum_{i \geq 0} c_i x^i, \textrm{ where } c_i = -\left(\sum_{j=1}^i b_j \alpha^{j}(c_{i-j})\right).\end{equation}
Using this, we can find the inverse of a general element $\sum_{i \geq n}a_i x^i$ by writing it in the form $a_nx^n \sum_{i \geq 0}\alpha^{-n}(a_{i+n}/a_n)x^i$; the resulting sum is now in the correct form to apply \eqref{eq:inverse coefficients}.

These computations can all be done by hand, of course, but since evaluating the coefficients at each step involves only commutative terms it is now a much simpler matter to delegate this process to a computer.  We can view elements of $k_q(y)(\!(x)\!)$ as infinite sequences of commutative terms, with addition defined pointwise and multiplication defined term-by-term by \eqref{eq:multiplication coefficients}; phrased in this manner, it is now possible to write the Magma functions which simulate the ring structure of $k_q(y)(\!(x)\!)$ without explicit reference to its non-commutativity.

In deference to the computer's dislike of infinite things we are unfortunately required to work with \textit{truncated} sequences, but for many applications this turns out to be sufficient: to eliminate a pair of potential $q$-commuting elements $f$ and $g$, for example, we need only compute the first few terms of the expression $fg-qgf$ and see whether the result is non-zero.  However, while it is easy to convert fractions to power series using \eqref{eq:inverse coefficients}, we would like to be able to pull our computations back to fractions at the end as well (where possible).  The results of the next section prove that this is indeed possible under certain circumstances.

\section{The theory behind $q$-commuting computations}\label{s:magma_theory}
One of the more useful things we can do with the Magma environment described above is to input an element of the form
\[g = \lambda y + \sum_{i \geq 1} g_i x^i, \quad \lambda \in k^{\times}, a_i \in k(y)\]
and construct an element $f \in k_q(y)(\!(x)\!)$ such that $fg = qgf$.  The catch is that even if $g \in k_q(x,y)$ there is no guarantee that the constructed element $f$ will also represent a fraction.  

The following two theorems aim to tackle this problem by describing under what conditions a power series $f$ will be the image of a fraction from $k_q(x,y)$.  They have the added benefit of being constructive, i.e. if $f \in k_q(x,y)$ they return elements $u, v \in k_q[x,y]$ such that $f = v^{-1}u$.  These results are classical in the commutative case and generalize easily to the case of an Ore extension by an automorphism, but since they don't seem to appear in the literature in this non-commutative form we provide the full proof here.

Let $K$ be a field, $K[x;\alpha]$ the Ore extension by an automorphism $\alpha$, and $K(x;\alpha)$, $K[[x;\alpha]]$ the ring of fractions and ring of power series respectively.  For the specific case of the $q$-division ring, we can take $K = k(y)$ and $\alpha: y \mapsto qy$.

\begin{theorem}\label{res:rec reln thm}The power series $\sum_{i \geq 0} a_ix^i \in K[\![x; \alpha]\!]$ represents a rational function $Q^{-1}P$ in $K(x;\alpha)$ if and only if there exists some integer $n$, and some constants $c_1, \dots, c_n \in K$ (of which some could be zero) such that for all $i \geq 0$ the coefficients of the power series satisfy the linear recurrence relation
\begin{equation}\label{eq:def of rec reln}a_{i+n} = c_1\alpha(a_{i+(n-1)}) + c_2\alpha^2(a_{i+(n-2)}) + \dots + c_n\alpha^n(a_i).\end{equation}
If this is the case, then $P$ is a polynomial of degree $\leq n-1$ which is constructed explicitly in the proof, and $Q = 1 - \sum_{i=1}^n c_ix^i$.\end{theorem}
The exposition of this proof follows closely the one from \cite{notes}.  We first require a technical lemma:
\begin{lemma}\label{res:rec reln technical lemma}Let $c_1, \dots, c_n$ be a set of elements of $K$; define a polynomial $c_1x + c_2x^2 + \dots + c_nx^n$, and let $\sum_{i \geq 0} a_i x^i$ be a power series in $K[\![x]\!]$.  Then
\begin{equation}\label{eq:rec reln technical lemma eqn}(c_1x + \dots + c_nx^n)\sum_{i \geq 0}a_ix^i = R + \sum_{i \geq 0} (c_1 \alpha(a_{i+n-1}) + \dots + c_n \alpha^n(a_i))x^{i+n},\end{equation}
where $R$ is a polynomial of degree at most $n-1$.\end{lemma}
\begin{proof}
We start by multiplying out the left hand side of \eqref{eq:rec reln technical lemma eqn} as follows:
\begin{align*}
 (c_1x + \dots + c_nx^n)\sum_{i \geq 0}a_ix^i &= c_1x\left(\sum_{i=0}^{n-2}a_i x^i + \sum_{i \geq n-1}a_ix^i\right) \\
&{} \qquad + c_2x^2\left(\sum_{i=0}^{n-3}a_ix^i + \sum_{i \geq n-2}a_i x^i\right) \\
&{} \qquad \quad \dots  \\
&{} \qquad + c_{n-1}x^{n-1}\left(\sum_{i=0}^0 a_ix^i + \sum_{i \geq 1} a_ix^i\right) \\
&{} \qquad + c_nx^n\left(0 + \sum_{i \geq0} a_i x^i\right)
\end{align*}
After moving all powers of $x$ to the right and re-indexing the second sum on each line so that it starts from $i=0$, we obtain
\begin{align*}
 (c_1x + \dots + c_nx^n)\sum_{i \geq 0}a_ix^i &= c_1\sum_{i=0}^{n-2}\alpha(a_i) x^{i+1} + \sum_{i \geq 0}c_1\alpha(a_{i+n-1})x^{i+n} \\
&{} \qquad + c_2\sum_{i=0}^{n-3}\alpha^2(a_i)x^{i+2} + \sum_{i \geq 0}c_2\alpha^2(a_{i+n-2}) x^{i+n}\\
&{} \qquad \quad \dots  \\
&{} \qquad + c_{n-1}\alpha^{n-1}(a_0)x^{n-1} + \sum_{i \geq 0}c_{n-1}\alpha^{n-1}(a_{i+1})x^{i+n} \\
&{} \qquad + \sum_{i \geq0} c_n\alpha^n(a_i) x^{i+n}
\end{align*}
By defining
\[R := c_1\sum_{i =0}^{n-2}\alpha(a_i)x^{i+1} + c_2\sum_{i =0}^{n-3}\alpha^2(a_i)x^{i+2} + \dots + c_{n-1}\alpha^{n-1}(a_0)x^{n-1}\]
it is now clear that
\[(c_1x + \dots + c_nx^n)\sum_{i \geq 0}a_ix^i = R + \sum_{i \geq 0} (c_1 \alpha(a_{i+n-1}) + \dots + c_n \alpha^n(a_i))x^{i+n}.\]
as required.
\end{proof}
\begin{proof}[Proof of Theorem~\ref{res:rec reln thm}]
Let $\sum_{i \geq 0}a_ix^i \in K[\![x;\alpha]\!]$ and suppose that this power series satisfies a linear recurrence relation of the form
\[a_{i+n} = c_1\alpha(a_{i+(n-1)}) + c_2\alpha^2(a_{i+(n-2)}) + \dots + c_n\alpha^n(a_i)\]
for all $i\geq 0$.  We will construct a left fraction $Q^{-1}P \in K(x;\alpha)$ such that the image of $Q^{-1}P$ in $K[\![x;\alpha]\!]$ is $\sum_{i \geq 0}a_ix^i$.

Define $Q := 1 - \sum_{i=1}^n c_ix^i$, and observe that
\[ (1-Q)\sum_{i \geq0}a_i x^i = (c_1x + c_2x^2 + \dots +c_nx^n)\sum_{i \geq0}a_ix^i.\]
This is in the correct form to apply Lemma~\ref{res:rec reln technical lemma}, and so we have
\begin{align*}
 (1-Q)\sum_{i \geq0}a_i x^i &= R + \sum_{i \geq0}(c_1\alpha(a_{i+n-1}) + \dots + c_n\alpha^n(a_i))x^{i+n} \\
&= R + \sum_{i \geq0}a_{n+i}x^{n+i} \textrm{\quad (by assumption)}\\
&= R + \sum_{i \geq n}a_ix^i \\
&= R - \sum_{i=0}^{n-1}a_ix^i + \sum_{ i\geq0}a_ix^i
\end{align*}
where $R$ is a polynomial in $K[x;\alpha]$ of degree $\leq n-1$.  After simplifying this becomes
\[Q\sum_{i \geq0}a_i x^i = -R + \sum_{i=0}^{n-1}a_ix^i\]
and hence
\[\sum_{i \geq0}a_i x^i = Q^{-1}P\]
where $P:= -R + \sum_{i=0}^{n-1}a_ix^i$.

Conversely, let $F = Q^{-1}P \in K(x;\alpha)$; we need to show that for $F = \sum_{i \geq m}f_ix^i \in K(\!(x;\alpha)\!)$, the sequence $(f_i)_{i \geq m}$ satisfies a recurrence relation of the form \eqref{eq:def of rec reln}.  We will do this by performing a series of reductions on the fraction $F$, none of them affecting whether it admits a recurrence relation or not, until $F$ is in a form that is easier to work with.

We first claim that it suffices to consider only the case where $P,Q \in K[x;\alpha]$ are not divisible by $x$, i.e. they have non-zero constant terms.  This is immediately clear for $P$ since the powers of $x$ are written on the right, so suppose that $x \nmid P$ and $Q = Q'x^{-m}$ with $x\nmid Q'$.  We can now observe that $F = x^mQ'^{-1}P$ and hence
\[Q'^{-1}P =  x^{-m}F = \sum_{i \geq m}\alpha^{-m}(f_i)x^{i-m} = \sum_{i \geq 0}\alpha^{-m}(f_{i+m})x^i.\]
Since $\alpha$ is an automorphism it is clear that the sequence $(\alpha^{-m}(f_{i+m}))_{i \geq 0}$ will satisfy a recurrence relation if and only if the original sequence $(f_{i})_{i \geq m}$ did.  We can therefore assume that 
\[F = Q^{-1}P = \sum_{i \geq 0}f_i x^i\]
where $Q = r_0 + r_1x + \dots + r_nx^n \in K[x;\alpha]$ satisfies $r_0 \neq 0$.

Next we replace $Q^{-1}P$ by a fraction $Q^{-1}V$ satisfying $deg_x(V) < deg_x(Q)$, in such a way that at most finitely many terms of the sequence $(f_i)$ are changed.  (Note that while this will change the recurrence relation itself, it will not affect the \textit{existence} of the recurrence relation: if we change the first $n$ terms in a sequence that admits a recurrence relation, we can always obtain a recurrence relation for the new sequence simply by appending $n$ zeroes to the old set of recurrence constants.)  

If $deg(P) \geq deg(Q)$ then we can use the division algorithm to write $P = QS + V$, where $deg(V) < deg(Q)$ and $S$ is a polynomial: now
\begin{equation}\label{eq:rec reln theorem misc eqn 2}Q^{-1}V = Q^{-1}P - S\end{equation}
and since $S$ is a polynomial we can see that the power series representations of $Q^{-1}V$ and $Q^{-1}P$ differ by at most the first $deg_x(S)$ terms.

Since the constant term $r_0$ of $Q$ is non-zero, we can scale \eqref{eq:rec reln theorem misc eqn 2} by $r_0$ to obtain
\[r_0(Q^{-1}P-S) = r_0Q^{-1}V = T^{-1}V = \sum_{i \geq 0} a_ix^i\]
where $T:=Q(x)r_0^{-1}=1 - c_1x - \dots - c_nx^n$ for some $c_i \in K$.  Since the power series for $Q^{-1}P$ and $Q^{-1}V$ differ by finitely many terms, and scaling the fraction by an element of $K$ does not affect the existence of a recurrence relation, we see that the sequence $(a_i)$ satisfies a recurrence relation if and only if the original sequence $(f_i)$ does.

We are now in a position to show that $T^{-1}V = \sum_{i \geq 0}a_i x^i$ satisfies a linear recurrence relation of the form \eqref{eq:def of rec reln}.  Indeed, we can rearrange the equality $T^{-1}V = \sum_{i \geq 0}a_i x^i$ to obtain
\[(1-T)\sum_{i \geq0} a_i x^i = -V + \sum_{i \geq 0}a_ix^i\]
and then apply Lemma~\ref{res:rec reln technical lemma} to rewrite this as
\[R + \sum_{ i \geq 0}(c_1\alpha(a_{i+n-1}) + \dots + c_n\alpha^n(a_i)) x^{i+n} = -V + \sum_{ i \geq 0}a_i x^i,\]
where $R$ is a polynomial of degree $<deg_x(T)$.  Rearranging this, we obtain
\begin{equation}\label{eq:rec reln thm proof misc equation 1}R + V = \sum_{ i \geq 0} a_i x^i - \sum_{ i \geq 0} (c_1\alpha(a_{i+n-1}) + \dots + c_n\alpha^n(a_i)) x^{i+n}.\end{equation}
Since $deg_x(V) < deg_x(Q)  = deg_x(T) =n$, the left hand side of \eqref{eq:rec reln thm proof misc equation 1} is zero in degree $n$ and above.  Hence by comparing coefficients of $x^{i+n}$ for $i \geq 0$, we obtain the required recurrence relation \eqref{eq:def of rec reln}.
\end{proof}

\begin{remark}
A very similar version of this proof yields a recurrence relation for right fractions $PQ^{-1}$; in this case, it helps to work with right coefficients, i.e. power series of the form $\sum_{i \geq n}x^i a_i$.
\end{remark}

While Theorem~\ref{res:rec reln thm} is extremely useful for turning a power series with a recurrence relation into a left fraction, it gives no indication as to how the recurrence relation should be found in the first place.  The following theorem, due to Kronecker in the commutative case, attempts to address this problem.
\begin{theorem}\label{res:rec reln det thm} A power series $\sum_{i \geq0}a_ix^i$ satisfies a linear recurrence relation
\[a_{i+n} = c_1\alpha(a_{i+(n-1)}) + c_2\alpha^2(a_{i+(n-2)}) + \dots + c_k\alpha^k(a_i)\]
if and only if there exists some $m \geq 1$ such that the determinants of the matrices
\begin{equation*}
 \Delta_k = \left[ \begin{array}{ccccc}
         \alpha^k(a_0) & \alpha^{k-1}(a_{1}) & \dots & \alpha(a_{k-1}) & a_{k} \\
	 \alpha^k(a_{1}) & \alpha^{k-1}(a_{2}) & \dots &\alpha(a_{k}) & a_{k+1} \\
	 \vdots & &\ddots & &\vdots\\
	 \alpha^k(a_{k}) & \alpha^{k+1}(a_{k+1}) & \dots& \alpha(a_{2k-1}) & a_{2k}
        \end{array} \right]
\end{equation*}
are zero for all $k \geq m$.\end{theorem}
\begin{proof}
Again, this follows the commutative proof closely; we base our exposition on \cite[Lemma~III]{AlgNumbersBook}.  Observe that $\Delta_k$ is a $(k+1)\times(k+1)$ matrix.

It is easy to see that if the power series satisfies a recurrence relation of length $n$ then for all $k \geq n$, the final column of $\Delta_k$ is linearly dependent on the previous $n$ columns and the determinant is zero.

Conversely, suppose $|\Delta_k| = 0$ for all $k \geq m$ for some integer $m$, and assume that $m$ is minimal with this property.  Since the columns of the matrix $\Delta_m$ must be linearly dependant, there exist fixed elements $c_1, \dots, c_m \in K$ such that
\begin{equation}\label{eq:def of c_i, rec reln det proof}
C_{m+1} - c_1C_m - \dots - c_mC_1 = 0,
\end{equation}
where $C_i$ denotes the $i$th column of $\Delta_m$.  We will prove that $\sum_{i \geq 0}a_i x^i$ satisfies a recurrence relation with constants $c_1, \dots, c_m$.

Define
\begin{equation}\label{eq:def of P, rec reln det proof}
P_{j+m} = a_{j+m} - c_1\alpha(a_{j+m-1}) - c_2\alpha^2(a_{j+m-2}) - \dots - c_m\alpha^m(a_j)
\end{equation}
for $j \geq 0$; we will prove by induction that $P_{j+m} = 0$ for all $j \geq 0$.

The base case $P_m = 0$ follows immediately from \eqref{eq:def of c_i, rec reln det proof}, so suppose it is true for all $j < r$ for some $r$.  If $r \leq m$ then $P_{r+m} = 0$ also follows immediately from \eqref{eq:def of c_i, rec reln det proof}, so we can assume that $r > m$.

By performing column operations on $\Delta_r$ and recalling that its determinant is zero (our initial premise was that $|\Delta_k| = 0$ for all $k \geq m$), we will be able to show that $P_{r+m} = 0$ as well.  Observe that we can write $\Delta_r$ as follows:
\begin{equation*}
\Delta_r = 
\left[
\begin{array}{ccc:cccc}
A & & & \alpha^{r-m}(a_m) &  \cdots &  \alpha(a_{r-1}) & a_r \\
& & &\vdots & \ddots & & \vdots \\  \hdashline

 \alpha^{r}(a_{m})  & & & \alpha^{r-m}(a_{2m}) &\cdots &\alpha(a_{r+m-1}) & a_{r+m} \\
 \vdots & \ddots & & \vdots &\ddots & & \vdots \\
 \alpha^r(a_r)& & & \alpha^{r-m}(a_{r+m}) & \cdots & \alpha(a_{2r-1}) & a_{2r}
 \end{array}
\right]
\end{equation*}
where the block denoted by $A$ is $\alpha^{r-m+1}(\Delta_{m-1})$.  As before, let $C_i$ denote the $i$th column of $\Delta_r$, and recall that $\Delta_r$ is an $(r+1)\times(r+1)$ matrix.  Working from right to left, we will perform column operations on the columns in the right-hand blocks, that is columns $C_{r+1}$ down to $C_{m+1}$.  The column operations are:
\begin{align*}
C_{r+1} &\mapsto C_{r+1} - c_1 C_r - c_2C_{r-1} - \dots - c_m C_{r-m} \\
C_{r} &\mapsto C_{r} - \alpha(c_1)C_{r-1} - \alpha(c_2)C_{r-2} - \dots - \alpha(c_m)C_{r-m-1} \\
& \vdots \\
C_{m+1} &\mapsto C_{m+1} - \alpha^{r-m}(c_1)C_{m} - \alpha^{r-m}(c_2)C_{m-1} - \dots - \alpha^{r-m}(c_m)C_{0}
\end{align*}
Having done this, we obtain the matrix
\begin{equation*}
\Delta = 
\left[
\begin{array}{ccc:cccc}
A & & & \alpha^{r-m}(P_m) &  \cdots &  \alpha(P_{r-1}) & P_r \\
& & &\vdots & \ddots & & \vdots \\  \hdashline

 \alpha^{r}(a_{m})  & & & \alpha^{r-m}(P_{2m}) &\cdots &\alpha(P_{r+m-1}) & P_{r+m} \\
 \vdots & \ddots & & \vdots &\ddots & & \vdots \\
 \alpha^r(a_r)& & & \alpha^{r-m}(P_{r+m}) & \cdots & \alpha(P_{2r-1}) & P_{2r}
 \end{array}
\right]
\end{equation*}
which still satisfies $|\Delta| = 0$.  Further, by the inductive assumption $P_{j+m} = 0$ for $j < r$ and so the top-right block of $\Delta$ is identically zero, as are all entries above the reverse-diagonal in the bottom-right block.  The determinant $|\Delta|$ can now easily be seen to be
\[|\Delta| = \alpha^{r-m+1}(\Delta_{m-1})\prod_{i=0}^{r-m} \alpha^i(P_{r+m}) = 0.\]
Since $m$ was assumed to be the minimal integer such that $|\Delta_k| = 0$ for all $k \geq m$, while $K$ is a field and $\alpha$ is an automorphism, we conclude that $P_{r+m} = 0$ as required.\end{proof}

It is worth taking a moment to note the limitations of these results, before we attempt to use them.  First, since we cannot evaluate infinitely many terms of a power series or check the determinants of infinitely many matrices, any results obtained in this fashion will always be an approximation.  These techniques should be viewed as a means of checking computations and finding inspiration for the correct elements to write down; any properties that they should satisfy, e.g $q$-commuting, will then need to be proved using other methods.

Second, applying this theory to a given power series will always yield one single left (or right) fraction.  This will be a problem if, for example, the element in question is a product of several smaller fractions: combining the result into one fraction will often make it hopelessly large and complicated, and factorizing the result into understandable factors is almost always impossible.  We give an example of this problem in \S\ref{s:magma_example} below.

Third, results are limited by the computational power available.  The coefficients of a power series can get large very quickly and hence evaluating the determinants in Theorem~\ref{res:rec reln det thm} quickly becomes impossible.

\section{Magma code for computations in $k_q(y)(\!(x)\!)$}\label{s:magma_code}
In this section we provide the Magma code used to simulate computations in $k_q(y)(\!(x)\!)$.  The code should be pasted directly into the Magma terminal, after which the functions described below can be used as required.

Elements of $k_q(y)(\!(x)\!)$ are represented by two-element lists \texttt{[*n,F*]}.  The integer \texttt{n} denotes the lowest power of $x$ appearing in the series, while \texttt{F} is a sequence of coefficients in $k(y)$.  Hence for example the element $\sum_{i \geq n}a_ix^i$ would be stored as
\begin{verbatim}
[*n,[a_n,a_{n+1},a_{n+2}, ... , a_r]*]
\end{verbatim}
where $r$ can be arbitrarily large.  Note that while $a_n$ can be zero, this will cause some functions to break.

The following functions and procedures are provided below:
\begin{itemize}
\item \texttt{inverseL}: Takes a truncated power series and returns its inverse.
\item \texttt{productL}: Takes two truncated power series $F$, $G$ and returns their product $FG$.
\item \texttt{findz}: Computes the element $z$ from \cite[Proposition~3.3]{AC1} for a given element $G$.
\item \texttt{checkrationalL}: Takes a truncated power series $F$ and uses Theorem~\ref{res:rec reln det thm} to check whether $F$ represents a fraction, up to a given bound.
\item \texttt{findrationalL}: If \texttt{checkrationalL} returns true, this constructs polynomials $P,Q \in k_q(y)[x;\alpha]$ such that $F = Q^{-1}P$.
\item \texttt{checkpowerrationalL}: Given $F$, $P$ and $Q$ from \texttt{findrationalL}, double-checks that $F = Q^{-1}P$ up to a given bound.
\item \texttt{qelement}: Given an element $G$ of the form \eqref{eq:standard form of g to construct f}, constructs an element $F$ such that $FG = qGF$ as described in \S\ref{ss:autos of q-comm structures}. (Note that $F$ need not be a fraction even if $G$ is.)
\end{itemize}
\begin{verbatim}
// Note that Magma indexes sequences, lists, etc from 1 not 0; 
// this leads to weird indexing in some of the loops.

// Change these for a different field, different generator names
// if needed.
// t represents \hat{q}, the square root of q.
field<w>:=CyclotomicField(3);
K<t,y>:=FunctionField(field,2);
q:=t^2;
gen:=Name(K,2);
alpha:= hom< K -> K | t, q*gen>;
beta:= hom< K -> K | t, q^(-1)*gen>; 

// Magma interprets 1 and 0 as integers rather than the identity
// elements in K; use these when the distinction matters.
zero:=K!0;
one:=K!1;

// Inverts a sequence Z of x-degree 0 (first term non-zero).
// Output is named Y.
// m indicates how many terms of the inverse to compute.
// In general, use inverseL below instead.
procedure inverse(Z,~Y,m)
m1:=#(Z);
if m gt m1 then for i:=1 to (m-m1) do Z[m1+i]:=0; end for; m1:=m; 
end if;
if m lt m1 then m1:=m; end if;
Y:= [];
Y[1]:= 1/(Z[1]);
for i:= 1 to m1-1 do
	var1:=zero;
	for j:= 1 to i do
		var1:= var1 - Y[1]*Z[i-j+2]*(alpha^(i-j+1))(Y[j]);
	end for;
	Y[i+1]:=var1;
	i;
end for;
end procedure;

// Takes a sequence Z of lowest x-power r and returns its inverse L.
// Computes the first m terms.
procedure inverseL(Z,r,~L,m)
n:=#(Z);
Z1:=[];
if r lt 0 then a:=alpha^(-r); else a:=beta^(r); end if;
for j:=1 to n do
	Z1[j]:=a(Z[j]);
end for;
inverse(Z1,~Y,m);
L:=[*-r,Y*];
end procedure;

// Takes 2 elements as input: [*r,Y*] and [*s,Z*].
// Returns their product L, computes the first m terms.
// If Y and Z are precise (i.e. polynomial rather than truncated 
// power series) use m=0 to get a precise, untruncated answer.
procedure productL(Y,r,Z,s,~L,m)
n1:=#(Y); n2:=#(Z);
if m eq 0 then
n:=n1+n2;
else n:=m;
end if;
for i:=1 to (n-n1) do Y[n1+i]:=0; end for;
for i:=1 to (n-n2) do Z[n2+i]:=0; end for;
P:=[];
for i:=1 to n do
	var1:=zero;
	for j:=1 to i do
		if j+r-1 lt 0 then 
			a:=beta^(1-j-r); else a:=alpha^(j+r-1); 
		end if;
		var1:=var1 + Y[j]*a(Z[i-j+1]);
	end for;
P[i]:=var1;
i;
end for;
L:=[*r+s,P*];
end procedure;

// This computes the element z from Artamonov and Cohn's paper.
// Input: sequence b, which must be in the form given in the paper; 
// s should be 1 or -1 corresponding to power of y in first term
// of b.
// Computes the first n terms.
procedure findz(b,s,~Z,n)
n1:=#(b);
if n gt n1 then 
	for i:=1 to (n-n1+1) do
		b[n1+i]:=zero;
	end for;
end if;
z:=[];
z[1]:=gen^(-s)*(1-q^s)^(-1)*b[2];
for i:=2 to n do
	var1:=zero;
	for j:=1 to i-1 do
		var1:= var1 + z[j]*(alpha^j)(b[i-j+1]);
	end for;
	z[i]:=gen^(-s)*(1-q^(i*s))^(-1)*(b[i+1] + var1);
	i;
end for;
Insert(~z,1,one);
Z:=[*0,z*];
end procedure;

// Takes a truncated series and checks whether it satisfies the 
// conditions to represent a fraction.
// Input: the sequence P to be checked; checks matrix size a to n
// (a must be at least 2), and from b to m iterations of each size
// (minimum 1, just set b=m=1 if you're not sure about this).
// Prints "Yes" every time a determinant is zero.
// This procedure checks for rationality as a left fraction.
procedure checkrationalL(P,a,n,b,m)
for i:=a to n do
	for r:=b to m do
		M:=ZeroMatrix(K,i,i);
		for j:=1 to i do
			for k:=1 to i do
				M[j,k]:=(alpha^(i-k))(P[j+k+r-2]);
			end for;
		end for;
		d:=Determinant(M);
		if d eq 0 then i, r, "Yes"; else i, r, "No"; end if;
	end for;
end for;
end procedure

// Having run checkrationalL and found some zero determinants,
// this tries to pull the power series back to a fraction.
// Input sequence S to be pulled back; the two numbers from 
// checkrationalL next to the first "Yes" it printed become 
// a,b here (same order).
// Output: denominator Q, numerator P (it's a left fraction Q^{-1}P) 
// and c the set of recurrence relation coefficients.
// Note that this procedure assumes S has x-degree 0, if not 
// simply multiply numerator on the right by the appropriate power 
// of x afterwards.
procedure findrationalL(S,a,b,~Q,~P,~c)
M:=ZeroMatrix(K,a,a);
for j:=1 to a do
	for k:=1 to a do
		M[j,k]:=(alpha^(a-k))(S[j+k+b-2]);
	end for;
end for;
M1:=ZeroMatrix(K,a-1,1);
for i:=1 to (a-1) do
	M1[i,1]:=M[i,a];
end for;
M2:=Submatrix(M,1,1,(a-1),(a-1));
d:=Determinant(M2); // checking that this is invertible
if d eq 0 then 
	"Matrix is not invertible, check your values of a and b."; 
else M3:=M2^(-1);
M4:=M3*M1;
// M4 is the c_i in reverse order.
c:=[];
for i:=1 to (a-1) do
	c[i]:= M4[a-i,1];
end for;
Insert(~c,1,-1); // insert c_0 = 1 for later, so c_i = c[i+1]
"c found, computing P and Q...";
if b gt 1 then for i:=#c+1 to #c+b do c[i]:=zero; end for; end if;
n:=a-1+b-1;  // for ease of notation
q1:=[];
q1[1]:=one;
for i:=2 to n+1 do
	q1[i]:=-c[i];
end for;
Q:=[*0,q1*];
p1:=[];
for i:=1 to n do
	p1[i]:=zero;
end for;
for j:=1 to n do   //recall that S is the original sequence.
	for i:=0 to (n-j) do
		p1[i+j]:= p1[i+j] - c[j]*(alpha^(j-1))(S[i+1]);
	end for;
end for;
P:=[*0,p1*];
Remove(~c,1);
end if;
"done.";
end procedure;

// Check that the fraction from findrationalL is correct.
// Input sequence S, elements Q and P from findrationalL.
// If procedure returns true, then the two expressions agree
// (up to the point they were truncated).
function checkpowerrationalL(S,Q,P)
if #S le 25 then n:=#S-5; else n:=25; end if;
inverseL(Q[2],Q[1],~Q1,n+5);
productL(Q1[2],Q1[1],P[2],P[1],~T,n+5);
t:=T[2];
u:=[];
for i:=1 to n do
u[i]:=t[i]-S[i];
end for;
v1:=[];
for i:=1 to n do
v1[i]:=zero;
end for;
c:= u eq v1;
return c;
end function;

// Given a sequence g of x-degree 0 and first coefficient ay 
// (where a is a scalar), constructs a power series F which 
// q-commutes with g.
// Input: sequence g, f1 a choice for the first coefficient of F 
// (can be anything in k(y)), returns an element F.
// Computes the first n terms.
// Note that there is no guarantee that F will represent a fraction;
// however, changing the choice of f1 will not affect whether F 
// represents a fraction or not.
procedure qelement (g, f1, ~F, n)
if #g lt n then n:=#g; end if;
f:=[];
f[1]:=f1;
for i:=2 to n do
	a:=zero;
	for j:=1 to i-1 do
		a:=a + q*g[j+1]*(alpha^j)(f[i-j]) - f[j]*(alpha^j)(g[i-j+1]);
	end for;
	f[i]:=a/(g[1]*(q^i-q));
	i;
end for;
F:=[*1,f*];
end procedure;
\end{verbatim}

\section{An example of a computation in $k_q(x,y)$}\label{s:magma_example}
It is noted in Remark~\ref{rem:order 3 snark} that the generators of the fixed ring $k_q(x,y)^{\sigma}$ can be expressed as a pair of single left fractions.  The purpose of this example is to illustrate how a comparatively simple element can balloon into something hopelessly complicated when expressed as a single fraction.

Recall the setup of \S\ref{s:more fixed rings}.  We define
\begin{gather*}
a = x + \omega y + \omega^2\hat{q}y^{-1}x^{-1}, \quad b = x^{-1} + \omega y^{-1} + \omega^2\hat{q}yx, \\
\pa = x + y + \hat{q}y^{-1}x^{-1}, \quad \pb = x^{-1} + y^{-1} + \hat{q}yx,\\
g = a^{-1}b, \quad f = \pb - \omega^2\pa g + (\omega^2-\omega)\hat{q}^{-1}(\omega^2g^2 + \hat{q}^2g^{-1}).
\end{gather*}
By Proposition~\ref{res:order_3_q_comm_proof}, we know that $fg = qgf$.  However, the element $f$ was originally constructed using the {\tt qelement} function from \S\ref{s:magma_code} and only appeared in its current form after much fruitless searching.  The original form of this element (which we denote by $f'$) is defined next.

Denominator:

$v:=\Big( q^{86}y^{19} + q^{77}y^{16} + (\omega  + 1)q^{76}y^{{16}} + 2\omega q^{{75}}y^{16} + (2\omega  + 1)q^{74}y^{16} + (\omega  - 1)q^{73}y^{16} - 4q^{72}y^{16} + (-\omega  - 2)q^{71}y^{16} + (-2\omega  - 1)q^{70}y^{16} + 
            (-2\omega  - 2)q^{69}y^{16} - \omega q^{68}y^{16} + q^{67}y^{16} + (\omega  - 1)q^{66}y^{13} - q^{65}y^{13} + (-\omega  - 1)q^{64}y^{13} + (-3\omega  - 5)q^{63}y^{13} + (-4\omega  - 3)q^{62}y^{13} - 
            6\omega q^{61}y^{13} - 5\omega q^{60}y^{13} + (-3\omega  + 3)q^{59}y^{13} + {10}q^{58}y^{13} + (3\omega  + 6)q^{57}y^{13} + (5\omega  + 5)q^{56}y^{13} + (6\omega  + 6)q^{55}y^{13} + (4\omega  + 1)q^{54}y^{13} - 
            q^{54}y^{10} + (3\omega  - 2)q^{53}y^{13} + (-\omega  - 2)q^{53}y^{10} + \omega q^{52}y^{13} - 3\omega q^{52}y^{10} - q^{51}y^{13} + (-3\omega  + 1)q^{51}y^{10} + (-\omega  - 2)q^{50}y^{13} + 4q^{50}y^{10} + (4\omega  
            + {11})q^{49}y^{10} + ({10}\omega  + {11})q^{48}y^{10} + ({15}\omega  + 5)q^{47}y^{10} + ({15}\omega  + 2)q^{46}y^{10} + (9\omega  - 8)q^{45}y^{10} - {20}q^{44}y^{10} + (-9\omega  - {17})q^{43}y^{10} + (-{15}\omega  - 
            {13})q^{42}y^{10} + (-{15}\omega  - {10})q^{41}y^{10} + (-{10}\omega  + 1)q^{40}y^{10} + (-4\omega  + 7)q^{39}y^{10} + 4q^{38}y^{10} + (\omega  - 1)q^{38}y^7 + (3\omega  + 4)q^{37}y^{10} - q^{37}y^7 + (3\omega  +
            3)q^{36}y^{10} + (-\omega  - 1)q^{36}y^7 + (\omega  - 1)q^{35}y^{10} + (-3\omega  - 5)q^{35}y^7 - q^{34}y^{10} + (-4\omega  - 3)q^{34}y^7 - 6\omega q^{33}y^7 - 5\omega q^{32}y^7 + (-3\omega  + 3)q^{31}y^7
            + {10}q^{30}y^7 + (3\omega  + 6)q^{29}y^7 + (5\omega  + 5)q^{28}y^7 + (6\omega  + 6)q^{27}y^7 + (4\omega  + 1)q^{26}y^7 + (3\omega  - 2)q^{25}y^7 + \omega q^{24}y^7 - q^{23}y^7 + (-\omega  - 
            2)q^{22}y^7 + q^{21}y^4 + (\omega  + 1)q^{20}y^4 + 2\omega q^{19}y^4 + (2\omega  + 1)q^{18}y^4 + (\omega  - 1)q^{17}y^4 - 4q^{16}y^4 + (-\omega  - 2)q^{15}y^4 + (-2\omega  - 1)q^{14}y^4 + (-2\omega  
            - 2)q^{13}y^4 - \omega q^{12}y^4 + q^{11}y^4 + q^2y\Big)$ 

            $+ 
             \Big(-\omega q^{84}y^{18} - \omega q^{83}y^{18} + q^{82}y^{18} + (\omega  + 2)q^{81}y^{18} + (\omega  + 1)q^{80}y^{18} + \omega q^{79}y^{18} + (\omega  - 1)q^{78}y^{18} - q^{77}y^{18} + (-\omega  - 1)q^{76}y^{18} + (-\omega  - 
            1)q^{75}y^{18} + q^{75}y^{15} + q^{74}y^{15} + (2\omega  + 2)q^{73}y^{15} + (4\omega  + 3)q^{72}y^{15} + (4\omega  + 1)q^{71}y^{15} + (5\omega  - 2)q^{70}y^{15} + (3\omega  - 5)q^{69}y^{15} + (-2\omega  - 
            8)q^{68}y^{15} + (-7\omega  - {10})q^{67}y^{15} + (-9\omega  - 7)q^{66}y^{15} + (-9\omega  - 2)q^{65}y^{15} + (-7\omega  + 3)q^{64}y^{15} + (-2\omega  + 6)q^{63}y^{15} + (\omega  + 1)q^{63}y^{12} + (3\omega  + 
            8)q^{62}y^{15} + \omega q^{62}y^{12} + (5\omega  + 7)q^{61}y^{15} + (\omega  - 3)q^{61}y^{12} + (4\omega  + 3)q^{60}y^{15} - 5q^{60}y^{12} + (4\omega  + 1)q^{59}y^{15} + (-6\omega  - 8)q^{59}y^{12} + 
            2\omega q^{58}y^{15} + (-{13}\omega  - {11})q^{58}y^{12} - q^{57}y^{15} + (-{16}\omega  - 6)q^{57}y^{12} - q^{56}y^{15} + (-{16}\omega  + 5)q^{56}y^{12} + (-9\omega  + {16})q^{55}y^{12} + (6\omega  + {26})q^{54}y^{12} + 
            ({21}\omega  + {30})q^{53}y^{12} + ({30}\omega  + {23})q^{52}y^{12} + ({30}\omega  + 7)q^{51}y^{12} + ({21}\omega  - 9)q^{50}y^{12} + (6\omega  - {20})q^{49}y^{12} + (-\omega  - 1)q^{49}y^9 + (-9\omega  - {25})q^{48}y^{12} -
            \omega q^{48}y^9 + (-{16}\omega  - {21})q^{47}y^{12} + (-\omega  + 3)q^{47}y^9 + (-{16}\omega  - {10})q^{46}y^{12} + 5q^{46}y^9 + (-{13}\omega  - 2)q^{45}y^{12} + (6\omega  + 8)q^{45}y^9 + (-6\omega  + 
            2)q^{44}y^{12} + ({13}\omega  + {11})q^{44}y^9 + 5q^{43}y^{12} + ({16}\omega  + 6)q^{43}y^9 + (\omega  + 4)q^{42}y^{12} + ({16}\omega  - 5)q^{42}y^9 + (\omega  + 1)q^{41}y^{12} + (9\omega  - {16})q^{41}y^9 + 
            \omega q^{40}y^{12} + (-6\omega  - {26})q^{40}y^9 + (-{21}\omega  - {30})q^{39}y^9 + (-{30}\omega  - {23})q^{38}y^9 + (-{30}\omega  - 7)q^{37}y^9 + (-{21}\omega  + 9)q^{36}y^9 + (-6\omega  + {20})q^{35}y^9 + (9\omega  +
            {25})q^{34}y^9 + ({16}\omega  + {21})q^{33}y^9 - q^{33}y^6 + ({16}\omega  + {10})q^{32}y^9 - q^{32}y^6 + ({13}\omega  + 2)q^{31}y^9 + (-2\omega  - 2)q^{31}y^6 + (6\omega  - 2)q^{30}y^9 + (-4\omega  - 
            3)q^{30}y^6 - 5q^{29}y^9 + (-4\omega  - 1)q^{29}y^6 + (-\omega  - 4)q^{28}y^9 + (-5\omega  + 2)q^{28}y^6 + (-\omega  - 1)q^{27}y^9 + (-3\omega  + 5)q^{27}y^6 - \omega q^{26}y^9 + (2\omega  + 
            8)q^{26}y^6 + (7\omega  + {10})q^{25}y^6 + (9\omega  + 7)q^{24}y^6 + (9\omega  + 2)q^{23}y^6 + (7\omega  - 3)q^{22}y^6 + (2\omega  - 6)q^{21}y^6 + (-3\omega  - 8)q^{20}y^6 + (-5\omega  - 
            7)q^{19}y^6 + (-4\omega  - 3)q^{18}y^6 + (-4\omega  - 1)q^{17}y^6 - 2\omega q^{16}y^6 + q^{15}y^6 + q^{14}y^6 + \omega q^{14}y^3 + \omega q^{13}y^3 - q^{12}y^3 + (-\omega  - 2)q^{11}y^3 + (-\omega  - 
            1)q^{10}y^3 - \omega q^9y^3 + (-\omega  + 1)q^8y^3 + q^7y^3 + (\omega  + 1)q^6y^3 + (\omega  + 1)q^5y^3\Big)x$

            $ + 
        \Big(-q^{88}y^{20} + (\omega  - 1)q^{87}y^{20} + (\omega  - 1)q^{86}y^{20} - q^{85}y^{20} + (-\omega  - 1)q^{84}y^{20} + (-\omega  - 1)q^{83}y^{20} + (-\omega  - 1)q^{81}y^{17} + (-\omega  - 1)q^{80}y^{17} + (-2\omega  - 
            2)q^{79}y^{17} + (-3\omega  - 1)q^{78}y^{17} + (-4\omega  - 1)q^{77}y^{17} + (-5\omega  - 1)q^{76}y^{17} - 4\omega q^{75}y^{17} + (-4\omega  + 3)q^{74}y^{17} + (-4\omega  + 5)q^{73}y^{17} + (-2\omega  + 
            6)q^{72}y^{17} - \omega q^{72}y^{14} + (3\omega  + 9)q^{71}y^{17} + (-\omega  + 1)q^{71}y^{14} + (6\omega  + {10})q^{70}y^{17} + (\omega  + 3)q^{70}y^{14} + (7\omega  + 8)q^{69}y^{17} + (3\omega  + 6)q^{69}y^{14} + 
            (6\omega  + 3)q^{68}y^{17} + (5\omega  + 7)q^{68}y^{14} + (5\omega  + 1)q^{67}y^{17} + (8\omega  + 6)q^{67}y^{14} + 3\omega q^{66}y^{17} + ({10}\omega  + 3)q^{66}y^{14} - q^{65}y^{17} + ({10}\omega  + 1)q^{65}y^{14} 
            - q^{64}y^{17} + (8\omega  - 1)q^{64}y^{14} + (6\omega  - 3)q^{63}y^{14} + (7\omega  - 2)q^{62}y^{14} + (8\omega  - 2)q^{61}y^{14} + ({10}\omega  - 5)q^{60}y^{14} + ({11}\omega  - {10})q^{59}y^{14} - \omega q^{59}y^{11}
            + (6\omega  - {16})q^{58}y^{14} + (\omega  + 1)q^{58}y^{11} + (-3\omega  - {23})q^{57}y^{14} + (3\omega  + 2)q^{57}y^{11} + (-{12}\omega  - {26})q^{56}y^{14} + (4\omega  + 1)q^{56}y^{11} + (-{19}\omega  - {23})q^{55}y^{14}
            + (3\omega  - 5)q^{55}y^{11} + (-{22}\omega  - {16})q^{54}y^{14} - {11}q^{54}y^{11} + (-{19}\omega  - 8)q^{53}y^{14} + (-9\omega  - {18})q^{53}y^{11} - {12}\omega q^{52}y^{14} + (-{20}\omega  - {21})q^{52}y^{11} + (-5\omega  
            + 6)q^{51}y^{14} + (-{29}\omega  - {21})q^{51}y^{11} + (-2\omega  + 6)q^{50}y^{14} + (-{33}\omega  - {12})q^{50}y^{11} + 3q^{49}y^{14} - {31}\omega q^{49}y^{11} + (2\omega  + 2)q^{48}y^{14} + (-{26}\omega  + 
            {12})q^{48}y^{11} + (\omega  + 1)q^{47}y^{14} + (-{15}\omega  + {22})q^{47}y^{11} + (-5\omega  + {31})q^{46}y^{11} + (5\omega  + {36})q^{45}y^{11} + ({15}\omega  + {37})q^{44}y^{11} - \omega q^{44}y^8 + ({26}\omega  + 
            {38})q^{43}y^{11} - 2\omega q^{43}y^8 + ({31}\omega  + {31})q^{42}y^{11} + 3q^{42}y^8 + ({33}\omega  + {21})q^{41}y^{11} + (2\omega  + 8)q^{41}y^8 + ({29}\omega  + 8)q^{40}y^{11} + (5\omega  + {11})q^{40}y^8 + 
            ({20}\omega  - 1)q^{39}y^{11} + ({12}\omega  + {12})q^{39}y^8 + (9\omega  - 9)q^{38}y^{11} + ({19}\omega  + {11})q^{38}y^8 - {11}q^{37}y^{11} + ({22}\omega  + 6)q^{37}y^8 + (-3\omega  - 8)q^{36}y^{11} + ({19}\omega  - 
            4)q^{36}y^8 + (-4\omega  - 3)q^{35}y^{11} + ({12}\omega  - {14})q^{35}y^8 + (-3\omega  - 1)q^{34}y^{11} + (3\omega  - {20})q^{34}y^8 - \omega q^{33}y^{11} + (-6\omega  - {22})q^{33}y^8 + (\omega  + 1)q^{32}y^{11} 
            + (-{11}\omega  - {21})q^{32}y^8 + (-{10}\omega  - {15})q^{31}y^8 + (-8\omega  - {10})q^{30}y^8 + (-7\omega  - 9)q^{29}y^8 + (-6\omega  - 9)q^{28}y^8 + (-8\omega  - 9)q^{27}y^8 - q^{27}y^5 + (-{10}\omega  - 
            9)q^{26}y^8 - q^{26}y^5 + (-{10}\omega  - 7)q^{25}y^8 + (-3\omega  - 3)q^{25}y^5 + (-8\omega  - 2)q^{24}y^8 + (-5\omega  - 4)q^{24}y^5 + (-5\omega  + 2)q^{23}y^8 + (-6\omega  - 3)q^{23}y^5 + 
            (-3\omega  + 3)q^{22}y^8 + (-7\omega  + 1)q^{22}y^5 + (-\omega  + 2)q^{21}y^8 + (-6\omega  + 4)q^{21}y^5 + (\omega  + 2)q^{20}y^8 + (-3\omega  + 6)q^{20}y^5 + (\omega  + 1)q^{19}y^8 + (2\omega  + 
            8)q^{19}y^5 + (4\omega  + 9)q^{18}y^5 + (4\omega  + 7)q^{17}y^5 + (4\omega  + 4)q^{16}y^5 + (5\omega  + 4)q^{15}y^5 + (4\omega  + 3)q^{14}y^5 + (3\omega  + 2)q^{13}y^5 + 2\omega q^{12}y^5 + 
            \omega q^{11}y^5 + \omega q^{10}y^5 + \omega q^8y^2 + \omega q^7y^2 - q^6y^2 + (-\omega  - 2)q^5y^2 + (-\omega  - 2)q^4y^2 - q^3y^2\Big)x^2$

            $ + 
        \Big(q^{92}y^{22} + q^{91}y^{22} + \omega q^{87}y^{19} + (2\omega  + 1)q^{86}y^{19} + (2\omega  + 1)q^{85}y^{19} + (\omega  - 2)q^{84}y^{19} + (\omega  - 3)q^{83}y^{19} + (\omega  - 2)q^{82}y^{19} - 2q^{81}y^{19} + (-2\omega  - 
            2)q^{80}y^{19} + (-2\omega  - 1)q^{79}y^{19} - q^{78}y^{19} - q^{78}y^{16} - 3q^{77}y^{16} + (-\omega  + 1)q^{76}y^{19} + (-2\omega  - 4)q^{76}y^{16} + (-2\omega  - 1)q^{75}y^{19} + (-6\omega  - 
            7)q^{75}y^{16} + (-\omega  - 1)q^{74}y^{19} + (-{11}\omega  - 8)q^{74}y^{16} + q^{73}y^{19} + (-{15}\omega  - 4)q^{73}y^{16} + q^{72}y^{19} + (-{15}\omega  + 2)q^{72}y^{16} + (-9\omega  + {11})q^{71}y^{16} + 
            {21}q^{70}y^{16} + ({10}\omega  + {24})q^{69}y^{16} + ({18}\omega  + {19})q^{68}y^{16} + (-\omega  - 1)q^{68}y^{13} + ({22}\omega  + {14})q^{67}y^{16} + (-2\omega  - 1)q^{67}y^{13} + ({20}\omega  + 5)q^{66}y^{16} + (-2\omega  
            + 1)q^{66}y^{13} + ({12}\omega  - 6)q^{65}y^{16} + 5q^{65}y^{13} + (4\omega  - {11})q^{64}y^{16} + (3\omega  + {10})q^{64}y^{13} + (-2\omega  - {12})q^{63}y^{16} + (8\omega  + {14})q^{63}y^{13} + (-6\omega  - 
            {11})q^{62}y^{16} + ({17}\omega  + {15})q^{62}y^{13} + (-8\omega  - 7)q^{61}y^{16} + ({26}\omega  + {14})q^{61}y^{13} + (-6\omega  - 3)q^{60}y^{16} + ({30}\omega  + 7)q^{60}y^{13} + (-3\omega  - 1)q^{59}y^{16} + 
            ({27}\omega  - 9)q^{59}y^{13} + (-\omega  + 1)q^{58}y^{16} + ({17}\omega  - {26})q^{58}y^{13} + (-\omega  + 1)q^{57}y^{16} + (-\omega  - {41})q^{57}y^{13} + (-\omega  - 1)q^{56}y^{16} + (-{23}\omega  - {51})q^{56}y^{13} - 
            q^{55}y^{16} + (-{42}\omega  - {49})q^{55}y^{13} - q^{55}y^{10} + (-{52}\omega  - {34})q^{54}y^{13} - 2q^{54}y^{10} + (-{50}\omega  - {13})q^{53}y^{13} - 2q^{53}y^{10} + (-{37}\omega  + {10})q^{52}y^{13} + (-2\omega  - 
            3)q^{52}y^{10} + (-{16}\omega  + {29})q^{51}y^{13} + (-4\omega  - 5)q^{51}y^{10} + (3\omega  + {36})q^{50}y^{13} + (-6\omega  - 7)q^{50}y^{10} + ({16}\omega  + {33})q^{49}y^{13} + (-9\omega  - 8)q^{49}y^{10} + 
            ({22}\omega  + {25})q^{48}y^{13} + (-{15}\omega  - 8)q^{48}y^{10} + ({21}\omega  + {16})q^{47}y^{13} + (-{21}\omega  - 5)q^{47}y^{10} + ({15}\omega  + 7)q^{46}y^{13} + (-{22}\omega  + 3)q^{46}y^{10} + (9\omega  + 
            1)q^{45}y^{13} + (-{16}\omega  + {17})q^{45}y^{10} + (6\omega  - 1)q^{44}y^{13} + (-3\omega  + {33})q^{44}y^{10} + (4\omega  - 1)q^{43}y^{13} + ({16}\omega  + {45})q^{43}y^{10} + (2\omega  - 1)q^{42}y^{13} + ({37}\omega  
            + {47})q^{42}y^{10} - 2q^{41}y^{13} + ({50}\omega  + {37})q^{41}y^{10} - 2q^{40}y^{13} + ({52}\omega  + {18})q^{40}y^{10} - q^{39}y^{13} + ({42}\omega  - 7)q^{39}y^{10} - q^{39}y^7 + ({23}\omega  - 
            {28})q^{38}y^{10} + \omega q^{38}y^7 + (\omega  - {40})q^{37}y^{10} + (\omega  + 2)q^{37}y^7 + (-{17}\omega  - {43})q^{36}y^{10} + (\omega  + 2)q^{36}y^7 + (-{27}\omega  - {36})q^{35}y^{10} + (3\omega  + 2)q^{35}y^7 + 
            (-{30}\omega  - {23})q^{34}y^{10} + (6\omega  + 3)q^{34}y^7 + (-{26}\omega  - {12})q^{33}y^{10} + (8\omega  + 1)q^{33}y^7 + (-{17}\omega  - 2)q^{32}y^{10} + (6\omega  - 5)q^{32}y^7 + (-8\omega  + 6)q^{31}y^{10} + 
            (2\omega  - {10})q^{31}y^7 + (-3\omega  + 7)q^{30}y^{10} + (-4\omega  - {15})q^{30}y^7 + 5q^{29}y^{10} + (-{12}\omega  - {18})q^{29}y^7 + (2\omega  + 3)q^{28}y^{10} + (-{20}\omega  - {15})q^{28}y^7 + (2\omega  + 
            1)q^{27}y^{10} + (-{22}\omega  - 8)q^{27}y^7 + \omega q^{26}y^{10} + (-{18}\omega  + 1)q^{26}y^7 + (-{10}\omega  + {14})q^{25}y^7 + {21}q^{24}y^7 + (9\omega  + {20})q^{23}y^7 + ({15}\omega  + {17})q^{22}y^7 + 
            q^{22}y^4 + ({15}\omega  + {11})q^{21}y^7 + q^{21}y^4 + ({11}\omega  + 3)q^{20}y^7 + \omega q^{20}y^4 + (6\omega  - 1)q^{19}y^7 + (2\omega  + 1)q^{19}y^4 + (2\omega  - 2)q^{18}y^7 + (\omega  + 2)q^{18}y^4 
            - 3q^{17}y^7 - q^{16}y^7 - q^{16}y^4 + (2\omega  + 1)q^{15}y^4 + 2\omega q^{14}y^4 - 2q^{13}y^4 + (-\omega  - 3)q^{12}y^4 + (-\omega  - 4)q^{11}y^4 + (-\omega  - 3)q^{10}y^4 + (-2\omega  - 
            1)q^9y^4 + (-2\omega  - 1)q^8y^4 + (-\omega  - 1)q^7y^4 + q^3y + q^2y\Big)x^3$

            $ + 
       \Big( (-\omega  - 1)q^{93}y^{21} - \omega q^{92}y^{21} + (-\omega  + 1)q^{91}y^{21} + (-\omega  + 3)q^{90}y^{21} + (-\omega  + 2)q^{89}y^{21} + (\omega  + 2)q^{88}y^{21} + (3\omega  + 2)q^{87}y^{21} + (2\omega  + 2)q^{86}y^{21} + (\omega 
            + 1)q^{85}y^{18} + (-\omega  - 1)q^{84}y^{21} + (\omega  + 2)q^{84}y^{18} - q^{83}y^{21} + (\omega  + 3)q^{83}y^{18} + (4\omega  + 4)q^{82}y^{18} + (8\omega  + 5)q^{81}y^{18} + ({10}\omega  + 5)q^{80}y^{18} + 
            ({11}\omega  + 2)q^{79}y^{18} + ({11}\omega  - 4)q^{78}y^{18} + (8\omega  - {10})q^{77}y^{18} + (3\omega  - {14})q^{76}y^{18} + \omega q^{76}y^{15} + (-5\omega  - {17})q^{75}y^{18} + 2\omega q^{75}y^{15} + (-{12}\omega  - 
            {18})q^{74}y^{18} + (\omega  - 2)q^{74}y^{15} + (-{15}\omega  - {15})q^{73}y^{18} + (-2\omega  - 6)q^{73}y^{15} + (-{13}\omega  - 7)q^{72}y^{18} + (-3\omega  - {10})q^{72}y^{15} - 9\omega q^{71}y^{18} + (-9\omega  - 
            {13})q^{71}y^{15} + (-6\omega  + 3)q^{70}y^{18} + (-{13}\omega  - {12})q^{70}y^{15} + (-2\omega  + 3)q^{69}y^{18} + (-{19}\omega  - {12})q^{69}y^{15} + (2\omega  + 3)q^{68}y^{18} + (-{23}\omega  - 9)q^{68}y^{15} + 
            (2\omega  + 2)q^{67}y^{18} + (-{26}\omega  - 6)q^{67}y^{15} + (-{22}\omega  + 3)q^{66}y^{15} - q^{65}y^{18} + (-{21}\omega  + 9)q^{65}y^{15} - q^{64}y^{18} + (-{16}\omega  + {21})q^{64}y^{15} + (-8\omega  + 
            {28})q^{63}y^{15} + \omega q^{63}y^{12} + (4\omega  + {37})q^{62}y^{15} - q^{62}y^{12} + ({17}\omega  + {39})q^{61}y^{15} + (-3\omega  - 2)q^{61}y^{12} + ({26}\omega  + {38})q^{60}y^{15} + (-4\omega  - 1)q^{60}y^{12} + 
            ({33}\omega  + {28})q^{59}y^{15} + (-5\omega  + 3)q^{59}y^{12} + ({31}\omega  + {17})q^{58}y^{15} + (-4\omega  + {11})q^{58}y^{12} + ({24}\omega  + 6)q^{57}y^{15} + (5\omega  + {22})q^{57}y^{12} + ({13}\omega  - 
            2)q^{56}y^{15} + ({19}\omega  + {31})q^{56}y^{12} + (6\omega  - 7)q^{55}y^{15} + ({34}\omega  + {36})q^{55}y^{12} - 7q^{54}y^{15} + ({50}\omega  + {34})q^{54}y^{12} + (2\omega  - 2)q^{53}y^{15} + ({58}\omega  + 
            {20})q^{53}y^{12} + \omega q^{52}y^{15} + ({56}\omega  - 1)q^{52}y^{12} + ({45}\omega  - {24})q^{51}y^{12} - q^{50}y^{15} + ({25}\omega  - {47})q^{50}y^{12} - q^{50}y^9 + (\omega  - 1)q^{49}y^{15} - {65}q^{49}y^{12} + 
            (-\omega  - 2)q^{49}y^9 - q^{48}y^{15} + (-{25}\omega  - {72})q^{48}y^{12} - q^{48}y^9 + (-{45}\omega  - {69})q^{47}y^{12} + (-{56}\omega  - {57})q^{46}y^{12} + (-\omega  - 1)q^{46}y^9 + (-{58}\omega  - {38})q^{45}y^{12}
            + (-2\omega  - 4)q^{45}y^9 + (-{50}\omega  - {16})q^{44}y^{12} - 7q^{44}y^9 + (-{34}\omega  + 2)q^{43}y^{12} + (-6\omega  - {13})q^{43}y^9 + (-{19}\omega  + {12})q^{42}y^{12} + (-{13}\omega  - {15})q^{42}y^9 + 
            (-5\omega  + {17})q^{41}y^{12} + (-{24}\omega  - {18})q^{41}y^9 + (4\omega  + {15})q^{40}y^{12} + (-{31}\omega  - {14})q^{40}y^9 + (5\omega  + 8)q^{39}y^{12} + (-{33}\omega  - 5)q^{39}y^9 + (4\omega  + 3)q^{38}y^{12} 
            + (-{26}\omega  + {12})q^{38}y^9 + (3\omega  + 1)q^{37}y^{12} + (-{17}\omega  + {22})q^{37}y^9 - q^{36}y^{12} + (-4\omega  + {33})q^{36}y^9 + (-\omega  - 1)q^{35}y^{12} + (8\omega  + {36})q^{35}y^9 + ({16}\omega  + 
            {37})q^{34}y^9 - q^{34}y^6 + ({21}\omega  + {30})q^{33}y^9 - q^{33}y^6 + ({22}\omega  + {25})q^{32}y^9 + ({26}\omega  + {20})q^{31}y^9 - 2\omega q^{31}y^6 + ({23}\omega  + {14})q^{30}y^9 + (-2\omega  + 
            1)q^{30}y^6 + ({19}\omega  + 7)q^{29}y^9 + (2\omega  + 5)q^{29}y^6 + ({13}\omega  + 1)q^{28}y^9 + (6\omega  + 9)q^{28}y^6 + (9\omega  - 4)q^{27}y^9 + (9\omega  + 9)q^{27}y^6 + (3\omega  - 7)q^{26}y^9
            + ({13}\omega  + 6)q^{26}y^6 + (2\omega  - 4)q^{25}y^9 + {15}\omega q^{25}y^6 + (-\omega  - 3)q^{24}y^9 + ({12}\omega  - 6)q^{24}y^6 + (-2\omega  - 2)q^{23}y^9 + (5\omega  - {12})q^{23}y^6 + (-\omega  - 
            1)q^{22}y^9 + (-3\omega  - {17})q^{22}y^6 + (-8\omega  - {18})q^{21}y^6 + (-{11}\omega  - {15})q^{20}y^6 + (-{11}\omega  - 9)q^{19}y^6 + (-{10}\omega  - 5)q^{18}y^6 + (-8\omega  - 3)q^{17}y^6 - 
            4\omega q^{16}y^6 + (-\omega  + 2)q^{15}y^6 - q^{15}y^3 + (-\omega  + 1)q^{14}y^6 + \omega q^{14}y^3 - \omega q^{13}y^6 - 2\omega q^{12}y^3 + (-3\omega  - 1)q^{11}y^3 + (-\omega  + 1)q^{10}y^3 + (\omega  + 
            3)q^9y^3 + (\omega  + 4)q^8y^3 + (\omega  + 2)q^7y^3 + (\omega  + 1)q^6y^3 + \omega q^5y^3\Big)x^4$ 

            $+ 
        \Big(\omega q^{99}y^{23} + \omega q^{98}y^{23} - q^{97}y^{23} + (-\omega  - 2)q^{96}y^{23} + (-\omega  - 2)q^{95}y^{23} - q^{94}y^{23} - q^{93}y^{20} - q^{92}y^{20} + (-2\omega  - 2)q^{91}y^{20} + (-4\omega  - 3)q^{90}y^{20} + 
            (-4\omega  - 2)q^{89}y^{20} - 5\omega q^{88}y^{20} + (-5\omega  + 2)q^{87}y^{20} + (-5\omega  + 2)q^{86}y^{20} + (-3\omega  + 4)q^{85}y^{20} + 7q^{84}y^{20} + (-\omega  - 1)q^{84}y^{17} + (2\omega  + 
            9)q^{83}y^{20} - \omega q^{83}y^{17} + (4\omega  + 8)q^{82}y^{20} + (-\omega  + 3)q^{82}y^{17} + (6\omega  + 6)q^{81}y^{20} + (-\omega  + 6)q^{81}y^{17} + (7\omega  + 4)q^{80}y^{20} + (3\omega  + 9)q^{80}y^{17} + 
            (6\omega  + 2)q^{79}y^{20} + ({10}\omega  + {13})q^{79}y^{17} + 3\omega q^{78}y^{20} + ({18}\omega  + {14})q^{78}y^{17} - 2q^{77}y^{20} + ({23}\omega  + 9)q^{77}y^{17} - 2q^{76}y^{20} + {22}\omega q^{76}y^{17} - 
            q^{75}y^{20} + ({15}\omega  - {12})q^{75}y^{17} + (6\omega  - {21})q^{74}y^{17} + (-5\omega  - {26})q^{73}y^{17} + (\omega  + 1)q^{73}y^{14} + (-{13}\omega  - {24})q^{72}y^{17} + (\omega  + 1)q^{72}y^{14} + (-{16}\omega  - 
            {20})q^{71}y^{17} + \omega q^{71}y^{14} + (-{16}\omega  - {15})q^{70}y^{17} + (2\omega  - 1)q^{70}y^{14} + (-{14}\omega  - 9)q^{69}y^{17} + (2\omega  - 3)q^{69}y^{14} + (-{14}\omega  - 5)q^{68}y^{17} - 8q^{68}y^{14} +
            (-{12}\omega  - 3)q^{67}y^{17} + (-5\omega  - {14})q^{67}y^{14} - 8\omega q^{66}y^{17} + (-{15}\omega  - {20})q^{66}y^{14} + (-4\omega  + 5)q^{65}y^{17} + (-{27}\omega  - {24})q^{65}y^{14} + (-\omega  + 7)q^{64}y^{17} + 
            (-{36}\omega  - {18})q^{64}y^{14} + (\omega  + 6)q^{63}y^{17} + (-{39}\omega  - 2)q^{63}y^{14} + (2\omega  + 4)q^{62}y^{17} + (-{31}\omega  + {20})q^{62}y^{14} + (3\omega  + 3)q^{61}y^{17} + (-{12}\omega  + {41})q^{61}y^{14}
            + (3\omega  + 2)q^{60}y^{17} + ({14}\omega  + {56})q^{60}y^{14} + q^{60}y^{11} + (\omega  + 1)q^{59}y^{17} + ({39}\omega  + {58})q^{59}y^{14} + (-\omega  + 2)q^{59}y^{11} + ({53}\omega  + {45})q^{58}y^{14} + 4q^{58}y^{11}
            + ({53}\omega  + {23})q^{57}y^{14} + (3\omega  + 6)q^{57}y^{11} + ({42}\omega  + 1)q^{56}y^{14} + (7\omega  + 7)q^{56}y^{11} + ({25}\omega  - {16})q^{55}y^{14} + ({11}\omega  + 7)q^{55}y^{11} + (8\omega  - {26})q^{54}y^{14}
            + ({14}\omega  + 6)q^{54}y^{11} + (-4\omega  - {27})q^{53}y^{14} + ({14}\omega  + 1)q^{53}y^{11} + (-{11}\omega  - {24})q^{52}y^{14} + ({13}\omega  - 6)q^{52}y^{11} + (-{13}\omega  - {19})q^{51}y^{14} + ({11}\omega  - 
            {13})q^{51}y^{11} + (-{14}\omega  - {13})q^{50}y^{14} + (4\omega  - {23})q^{50}y^{11} + (-{14}\omega  - 8)q^{49}y^{14} + (-8\omega  - {34})q^{49}y^{11} + (-{11}\omega  - 4)q^{48}y^{14} + (-{25}\omega  - {41})q^{48}y^{11} -
            7\omega q^{47}y^{14} + (-{42}\omega  - {41})q^{47}y^{11} + (-3\omega  + 3)q^{46}y^{14} + (-{53}\omega  - {30})q^{46}y^{11} + 4q^{45}y^{14} + (-{53}\omega  - 8)q^{45}y^{11} + (\omega  + 3)q^{44}y^{14} + (-{39}\omega  + 
            {19})q^{44}y^{11} - \omega q^{44}y^8 + q^{43}y^{14} + (-{14}\omega  + {42})q^{43}y^{11} + (-3\omega  - 1)q^{43}y^8 + ({12}\omega  + {53})q^{42}y^{11} - 3\omega q^{42}y^8 + ({31}\omega  + {51})q^{41}y^{11} + (-2\omega  + 
            2)q^{41}y^8 + ({39}\omega  + {37})q^{40}y^{11} + (-\omega  + 5)q^{40}y^8 + ({36}\omega  + {18})q^{39}y^{11} + (\omega  + 8)q^{39}y^8 + ({27}\omega  + 3)q^{38}y^{11} + (4\omega  + 9)q^{38}y^8 + ({15}\omega  - 
            5)q^{37}y^{11} + (8\omega  + 8)q^{37}y^8 + (5\omega  - 9)q^{36}y^{11} + ({12}\omega  + 9)q^{36}y^8 - 8q^{35}y^{11} + ({14}\omega  + 9)q^{35}y^8 + (-2\omega  - 5)q^{34}y^{11} + ({14}\omega  + 5)q^{34}y^8 +
            (-2\omega  - 3)q^{33}y^{11} + ({16}\omega  + 1)q^{33}y^8 + (-\omega  - 1)q^{32}y^{11} + ({16}\omega  - 4)q^{32}y^8 - \omega q^{31}y^{11} + ({13}\omega  - {11})q^{31}y^8 - \omega q^{30}y^{11} + (5\omega  - {21})q^{30}y^8 +
            (-6\omega  - {27})q^{29}y^8 + (-{15}\omega  - {27})q^{28}y^8 - q^{28}y^5 + (-{22}\omega  - {22})q^{27}y^8 - 2q^{27}y^5 + (-{23}\omega  - {14})q^{26}y^8 - 2q^{26}y^5 + (-{18}\omega  - 4)q^{25}y^8 + (-3\omega 
            - 3)q^{25}y^5 + (-{10}\omega  + 3)q^{24}y^8 + (-6\omega  - 4)q^{24}y^5 + (-3\omega  + 6)q^{23}y^8 + (-7\omega  - 3)q^{23}y^5 + (\omega  + 7)q^{22}y^8 - 6\omega q^{22}y^5 + (\omega  + 4)q^{21}y^8 + 
            (-4\omega  + 4)q^{21}y^5 + (\omega  + 1)q^{20}y^8 + (-2\omega  + 7)q^{20}y^5 + \omega q^{19}y^8 + 7q^{19}y^5 + (3\omega  + 7)q^{18}y^5 + (5\omega  + 7)q^{17}y^5 + (5\omega  + 7)q^{16}y^5 + (5\omega  + 
            5)q^{15}y^5 + (4\omega  + 2)q^{14}y^5 + (4\omega  + 1)q^{13}y^5 + 2\omega q^{12}y^5 - q^{11}y^5 - q^{10}y^5 - q^9y^2 + (\omega  - 1)q^8y^2 + (\omega  - 1)q^7y^2 - q^6y^2 + (-\omega  - 
            1)q^5y^2 + (-\omega  - 1)q^4y^2\Big)x^5 $

            $+ 
        \Big(q^{106}y^{25} + q^{100}y^{22} + (\omega  + 1)q^{99}y^{22} + 2\omega q^{98}y^{22} + (2\omega  + 1)q^{97}y^{22} + (\omega  - 1)q^{96}y^{22} - 4q^{95}y^{22} + (-\omega  - 2)q^{94}y^{22} + (-2\omega  - 1)q^{93}y^{22} + 
            (-2\omega  - 2)q^{92}y^{22} + (\omega  - 1)q^{92}y^{19} - \omega q^{91}y^{22} - q^{91}y^{19} + q^{90}y^{22} + (-\omega  - 1)q^{90}y^{19} + (-3\omega  - 5)q^{89}y^{19} + q^{88}y^{22} + (-4\omega  - 3)q^{88}y^{19} + 
            q^{87}y^{22} - 6\omega q^{87}y^{19} - 5\omega q^{86}y^{19} + (-3\omega  + 3)q^{85}y^{19} + {10}q^{84}y^{19} + (3\omega  + 6)q^{83}y^{19} - q^{83}y^{16} + (5\omega  + 6)q^{82}y^{19} + (-\omega  - 2)q^{82}y^{16} + 
            (7\omega  + 8)q^{81}y^{19} - 3\omega q^{81}y^{16} + (7\omega  + 2)q^{80}y^{19} + (-3\omega  + 1)q^{80}y^{16} + (7\omega  - 1)q^{79}y^{19} + 4q^{79}y^{16} + 4\omega q^{78}y^{19} + (4\omega  + {11})q^{78}y^{16} + (\omega 
            - 6)q^{77}y^{19} + ({10}\omega  + {11})q^{77}y^{16} + (-2\omega  - 8)q^{76}y^{19} + ({15}\omega  + 5)q^{76}y^{16} + (-3\omega  - 3)q^{75}y^{19} + ({15}\omega  + 2)q^{75}y^{16} + (-4\omega  - 3)q^{74}y^{19} + 
            ({10}\omega  - 9)q^{74}y^{16} + (-3\omega  - 2)q^{73}y^{19} + (\omega  - {22})q^{73}y^{16} + (-\omega  + 1)q^{72}y^{19} + (-{10}\omega  - {19})q^{72}y^{16} + q^{71}y^{19} + (-{19}\omega  - {19})q^{71}y^{16} + (-{22}\omega  - 
            {18})q^{70}y^{16} + (\omega  - 1)q^{70}y^{13} + q^{69}y^{19} + (-{20}\omega  - 2)q^{69}y^{16} - q^{69}y^{13} + (-{15}\omega  + 7)q^{68}y^{16} + (-\omega  - 1)q^{68}y^{13} + (-8\omega  + 7)q^{67}y^{16} + (-3\omega  - 
            5)q^{67}y^{13} + {17}q^{66}y^{16} + (-4\omega  - 3)q^{66}y^{13} + (6\omega  + {19})q^{65}y^{16} + (-6\omega  - 1)q^{65}y^{13} + (9\omega  + {10})q^{64}y^{16} + (-6\omega  - 3)q^{64}y^{13} + ({11}\omega  + 
            {11})q^{63}y^{16} + (-7\omega  + 1)q^{63}y^{13} + ({11}\omega  + 8)q^{62}y^{16} + (-6\omega  + {11})q^{62}y^{13} + (9\omega  - 1)q^{61}y^{16} + {11}q^{61}y^{13} + (6\omega  - 1)q^{60}y^{16} + (9\omega  + 
            {20})q^{60}y^{13} + (2\omega  - 2)q^{59}y^{16} + ({20}\omega  + {28})q^{59}y^{13} + (-\omega  - 7)q^{58}y^{16} + ({29}\omega  + {17})q^{58}y^{13} + (-2\omega  - 4)q^{57}y^{16} + ({33}\omega  + 5)q^{57}y^{13} + (-2\omega  -
            1)q^{56}y^{16} + ({25}\omega  - 6)q^{56}y^{13} + q^{56}y^{10} + (-2\omega  - 2)q^{55}y^{16} + ({10}\omega  - {30})q^{55}y^{13} + (\omega  + 1)q^{55}y^{10} - \omega q^{54}y^{16} + (-{10}\omega  - {40})q^{54}y^{13} + 
            2\omega q^{54}y^{10} + q^{53}y^{16} + (-{25}\omega  - {31})q^{53}y^{13} + (2\omega  + 1)q^{53}y^{10} + (-{33}\omega  - {28})q^{52}y^{13} + (2\omega  - 2)q^{52}y^{10} + (-{29}\omega  - {12})q^{51}y^{13} + (\omega  - 
            6)q^{51}y^{10} + (-{20}\omega  + 8)q^{50}y^{13} + (-2\omega  - 4)q^{50}y^{10} + (-9\omega  + {11})q^{49}y^{13} + (-6\omega  - 7)q^{49}y^{10} + {11}q^{48}y^{13} + (-9\omega  - {10})q^{48}y^{10} + (6\omega  + 
            {17})q^{47}y^{13} + (-{11}\omega  - 3)q^{47}y^{10} + (7\omega  + 8)q^{46}y^{13} - {11}\omega q^{46}y^{10} + (6\omega  + 3)q^{45}y^{13} + (-9\omega  + 1)q^{45}y^{10} + (6\omega  + 5)q^{44}y^{13} + (-6\omega  + 
            {13})q^{44}y^{10} + (4\omega  + 1)q^{43}y^{13} + {17}q^{43}y^{10} + (3\omega  - 2)q^{42}y^{13} + (8\omega  + {15})q^{42}y^{10} + \omega q^{41}y^{13} + ({15}\omega  + {22})q^{41}y^{10} - q^{40}y^{13} + ({20}\omega  + 
            {18})q^{40}y^{10} + q^{40}y^7 + (-\omega  - 2)q^{39}y^{13} + ({22}\omega  + 4)q^{39}y^{10} + {19}\omega q^{38}y^{10} + q^{38}y^7 + ({10}\omega  - 9)q^{37}y^{10} + (\omega  + 2)q^{37}y^7 + (-\omega  - {23})q^{36}y^{10}
            + (3\omega  + 1)q^{36}y^7 + (-{10}\omega  - {19})q^{35}y^{10} + (4\omega  + 1)q^{35}y^7 + (-{15}\omega  - {13})q^{34}y^{10} + 3\omega q^{34}y^7 + (-{15}\omega  - {10})q^{33}y^{10} + (2\omega  - 6)q^{33}y^7 + 
            (-{10}\omega  + 1)q^{32}y^{10} + (-\omega  - 7)q^{32}y^7 + (-4\omega  + 7)q^{31}y^{10} + (-4\omega  - 4)q^{31}y^7 + 4q^{30}y^{10} + (-7\omega  - 8)q^{30}y^7 + (3\omega  + 4)q^{29}y^{10} + (-7\omega  - 
            5)q^{29}y^7 + (3\omega  + 3)q^{28}y^{10} + (-7\omega  + 1)q^{28}y^7 + (\omega  - 1)q^{27}y^{10} + (-5\omega  + 1)q^{27}y^7 - q^{26}y^{10} + (-3\omega  + 3)q^{26}y^7 + {10}q^{25}y^7 + (3\omega  + 
            6)q^{24}y^7 + (5\omega  + 5)q^{23}y^7 + (6\omega  + 6)q^{22}y^7 + q^{22}y^4 + (4\omega  + 1)q^{21}y^7 + q^{21}y^4 + (3\omega  - 2)q^{20}y^7 + \omega q^{19}y^7 + q^{19}y^4 - q^{18}y^7 + (\omega  + 
            1)q^{18}y^4 + (-\omega  - 2)q^{17}y^7 + 2\omega q^{17}y^4 + (2\omega  + 1)q^{16}y^4 + (\omega  - 1)q^{15}y^4 - 4q^{14}y^4 + (-\omega  - 2)q^{13}y^4 + (-2\omega  - 1)q^{12}y^4 + (-2\omega  - 
            2)q^{11}y^4 - \omega q^{10}y^4 + q^9y^4 + q^3y\Big)x^6$

Numerator:

$u:=\Big(q^{86}y^{19} + q^{77}y^{16} + (\omega  + 1)q^{76}y^{16} + 2\omega q^{75}y^{16} + (2\omega  + 1)q^{74}y^{16} + (\omega  - 1)q^{73}y^{16} - 4q^{72}y^{16} + (-\omega  - 2)q^{71}y^{16} + (-2\omega  - 1)q^{70}y^{16} + 
            (-2\omega  - 2)q^{69}y^{16} - \omega q^{68}y^{16} + q^{67}y^{16} + (\omega  - 1)q^{66}y^{13} - q^{65}y^{13} + (-\omega  - 1)q^{64}y^{13} + (-3\omega  - 5)q^{63}y^{13} + (-4\omega  - 3)q^{62}y^{13} - 
            6\omega q^{61}y^{13} - 5\omega q^{60}y^{13} + (-3\omega  + 3)q^{59}y^{13} + {10}q^{58}y^{13} + (3\omega  + 6)q^{57}y^{13} + (5\omega  + 5)q^{56}y^{13} + (6\omega  + 6)q^{55}y^{13} + (4\omega  + 1)q^{54}y^{13} - 
            q^{54}y^{10} + (3\omega  - 2)q^{53}y^{13} + (-\omega  - 2)q^{53}y^{10} + \omega q^{52}y^{13} - 3\omega q^{52}y^{10} - q^{51}y^{13} + (-3\omega  + 1)q^{51}y^{10} + (-\omega  - 2)q^{50}y^{13} + 4q^{50}y^{10} + (4\omega  
            + {11})q^{49}y^{10} + ({10}\omega  + {11})q^{48}y^{10} + ({15}\omega  + 5)q^{47}y^{10} + ({15}\omega  + 2)q^{46}y^{10} + (9\omega  - 8)q^{45}y^{10} - {20}q^{44}y^{10} + (-9\omega  - {17})q^{43}y^{10} + (-{15}\omega  - 
            {13})q^{42}y^{10} + (-{15}\omega  - {10})q^{41}y^{10} + (-{10}\omega  + 1)q^{40}y^{10} + (-4\omega  + 7)q^{39}y^{10} + 4q^{38}y^{10} + (\omega  - 1)q^{38}y^7 + (3\omega  + 4)q^{37}y^{10} - q^{37}y^7 + (3\omega  +
            3)q^{36}y^{10} + (-\omega  - 1)q^{36}y^7 + (\omega  - 1)q^{35}y^{10} + (-3\omega  - 5)q^{35}y^7 - q^{34}y^{10} + (-4\omega  - 3)q^{34}y^7 - 6\omega q^{33}y^7 - 5\omega q^{32}y^7 + (-3\omega  + 3)q^{31}y^7
            + {10}q^{30}y^7 + (3\omega  + 6)q^{29}y^7 + (5\omega  + 5)q^{28}y^7 + (6\omega  + 6)q^{27}y^7 + (4\omega  + 1)q^{26}y^7 + (3\omega  - 2)q^{25}y^7 + \omega q^{24}y^7 - q^{23}y^7 + (-\omega  - 
            2)q^{22}y^7 + q^{21}y^4 + (\omega  + 1)q^{20}y^4 + 2\omega q^{19}y^4 + (2\omega  + 1)q^{18}y^4 + (\omega  - 1)q^{17}y^4 - 4q^{16}y^4 + (-\omega  - 2)q^{15}y^4 + (-2\omega  - 1)q^{14}y^4 + (-2\omega  
            - 2)q^{13}y^4 - \omega q^{12}y^4 + q^{11}y^4 + q^2y\Big)x $

            $+ 
       \Big( (-\omega  - 1)q^{90}y^{21} + (-\omega  - 1)q^{89}y^{21} - \omega q^{83}y^{18} + q^{82}y^{18} + q^{81}y^{18} - \omega q^{80}y^{18} + 2q^{79}y^{18} + 2q^{78}y^{18} + 2q^{77}y^{18} + (4\omega  + 5)q^{76}y^{18} + (5\omega  +
            4)q^{75}y^{18} + (\omega  + 1)q^{75}y^{15} + 3\omega q^{74}y^{18} + (3\omega  - 1)q^{73}y^{18} + (2\omega  - 1)q^{72}y^{18} + (3\omega  + 1)q^{72}y^{15} + (-\omega  - 2)q^{71}y^{18} + 2\omega q^{71}y^{15} + (-\omega  - 
            1)q^{70}y^{18} + 2\omega q^{70}y^{15} + (4\omega  - 1)q^{69}y^{15} + (\omega  - 5)q^{68}y^{15} + (-\omega  - 6)q^{67}y^{15} - 5q^{66}y^{15} + (-5\omega  - 9)q^{65}y^{15} + (-7\omega  - 8)q^{64}y^{15} + (-2\omega  
            - 1)q^{64}y^{12} + (-5\omega  - 5)q^{63}y^{15} + q^{63}y^{12} + (-{10}\omega  - 8)q^{62}y^{15} + (\omega  + 1)q^{62}y^{12} + (-{11}\omega  - 6)q^{61}y^{15} - \omega q^{61}y^{12} - 7\omega q^{60}y^{15} + (\omega  - 
            1)q^{60}y^{12} + (-7\omega  + 1)q^{59}y^{15} - 2q^{59}y^{12} + (-5\omega  + 3)q^{58}y^{15} + (-7\omega  - 5)q^{58}y^{12} + (\omega  + 7)q^{57}y^{15} + (-7\omega  - 1)q^{57}y^{12} + (2\omega  + 5)q^{56}y^{15}
            + (-4\omega  + 3)q^{56}y^{12} + (\omega  + 2)q^{55}y^{15} + (-7\omega  + 6)q^{55}y^{12} + (3\omega  + 2)q^{54}y^{15} + (\omega  + {13})q^{54}y^{12} + (2\omega  + 1)q^{53}y^{15} + (6\omega  + {13})q^{53}y^{12} + (5\omega 
            + {11})q^{52}y^{12} - \omega q^{52}y^9 + (9\omega  + {13})q^{51}y^{12} + (-\omega  + 1)q^{51}y^9 + ({13}\omega  + {13})q^{50}y^{12} + (3\omega  + 3)q^{50}y^9 + ({11}\omega  + {10})q^{49}y^{12} + (3\omega  + 
            2)q^{49}y^9 + ({18}\omega  + {13})q^{48}y^{12} + 3\omega q^{48}y^9 + ({21}\omega  + 7)q^{47}y^{12} + (6\omega  - 1)q^{47}y^9 + ({14}\omega  - 4)q^{46}y^{12} + (\omega  - 5)q^{46}y^9 + ({10}\omega  - 9)q^{45}y^{12} 
            + (-4\omega  - 7)q^{45}y^9 + (3\omega  - {14})q^{44}y^{12} - 4q^{44}y^9 + (-8\omega  - {17})q^{43}y^{12} + (-\omega  - 3)q^{43}y^9 + (-{10}\omega  - {11})q^{42}y^{12} + (-3\omega  - 3)q^{42}y^9 + (-7\omega  - 
            4)q^{41}y^{12} + (3\omega  - 4)q^{41}y^9 + (-6\omega  - 1)q^{40}y^{12} + (-2\omega  - {10})q^{40}y^9 + (-2\omega  + 2)q^{39}y^{12} + (-{10}\omega  - {13})q^{39}y^9 + (2\omega  + 3)q^{38}y^{12} + (-{11}\omega  - 
            {12})q^{38}y^9 + (\omega  + 1)q^{37}y^{12} + (-{16}\omega  - {10})q^{37}y^9 + (-{17}\omega  - 2)q^{36}y^9 + (-2\omega  - 1)q^{36}y^6 + (-8\omega  + 6)q^{35}y^9 - \omega q^{35}y^6 + (-4\omega  + 6)q^{34}y^9 +
            q^{34}y^6 + (-2\omega  + 7)q^{33}y^9 + (-2\omega  + 2)q^{33}y^6 + (4\omega  + 7)q^{32}y^9 + (\omega  + 3)q^{32}y^6 + (2\omega  + 2)q^{31}y^9 + (4\omega  + 4)q^{31}y^6 + (-\omega  + 1)q^{30}y^9 + (\omega  
            + 2)q^{30}y^6 + (\omega  + 3)q^{29}y^9 + (2\omega  + 2)q^{29}y^6 + (\omega  + 2)q^{28}y^9 + (5\omega  + 2)q^{28}y^6 + q^{27}y^9 + 2q^{27}y^6 + (2\omega  + 2)q^{26}y^9 + (2\omega  + 3)q^{26}y^6 
            + (2\omega  + 1)q^{25}y^9 + (6\omega  + 3)q^{25}y^6 + (4\omega  + 2)q^{24}y^6 + (3\omega  + 1)q^{23}y^6 + (5\omega  - 1)q^{22}y^6 + (\omega  - 3)q^{21}y^6 + (\omega  - 1)q^{20}y^6 + (\omega  - 
            2)q^{19}y^6 + \omega q^{19}y^3 + (-\omega  - 3)q^{18}y^6 - q^{18}y^3 + (-\omega  - 2)q^{17}y^6 + (-2\omega  - 2)q^{17}y^3 - q^{16}y^6 + (-\omega  - 2)q^{16}y^3 + (-\omega  - 1)q^{15}y^6 + (-2\omega  - 
            1)q^{15}y^3 - \omega q^{14}y^6 - 3\omega q^{14}y^3 + q^{13}y^3 + (\omega  + 1)q^{12}y^3 - \omega q^{11}y^3 + (-\omega  + 1)q^8y^3 + q^7y^3 + \omega \Big)x^2 $

            $+ 
        \Big(\omega q^{95}y^{23} - q^{89}y^{20} - q^{88}y^{20} + (-2\omega  - 1)q^{87}y^{20} + (-\omega  - 1)q^{86}y^{20} + (-\omega  - 1)q^{85}y^{20} - 2\omega q^{84}y^{20} - 2\omega q^{81}y^{20} - \omega q^{81}y^{17} + q^{80}y^{20} - 
            \omega q^{80}y^{17} + (\omega  + 2)q^{79}y^{20} + 2q^{78}y^{20} + (-\omega  + 3)q^{78}y^{17} + (\omega  + 1)q^{77}y^{20} + (-\omega  + 3)q^{77}y^{17} + \omega q^{76}y^{20} + (4\omega  + 4)q^{76}y^{17} + (4\omega  + 
            6)q^{75}y^{17} + (4\omega  + 5)q^{74}y^{17} + (8\omega  + 1)q^{73}y^{17} + (5\omega  + 1)q^{72}y^{17} + (3\omega  - 1)q^{71}y^{17} + q^{71}y^{14} + (4\omega  - 4)q^{70}y^{17} + 2\omega q^{70}y^{14} - 
            5q^{69}y^{17} + 2\omega q^{69}y^{14} + (-3\omega  - 5)q^{68}y^{17} - q^{68}y^{14} - 5q^{67}y^{17} + (\omega  - 4)q^{67}y^{14} + (-4\omega  - 4)q^{66}y^{17} + (2\omega  - 3)q^{66}y^{14} + (-4\omega  - 
            2)q^{65}y^{17} + (-3\omega  - 3)q^{65}y^{14} + (-\omega  - 2)q^{64}y^{17} + (-\omega  - 7)q^{64}y^{14} + (-2\omega  - 1)q^{63}y^{17} + (-\omega  - 8)q^{63}y^{14} + (-3\omega  - 1)q^{62}y^{17} + (-{10}\omega  - 
            8)q^{62}y^{14} + (-{13}\omega  - {13})q^{61}y^{14} - \omega q^{60}y^{17} + (-{13}\omega  - 8)q^{60}y^{14} - q^{60}y^{11} + (-\omega  + 1)q^{59}y^{17} + (-{19}\omega  + 1)q^{59}y^{14} + (-2\omega  - 1)q^{59}y^{11} + 
            (-{12}\omega  + 7)q^{58}y^{14} + (-3\omega  - 1)q^{58}y^{11} + (-\omega  + {14})q^{57}y^{14} + (-3\omega  - 1)q^{57}y^{11} + (3\omega  + {21})q^{56}y^{14} + (-3\omega  + 3)q^{56}y^{11} + ({11}\omega  + {16})q^{55}y^{14} 
            + (-3\omega  + 4)q^{55}y^{11} + ({15}\omega  + {10})q^{54}y^{14} + (3\omega  + 7)q^{54}y^{11} + (9\omega  + 7)q^{53}y^{14} + (4\omega  + {11})q^{53}y^{11} + (9\omega  + 3)q^{52}y^{14} + (6\omega  + {11})q^{52}y^{11} +
            (7\omega  - 1)q^{51}y^{14} + ({15}\omega  + 9)q^{51}y^{11} + 2\omega q^{50}y^{14} + ({16}\omega  + {12})q^{50}y^{11} + 4\omega q^{49}y^{14} + ({13}\omega  + 4)q^{49}y^{11} + (4\omega  - 2)q^{48}y^{14} + ({16}\omega  - 
            5)q^{48}y^{11} - 2q^{47}y^{14} + ({11}\omega  - 7)q^{47}y^{11} - 2q^{46}y^{14} + (-\omega  - {14})q^{46}y^{11} + q^{46}y^8 - 2q^{45}y^{14} + (-2\omega  - {21})q^{45}y^{11} + (\omega  + 1)q^{45}y^8 + 
            (-2\omega  - 1)q^{44}y^{14} + (-{10}\omega  - {16})q^{44}y^{11} + (3\omega  + 1)q^{44}y^8 - \omega q^{43}y^{14} + (-{15}\omega  - {12})q^{43}y^{11} + 3\omega q^{43}y^8 + (-{12}\omega  - {13})q^{42}y^{11} + (3\omega  - 
            4)q^{42}y^8 + (-{12}\omega  - 6)q^{41}y^{11} + (\omega  - 7)q^{41}y^8 + (-{13}\omega  - 1)q^{40}y^{11} + (-6\omega  - {10})q^{40}y^8 + (-6\omega  - 1)q^{39}y^{11} + (-{11}\omega  - {13})q^{39}y^8 + (-6\omega  + 
            1)q^{38}y^{11} + (-{12}\omega  - 8)q^{38}y^8 + (-4\omega  + 6)q^{37}y^{11} - {15}\omega q^{37}y^8 + (\omega  + 4)q^{36}y^{11} + (-{10}\omega  + 3)q^{36}y^8 + (\omega  + 3)q^{35}y^{11} + (-2\omega  + {10})q^{35}y^8
            + (\omega  + 3)q^{34}y^{11} + (\omega  + {15})q^{34}y^8 + (3\omega  + 2)q^{33}y^{11} + (4\omega  + 9)q^{33}y^8 + \omega q^{32}y^{11} + (9\omega  + 8)q^{32}y^8 - \omega q^{31}y^{11} + (5\omega  + {10})q^{31}y^8 + (7\omega 
            + 7)q^{30}y^8 - \omega q^{30}y^5 + ({10}\omega  + 4)q^{29}y^8 - \omega q^{29}y^5 + (8\omega  + 6)q^{28}y^8 + 2q^{28}y^5 + (8\omega  + 1)q^{27}y^8 + (\omega  + 3)q^{27}y^5 + (7\omega  - 4)q^{26}y^8 + 
            (3\omega  + 3)q^{26}y^5 + (\omega  - 5)q^{25}y^8 + (4\omega  + 3)q^{25}y^5 + (-\omega  - 4)q^{24}y^8 + (4\omega  + 1)q^{24}y^5 + (-\omega  - 4)q^{23}y^8 + (4\omega  - 2)q^{23}y^5 + (-3\omega  - 
            2)q^{22}y^8 + (2\omega  - 2)q^{22}y^5 - \omega q^{21}y^8 + (-\omega  - 4)q^{21}y^5 + (-\omega  - 6)q^{20}y^5 + (-3\omega  - 4)q^{19}y^5 + (-6\omega  - 3)q^{18}y^5 + (-4\omega  - 3)q^{17}y^5 - 
            3\omega q^{16}y^5 + (-3\omega  + 2)q^{15}y^5 + (-\omega  + 1)q^{14}y^5 + q^{13}y^5 + q^{12}y^5 + (\omega  + 1)q^{11}y^5 + (-\omega  - 1)q^{11}y^2 + (-\omega  - 1)q^{10}y^2 - \omega q^9y^2 + q^8y^2 + 
            (\omega  + 1)q^7y^2 + (\omega  + 1)q^6y^2 + \omega q^5y^2 + \omega q^4y^2\Big)x^3$

            $ + 
        \Big(\omega q^{94}y^{22} + \omega q^{93}y^{22} + (-\omega  - 1)q^{91}y^{22} + q^{89}y^{22} + q^{88}y^{22} - q^{88}y^{19} - q^{87}y^{19} + (-\omega  - 1)q^{86}y^{19} + (-2\omega  - 2)q^{85}y^{19} + (-\omega  - 1)q^{84}y^{19} + 
            (-2\omega  + 1)q^{83}y^{19} + (-2\omega  + 1)q^{82}y^{19} + \omega q^{81}y^{19} + (\omega  + 2)q^{80}y^{19} + (-\omega  + 1)q^{79}y^{19} + (-\omega  - 1)q^{79}y^{16} + (\omega  - 1)q^{78}y^{19} + (-\omega  - 1)q^{78}y^{16}
            + (2\omega  + 1)q^{77}y^{19} + (-\omega  + 2)q^{77}y^{16} + q^{76}y^{19} + (-2\omega  + 2)q^{76}y^{16} + (\omega  - 2)q^{75}y^{19} + 2q^{75}y^{16} + (\omega  - 2)q^{74}y^{19} + (4\omega  + 6)q^{74}y^{16} + (-\omega  
            - 1)q^{73}y^{19} + (3\omega  + 6)q^{73}y^{16} + (-2\omega  - 2)q^{72}y^{19} + (5\omega  + 1)q^{72}y^{16} + (-\omega  - 1)q^{71}y^{19} + (6\omega  + 2)q^{71}y^{16} - \omega q^{70}y^{19} + (2\omega  + 2)q^{70}y^{16}
            - \omega q^{69}y^{19} + (\omega  - 4)q^{69}y^{16} + (3\omega  - 3)q^{68}y^{16} + (\omega  + 1)q^{68}y^{13} - 3\omega q^{67}y^{16} + (\omega  + 1)q^{67}y^{13} + (-3\omega  - 3)q^{66}y^{16} - 2q^{66}y^{13} - 
            3q^{65}y^{16} + (2\omega  - 1)q^{65}y^{13} + (-3\omega  + 3)q^{64}y^{16} + (\omega  + 1)q^{64}y^{13} + (-4\omega  + 1)q^{63}y^{16} + (-\omega  - 3)q^{63}y^{13} + (2\omega  + 2)q^{62}y^{16} + (3\omega  - 
            2)q^{62}y^{13} + (2\omega  + 6)q^{61}y^{16} + (2\omega  + 1)q^{61}y^{13} + (\omega  + 5)q^{60}y^{16} + (-2\omega  - 6)q^{60}y^{13} + (6\omega  + 3)q^{59}y^{16} - 9q^{59}y^{13} + (6\omega  + 4)q^{58}y^{16} + 
            (-3\omega  - 4)q^{58}y^{13} + 2\omega q^{57}y^{16} + (-9\omega  - 7)q^{57}y^{13} + (2\omega  - 2)q^{56}y^{16} + (-4\omega  - 7)q^{56}y^{13} + (2\omega  - 1)q^{55}y^{16} + (-2\omega  + 2)q^{55}y^{13} + (\omega  + 
            1)q^{55}y^{10} + (-\omega  - 1)q^{54}y^{16} + (-5\omega  + 1)q^{54}y^{13} + (-\omega  - 1)q^{54}y^{10} + (-\omega  - 1)q^{53}y^{16} + (\omega  - 5)q^{53}y^{13} + (-\omega  - 2)q^{53}y^{10} + (2\omega  - 
            2)q^{52}y^{13} - \omega q^{52}y^{10} + (-7\omega  - 4)q^{51}y^{13} - 4\omega q^{51}y^{10} + (-7\omega  - 9)q^{50}y^{13} + (-3\omega  - 2)q^{50}y^{10} + (-4\omega  - 3)q^{49}y^{13} + 4q^{49}y^{10} - 
            9\omega q^{48}y^{13} + (-3\omega  + 3)q^{48}y^{10} + (-6\omega  - 2)q^{47}y^{13} + (-4\omega  - 1)q^{47}y^{10} + (\omega  + 2)q^{46}y^{13} + (\omega  + 5)q^{46}y^{10} + (-2\omega  + 3)q^{45}y^{13} + (-4\omega  + 
            8)q^{45}y^{10} + (-3\omega  - 1)q^{44}y^{13} + (-2\omega  + 6)q^{44}y^{10} + (\omega  + 1)q^{43}y^{13} + (6\omega  + {13})q^{43}y^{10} + (-\omega  + 2)q^{42}y^{13} + (9\omega  + {18})q^{42}y^{10} - 
            2\omega q^{41}y^{13} + (9\omega  + 9)q^{41}y^{10} + (\omega  + 1)q^{40}y^{13} + ({18}\omega  + 9)q^{40}y^{10} + (\omega  + 1)q^{40}y^7 + (\omega  + 1)q^{39}y^{13} + ({13}\omega  + 6)q^{39}y^{10} + (\omega  + 2)q^{39}y^7 
            + (6\omega  - 2)q^{38}y^{10} + \omega q^{38}y^7 + (8\omega  - 4)q^{37}y^{10} + 4\omega q^{37}y^7 + (5\omega  + 1)q^{36}y^{10} + 4\omega q^{36}y^7 + (-\omega  - 4)q^{35}y^{10} - 5q^{35}y^7 + (3\omega  - 
            3)q^{34}y^{10} - 6q^{34}y^7 + 4\omega q^{33}y^{10} + (-\omega  - 4)q^{33}y^7 + (-2\omega  - 3)q^{32}y^{10} + (-6\omega  - 7)q^{32}y^7 - 4q^{31}y^{10} + (-4\omega  - 6)q^{31}y^7 - q^{30}y^{10} + (-\omega 
            - 1)q^{30}y^7 + (-2\omega  - 1)q^{29}y^{10} + (-5\omega  - 4)q^{29}y^7 + (-\omega  - 1)q^{28}y^{10} + (-4\omega  - 5)q^{28}y^7 + (\omega  + 1)q^{27}y^{10} + (-\omega  - 1)q^{27}y^7 + (-6\omega  - 
            4)q^{26}y^7 + (-7\omega  - 6)q^{25}y^7 + (-4\omega  - 1)q^{24}y^7 - 6\omega q^{23}y^7 - \omega q^{23}y^4 - 5\omega q^{22}y^7 - \omega q^{22}y^4 + 4q^{21}y^7 + 2q^{21}y^4 + 4q^{20}y^7 + 
            2q^{20}y^4 + q^{19}y^7 + q^{19}y^4 + (2\omega  + 1)q^{18}y^7 + (3\omega  + 3)q^{18}y^4 + (\omega  + 1)q^{17}y^7 + (2\omega  + 2)q^{17}y^4 + (2\omega  + 2)q^{15}y^4 + (3\omega  + 3)q^{14}y^4 + 
            \omega q^{13}y^4 + 2\omega q^{12}y^4 + 2\omega q^{11}y^4 - q^{10}y^4 - q^9y^4 + (-\omega  - 1)q^4y + (-\omega  - 1)q^3y\Big)x^4$

             $+ 
        \Big(q^{10}0y^{24} + q^{94}y^{21} + (\omega  + 1)q^{93}y^{21} + 2\omega q^{92}y^{21} + (2\omega  + 1)q^{91}y^{21} + \omega q^{90}y^{21} - 2q^{89}y^{21} - 2q^{86}y^{21} + (\omega  - 1)q^{86}y^{18} + (-\omega  - 1)q^{85}y^{21} 
            - q^{85}y^{18} + (-\omega  - 1)q^{84}y^{21} + (-\omega  - 2)q^{83}y^{21} + (-\omega  - 3)q^{83}y^{18} - \omega q^{82}y^{21} + (-\omega  - 2)q^{82}y^{18} - \omega q^{81}y^{21} + (-2\omega  - 1)q^{81}y^{18} + (-2\omega  - 
            4)q^{80}y^{18} + (-4\omega  - 4)q^{79}y^{18} - 5\omega q^{78}y^{18} + (-5\omega  - 3)q^{77}y^{18} - q^{77}y^{15} - 5\omega q^{76}y^{18} + (-\omega  - 2)q^{76}y^{15} + (-4\omega  + 4)q^{75}y^{18} - 
            2\omega q^{75}y^{15} + (-\omega  + 3)q^{74}y^{18} - 2\omega q^{74}y^{15} + (\omega  + 5)q^{73}y^{18} - 2\omega q^{73}y^{15} + (\omega  + 8)q^{72}y^{18} + (-2\omega  + 4)q^{72}y^{15} + (5\omega  + 4)q^{71}y^{18} + 
            4q^{71}y^{15} + (6\omega  + 4)q^{70}y^{18} + 2q^{70}y^{15} + (4\omega  + 4)q^{69}y^{18} + (-\omega  + 7)q^{69}y^{15} + (3\omega  - 1)q^{68}y^{18} + (3\omega  + 9)q^{68}y^{15} + (3\omega  - 1)q^{67}y^{18} + 
            (7\omega  + 9)q^{67}y^{15} + ({10}\omega  + {15})q^{66}y^{15} - q^{65}y^{18} + ({16}\omega  + {11})q^{65}y^{15} - q^{64}y^{18} + ({21}\omega  + 3)q^{64}y^{15} - q^{64}y^{12} + ({14}\omega  - 1)q^{63}y^{15} + 
            q^{63}y^{12} + (7\omega  - {12})q^{62}y^{15} + (2\omega  + 3)q^{62}y^{12} + (\omega  - {19})q^{61}y^{15} + (3\omega  + 1)q^{61}y^{12} + (-8\omega  - {13})q^{60}y^{15} + (3\omega  + 1)q^{60}y^{12} + (-{13}\omega  - 
            {13})q^{59}y^{15} + (4\omega  + 1)q^{59}y^{12} + (-8\omega  - {10})q^{58}y^{15} + (6\omega  - 4)q^{58}y^{12} + (-8\omega  - 1)q^{57}y^{15} + (\omega  - 6)q^{57}y^{12} + (-7\omega  - 1)q^{56}y^{15} + (-\omega  - 
            6)q^{56}y^{12} + (-3\omega  - 3)q^{55}y^{15} + (-\omega  - {13})q^{55}y^{12} + (-3\omega  + 2)q^{54}y^{15} + (-6\omega  - {12})q^{54}y^{12} + (-4\omega  + 1)q^{53}y^{15} + (-{13}\omega  - {12})q^{53}y^{12} - 
            \omega q^{52}y^{15} + (-{12}\omega  - {15})q^{52}y^{12} + 2q^{51}y^{15} + (-{16}\omega  - {10})q^{51}y^{12} + 2q^{50}y^{15} + (-{21}\omega  - 2)q^{50}y^{12} + \omega q^{49}y^{15} + (-{14}\omega  - 1)q^{49}y^{12} - 
            q^{49}y^9 + (-7\omega  + {11})q^{48}y^{12} + (-2\omega  - 3)q^{48}y^9 + (-5\omega  + {16})q^{47}y^{12} + (-4\omega  - 1)q^{47}y^9 + (4\omega  + {13})q^{46}y^{12} + (-4\omega  - 1)q^{46}y^9 + ({12}\omega  + 
            {16})q^{45}y^{12} + (-5\omega  + 1)q^{45}y^9 + (9\omega  + {15})q^{44}y^{12} + (-4\omega  + 7)q^{44}y^9 + ({11}\omega  + 6)q^{43}y^{12} + (\omega  + 8)q^{43}y^9 + ({11}\omega  + 4)q^{42}y^{12} + (6\omega  + 
            8)q^{42}y^9 + (7\omega  + 3)q^{41}y^{12} + (4\omega  + {10})q^{41}y^9 + (4\omega  - 3)q^{40}y^{12} + (7\omega  + 7)q^{40}y^9 + (3\omega  - 3)q^{39}y^{12} + ({10}\omega  + 5)q^{39}y^9 + (-\omega  - 
            3)q^{38}y^{12} + (8\omega  + 9)q^{38}y^9 + (-\omega  - 3)q^{37}y^{12} + (9\omega  + 4)q^{37}y^9 + (-\omega  - 2)q^{36}y^{12} + ({15}\omega  + 1)q^{36}y^9 - \omega q^{35}y^{12} + ({10}\omega  - 2)q^{35}y^9 + 
            (3\omega  - {10})q^{34}y^9 + (\omega  + 1)q^{34}y^6 - {15}q^{33}y^9 + \omega q^{33}y^6 + (-8\omega  - {12})q^{32}y^9 + \omega q^{32}y^6 + (-{13}\omega  - {11})q^{31}y^9 + (\omega  - 1)q^{31}y^6 + (-{10}\omega  - 
            6)q^{30}y^9 + (2\omega  - 3)q^{30}y^6 + (-7\omega  + 1)q^{29}y^9 - 3q^{29}y^6 + (-4\omega  + 3)q^{28}y^9 + (-3\omega  - 4)q^{28}y^6 + 3q^{27}y^9 + (-3\omega  - 6)q^{27}y^6 + (\omega  + 
            3)q^{26}y^9 + (-4\omega  - 3)q^{26}y^6 + (\omega  + 1)q^{25}y^9 + (-6\omega  - 1)q^{25}y^6 + \omega q^{24}y^9 + (-4\omega  - 1)q^{24}y^6 + (-2\omega  + 2)q^{23}y^6 + (-2\omega  + 4)q^{22}y^6 + (\omega  +
            4)q^{21}y^6 + (3\omega  + 4)q^{20}y^6 + (3\omega  + 3)q^{19}y^6 + (3\omega  + 1)q^{18}y^6 + 2\omega q^{17}y^6 - q^{16}y^6 + q^{16}y^3 - q^{15}y^6 + q^{15}y^3 + (\omega  + 1)q^{14}y^3 + (\omega  + 
            1)q^{13}y^3 + \omega q^{12}y^3 - q^{11}y^3 + (-\omega  - 1)q^{10}y^3 + (-\omega  - 1)q^9y^{30}\Big)x^5$ 

            $+\Big((-\omega  - 1)q^{95}y^{23} + (-\omega  - 1)q^{94}y^{23} + (-\omega  - 1)q^{89}y^{20} + (-2\omega  - 1)q^{88}y^{20} + (-\omega  + 2)q^{87}y^{20} + (-\omega  + 3)q^{86}y^{20} + 3q^{85}y^{20} + (4\omega  + 5)q^{84}y^{20} + 
            (5\omega  + 4)q^{83}y^{20} + 2\omega q^{82}y^{20} + 2\omega q^{81}y^{20} + (\omega  + 2)q^{81}y^{17} + 2\omega q^{80}y^{20} + (2\omega  + 3)q^{80}y^{17} - q^{79}y^{20} + (2\omega  + 1)q^{79}y^{17} + \omega q^{78}y^{20} + 
            (5\omega  + 2)q^{78}y^{17} + \omega q^{77}y^{20} + (7\omega  + 1)q^{77}y^{17} - q^{76}y^{20} + (3\omega  - 5)q^{76}y^{17} + (\omega  - 7)q^{75}y^{17} - 7q^{74}y^{17} + (-6\omega  - {11})q^{73}y^{17} + (-8\omega  - 
            {10})q^{72}y^{17} + (\omega  + 1)q^{72}y^{14} + (-5\omega  - 5)q^{71}y^{17} + (3\omega  + 2)q^{71}y^{14} + (-8\omega  - 7)q^{70}y^{17} + (2\omega  - 2)q^{70}y^{14} + (-9\omega  - 5)q^{69}y^{17} + (-\omega  - 
            6)q^{69}y^{14} - 5\omega q^{68}y^{17} + (-4\omega  - 7)q^{68}y^{14} + (-6\omega  - 1)q^{67}y^{17} + (-{11}\omega  - {10})q^{67}y^{14} + (-5\omega  + 1)q^{66}y^{17} + (-{17}\omega  - 8)q^{66}y^{14} + (-\omega  + 
            4)q^{65}y^{17} + (-{14}\omega  + 3)q^{65}y^{14} + 2q^{64}y^{17} + (-9\omega  + {10})q^{64}y^{14} + 2q^{63}y^{17} + (-4\omega  + {14})q^{63}y^{14} + (\omega  + 3)q^{62}y^{17} + (7\omega  + {21})q^{62}y^{14} + 
            ({13}\omega  + {18})q^{61}y^{14} + ({10}\omega  + {11})q^{60}y^{14} + (\omega  + 1)q^{59}y^{17} + ({13}\omega  + {13})q^{59}y^{14} + (\omega  + 2)q^{59}y^{11} + ({13}\omega  + 9)q^{58}y^{14} + (2\omega  + 2)q^{58}y^{11} + 
            ({11}\omega  + 5)q^{57}y^{14} + \omega q^{57}y^{11} + ({13}\omega  + 6)q^{56}y^{14} + (2\omega  + 1)q^{56}y^{11} + ({13}\omega  + 1)q^{55}y^{14} + (3\omega  + 1)q^{55}y^{11} + (6\omega  - 7)q^{54}y^{14} + (\omega  - 
            1)q^{54}y^{11} + (3\omega  - 4)q^{53}y^{14} + (2\omega  + 2)q^{53}y^{11} + (-\omega  - 7)q^{52}y^{14} + (7\omega  + 4)q^{52}y^{11} + (-5\omega  - 7)q^{51}y^{14} + (7\omega  - 2)q^{51}y^{11} - 2\omega q^{50}y^{14}
            + (6\omega  - 4)q^{50}y^{11} + (-\omega  + 1)q^{49}y^{14} + (6\omega  - 8)q^{49}y^{11} - q^{48}y^{14} + (-2\omega  - {17})q^{48}y^{11} + (\omega  + 1)q^{47}y^{14} + (-{10}\omega  - {16})q^{47}y^{11} + \omega q^{46}y^{14} 
            + (-{12}\omega  - {11})q^{46}y^{11} + (-\omega  - 2)q^{45}y^{14} + (-{13}\omega  - {10})q^{45}y^{11} - q^{45}y^8 + (-{10}\omega  - 2)q^{44}y^{11} + (-\omega  - 1)q^{44}y^8 + (-4\omega  + 3)q^{43}y^{11} - 
            \omega q^{43}y^8 + (-3\omega  - 3)q^{42}y^{11} + (-2\omega  - 1)q^{42}y^8 + (-3\omega  - 1)q^{41}y^{11} + (-3\omega  - 1)q^{41}y^8 - 4\omega q^{40}y^{11} + (-2\omega  + 1)q^{40}y^8 + (-7\omega  - 
            4)q^{39}y^{11} + (-\omega  + 1)q^{39}y^8 + (-5\omega  + 1)q^{38}y^{11} + (-3\omega  + 1)q^{38}y^8 + (-\omega  + 6)q^{37}y^{11} + (-\omega  + 5)q^{37}y^8 + 3q^{36}y^{11} + (\omega  + 3)q^{36}y^8 + (2\omega  +
            3)q^{35}y^{11} + (2\omega  + 4)q^{35}y^8 + (3\omega  + 3)q^{34}y^{11} + (3\omega  + 6)q^{34}y^8 + (\omega  - 1)q^{33}y^{11} + (3\omega  + 2)q^{33}y^8 - q^{32}y^{11} + 2\omega q^{32}y^8 + (2\omega  + 
            5)q^{31}y^8 + (2\omega  + 2)q^{30}y^8 + (2\omega  + 1)q^{29}y^8 + (4\omega  + 4)q^{28}y^8 + (3\omega  + 1)q^{27}y^8 + \omega q^{27}y^5 + (2\omega  - 2)q^{26}y^8 + (\omega  - 1)q^{26}y^5 + 
            \omega q^{25}y^8 - q^{24}y^8 + (-\omega  - 2)q^{23}y^8 - q^{23}y^5 + (\omega  + 1)q^{22}y^5 + \omega q^{21}y^5 - 3q^{20}y^5 + (-\omega  - 2)q^{19}y^5 + (-2\omega  - 1)q^{18}y^5 + (-2\omega  - 
            2)q^{17}y^5 - \omega q^{16}y^5 + q^{15}y^5 + q^9y^2\Big)x^6$

             $+\Big(\omega q^{88}y^{22} + \omega q^{82}y^{19} - q^{81}y^{19} + (-2\omega  - 2)q^{80}y^{19} + (-\omega  - 2)q^{79}y^{19} + (-2\omega  - 1)q^{78}y^{19} - 4\omega q^{77}y^{19} + (-\omega  + 1)q^{76}y^{19} + (\omega  + 2)q^{75}y^{19} + 
            2q^{74}y^{19} + (-2\omega  - 1)q^{74}y^{16} + (\omega  + 1)q^{73}y^{19} - \omega q^{73}y^{16} + \omega q^{72}y^{19} + q^{72}y^{16} + (-2\omega  + 3)q^{71}y^{16} + (\omega  + 4)q^{70}y^{16} + (6\omega  + 6)q^{69}y^{16} 
            + (5\omega  + 5)q^{68}y^{16} + (6\omega  + 3)q^{67}y^{16} + {10}\omega q^{66}y^{16} + (3\omega  - 3)q^{65}y^{16} - \omega q^{65}y^{13} - 5q^{64}y^{16} + (-\omega  + 1)q^{64}y^{13} - 6q^{63}y^{16} + (3\omega  + 
            3)q^{63}y^{13} + (-3\omega  - 4)q^{62}y^{16} + (4\omega  + 3)q^{62}y^{13} + (-5\omega  - 3)q^{61}y^{16} + 4\omega q^{61}y^{13} + (-\omega  - 1)q^{60}y^{16} + (7\omega  - 4)q^{60}y^{13} - \omega q^{59}y^{16} + (\omega  -
            {10})q^{59}y^{13} + (-\omega  + 1)q^{58}y^{16} + (-{10}\omega  - {15})q^{58}y^{13} + (-{13}\omega  - {15})q^{57}y^{13} + (-{17}\omega  - 9)q^{56}y^{13} - {20}\omega q^{55}y^{13} + (-8\omega  + 9)q^{54}y^{13} + (2\omega  + 
            {15})q^{53}y^{13} + (5\omega  + {15})q^{52}y^{13} + (-2\omega  - 1)q^{52}y^{10} + ({11}\omega  + {10})q^{51}y^{13} - \omega q^{51}y^{10} + ({11}\omega  + 4)q^{50}y^{13} + q^{50}y^{10} + 4\omega q^{49}y^{13} + (-2\omega  + 
            3)q^{49}y^{10} + (\omega  - 3)q^{48}y^{13} + (\omega  + 4)q^{48}y^{10} - 3q^{47}y^{13} + (6\omega  + 6)q^{47}y^{10} + (-2\omega  - 1)q^{46}y^{13} + (5\omega  + 5)q^{46}y^{10} - \omega q^{45}y^{13} + (6\omega  + 
            3)q^{45}y^{10} + {10}\omega q^{44}y^{10} + (3\omega  - 3)q^{43}y^{10} - 5q^{42}y^{10} - 6q^{41}y^{10} + (-3\omega  - 4)q^{40}y^{10} + (-5\omega  - 3)q^{39}y^{10} + (-\omega  - 1)q^{38}y^{10} + \omega q^{38}y^7 
            - \omega q^{37}y^{10} - q^{37}y^7 + (-\omega  + 1)q^{36}y^{10} + (-2\omega  - 2)q^{36}y^7 + (-\omega  - 2)q^{35}y^7 + (-2\omega  - 1)q^{34}y^7 - 4\omega q^{33}y^7 + (-\omega  + 1)q^{32}y^7 + (\omega  + 
            2)q^{31}y^7 + 2q^{30}y^7 + (\omega  + 1)q^{29}y^7 + \omega q^{28}y^7 + \omega q^{22}y^4\Big)x^7$

Defining $f' = v^{-1}u$, it can be verified computationally using Magma that $\sigma(f') = f'$, $f'g = qgf'$ and $k(x,y)^{\sigma} = k_q(f',g)$.

\section{Computation of prime ideals in $\GL{3}$}\label{s:magma_gl3}

In Theorem~\ref{res:checking the generators of primitives in GL3}, we use computation in Magma to verify that certain ideals are prime.  The ideals in question are 
\[Q_{\lambda} = \langle e_1 - \lambda_1 f_1, \dots, e_n - \lambda_n f_n\rangle \subset \GL{3}/I_{\omega}\]
as defined in \eqref{eq:prediction of generators of primitives}.  

Since $Q_{\lambda} \subseteq Q_{\lambda}B_{\omega} \cap \GL{3}/I_{\omega} = P_{\lambda} \cap \GL{3}/I_{\omega}$ and $P_{\lambda}$ is known to be a non-trivial ideal in $B_{\omega}$, we can conclude that $D \not\in Q_{\lambda}$.  The ideal $Q_{\lambda}$ is therefore prime in $\GL{3}/I_{\omega}$ if and only if $I_{\omega} + Q_{\lambda}$ is prime in $\GL{3}$ if and only if $(I_{\omega} + Q_{\lambda}) \cap \ML{3}$ is prime in $\ML{3}$.

We are only interested in the commutative algebra structure of $\ML{3}$ rather than the Poisson algebra structure, so we may view $\ML{3}$ as a polynomial ring in 9 variables.  It is now easy to verify that the four ideals $I_{\omega} + Q_{\lambda}$ are prime in $\ML{3}$ for the appropriate values of $\omega$, which we do as follows.

\begin{verbatim}
> field<i>:=CyclotomicField(4);
> K<x11,x12,x13,x21,x22,x23,x31,x32,x33>:=PolynomialRing(field,9);
                            
> Det:=x11*(x22*x33-x23*x32) - x12*(x21*x33-x23*x31) \
   + x13*(x21*x32-x22*x31);
> M13:=x21*x32 - x22*x31;
> M31:=x12*x23 - x22*x13;
> M21:=x12*x33 - x13*x32;
> M32:=x11*x23 - x13*x21;
> 
> // (321,321)
> I1:=ideal<K|Det-1,M13-x13,M31-x31>;
> 
> // (321,312)
> I2:=ideal<K|Det-1,x13,x12*x23-x31>;
> 
> // (231,231)
> I3:=ideal<K|Det-1,x31,M31,M21-x21,M32-x32>;
> 
> // (132,312)
> I4:=ideal<K|Det-1,x13,x21,x31,x11*x32-x23>;
> 
> IsPrime(I1);
true
> IsPrime(I2);
true
> IsPrime(I3);
true
> IsPrime(I4);
true
\end{verbatim}

\chapter{$\HH$-prime Figures}\label{c:H-prime figures}
\begin{figure}[h!]
\setlength{\tabcolsep}{2pt}

\centering
\begin{tabular}{cc|cccccc}

& $\omega_{-}$ & \multirow{2}{*}{321} & \multirow{2}{*}{231} & \multirow{2}{*}{312} & \multirow{2}{*}{132}& \multirow{2}{*}{213}& \multirow{2}{*}{123} \\ 
$\omega_{+}$& &&&&&& \\ \hline 

\multicolumn{2}{c|} {\raisebox{1em}{321}} &

\begin{smallarray}{m321321}
{
 \circ & \circ & \circ \\
\circ & \circ & \circ \\
\circ & \circ & \circ \\
};
\end{smallarray}
&

\begin{smallarray}{m321231}
{
 \circ & & \\
\circ & & \\
\circ & \circ & \circ \\
};
\MyZ(1,2)
\end{smallarray}

&
\begin{smallarray}{m321312}
{
 \circ & \circ & \bullet \\
\circ & \circ & \circ \\
\circ & \circ & \circ \\
};
\end{smallarray}
&

\begin{smallarray}{m321132}
{
 \circ & \bullet & \bullet \\
\circ & \circ & \circ \\
\circ & \circ & \circ \\
};
\end{smallarray}
&

\begin{smallarray}{m321213}
{
 \circ & \circ & \bullet \\
\circ & \circ & \bullet \\
\circ & \circ & \circ \\
};
\end{smallarray}

&
\begin{smallarray}{m321123}
{
 \circ & \bullet & \bullet \\
\circ & \circ & \bullet \\
\circ & \circ & \circ \\
};
\end{smallarray}
\\

\multicolumn{2}{c|}{\raisebox{1em}{231}}
&
\begin{smallarray}{m231321}
{
 \circ & \circ & \circ \\
\circ & \circ & \circ \\
\bullet & \circ & \circ \\
};
\end{smallarray}

&
\begin{smallarray}{m231231}
{
 \circ &  &  \\
\circ &  &  \\
\bullet & \circ & \circ \\
};
\MyZ(1,2)
\end{smallarray}

&
\begin{smallarray}{m231312}
{
 \circ & \circ & \bullet \\
\circ & \circ & \circ \\
\bullet & \circ & \circ \\
};
\end{smallarray}
&

\begin{smallarray}{m231132}
{
 \circ & \bullet & \bullet \\
\circ & \circ & \circ \\
\bullet & \circ & \circ \\
};
\end{smallarray}

&
\begin{smallarray}{m231213}
{
 \circ & \circ & \bullet \\
\circ & \circ & \bullet \\
\bullet & \circ & \circ \\
};
\end{smallarray}

&
\begin{smallarray}{m231123}
{
 \circ & \bullet & \bullet \\
\circ & \circ & \bullet \\
\bullet & \circ & \circ \\
};
\end{smallarray}

\\

\multicolumn{2}{c|}{\raisebox{1em}{312}} 
&
\begin{smallarray}{m312321}
{
 \circ & \circ & \circ \\
 &  & \circ \\
 &  & \circ \\
};
\MyZ(2,1)
\end{smallarray}

&
\begin{smallarray}{m312231}
{
 \circ &  &  \\
 &  &  \\
 &  & \circ \\
};
\MyZ(1,2)
\MyZ(2,1)
\end{smallarray}

&
\begin{smallarray}{m312312}
{
 \circ & \circ & \bullet \\
 &  & \circ \\
 &  & \circ \\
};
\MyZ(2,1)
\end{smallarray}

&
\begin{smallarray}{m312132}
{
 \circ & \bullet & \bullet \\
 &  & \circ \\
 &  & \circ \\
};
\MyZ(2,1)
\end{smallarray}

&
\begin{smallarray}{m312213}
{
 \circ & \circ & \bullet \\
 &  & \bullet \\
 &  & \circ \\
};
\MyZ(2,1)
\end{smallarray}

&
\begin{smallarray}{m312123}
{
 \circ & \bullet & \bullet \\
 &  & \bullet \\
 &  & \circ \\
};
\MyZ(2,1)
\end{smallarray}

\\

\multicolumn{2}{c|}{\raisebox{1em}{132}} 

&
\begin{smallarray}{m132321}
{
 \circ & \circ & \circ \\
 \bullet & \circ & \circ \\
 \bullet & \circ & \circ \\
};
\end{smallarray}
&
\begin{smallarray}{m132231}
{
 \circ &  &  \\
 \bullet &  &  \\
 \bullet & \circ & \circ \\
};
\MyZ(1,2)
\end{smallarray}

&
\begin{smallarray}{m132312}
{
 \circ & \circ & \bullet \\
 \bullet & \circ & \circ \\
 \bullet & \circ & \circ \\
};
\end{smallarray}
&
\begin{smallarray}{m132132}
{
 \circ & \bullet & \bullet \\
 \bullet & \circ & \circ \\
 \bullet & \circ & \circ \\
};
\end{smallarray}
&
\begin{smallarray}{m132213}
{
 \circ & \circ & \bullet \\
 \bullet & \circ & \bullet \\
 \bullet & \circ & \circ \\
};
\end{smallarray}
&
\begin{smallarray}{m132123}
{
 \circ & \bullet & \bullet \\
 \bullet & \circ & \bullet \\
 \bullet & \circ & \circ \\
};
\end{smallarray}
\\

\multicolumn{2}{c|}{\raisebox{1em}{213}} 

&

\begin{smallarray}{m213321}
{
 \circ & \circ & \circ \\
 \circ & \circ & \circ \\
 \bullet & \bullet & \circ \\
};
\end{smallarray}
&
\begin{smallarray}{m213231}
{
 \circ &  &  \\
 \circ &  &  \\
 \bullet & \bullet & \circ \\
};
\MyZ(1,2)
\end{smallarray}

&
\begin{smallarray}{m213312}
{
 \circ & \circ & \bullet \\
 \circ & \circ & \circ \\
 \bullet & \bullet & \circ \\
};
\end{smallarray}
&
\begin{smallarray}{m213132}
{
 \circ & \bullet & \bullet \\
 \circ & \circ & \circ \\
 \bullet & \bullet & \circ \\
};
\end{smallarray}
&
\begin{smallarray}{m213213}
{
 \circ & \circ & \bullet \\
 \circ & \circ & \bullet \\
 \bullet & \bullet & \circ \\
};
\end{smallarray}
&
\begin{smallarray}{m213123}
{
 \circ & \bullet & \bullet \\
 \circ & \circ & \bullet \\
 \bullet & \bullet & \circ \\
};
\end{smallarray}
\\

\multicolumn{2}{c|}{\raisebox{1em}{123}}
&
\begin{smallarray}{m123321}
{
 \circ & \circ & \circ \\
 \bullet & \circ & \circ \\
 \bullet & \bullet & \circ \\
};
\end{smallarray}
&
\begin{smallarray}{m123231}
{
 \circ & & \\
 \bullet &  &  \\
 \bullet & \bullet & \circ \\
};
\MyZ(1,2)
\end{smallarray}

&
\begin{smallarray}{m123312}
{
 \circ & \circ & \bullet \\
 \bullet & \circ & \circ \\
 \bullet & \bullet & \circ \\
};
\end{smallarray}
&
\begin{smallarray}{m123132}
{
 \circ & \bullet & \bullet \\
 \bullet & \circ & \circ \\
 \bullet & \bullet & \circ \\
};
\end{smallarray}
&
\begin{smallarray}{m123213}
{
 \circ & \circ & \bullet \\
 \bullet & \circ & \bullet \\
 \bullet & \bullet & \circ \\
};
\end{smallarray}
&
\begin{smallarray}{m123123}
{
 \circ & \bullet & \bullet \\
 \bullet & \circ & \bullet \\
 \bullet & \bullet & \circ \\
};
\end{smallarray}

\end{tabular}

\caption{Generators for $\HH$-primes in $\QGL{3}$ and $\GL{3}$.}\label{fig:H_primes_gens}
\end{figure}

This figure is reproduced from \cite[Figure~1]{GL1} and represents the 36 $\HH$-primes in $\QGL{3}$ and $\GL{3}$.  Each ideal is represented pictorially by a $3 \times 3$ grid of dots: a black dot in position $(i,j)$ denotes the element $X_{ij}$, and a square represents a $2\times2$ (quantum) minor in the natural way.  For example, the ideal in position $(231,231)$ denotes the ideal generated by $X_{31}$ and $[\wt{3}|\wt{1}]_q$ in $\QGL{3}$, or the ideal generated by $x_{31}$ and $[\wt{3}|\wt{1}]$ in $\GL{3}$, as appropriate.

These ideals are indexed by $\omega = (\omega_{+}, \omega_{-}) \in S_3 \times S_3$, following the notation of \cite{GL1}.

\clearpage
\begin{figure}
 \setlength{\tabcolsep}{12pt}
\centering

\begin{tabular}{ccccccc}
\hline
\begin{smallarray}{m321321}
{
 \circ & \circ & \circ \\
\circ & \circ & \circ \\
\circ & \circ & \circ \\
};
\end{smallarray}
& & & & & &

\\[-4px]
\small(321,321)
& & & & & &

\\
\hline
\begin{smallarray}{m321312}
{
 \circ & \circ & \bullet \\
\circ & \circ & \circ \\
\circ & \circ & \circ \\
};
\end{smallarray}
&
$\stackrel{\longrightarrow}{\tau}$
&
\begin{smallarray}{m231321}
{
 \circ & \circ & \circ \\
\circ & \circ & \circ \\
\bullet & \circ & \circ \\
};
\end{smallarray}

&
$\stackrel{\longrightarrow}{S}$
&
\begin{smallarray}{m312321}
{
 \circ & \circ & \circ \\
 &  & \circ \\
 &  & \circ \\
};
\MyZ(2,1)
\end{smallarray}
&
$\stackrel{\longrightarrow}{\tau}$
& 
\begin{smallarray}{m321231}
{
 \circ & & \\
\circ & & \\
\circ & \circ & \circ \\
};
\MyZ(1,2)
\end{smallarray}
\\[-4px]
\small(321,312)
& &
\small(231,321)
& & 
\small(312,321)
& &
\small(321,231)
\\

\hline
\begin{smallarray}{m231231}
{
 \circ &  &  \\
\circ &  &  \\
\bullet & \circ & \circ \\
};
\MyZ(1,2)
\end{smallarray}
&
$\stackrel{\longrightarrow}{\tau}$
& 
\begin{smallarray}{m312312}
{
 \circ & \circ & \bullet \\
 &  & \circ \\
 &  & \circ \\
};
\MyZ(2,1)
\end{smallarray}
& & & & 
\\[-4px]
\small(231,231)
& & 
\small(312,312)
& & & & 
\\
\hline
\begin{smallarray}{m231312}
{
 \circ & \circ & \bullet \\
\circ & \circ & \circ \\
\bullet & \circ & \circ \\
};
\end{smallarray}
&$\stackrel{\longrightarrow}{S}$& 
\begin{smallarray}{m312231}
{
 \circ &  &  \\
 &  &  \\
 &  & \circ \\
};
\MyZ(1,2)
\MyZ(2,1)
\end{smallarray}
& & & &
\\[-4px]
\small(231,312)
& & 
\small(312,231)
& & & &
\\
\hline
\begin{smallarray}{m321132}
{
 \circ & \bullet & \bullet \\
\circ & \circ & \circ \\
\circ & \circ & \circ \\
};
\end{smallarray}
& 
$\stackrel{\longrightarrow}{\tau}$
&
\begin{smallarray}{m132321}
{
 \circ & \circ & \circ \\
 \bullet & \circ & \circ \\
 \bullet & \circ & \circ \\
};
\end{smallarray}
&
$\stackrel{\longrightarrow}{\rho}$
&
\begin{smallarray}{m213321}
{
 \circ & \circ & \circ \\
 \circ & \circ & \circ \\
 \bullet & \bullet & \circ \\
};
\end{smallarray}
&
$\stackrel{\longrightarrow}{\tau}$
&
\begin{smallarray}{m321213}
{
 \circ & \circ & \bullet \\
\circ & \circ & \bullet \\
\circ & \circ & \circ \\
};
\end{smallarray}
\\[-4px]
\small(321,132)
& & 
\small(132,321)
& &
\small(213,321)
& & 
\small(321,213)
\\
\hline
\begin{smallarray}{m321123}
{
 \circ & \bullet & \bullet \\
\circ & \circ & \bullet \\
\circ & \circ & \circ \\
};
\end{smallarray}
&
$\stackrel{\longrightarrow}{\tau}$
&
\begin{smallarray}{m123321}
{
 \circ & \circ & \circ \\
 \bullet & \circ & \circ \\
 \bullet & \bullet & \circ \\
};
\end{smallarray}
& & & & 
\\[-4px]
\small(321,123)
& &
\small(123,321)
& & & &
\\
\hline
\begin{smallarray}{m132312}
{
 \circ & \circ & \bullet \\
 \bullet & \circ & \circ \\
 \bullet & \circ & \circ \\
};
\end{smallarray}
& 
$\stackrel{\longrightarrow}{\tau}$
&
\begin{smallarray}{m231132}
{
 \circ & \bullet & \bullet \\
\circ & \circ & \circ \\
\bullet & \circ & \circ \\
};
\end{smallarray}
&
$\stackrel{\longrightarrow}{\rho}$
&
\begin{smallarray}{m231213}
{
 \circ & \circ & \bullet \\
\circ & \circ & \bullet \\
\bullet & \circ & \circ \\
};
\end{smallarray}
&
$\stackrel{\longrightarrow}{\tau}$
&
\begin{smallarray}{m213312}
{
 \circ & \circ & \bullet \\
 \circ & \circ & \circ \\
 \bullet & \bullet & \circ \\
};
\end{smallarray}
\\[-4px]
\small(132,312)
& & 
\small(231,132)
& &
\small(231,213)
& &
\small(213,312)
\\
\parbox[c][3em][c]{3em}{$S \circ \rho \downarrow$}
& &
$S \circ \rho \downarrow$
& &
$S \circ \rho \downarrow$
& &
$S \circ \rho \downarrow$
\\
\begin{smallarray}{m213231}
{
 \circ &  &  \\
 \circ &  &  \\
 \bullet & \bullet & \circ \\
};
\MyZ(1,2)
\end{smallarray}
&
$\stackrel{\longrightarrow}{\tau}$
&
\begin{smallarray}{m312213}
{
 \circ & \circ & \bullet \\
 &  & \bullet \\
 &  & \circ \\
};
\MyZ(2,1)
\end{smallarray}
&
$\stackrel{\longrightarrow}{\rho}$
&
\begin{smallarray}{m312132}
{
 \circ & \bullet & \bullet \\
 &  & \circ \\
 &  & \circ \\
};
\MyZ(2,1)
\end{smallarray}
&
$\stackrel{\longrightarrow}{\tau}$
&
\begin{smallarray}{m132231}
{
 \circ &  &  \\
 \bullet &  &  \\
 \bullet & \circ & \circ \\
};
\MyZ(1,2)
\end{smallarray}
\\[-4px]
\small(213,231)
& &
\small(312,213)
& &
\small(312,132)
& &
\small(132,231)
\\
\hline

\begin{smallarray}{m132132}
{
 \circ & \bullet & \bullet \\
 \bullet & \circ & \circ \\
 \bullet & \circ & \circ \\
};
\end{smallarray}
&
$\stackrel{\longrightarrow}{\rho}$
&
\begin{smallarray}{m213213}
{
 \circ & \circ & \bullet \\
 \circ & \circ & \bullet \\
 \bullet & \bullet & \circ \\
};
\end{smallarray}
& & & &
\\[-4px]
\small(132,132)
& &
\small(213,213)
& & & &
\\
\hline
\begin{smallarray}{m123312}
{
 \circ & \circ & \bullet \\
 \bullet & \circ & \circ \\
 \bullet & \bullet & \circ \\
};
\end{smallarray}
&
$\stackrel{\longrightarrow}{\tau}$
&
\begin{smallarray}{m231123}
{
 \circ & \bullet & \bullet \\
\circ & \circ & \bullet \\
\bullet & \circ & \circ \\
};
\end{smallarray}
&
$\stackrel{\longrightarrow}{S}$
&
\begin{smallarray}{m312123}
{
 \circ & \bullet & \bullet \\
 &  & \bullet \\
 &  & \circ \\
};
\MyZ(2,1)
\end{smallarray}
&
$\stackrel{\longrightarrow}{\tau}$
&
\begin{smallarray}{m123231}
{
 \circ & & \\
 \bullet &  &  \\
 \bullet & \bullet & \circ \\
};
\MyZ(1,2)
\end{smallarray}
\\[-4px]
\small(123,312)
& & 
\small(231,123)
& &
\small(312,123)
& &
\small(123,231)
\\
\hline

\begin{smallarray}{m213132}
{
 \circ & \bullet & \bullet \\
 \circ & \circ & \circ \\
 \bullet & \bullet & \circ \\
};
\end{smallarray}
& 
$\stackrel{\longrightarrow}{\tau}$
&
\begin{smallarray}{m132213}
{
 \circ & \circ & \bullet \\
 \bullet & \circ & \bullet \\
 \bullet & \circ & \circ \\
};
\end{smallarray}
& & & &
\\[-4px]
\small(213,132)
& & 
\small(132,213)
& & & &
\\
\hline
\begin{smallarray}{m123132}
{
 \circ & \bullet & \bullet \\
 \bullet & \circ & \circ \\
 \bullet & \bullet & \circ \\
};
\end{smallarray}
&
$\stackrel{\longrightarrow}{\tau}$
&
\begin{smallarray}{m132123}
{
 \circ & \bullet & \bullet \\
 \bullet & \circ & \bullet \\
 \bullet & \circ & \circ \\
};
\end{smallarray}
&
$\stackrel{\longrightarrow}{\rho}$
&
\begin{smallarray}{m213123}
{
 \circ & \bullet & \bullet \\
 \circ & \circ & \bullet \\
 \bullet & \bullet & \circ \\
};
\end{smallarray}
&
$\stackrel{\longrightarrow}{\tau}$
&
\begin{smallarray}{m123213}
{
 \circ & \circ & \bullet \\
 \bullet & \circ & \bullet \\
 \bullet & \bullet & \circ \\
};
\end{smallarray}
\\[-4px]
\small(123,132)
& & 
\small(132,123)
& &
\small(213,123)
& & 
\small(123,213)
\\
\hline
\begin{smallarray}{m123123}
{
 \circ & \bullet & \bullet \\
 \bullet & \circ & \bullet \\
 \bullet & \bullet & \circ \\
};
\end{smallarray}
& & & & & &
\\[-4px]
\small(123,123)
& & & & & & \\
\hline
\end{tabular}
\caption{$\HH$-primes grouped by orbit}\label{fig:H_primes_nice}
\end{figure}

\clearpage
\begin{figure}
\centering
\begin{tabular}{ccm{2em}m{12em}m{10em}}
& $(\omega_{+},\omega_{-})$ && Simplified generators for original Ore set & Additional generators \\
\hline
\begin{smallarray}{m321321}
{
 \circ & \circ & \circ \\
\circ & \circ & \circ \\
\circ & \circ & \circ \\
};
\end{smallarray} & 
$(321,321)$ && $X_{31}$, $X_{13}$, $[\wt{1}|\wt{3}]_q$, $[\wt{3}|\wt{1}]_q$ &$X_{11}$, $X_{12}$, $X_{21}$, $[\wt{3}|\wt{3}]_q$ \\ 
\begin{smallarray}{m321312}
{
 \circ & \circ & \bullet \\
\circ & \circ & \circ \\
\circ & \circ & \circ \\
};
\end{smallarray} &
$(321,312)$ && $X_{31}$, $[\wt{1}|\wt{3}]_q$, $X_{23}$, $X_{12}$ & $X_{11}$, $X_{21}$, $[\wt{3}|\wt{3}]_q$ \\ 
\begin{smallarray}{m231231}
{
 \circ &  &  \\
\circ &  &  \\
\bullet & \circ & \circ \\
};
\MyZ(1,2)
\end{smallarray} &
$(231,231)$ && $X_{21}$, $X_{32}$, $[\wt{2}|\wt{1}]_q$, $[\wt{3}|\wt{2}]_q$ & $X_{33}$, $[\wt{1}|\wt{1}]_q$ \\
\begin{smallarray}{m231312}
{
 \circ & \circ & \bullet \\
\circ & \circ & \circ \\
\bullet & \circ & \circ \\
};
\end{smallarray} & 
$(231,312)$ && $X_{21}$, $ X_{32}$, $ X_{23}$, $X_{12}$ &  $[\wt{1}|\wt{1}]_q$, $X_{33}$\\
\begin{smallarray}{m321132}
{
 \circ & \bullet & \bullet \\
\circ & \circ & \circ \\
\circ & \circ & \circ \\
};
\end{smallarray} & 
$(321,132)$ && $X_{31}$, $[\wt{1}|\wt{3}]_q$, $X_{23}$ & $X_{32}$, $X_{33}$\\
\begin{smallarray}{m321123}
{
 \circ & \bullet & \bullet \\
\circ & \circ & \bullet \\
\circ & \circ & \circ \\
};
\end{smallarray} & 
$(321,123)$ && $X_{31}$, $[\wt{1}|\wt{3}]_q$ & $X_{21}$\\
\begin{smallarray}{m132312}
{
 \circ & \circ & \bullet \\
 \bullet & \circ & \circ \\
 \bullet & \circ & \circ \\
};
\end{smallarray} & 
$(132,312)$ && $X_{32}$, $ X_{23}$, $X_{12}$ & $X_{33}$ \\
\begin{smallarray}{m132132}
{
 \circ & \bullet & \bullet \\
 \bullet & \circ & \circ \\
 \bullet & \circ & \circ \\
};
\end{smallarray} & 
$(132,132)$ && $X_{32}$, $X_{23}$ & $X_{33}$\\
\begin{smallarray}{m123312}
{
 \circ & \circ & \bullet \\
 \bullet & \circ & \circ \\
 \bullet & \bullet & \circ \\
};
\end{smallarray} & 
$(123,312)$ && $X_{23}$, $X_{12}$ & \\
\begin{smallarray}{m213132}
{
 \circ & \bullet & \bullet \\
 \circ & \circ & \circ \\
 \bullet & \bullet & \circ \\
};
\end{smallarray} &
$(213,132)$ && $X_{21}$, $X_{23}$ & \\
\begin{smallarray}{m123132}
{
 \circ & \bullet & \bullet \\
 \bullet & \circ & \circ \\
 \bullet & \bullet & \circ \\
};
\end{smallarray} &
$(123,132)$ && $X_{23}$ & \\
\begin{smallarray}{m123123}
{
 \circ & \bullet & \bullet \\
 \bullet & \circ & \bullet \\
 \bullet & \bullet & \circ \\
};
\end{smallarray} &
$(123,123)$ &&  &
\end{tabular}
\caption{Definitive sets of generators for the Ore sets $E_{\omega}$.}\label{fig:Ore_sets_for_tori}
\end{figure}
$E_{\omega}$ is defined to be the multiplicative set generated by all of the elements in the row corresponding to $I_{\omega}$; those cases not listed explicitly here can be obtained by applying the appropriate combination of $\tau$, $\rho$ and $S$ from Figure~\ref{fig:H_primes_nice}.  Elements of $E_{\omega}$ are considered as coset representatives in $\GL{3}/I_{\omega}$.

$E_{\omega}$ also denotes the corresponding multiplcative set in $\GL{3}/I_{\omega}$, where we replace each $X_{ij}$ with $x_{ij}$ and $[\wt{i}|\wt{j}]_q$ with $[\wt{i}|\wt{j}]$.

\clearpage
\begin{figure}
\centering
\begin{tabular}{ccp{3em}p{22em}}
& $(\omega_{+},\omega_{-})$ &  & Localization is quantum torus on these generators \\
\hline
\begin{smallarray}{m321321}
{
 \circ & \circ & \circ \\
\circ & \circ & \circ \\
\circ & \circ & \circ \\
};
\end{smallarray} & 
$(321,321)$ & & 
$X_{11}$, $X_{12}$, $X_{13}$, $X_{21}$, $X_{31}$, $[\wt{3}|\wt{3}]_q$, $[\wt{3}|\wt{1}]_q$, $[\wt{1}|\wt{3}]_q$, $Det_q$
\\ 
\begin{smallarray}{m321312}
{
 \circ & \circ & \bullet \\
\circ & \circ & \circ \\
\circ & \circ & \circ \\
};
\end{smallarray} &
$(321,312)$ &  & 
$X_{11}$, $X_{12}$, $X_{21}$, $X_{23}$, $X_{31}$, $[\wt{3}|\wt{3}]_q$, $[\wt{1}|\wt{3}]_q$, $Det_q$
\\ 
\begin{smallarray}{m231231}
{
 \circ &  &  \\
\circ &  &  \\
\bullet & \circ & \circ \\
};
\MyZ(1,2)
\end{smallarray} &
$(231,231)$ &  & 
$X_{13}$, $X_{21}$, $X_{32}$, $X_{33}$, $[\wt{2}|\wt{1}]_q$, $[\wt{1}|\wt{1}]_q$, $Det_q$
\\
\begin{smallarray}{m231312}
{
 \circ & \circ & \bullet \\
\circ & \circ & \circ \\
\bullet & \circ & \circ \\
};
\end{smallarray} & 
$(231,312)$ & & 
$X_{12}$, $X_{21}$, $X_{23}$, $X_{32}$, $X_{33}$, $[\wt{1}|\wt{1}]_q$, $Det_q$
\\
\begin{smallarray}{m321132}
{
 \circ & \bullet & \bullet \\
\circ & \circ & \circ \\
\circ & \circ & \circ \\
};
\end{smallarray} & 
$(321,132)$ & & 
$X_{11}$, $X_{23}$, $X_{31}$, $X_{32}$, $X_{33}$, $[\wt{1}|\wt{3}]_q$, $[\wt{1}|\wt{1}]_q$
\\
\begin{smallarray}{m321123}
{
 \circ & \bullet & \bullet \\
\circ & \circ & \bullet \\
\circ & \circ & \circ \\
};
\end{smallarray} & 
$(321,123)$ & & 
$X_{11}$, $X_{21}$, $X_{22}$, $X_{31}$, $[\wt{1}|\wt{3}]_q$, $X_{33}$
\\
\begin{smallarray}{m132312}
{
 \circ & \circ & \bullet \\
 \bullet & \circ & \circ \\
 \bullet & \circ & \circ \\
};
\end{smallarray} & 
$(132,312)$ &  & 
$X_{11}$, $X_{12}$, $X_{23}$, $X_{32}$, $X_{33}$, $[\wt{1}|\wt{1}]_q$
\\
\begin{smallarray}{m132132}
{
 \circ & \bullet & \bullet \\
 \bullet & \circ & \circ \\
 \bullet & \circ & \circ \\
};
\end{smallarray} & 
$(132,132)$ & & 
$X_{11}$, $X_{23}$, $X_{32}$, $X_{33}$, $[\wt{1}|\wt{1}]_q$
\\
\begin{smallarray}{m123312}
{
 \circ & \circ & \bullet \\
 \bullet & \circ & \circ \\
 \bullet & \bullet & \circ \\
};
\end{smallarray} & 
$(123,312)$ &  & 
$X_{11}$, $X_{12}$, $X_{22}$, $X_{23}$, $X_{33}$
\\
\begin{smallarray}{m213132}
{
 \circ & \bullet & \bullet \\
 \circ & \circ & \circ \\
 \bullet & \bullet & \circ \\
};
\end{smallarray} &
$(213,132)$ &  & 
$X_{11}$, $X_{21}$, $X_{22}$, $X_{23}$, $X_{33}$
\\
\begin{smallarray}{m123132}
{
 \circ & \bullet & \bullet \\
 \bullet & \circ & \circ \\
 \bullet & \bullet & \circ \\
};
\end{smallarray} &
$(123,132)$ &   & 
$X_{11}$, $X_{22}$, $X_{23}$, $X_{33}$
\\
\begin{smallarray}{m123123}
{
 \circ & \bullet & \bullet \\
 \bullet & \circ & \bullet \\
 \bullet & \bullet & \circ \\
};
\end{smallarray} &
$(123,123)$ & & 
$X_{11}$, $X_{22}$, $X_{33}$
\end{tabular}
\caption{Generators for the quantum tori $A_{\omega} = \qr{\QGL{3}}{I_{\omega}}\big[E_{\omega}^{-1}\big]$.}\label{fig:localizations}
\end{figure}
The localizations $\QGL{3}/I_{\omega}$ at the Ore sets $E_{\omega}$ listed in Figure~\ref{fig:Ore_sets_for_tori} are computed in \cite[\S4]{GL1}, and we reproduce this information here for convenience.  

$A_{\omega}$ is always isomorphic to a quantum torus $k_{\mathbf{q}}[R_1^{\pm1}, \dots, R_m^{\pm1}]$, and Figure~\ref{fig:localizations} lists a choice for the generators $R_i$ for each case (those not listed explicitly here can be obtained by applying the appropriate combination of $\tau$, $\rho$ and $S$ from Figure~\ref{fig:H_primes_nice}).  The $q$-commuting relations $R_iR_j = q^{a_{ij}}R_jR_i$ are not needed for this thesis, but can easily be computed from the relations in $\QGL{3}$.

By Proposition~\ref{res:the poisson localizations are scls}, Figure~\ref{fig:localizations} also describes sets of generators for the Poisson algebras $B_{\omega}$, subject to replacing $X_{ij}$ with $x_{ij}$, $[\wt{i}|\wt{j}]_q$ with $[\wt{i}|\wt{j}]$ and $Det_q$ with $Det$.

\clearpage
\begin{figure}
\centering
\begin{tabular}{ccm{2em}l}
& $(\omega_{+},\omega_{-})$ & &  Generators of the centre \\
\hline
\begin{smallarray}{m321321}
{
 \circ & \circ & \circ \\
\circ & \circ & \circ \\
\circ & \circ & \circ \\
};
\end{smallarray} & 
$(321,321)$ & &  
$Det_q$,\ \ $[\wt{1}|\wt{3}]X_{13}^{-1}$,\ \ $[\wt{3}|\wt{1}]X_{31}^{-1}$\\ 
\begin{smallarray}{m321312}
{
 \circ & \circ & \bullet \\
\circ & \circ & \circ \\
\circ & \circ & \circ \\
};
\end{smallarray} &
$(321,312)$ & & 
$Det_q$,\ \ $X_{12}X_{23}X_{31}^{-1}$ \\ 
\begin{smallarray}{m231231}
{
 \circ &  &  \\
\circ &  &  \\
\bullet & \circ & \circ \\
};
\MyZ(1,2)
\end{smallarray} &
$(231,231)$ & &  
$Det_q$,\ \ $[\wt{2}|\wt{1}]X_{21}^{-1}$,\ \ $[\wt{3}|\wt{2}]X_{32}^{-1}$\\
\begin{smallarray}{m231312}
{
 \circ & \circ & \bullet \\
\circ & \circ & \circ \\
\bullet & \circ & \circ \\
};
\end{smallarray} & 
$(231,312)$ & & 
$Det_q$\\
\begin{smallarray}{m321132}
{
 \circ & \bullet & \bullet \\
\circ & \circ & \circ \\
\circ & \circ & \circ \\
};
\end{smallarray} & 
$(321,132)$ & & 
$Det_q$\\
\begin{smallarray}{m321123}
{
 \circ & \bullet & \bullet \\
\circ & \circ & \bullet \\
\circ & \circ & \circ \\
};
\end{smallarray} & 
$(321,123)$ & & 
$Det_q$,\ \ $X_{22}[\wt{1}|\wt{3}]X_{31}^{-1}$\\
\begin{smallarray}{m132312}
{
 \circ & \circ & \bullet \\
 \bullet & \circ & \circ \\
 \bullet & \circ & \circ \\
};
\end{smallarray} & 
$(132,312)$ & &  
$Det_q$,\ \ $X_{11}X_{32}X_{23}^{-1}$\\
\begin{smallarray}{m132132}
{
 \circ & \bullet & \bullet \\
 \bullet & \circ & \circ \\
 \bullet & \circ & \circ \\
};
\end{smallarray} & 
$(132,132)$ & & 
$Det_q$,\ \ $X_{11}$,\ \ $X_{23}X_{32}^{-1}$\\
\begin{smallarray}{m123312}
{
 \circ & \circ & \bullet \\
 \bullet & \circ & \circ \\
 \bullet & \bullet & \circ \\
};
\end{smallarray} & 
$(123,312)$ & &  
$Det_q$\\
\begin{smallarray}{m213132}
{
 \circ & \bullet & \bullet \\
 \circ & \circ & \circ \\
 \bullet & \bullet & \circ \\
};
\end{smallarray} &
$(213,132)$ & &  
$Det_q$\\
\begin{smallarray}{m123132}
{
 \circ & \bullet & \bullet \\
 \bullet & \circ & \circ \\
 \bullet & \bullet & \circ \\
};
\end{smallarray} &
$(123,132)$ & &   
$Det_q$,\ \ $X_{11}$ \\
\begin{smallarray}{m123123}
{
 \circ & \bullet & \bullet \\
 \bullet & \circ & \bullet \\
 \bullet & \bullet & \circ \\
};
\end{smallarray} &
$(123,123)$ & & 
$Det_q$,\ \ $X_{11}$,\ \ $X_{22}$
\end{tabular}
\caption{Generators for the centres of the localizations $A_{\omega}$.}\label{fig:centres}
\end{figure}

By the Stratification Theorem, the centre of $A_{\omega}$ is always a Laurent polynomial ring.  This figure lists a set of generators for $Z(A_{\omega})$, reproduced from the results of \cite[\S5]{GL1}.  The other 24 cases may be obtained by applying the appropriate combinations of $\tau$, $\rho$, and $S$ from Figure~\ref{fig:H_primes_nice}; implicitly, we ignore any extra factors of the central element $Det_q$ which might appear after applying $S$.

By Proposition~\ref{res:centres of gl and qgl localizations agree}, this figure also lists generating sets for the Poisson centres $PZ(B_{\omega})$, subject only to replacing $X_{ij}$ by $x_{ij}$, $[\wt{i}|\wt{j}]_q$ by $[\wt{i}|\wt{j}]$ and $Det_q$ by $Det$.

\chapter*{Reviews for ``The $q$-Division Ring, Quantum Matrices and Semi-classical Limits''}
\def\baselinestretch{1.0}\normalsize

\blurb{This thesis introduces and discusses the objects mentioned in the title. Some new results are proved.}{Christopher Tedd (Logician)}

\blurb{I was with you up to ``Index of Notation''.}{Dr. Andrew Taylor PhD}

\blurb{I'm quite bored. What's the point of an algebraic structure you can't eat or cuddle?}{Stumpy (A Mighty Stegosaur)}

\blurb{This piece of work had so much potential. Yet I am still finding myself utterly disappointed in the lack of Harry Potter citations. I expected better from someone at this level.}{Daisy Fields (Humanities Student)}

\blurb{such limit. \\very quantum. \\wow.}{Matthew Taylor}

\blurb{A+++, would read again.}{user5937}

\clearpage
\def\baselinestretch{1.5}\normalsize

\bibliographystyle{plain}
\bibliography{thesis_corrected}

\end{document}